\documentclass[oneside,12pt]{book}

\usepackage{amssymb}
\usepackage{amsthm}
\usepackage{amsmath}
\usepackage{anysize}
\usepackage{hyperref}
\usepackage{bbm}
\usepackage{enumitem}

\marginsize{2.6cm}{2.6cm}{2cm}{2cm}
\newtheorem*{theorem*}{Theorem}
\newtheorem{theorem}{Theorem}[chapter]

\newtheorem{claim}[theorem]{Claim}
\newtheorem{lemma}[theorem]{Lemma}
\newtheorem{corollary}[theorem]{Corollary}
\newtheorem{remark}[theorem]{Remark}
\newtheorem{proposition}[theorem]{Proposition}

\def\del{\partial}
\def\dbar{\bar\partial}
\def\ddbar{\del\dbar}
\DeclareMathOperator{\Ric}{Ric}

\def\o{\omega}

\numberwithin{equation}{chapter}

\title{\textbf{\LARGE{Geometric pluripotential theory on K\"ahler manifolds}}}
\author{\Large{Tam\'as Darvas}\vspace{0.2cm}}
\date{\normalsize{UNIVERSITY OF MARYLAND}}
\begin{document}
\maketitle 

\chapter*{\centering \begin{normalsize}Abstract\end{normalsize}}
\thispagestyle{empty}
\begin{quotation}
\noindent 
Finite energy pluripotential theory accommodates the variational theory of equations of complex Monge--Amp\`ere type arising in K\"ahler geometry. Recently it has been discovered that many of the potential spaces involved have a rich metric geometry, effectively turning the variational problems in question into problems of infinite dimensional convex optimization, yielding existence results for solutions of the underlying  complex Monge--Amp\`ere equations. The purpose of this survey is to describe these developments from basic principles. 
\end{quotation}
\clearpage

\tableofcontents 
\setcounter{page}{1}

\chapter*{Preface}
\addcontentsline{toc}{chapter}{Preface}  
A circle of problems, going back to Calabi \cite{clb}, asks to find K\"ahler metrics with special curvature properties on a compact K\"ahler manifold $(X,\o)$. Of special interest are the \emph{K\"ahler--Einstein} (KE) metrics $\tilde \o$, that are cohomologous to $\o$, and  whose Ricci curvature is proportional to the metric tensor, i.e., 
$$\Ric \tilde \o = \lambda \tilde \o.$$ 
Existence of such metrics on $(X,\o)$ is only possible under cohomological restrictions, in particular the first Chern class needs to be a scalar multiple of the K\"ahler class $[\o]$: 
\begin{equation}\label{eq: coh_cond}\tag{$\star$}
c_1(X) = \lambda [\o].
\end{equation}
When $\lambda \leq 0$, by the work of Aubin and Yau it is always possible to find a unique KE metric on $(X,\o)$ \cite{A,Y}. The case of Fano manifolds, structures that satisfy \eqref{eq: coh_cond} with $\lambda >0$, is much more  intricate. In particular, the problem of finding KE metrics in this case is equivalent with solving the following global scalar equation of complex Monge--Amp\`ere type on $X$:
\begin{equation}\label{eq: intr_KEeq}\tag{$\star\star$}
(\o + i\ddbar u)^n = e^{-\lambda u + f_0}\o^n,
\end{equation}
where $f_0$ is a fixed smooth function on $X$.
The solution $u$ belongs to $\mathcal H_\o$, the set of smooth functions (potentials) that satisfy $\o + i\ddbar u >0$, which is an open subset of $C^\infty(X)$. As is well known, this  equation does not always admit a solution, and our desire is to characterize Fano manifolds that admit KE metrics.

Switching point of view, the KE problem has a very rich \emph{variational} theory as well. Indeed,  Mabuchi and Ding \cite{m, ding} introduced functionals 
$$\mathcal K: \mathcal H_\o \to \Bbb R \ \textup{ and } \ \mathcal F: \mathcal H_\o \to \Bbb R$$  whose minimizers are exactly the KE potentials, the solutions of \eqref{eq: intr_KEeq}  (see \eqref{eq: F_def} and \eqref{eq: Ken_def} for precise definition of these functionals). As a result, KE metrics exist if and only if the minimizer set of $\mathcal K$ (or $\mathcal F$) is non--empty. Along these lines we ask ourselves: what conditions guarantee existence/uniqueness of minimizers? Our source of inspiration will be the following elementary finite--dimensional result, which will allow to turn the variational approach into a problem of infinite--dimensional \emph{convex optimization}:

\begin{theorem*} Suppose $F:\Bbb R^n \to \Bbb R$ is a strictly convex functional. If $F$ has a minimizer it has to be unique. Regarding the existence, the following are equivalent:\vspace{0.1cm}\\
(i) $F$ has a (unique) minimizer $x_0 \in \Bbb R^n$.\\
(ii) $F$ is proper, i.e., there exists $C,D>0$ such that $F(x) \geq C|x| -D, \ x \in \Bbb R^n$.
\end{theorem*}

We give a sketch of the elementary proof. Uniqueness of minimizers is a consequence of strict convexity. That properness of $F$ implies existence of a minimizer follows from the fact that a bounded $F$--minimizing sequence subconverges to some $x_0 \in \Bbb R$, and that convex functions are continuous. That existence of a (unique) minimizer $x_0 \in \Bbb R$ implies properness of $F$ follows from the fact that the unit sphere $\Bbb S^{n-1}(x_0, 1)$ is compact, hence $C:=\inf_{\Bbb S^{n-1}(x_0, 1)}(F(x) - F(x_0))=F(y)-F(x_0)$ for some  $y \in \Bbb S^{n-1}(x_0, 1)$. Uniqueness of minimizers  implies that $C>0$. Using convexity of $F$, one concludes that $F(x) \geq C|x|-D$ for some $D>0$. 

The above simple argument already sheds light on what needs to be accomplished in our infinite dimensional setting to obtain an analogous result for existence/uniqueness of KE metrics. First, we need to understand the convexity of $\mathcal K$ and $\mathcal F$. Second, even in the above short argument we have used twice that $\Bbb R^n$ is complete. Consequently, an adequate metric structure needs to be chosen on $\mathcal H_\o$, and its completion needs to be understood. Third, (pre)compactness of spheres/balls in this new metric geometry needs to be explored.  

Regarding convexity, unfortunately $\mathcal K$ and $\mathcal F$ are not convex along the straight line segments of $\mathcal H_\o$. In order to address this, Mabuchi, Semmes and Donaldson independently introduced a non-positively curved  Riemannian $L^2$ type metric on $\mathcal H_\o$ that produces geodesics along which $\mathcal K$ and $\mathcal F$ are indeed convex \cite{m,s,do}. Inspired from this, a careful analysis of infinite dimensional spaces led Bando--Mabuchi to prove uniqueness of KE metrics \cite{bm}, and later Berman--Berndtsson to discover even more general uniqueness results \cite{bb}. 

On the other hand, there is strong evidence to suggest that the $L^2$ geometry of Mabuchi--Semmes--Donaldson alluded to above does not have the right compactness properties to allow for a characterization of existence of KE metrics. In order to address this, one needs to introduce more general $L^p$ type  Finsler metrics on $\mathcal H_\o$ and compute the metric completion of the related path length metric spaces $(\mathcal H_\o,d_p)$ \cite{da1,da2}. 

After sufficient metric theory is developed, it is apparent that the $L^1$ geometry of $\mathcal H_\o$ will be the one we should focus on. Indeed, sublevel sets of $\mathcal K$ restricted to spheres/balls are $d_1$--precompact, allowing to establish an equivalence between existence of KE metrics and $d_1$--properness of $\mathcal K$ and $\mathcal F$. Lastly, $d_1$--properness can be expressed using simple analytic means. This allowed the author and Y.A. Rubinstein to verify numerous related conjectures of Tian \cite{dr2} going back to the nineties \cite{t1,t3}.

\vspace{-0.3cm}\paragraph{Structure of the survey.} The aim of this work is to give a self contained introduction to special K\"ahler metrics using pluripotential theory/infinite--dimensional geometry. 

In Chapter 1, we give a very brief introduction to Orlicz spaces that are generalizations of the classical $L^p$ spaces. Our treatment will be rather minimalistic and we refer to \cite{rr} for a complete treatment.

In Chapter 2, we develop some background in finite energy pluripotential theory, necessary for later developments, closely following the original treatises of Guedj--Zeriahi and collaborators \cite{gz05,gz,BBGZ,BBEGZ}, that were inspired by work of Cegrell \cite{Ce} in the local case. We refer to these works for a comprehensive treatment, as well as the recent excellent textbook \cite{gzbook}. 

Chapter 3 contains the main technical machinery presented in this work. Here we introduce the $L^p$ Finsler geometry of the space of K\"ahler potentials $\mathcal H_\o$, and compute the metric completion of this space with respect to the corresponding path length metrics $d_p$. The $d_p$--completions of $\mathcal H_\o$ will be identified with $\mathcal E_p(X,\o)$, the finite energy spaces of Guedj--Zeriahi described in the previous chapter (Theorem \ref{thm: EpComplete}). In particular, we can endow these spaces with a rich \emph{metric geometry}, inspiring the title of this work. 

In Chapter 4 we discuss applications to existence/uniqueness of KE metrics on Fano manifolds. First we describe an abstract  properness/existence principle (Theorem \ref{thm: ExistencePrinc}) that adapts the above finite--dimensional Theorem to our infinite--dimensional setting. As we verify the assumptions of this principle, we will present self contained proofs of the Bando--Mabuchi uniqueness theorem (Theorem \ref{thm: BMuniqueness}) \cite{bm} and the Matsushima theorem about reductivity of the automorphism group of a KE manifold (Proposition \ref{prop: Matsushima_thm}) \cite{mat}. After this, in Theorem \ref{thm: F_func_properness} and Theorem \ref{thm: K-energy_properness} we resolve different versions of Tian's conjectures \cite{t1,t3} characterizing existence of KE metrics in terms of energy properness, following \cite{dr2}.

\vspace{-0.3cm}\paragraph{Prerequisites.} An effort has been made to keep prerequisites at a minimum. However due to size constraints, such requirements on part of the reader are inevitable. We assume that our reader is familiar with the basics of Bedford--Taylor theory of the complex Monge--Amp\`ere operator. Mastery of \cite[Chapters I-III]{bl3} or \cite[Chapter I and Chapter III.1-3]{De} is more then sufficient, and for a thorough treatment we highly recommend the recent textbook \cite{gzbook}. We also assume that our reader is familiar with the basics of K\"ahler geometry, though we devote a section in the appendix to introduce our terminology, and recall some of the essentials. For a comprehensive introduction into K\"ahler geometry we refer the reader to the recent textbook \cite{sze}, as well as \cite{De,we}.  

Though our main focus is the pluripotential theoretic point of view, some results in this survey rest on important regularity theorems regarding equations of complex Monge--Amp\`ere type. Due to space constraints we cannot present a  detailed proof of these theorems, but we will isolate their statements and keep them at a minimum, while providing precise references at all times.
 
\vspace{-0.3cm}\paragraph{Relation to other works.} As stressed above, our focus in this work is on self--contained treatment of the chosen topics. On the down side, we could not devote enough space to the vast historical developments of the subject, and for such a treatment we refer to the survey \cite{r2}, that discusses similar topics using a more chronological approach. 

Without a doubt the choice of topics represent our bias and limitations, and many important recent developments could not be surveyed. In particular, recent breakthroughs on K--stability (the work of Chen--Donaldson--Sun \cite{cds}, Tian \cite{t2}, Chen--Wang--Sun \cite{cws}, Berman--Boucksom--Jonsson \cite{bbj}) could not be presented, and we refer to \cite{do2,t5} for recent surveys on this topic. For results about the quantization of the geometry of the space of K\"ahler metrics, we refer to the original papers \cite{br2,CS12,do1,PS06,PS07,dlr}, as well as the survey \cite{PS09} along with references therein.
Geometric flows  could not be discussed either, and we refer to \cite{beg,BBEGZ,bdl1,dh, st1,st2} for work on the Ricci and Calabi flows that uses the theoretical machinery described in this survey.

The relation with constant scalar curvature K\"ahler (csck) metrics is also not elaborated. A preliminary version of this memoir appeared on the website of the author in the early months of 2017. Since then a number of important works have appeared building on the topics presented in this work: Chen--Cheng cracked the PDE theory of the csck equation \cite{cc2,cc3}, allowing to fully prove a converse of a theorem by Berman--Darvas--Lu \cite{bdl2}. Very recently He--Li pointed out that the contents of this survey generalize to Sasakian manifolds \cite{hl18}, paving the way to existence theorems for canonical metrics in that context as well.

\vspace{-0.3cm}\paragraph{Acknowledgments.} I thank  P. Gupta, W. He, L. Lempert, J. Li, C.H. Lu, Y.A. Rubinstein, K. Smith, V. Tosatti and the anonymous referees for their suggested corrections, careful remarks, and precisions. Also, I thank the students of MATH868D at the University of Maryland for their intriguing questions and relentless interest throughout the Fall of 2016. Chapter 3 is partly based on work done as a graduate student at Purdue University, and I am indebted to L. Lempert for encouragement and guidance. Chapter 4  surveys to some extent joint work with Y. Rubinstein, and I am grateful for his mentorship over the years. This research was partially supported by
NSF grant DMS-1610202 and BSF grant 2012236.

\chapter{A primer on Orlicz spaces}
Plainly speaking, Orlicz spaces are generalizations of $L^p$ Banach spaces. As we will see in our later study, we will prefer working with Orlicz norms over $L^p$ norms, since a careful choice of weight makes Orlicz norms smooth away from the origin. The same cannot be said about $L^p$ norms. We give here a brief and self-contained introduction, only touching on aspects that will be needed later. For a more thorough treatment we refer to \cite{rr}.

Suppose $(\Omega,\Sigma, \mu)$ is a measure space with $\mu(\Omega)=1$ and $(\chi,\chi^*)$ is a complementary pair of {Young weights}. This means that $\chi:\Bbb R \to \Bbb R^+ \cup \{ \infty \}$ is convex, even, lower semi-continuous (lsc)  and satisfies the normalizing conditions 
$$\chi(0)=0, \ \  1 \in \partial \chi(1).$$ 
Recall that $\partial \chi(l) \subset \Bbb R$ is the set of subgradients to $\chi$ at $l$, i.e., $v \in \partial \chi(l)$ if and only if $\chi(l) + v h \leq \chi(l+h), \ h \in \Bbb R$. As described, $\chi$ is simply a \emph{normalized Young weight}. The complement $\chi^*$ is the Legendre transform of $\chi$:
\begin{equation}\label{eq: Legendre_trans_def}
\chi^*(h) = \sup_{l \in \Bbb R} (lh - \chi(l)).
\end{equation}
Using convexity of $\chi$ and the above identity, one can verify that $\chi^*$ is also a normalized Young weight. Additionally $(\chi,\chi^*)$ satisfies the \emph{Young identity} and \emph{inequality}:
\begin{equation}\label{eq: YoungIdIneq}
\chi(a) + \chi^*(\chi'(a))=a\chi'(a), \ \chi(a) + \chi^*(b) \geq ab, \ a,b \in \Bbb R, \ \chi'(a) \in \partial \chi(a),
\end{equation}
in particular, due to our normalization: $\chi(1) + \chi^*(1)=1$. 

The most typical example to keep in mind is the pair $\chi_p(l)=|l|^p/p$ and $\chi^*_p(l)=|l|^q/q$, where $p,q > 1$ and $1/p + 1/q = 1$.
Let $L^\chi(\mu)$ be the following space of measurable functions:
$$L^\chi(\mu)=\Big\{ f:\Omega \to \Bbb R \cup \{ \infty,-\infty\}: \ \exists r >0 \textup{ s.t. } \int_\Omega \chi(rf) d\mu < \infty\Big\}.$$
One can introduce the following  norm on $L^\chi(\mu)$:
\begin{equation}\label{eq: OrliczNormDef}
\|f\|_{\chi,\mu}=\inf \Big\{ r > 0 : \ \int_\Omega \chi\Big(\frac{f}{r}\Big) d\mu \leq \chi(1) \Big\}.
\end{equation}
The set $\{f \in L^\chi(\mu): \int_\Omega \chi(f) d\mu \leq \chi(1)\}$ is convex and symmetric in $L^\chi(\mu)$, hence $\|\cdot\|_{\chi,\mu}$ is nothing but the Minkowski seminorm of this set. This is the content of the following lemma:

\begin{lemma} Suppose $f,g \in L^\chi(\mu)$. Then $\| f+ g\|_{\chi,\mu} \leq \| f\|_{\chi,\mu} + \| g\|_{\chi,\mu}.$
\end{lemma}

\begin{proof}Suppose $\int_\Omega \chi(f/r_1)d\mu\leq \chi(1), \ \int_\Omega \chi(g/r_2)d\mu \leq \chi(1)$ for some $r_1,r_2 > 0$. Convexity of $\chi$ implies that 
\begin{flalign}
\int_\Omega \chi\Big(\frac{f+g}{r_1 + r_2}\Big) d\mu \leq \frac{r_1}{r_1 + r_2}\int_\Omega \chi\Big(\frac{f}{r_1}\Big) d\mu  + \frac{r_2}{r_1 + r_2}\int_\Omega \chi\Big(\frac{g}{r_2}\Big) d\mu \leq \chi(1).
\end{flalign}
Hence $\|f+g \|_{\chi,\mu} \leq r_1 + r_2$, finishing the argument.
\end{proof}

\noindent Together with the previous one, the next lemma implies that $(L^\chi(\mu),\| \cdot \|_{\chi,\mu})$ is a normed space:

\begin{lemma} Suppose $f \in L^\chi(\mu)$. Then $\| f\|_{\chi,\mu} =0$ implies that $f=0$ a.e. with respect to $\mu$.
\end{lemma}

\begin{proof} As $1 \in \partial \chi(1)$ and $\chi \geq 0$, it follows that 
\begin{equation}
l \leq \chi(1) + l \leq \chi(1 + l)  \textup{ for all } l \geq 0.\label{eq: chi_ineq}
\end{equation}
As a consequence of this inequality, $L^\chi(\mu) \subset L^1(\mu)$. Also, since $\|f \|_{\chi,\mu}=0$, it follows that $\int_\Omega \chi(n f) d\mu \leq \chi(1)$ for all $n \in \Bbb N$. By \eqref{eq: chi_ineq} we can write:
$$ \int_{\{ n|f| >1 \}} \Big(|f| - \frac{1}{n}\Big)d\mu \leq \int_{\{ n|f| >1\}} \frac{\chi(n |f|)}{n} d\mu \leq \int_\Omega \frac{\chi(n f)}{n} d\mu \leq \frac{\chi(1)}{n}.$$
Applying the dominated convergence theorem to this inequality gives  $\int_{|f| > 0}|f|d\mu =0$, finishing the proof.
\end{proof}

Though we will not make use of it, one can also show that $(L^\chi(\mu),\| \cdot \|_{\chi,\mu})$ is complete, hence it is a Banach space (see \cite[Theorem 3.3.10]{rr}). The reason we work with a complementary pair of Young weights is because in this setting the H\"older inequality holds:
\begin{proposition}\label{prop: Holder} For $f \in L^\chi(\mu)$ and $g \in L^{\chi^*}(\mu)$ we have
\begin{equation}\label{eq: HolderIneq}
\int_\Omega fg d\mu \leq \|f\|_{\chi,\mu} \|g\|_{\chi^*,\mu}, \ f \in L^\chi(\mu), \ g \in L^{\chi^*}(\mu).
\end{equation}
\end{proposition}
\begin{proof}  Let $r_1 > \|f\|_{\chi,\mu}$ and $r_2 > \|g\|_{\chi^*,\mu}$.  Using both the Young inequality and the identity \eqref{eq: YoungIdIneq}, \eqref{eq: HolderIneq} follows in the following manner:
$$\int_\Omega \frac{fg}{r_1r_2} d\mu \leq \int_\Omega \chi\Big(\frac{f}{r_1}\Big) d\mu  + \int_\Omega \chi^*\Big(\frac{g}{r_2}\Big) d\mu \leq \chi(1) + \chi^*(1)=1. \vspace{-0.2in}$$
\end{proof}

\noindent Orlicz spaces can be quite general and in our study we will be interested in spaces whose normalized Young weight is finite and  satisfies the  growth estimate
\begin{equation}\label{eq: GrowthEst}
l\chi'(l) \leq p\chi(l), \ l >0,
\end{equation}
for some $p \geq 1$. To clarify, $\chi'(l)$ is just an arbitrary subgradient of $\chi$ at $l$. For such weights we write $\chi \in \mathcal W^+_p$, following the notation of \cite{gz}. As it turns out, weights $\chi$ that satisfy \eqref{eq: GrowthEst} can be thought of as distant cousins  of the homogeneous $L^p$ weight $|l|^p/p \in \mathcal W^+_p$:
\begin{proposition} For $\chi \in \mathcal W^+_p, \ p \geq 1$ and $0 < \varepsilon < 1$ we have
\begin{equation}\label{eq: GrowthControl}
\varepsilon^p \chi(l) \leq \chi(\varepsilon l) \leq \varepsilon \chi(l), \ l > 0.
\end{equation}
\end{proposition}
\begin{proof}
The second estimate follows from convexity of $\chi$. For the first estimate we notice that for any $\delta >0$ the weight $\chi_\delta(l) := \chi(l) + \delta|l|$ also satisfies \eqref{eq: GrowthEst}.  As $\chi_{\delta}(h) >0$ for $h >0$, we can integrate $\chi'_\delta(h)/\chi_\delta(h) \leq p/h$ from $\varepsilon l$ to $l$ to obtain:
$$\varepsilon^p \chi_\delta(l) \leq \chi_\delta(\varepsilon l).$$
Letting $\delta \to 0$ the desired estimate follows.
\end{proof}

Estimate \eqref{eq: GrowthControl} immediately implies that for $f \in L^\chi(\mu)$ the function $l \to \int_\Omega \chi(l f)d\mu$ is continuous, hence we have
\begin{equation}\label{eq: OrliczNormId}
\|f\|_{\chi,\mu}=\alpha>0 \textup{ if and only if } \int_\Omega \chi\Big( \frac{f}{\alpha}\Big)d\mu=\chi(1).
\end{equation}

To simplify future notation, we introduce the increasing functions 
$$M_p(l) = \max\{l,l^p\},  \ \ m_p(l) = \min\{l,l^p\},$$ for $p >0, l \geq 0$. Observe that $M_p \circ m_{1/p}(l) = m_p \circ M_{1/p}(l)=l$. As a consequence of \eqref{eq: GrowthControl} and \eqref{eq: OrliczNormId}, we immediately obtain the following  estimates, characterizing Orlicz norm convergence:
\begin{proposition} \label{prop: NormIntegralEst} If $\chi \in \mathcal W_p^+$  and $f \in L^\chi(\mu)$ then
\begin{equation}
 m_p(\|f\|_{\chi,\mu}) \leq \frac{\int_\Omega \chi(f)d\mu}{\chi(1)} \leq M_p(\|f\|_{\chi,\mu}).
\end{equation}
\begin{equation}
 m_{1/p}\Big(\frac{\int_\Omega \chi(f)d\mu}{\chi(1)}\Big) \leq \|f\|_{\chi,\mu} \leq M_{1/p}\Big(\frac{\int_\Omega \chi(f)d\mu}{\chi(1)}\Big).
\end{equation}
As a result, for a sequence $\{ f_j\}_{j \in \Bbb N}$ we have $\|f_j\|_{\chi,\mu} \to 0$ if and only if $\int_\Omega \chi(f_j)d\mu \to 0$. Also, $\|f_j\|_{\chi,\mu} \to N$ for $N > 0$ if and only if $\int_\Omega \chi(f_j/N)d\mu \to \chi(1)$.
\end{proposition}

Later in this survey we will need to approximate certain Orlicz norms with Orlicz norms having \emph{smooth} $\mathcal W^+_p$-weights. The following two approximation results, which are by no means optimal, will be useful in our treatment:

\begin{proposition} \label{prop: approx_lemma}Suppose $\chi \in \mathcal W^+_p$ and $\{\chi_k \}_{k \in \Bbb N}$ is a sequence of normalized Young weights that converges uniformly on compacts to $\chi.$ Let $f$ be a bounded $\mu$--measurable function on $X$. Then $f \in L^\chi(\mu),L^{\chi_k}(\mu), \ k \in \Bbb N$ and we have that
$$\lim_{k \to +\infty}\| f\|_{\chi_k,\mu} = \| f\|_{\chi,\mu}.$$
\end{proposition}
\begin{proof} Suppose $N=\| f\|_{\chi,\mu}$. If $N=0$, then $f=0$ a.e. with respect to $\mu$ implying that $\| f\|_{\chi_k,\mu} = \| f\|_{\chi,\mu}=0$. So we assume that $N >0$. As $\chi \in \mathcal W^+_p$, by \eqref{eq: OrliczNormId}, for any $\varepsilon >0$ there exists $\delta >0$ such that
$$\int_X \chi\Big(\frac{f}{(1+\varepsilon)N}\Big) d\mu<\chi(1)-2\delta< \chi(1) <\chi(1)+2\delta<\int_X \chi\Big(\frac{f}{(1-\varepsilon)N}\Big) d\mu.$$
As $\chi_k$ tends uniformly on compacts to $\chi$, $f$ is bounded and $\mu(X)=1$, it follows from the dominated convergence theorem that for $k$ big enough we have
$$\int_X \chi_k\Big(\frac{f}{(1+\varepsilon)N}\Big) d\mu<\chi_k(1)-\delta< \chi_k(1) <\chi_k(1)+\delta<\int_X \chi_k\Big(\frac{f}{(1-\varepsilon)N}\Big) d\mu.$$
This implies that $(1-\varepsilon)N \leq  \| f\|_{\chi_k,\mu} \leq (1+\varepsilon)N$, from which the conclusion follows.
\end{proof}
\begin{proposition} \label{prop: approx_lemma2}Given $\chi \in \mathcal W^+_p$, there exists $\chi_k \in \mathcal W^+_{p_k} \cap C^\infty(\Bbb R),$ $k \in \Bbb N$, with $\{p_k\}_k$ possibly unbounded, such that $\chi_k \to \chi$ uniformly on compacts.
\end{proposition}
\begin{proof} There is great freedom in constructing the sequence $\chi_k$. First we smoothen and normalize $\chi$, introducing the sequence $\tilde \chi_k$ in the process:
$$\tilde \chi_k(l) = (\delta_{k} \star \chi)(h_k l) - (\delta_{k} \star \chi)(0), \ l \in \Bbb R, k \in \Bbb N^*.$$
Here $\delta$ is a typical choice of bump function that is smooth, even, has support in $(-1,1)$, and $\int_\Bbb R \delta(t)dt =1$. Accordingly, $\delta_k(\cdot):=k \delta(k(\cdot))$, and $h_k >0$ is chosen in such a way that $\tilde \chi_k$ becomes normalized ($\tilde \chi_k(0)=0$ and $\tilde \chi_k'(1)=1$). As $\chi$ is normalized it follows that $\tilde \chi'_k(1 - \frac{1}{k}) \leq 1 \leq \tilde \chi'_k(1 + \frac{1}{k})$. In particular,
$1-1/k \leq h_k \leq 1 +1/k$, hence the $\tilde \chi_k$ are normalized Young weights that converge to $\chi$ uniformly on compacts.

As $\chi$ is strictly increasing on $\Bbb R^+$, so is each $\tilde \chi_k$, but
our construction does not seem to guarantee that $\tilde \chi_k \in \mathcal W^+_{p_k}$ for some $p_k \geq 1$. This can be fixed by changing smoothly the values of $\tilde \chi_k$ on the sets $|x| < 1/k$ and $|x| > k$ in such a manner that the altered weights $\chi_k$ satisfy the estimate
$$l \chi'_k(l) \leq p_k \chi_k(l), \ l >0,$$
for some $p_k  \geq 1$. Since $\tilde \chi_k'(t)>0, \ t >0$ and  $\tilde \chi'_k(0)=0$, this can be done using elementary calculus.
The sequence $\{\chi_k\}_k$ thus obtained satisfies the required properties.
\end{proof}

\chapter{Finite energy pluripotential theory on K\"ahler manifolds}

In this section we introduce the basics of finite energy pluripotential theory on K\"ahler manifolds. Our treatment is by no means comprehensive and we refer the readers to \cite[Chapter 8-10]{gzbook} for a more elaborate study. Our short introduction closely follows \cite{gz} and \cite{da1}. More precisely, Sections 2.1--2.3 are based on \cite{gz}, and Section 2.4 is based on \cite{da1}.

The \emph{plurifine} topology of an open set $U \subset \Bbb C^n$ is the coarsest topology making all plurisubharmonic (psh) functions on $U$ continuous. Clearly, the resulting topology on $U$ is finer then the Euclidean one. The motivating result behind this notion is that of Bedford--Taylor \cite{BT1}, according to which the complex Monge--Amp\`ere operator is local with respect to the plurifine topology, i.e., if $\phi^j,\psi^j \in \textup{PSH}\cap L^\infty_{\textup{loc}}(U)$ and $V \subset U$ is plurifine open such that $\phi^j|_V=\psi^j|_V$, then 
\begin{equation}\label{eq: GZ_MP_local}
\mathbbm{1}_{V}i\ddbar \phi^1\wedge \ldots \wedge i\ddbar \phi^n=\mathbbm{1}_{V}i\ddbar \psi^1 \wedge \ldots \wedge i \ddbar \psi^n.
\end{equation}
When dealing with a $n$ dimensional  K\"ahler manifold $(X,\o)$, the largest class of potentials one can consider is that of $\o$-\emph{psh}  functions:
$$\textup{PSH}(X,\o)=\{ u \textup{ is quasi-plurisubharmonic on }X \textup{, and } \o_u := \o + i\ddbar u \geq 0 \textup{ as currents} \}.$$
In concrete terms, $u \in \textup{PSH}(X,\o)$ means that $g + u|_U \in \textup{PSH}(U)$, for all open sets $U \subset X$ on which the metric $\o$ can be written as $\o = i\ddbar g, \ g \in C^\infty(U)$. We immediately get that $u$ is usc and $u \in L^1(X)$.

Given the local nature of \eqref{eq: GZ_MP_local}, it generalizes to K\"ahler manifolds in a straightforward manner, and in particular we have the following identity for $u^j,v \in \textup{PSH}(X,\o) \cap L^\infty$ that we will use numerous times below: 
\begin{equation}\label{eq: GZ_MP}
\mathbbm{1}_{\cap_j \{u^j>v\}}\o_{\max(u^1,v)} \wedge \ldots \wedge \o_{\max(u^n,v)}=\mathbbm{1}_{\cap_j \{u^j>v\}}\o_{u^1} \wedge \ldots \wedge \o_{u^n}.
\end{equation}

\section{Full mass $\o$-psh functions}

One of the cornerstones of Bedford--Taylor theory is associating a complex Monge--Amp\`ere measure to bounded psh functions. Their construction adapts  to elements of $\textup{PSH}(X,\o)\cap L^\infty$ in a straightforward manner. 

As it turns out, it is possible to generalize the Bedford--Taylor construction to all elements of $\textup{PSH}(X,\o)$, as we describe now. First we consider the case of products. Let $v^j \in \textup{PSH}(X,\o), \ j \in \{1,\ldots,n\}$ and by  $v^j_h = \max (v^j,-h) \in \textup{PSH}(X,\o) \cap L^\infty, \ h \in \Bbb R$ we denote the \emph{canonical cutoffs} of $v^j$. If $h_1 < h_2$ then \eqref{eq: GZ_MP} implies that
$$\mathbbm{1}_{\cap_j \{v^j>-h_1\}}\o_{v^1_{h_1}}\wedge \ldots \wedge \o_{v^n_{h_1}}=\mathbbm{1}_{\cap_j \{v^j>-h_1\}}\o_{v^1_{h_2}} \wedge \ldots \wedge \o_{v^n_{h_2}} \leq \mathbbm{1}_{\cap_j \{v^j>-h_2\}}\o_{v^1_{h_2}} \wedge \ldots \wedge \o_{v^n_{h_2}},
$$
hence $\{ \mathbbm{1}_{\cap_j \{v^j>-h\}}\o_{v^1_{h}} \wedge \ldots \wedge \o_{v^n_{h}}\}_h$ is an increasing sequence of Borel measures on $X$. This leads to the notion of the \emph{non-pluripolar product} $\o_{v^1} \wedge \ldots \wedge \o_{v^n}$, that generalizes the corresponding Bedford--Taylor construction:
\begin{equation}\label{eq: CMA_general_def}
\o_{v^1} \wedge \ldots \wedge \o_{v^n} := \lim_{h \to \infty} \mathbbm{1}_{\cap_j \{v^j>-h\}}\o_{v^1_{h}} \wedge \ldots \wedge \o_{v^n_{h}}.
\end{equation}
When $v= v^1 = \ldots =v^n$, the above construction yields $\omega^n_u:= \omega_u \wedge \ldots \wedge \omega_u$, the \emph{non-pluripolar complex Monge-Ampere measure} of $v$:
$$\omega_v^n:= \lim_{h \to \infty} \mathbbm{1}_{\{v >-h\}} \omega_{v_h}^n.$$
 By the above definition $\int_B \o_v^n = \lim_{h \to \infty}  \int_B \mathbbm{1}_{\{v>-h\}} \o_{v_h}^n$ for all Borel sets $B \subset X$. This is stronger than simply saying that $\o_v^n$ is the weak limit of $\mathbbm{1}_{\{v>-h\}} \o_{v_h}^n$. 

Before we examine this construction more closely, let us recall a few facts about smooth elements of $\textup{PSH}(X,\o)$. Of great importance in this work is the set of smooth \emph{K\"ahler potentials}:  
\begin{equation}\label{eq: H_o_def}
\mathcal H_\o = \{u \in C^\infty(X), \ \o + i\ddbar u >0 \}.
\end{equation}
Clearly, $\mathcal H_\o \subset \textup{PSH}(X,\o)$ and in fact any element of $\textup{PSH}(X,\o)$ can be approximated by a decreasing sequence in $\mathcal H_\o$:

\begin{theorem}[\cite{bk}] \label{thm: BK_approx} Given $u \in \textup{PSH}(X,\o)$, there exists a decreasing sequence $\{u_k\}_k \subset \mathcal H_\o$ such that $u_k \searrow u$. 
\end{theorem}

We give a proof of this result in Appendix A.2. Among other things, this theorem implies that the total volume of $X$ is the same for all currents $\o_u$ with  bounded potential:
\begin{lemma} \label{lem: const_volume_lem} if $v \in \textup{PSH}(X,\o) \cap L^\infty$, then $\int_X \o_v^n = \int_X \o^n =: \textup{Vol}(X)$.
\end{lemma}
\begin{proof} By an application of Stokes theorem, the statement holds for $v \in \mathcal H_\o$. The general result follows after we approximate $v \in \textup{PSH}(X,\o) \cap L^\infty$ with a decreasing sequence $v_k \in \mathcal H_\o$ (Theorem \ref{thm: BK_approx}), and we use Bedford--Taylor theory (\cite[Theorem 2.2.5]{bl3}) to conclude that $\o_{v_k}^n \to \o_v^n$ weakly. 
\end{proof}

In contrast with the above, given our definition \eqref{eq: CMA_general_def}, it is clear that we only have $\int_X\o_v^n \leq \textup{Vol}(X)$ for $v \in \textup{PSH}(X,\o)$. This leads to the natural definition of $\o$-psh functions of \emph{full mass}:
$$\mathcal E(X,\o) := \{v \in \textup{PSH}(X,\o) \textup{ s.t. } \int_X \o_v^n = \textup{Vol}(X) \}.$$

Though $\textup{PSH}(X,\o) \cap L^\infty \subsetneq \mathcal E(X,\o)$ (see the examples of Section 2.3), many of the properties that hold for bounded $\o$-psh functions, still hold for (unbounded) elements of $\mathcal E(X,\o)$ as well, like the \emph{comparison principle}:
\begin{proposition}\textup{\cite[Theorem 1.5]{gz}}\label{prop: comp_princ_E} Suppose $u,v \in \mathcal E(X,\o)$. Then
\begin{equation}\label{eq: comp_princ_E}
\int_{\{v < u\}}\o_u^n \leq \int_{\{v < u\}}\o_v^n. 
\end{equation} 
\end{proposition}
\begin{proof} First we show \eqref{eq: comp_princ_E} for $u,v$ bounded. Using \eqref{eq: GZ_MP} we can write:
\begin{flalign*}
\int_{\{v < u\}}\o_u^n =& \int_{\{v < u\}}\o_{\max(u,v)}^n=\int_X \o_{\max(u,v)}^n - \int_{\{ u \leq v\}} \o_{\max(u,v)}^n \\
\leq& \int_X \o_{\max(u,v)}^n - \int_{\{ u < v\}} \o_{\max(u,v)}^n.
\end{flalign*}
Using Lemma \ref{lem: const_volume_lem} and \eqref{eq: GZ_MP} we can continue to write:
$$\int_{\{v<u\}}\o_u^n  \leq \int_X \o_{v}^n - \int_{\{ u < v\}} \o_{v}^n = \int_{\{v\leq u\}}\o_v^n.$$
Noticing that $\{ v < u\} = \cup_{\varepsilon > 0} \{v + \varepsilon < u\}=\cup_{\varepsilon > 0} \{v + \varepsilon \leq  u\}$, we can replace $v$ with $v + \varepsilon$ in the above inequality and let $\varepsilon \to 0$ to obtain \eqref{eq: comp_princ_E} for bounded potentials.

In general, let $u_l = \max(u,-l), v_k=\max(u,-k), \ l,k \in \Bbb N$ be the canonical cutoffs of $u,v$. For these, by the above we have
$\int_{\{v_l < u_k\}} \o_{u_k}^n \leq \int_{\{v_l < u_k\}} \o_{v_l}^n$. This inequality together with the inclusions $\{ v_l < u \} \subset \{v_l < u_k \} \subset \{v < u_k \}$ gives 
$$\int_{\{v_l < u\}} \o_{u_k}^n \leq \int_{\{v < u_k\}} \o_{v_l}^n.$$
We can let $l \to \infty$, and using the definition of $\o_{v_l}^n$ from \eqref{eq: CMA_general_def} it follows that $\int_{\{v < u\}} \o_{u_k}^n \leq \int_{\{v < u_k\}} \o_{v}^n.$ Now letting $k \to \infty$ and using again \eqref{eq: CMA_general_def} we obtain
$$\int_{\{v < u\}} \o_{u}^n \leq \int_{\{v <  u\}} \o_{v}^n.$$
\end{proof}

Of great importance in what follows will be the \emph{monotonicity property} of $\mathcal E(X,\o)$: 
\begin{proposition}\textup{\cite[Proposition 1.6]{gz}} \label{prop: monotonicity_E} Suppose $u \in \mathcal E(X,\o)$ and $v \in \textup{PSH}(X,\o)$. If $u \leq v$ then $v \in \mathcal E(X,\o)$.
\end{proposition}
\begin{proof} Set $\psi := v/2$. We can assume without loss of generality that $u \leq v < -2$, hence $\psi < -1$.  This choice of normalization allows to write the following inclusions for the canonical cutoffs $u_j,v_j,\psi_j$:
$$\{\psi \leq -j  \}=\{\psi_j \leq -j  \} \subset \{u_{2j} < \psi_j -j +1\} \subset \{u_{2j} \leq -j \}.$$
Using these inclusions and the comparison principle (Proposition \ref{prop: comp_princ_E}) we conclude:
$$\int_{\{\psi_j \leq -j\}}\o_{\psi_j}^n \leq \int_{\{u_{2j} < \psi_j -j +1\}}\o_{\psi_j}^n \leq \int_{\{u_{2j} < \psi_j -j +1\}}\o_{u_{2j}}^n \leq \int_{\{u_{2j} \leq -j\}}\o_{u_{2j}}^n.$$
Using Lemma \ref{lem: const_volume_lem} and \eqref{eq: GZ_MP} we get $\int_{\{u_{2j} \leq -j\}}\o_{u_{2j}}^n = \int_{\{u_{j} \leq -j\}}\o_{u_{j}}^n=\int_{\{u \leq -j\}}\o_{u_{j}}^n$, hence $\int_{\{\psi_j \leq -j\}}\o_{\psi_j}^n \to 0$ as $j \to \infty$, implying that $\psi=v/2 \in \mathcal E(X,\o)$.

Now we argue that in fact $v \in \mathcal E(X,\o)$. For the canonical cutoffs we have the identity $v_{2j}/2 = \psi_j, \ j \in \Bbb N$, hence we get the estimate $\o_{\psi_j} = \o/2 + \o_{v_{2j}}/2 \geq \o_{v_{2j}}/2$. This allows to conclude that
$$\int_{\{v \leq -2j\}}\o_{v_{2j}}^n \leq 2^n \int_{\{v \leq -2j\}}\o_{\psi_{j}}^n = 2^n \int_{\{\psi \leq -j\}}\o_{\psi_{j}}^n,$$
hence the left hand integral converges to zero as $j \to \infty$, since $\psi \in \mathcal E(X,\omega)$. Using \eqref{eq: CMA_general_def}, this implies that $\int_X \o_v^n = \textup{Vol}(X)$, i.e., $v \in \mathcal E(X,\o)$ as desired. 
\end{proof}

Lastly we note that formula \eqref{eq: GZ_MP} generalizes to full mass $\omega$-psh functions as well:
\begin{lemma}\label{lem: MP_forE} For $u,v \in \mathcal E(X,\o)$ we have $ \mathbbm{1}_{\{u>v\}}\o_{\max(u,v)}^n=\mathbbm{1}_{\{u>v\}}\o_u^n$.
\end{lemma}
\begin{proof}Proposition \ref{prop: monotonicity_E} implies that $\alpha:= \max(u,v) \in \mathcal E(X,\o)$. For the canonical cutoffs we observe that $\max(u_j,v_{j+1}) = \max(u,v, -j) = \alpha_j$.
As the cutoffs are bounded we have 
\begin{equation} \label{eq: MP_intermediate}
\mathbbm{1}_{\{u_j>v_{j+1}\}}\o_{\alpha_j}^n=\mathbbm{1}_{\{u_j>v_{j+1}\}}\o_{u_j}^n.
\end{equation}
The definition of $\o_{u}^n$ and $\o_\alpha^n$ \eqref{eq: CMA_general_def} implies that $\mathbbm{1}_{\{u>v\}}\o_{u_j}^n \to \mathbbm{1}_{\{u>v\}}\o_{u}^n$ and $\mathbbm{1}_{\{u>v\}}\o_{\alpha_j}^n \to \mathbbm{1}_{\{u>v\}}\o_{\alpha}^n$. Since $\{u > v \} \subset \{u_j > v_{j+1} \}$ 
and $ \{u_j > v_{j+1}\} \setminus \{u > v \} \subset \{ u \leq  -j\}$, it follows that 
$$0 \leq (\mathbbm{1}_{\{u_j>v_{j+1}\}}-\mathbbm{1}_{\{u>v\}})\o_{u_j}^n \leq \mathbbm{1}_{\{u \leq -j\}}\o_{u_j}^n \to 0.$$
Since $ \{u_j > v_{j+1}\} \setminus \{u > v \} \subset \{\max(u,v) \leq  -j\}$ we also obtain that $0 \leq (\mathbbm{1}_{\{u>v\}}-\mathbbm{1}_{\{u_j>v_{j+1}\}})\o_{\alpha_j}^n \leq \mathbbm{1}_{\{\alpha \leq -j\}}\o_{\alpha_j}^n \to 0.$
Combining these last facts with \eqref{eq: MP_intermediate}, and taking the limit, we arrive at the desired result.
\end{proof}




\section{Finite energy classes}

By considering weights $\chi \in \mathcal W^+_p$, one can introduce various finite energy subclasses of $\mathcal E(X,\o)$, important in our later geometric study:  
\begin{equation}
\label{eq: E_chi_def}
\mathcal E_\chi(X,\o) := \{u \in \mathcal E(X,\o) \textup{ s.t. } E_\chi(u) < \infty \},
\end{equation}
where $E_\chi$ is the $\chi$\emph{--energy} defined by the expression:
$$E_\chi(u) := \int_X \chi(u) \o_u^n.$$ 
Recall from Chapter 1 that the condition $E_\chi(u) < \infty$ is equivalent to $u \in L^\chi(\o_u^n)$. Of special importance are the weights $\chi_p(t):=|t|^p/p$ and the associated finite energy classes 
\begin{equation}\label{eq: E_p_def}
\mathcal E_p(X,\o):= \mathcal E_{\chi_p}(X,\o).
\end{equation}

By the next result, to test membership in $\mathcal E_\chi(X,\omega)$ it is enough to test the finiteness condition $E_\chi(u) < \infty$ on the canonical cutoffs: 

\begin{proposition}\textup{\cite[Proposition 1.4]{gz}}\label{prop: cutoff_aprox} Suppose $u \in \mathcal E(X,\o)$ with canonical cutoffs $\{u_k\}_{k \in \Bbb N}$. If $h: \Bbb R^+ \to \Bbb R^+$ is continuous and increasing 
 then
$$\int_X h(|u|) \o_u^n< \infty \textup{ if and only if } \limsup_{k}\int_X h(|u_k|)\o_{u_k}^{n}<\infty.$$
Additionally, if the above conditions hold, then $\int_X h(|u|) \o_u^n= \lim_{k \to \infty}\int_X h(|u_k|)\o_{u_k}^{n}.$
\end{proposition}
\begin{proof} If $\limsup_{k}\int_X h(|u_k|)\o_{u_k}^{n}<C$ then the family of measures $\{h(\max(0,|u_k|))\o_{u_k}^n\}_k$ is precompact, hence one can extract a weakly converging subsequence of measures $\{h(\max(0,|u_{k_j}|))\o_{u_{k_j}}^n\}_{k_j} \to \mu$ with $\mu(X) < C$. By the definition of $\o_u^n$ \eqref{eq: CMA_general_def} it follows that  $\o_{u_{k_j}}^n \to \o_u^n$ weakly. Since $\{h(\max(0,|u_{k_j}|)) \}_{k_j}$ is a sequence of lsc functions increasing to $h(\max(0,|u|))$, a standard measure theoretic lemma implies that $ \int_X h(\max(0,|u|)) \o_u^n \leq \mu(X)$ (see \cite[Lemma A2.2]{bl3}). This implies that $ \int_X h(|u|) \o_u^n$ is finite.

Now assume that $\int_X h(|u|)\o_u^n \leq C$. By \eqref{eq: GZ_MP} and the definition of $\o_u^n$ it follows that $\mathbbm{1}_{\{u>-k\}}\o_{u}^n=\mathbbm{1}_{\{u>-k\}}\o_{u_k}^n$, implying that $\int_{\{u \leq -k\}}\o_{u_k}^n=\int_{\{u \leq -k\}}\o_{u}^n$. As a consequence, we can write the following:
\begin{flalign*}
\Big|\int_X h (|u_k|)\o_{u_k}^n -  \int_X h (|u|)\o_{u}^n\Big| &\leq \int_{\{u \leq  -k\}} h(k) \o_{u_k}^n + \int_{\{u \leq  -k\}} h(|u|) \o_{u}^n\\
&=  h(k)\int_{\{u \leq  -k\}} \o_{u}^n + \int_{\{u \leq  -k\}} h(|u|) \o_{u}^n\\
&\leq  2\int_{\{u \leq  -k\}} h(|u|) \o_{u}^n.
\end{flalign*}
Since $\int_{\{u=-\infty \}} h(|u|)\o_u^n=0$, it follows that $\int_X h (|u_k|)\o_{u_k}^n$ is bounded above and moreover $\int_X h (|u_k|)\o_{u_k}^n \to \int_X h (|u|)\o_{u}^n$.
\end{proof}
With the aid of the previous proposition, we can prove our next result, sometimes called the \emph{``fundamental estimate"}:

\begin{proposition} \textup{\cite[Lemma 3.5]{gz}}\label{prop: Energy_est} Suppose $\chi\in \mathcal W^+_p$ and $u,v \in \mathcal E_\chi(X,\o)$ satisfies $u\leq v \leq 0$. Then 
\begin{equation}\label{eq: Fund_Est}
E_\chi(v) \leq  (p+1)^n E_\chi(u).
\end{equation} 
\end{proposition}
\begin{proof} First assume that $u,v \in \mathcal H_\omega$. We first show the following estimate:
\begin{equation}\label{eq: Fund_Est_intermediate}
\int_{X} \chi(u) \o_v^{j+1} \wedge \o_u^{n-j-1} \leq (p+1)\int_{X} \chi(u) \o_v^{j} \wedge \o_u^{n-j}.
\end{equation}
To show this estimate, integration by parts gives 
\begin{equation}\label{eq: integ_parts}
\int_X \chi(u)\o_v^{j+1} \wedge \o_u^{n-j-1} = \int_X \chi(u)\o \wedge \o_v^{j} \wedge \o_u^{n-j-1} + \int_X v i\ddbar \chi(u) \wedge \o_v^{j} \wedge \o_u^{n-j-1}.
\end{equation}
The first integral on the right hand side is bounded above by $\int_X \chi(u)\o_v^{j} \wedge \o_u^{n-j}$. Indeed, as $u \leq 0$  we have $\int_X i\chi'(u)\del u \wedge \dbar u \wedge \o_{v}^{j} \wedge \o_u^{n-j-1} \leq 0$.  

Concerning the second integral on the right hand side of \eqref{eq: integ_parts} we notice that  $i\ddbar \chi(u) = i\chi''(u)\partial u \wedge \dbar u + \chi'(u)i\ddbar u \geq \chi'(u) \o_u$, as currents. Indeed, this is straitforward for smooth $\chi$, and in general it is possible to use approximation.

Since $v \leq 0$ we can write
\begin{flalign*}\int_X v i\ddbar \chi(u) \wedge \o_v^{j} \wedge \o_u^{n-j-1} &\leq \int_X v \chi'(u) \o_v^{j} \wedge \o_u^{n-j}  = \int_X |v| \chi'(|u|) \o_v^{j} \wedge \o_u^{n-j}\\
& \leq \int_X |u| \chi'(|u|) \o_v^{j} \wedge \o_u^{n-j} \leq p \int_X \chi(u) \o_v^{j} \wedge \o_u^{n-j},
\end{flalign*}
where in the last inequality we have used that $\chi \in \mathcal W^+_p$. Combining the above with \eqref{eq: integ_parts} yields  \eqref{eq: Fund_Est_intermediate}. Iterating \eqref{eq: Fund_Est_intermediate} $n$ times and using the fact that $u \leq v$ gives \eqref{eq: Fund_Est} for $u,v \in \mathcal H_\omega$.

In case $u,v \in \textup{PSH}(X,\omega) \cap L^\infty$, by adding small constants, we can assume that $u < v < 0$. We get $E_\chi(v_k) \leq (p+1)^n E_\chi(u_k)$, where $u_k,v_k \in \mathcal H_\omega$ are decreasing smooth approximants of $u,v$, given by Theorem \ref{thm: BK_approx}, satisfying $u_k \leq v_k \leq 0$. An application of \cite[Theorem 4.26]{gzbook} now gives the result for $u,v$ bounded.

In case $u,v \in \mathcal E_\chi(X,\o)$, we know that $E_\chi(u_k) \leq (p+1)^n E_\chi(v_k)$ for the canonical cutoffs $u_k:=\max(u,-k),v_k:=\max(v,-k)$. Proposition \ref{prop: cutoff_aprox} allows to take the limit $k \to \infty$, to obtain \eqref{eq: Fund_Est}.
\end{proof}

As a corollary of this result, we obtain the \emph{monotonicity property} for $\mathcal E_\chi(X,\o)$:

\begin{corollary}\label{cor: monotonicity_E_chi} Suppose $u \in \mathcal E_\chi(X,\o)$ and $v \in \textup{PSH}(X,\o)$. If $u \leq v$ then $v \in \mathcal E_\chi(X,\o)$.
\end{corollary}
\begin{proof}We can assume without loss of generality that $u \leq v \leq 0$. Proposition \ref{prop: monotonicity_E} implies that $v \in \mathcal E(X,\o)$. Also for the canonical cutoffs $v_k$ we have $u \leq v_k$, hence $E_\chi(v_k) \leq (p+1)^n E_\chi(u)$ for all $k \in \Bbb N$. Proposition \ref{prop: cutoff_aprox} now gives that $E_\chi(v) \leq (p+1)^n E_\chi(u)$, finishing the proof.
\end{proof}

\noindent The next result says that if $u,v \in \mathcal E_\chi(X,\o)$, then in fact $u \in L^\chi(\o_v^n)$:
\begin{proposition}\textup{\cite[Proposition 3.6]{gz}}\label{prop: mixed_finite_prop: Energy_est} Suppose $u,v \in \mathcal E_\chi(X,\o), \ \chi \in \mathcal W^+_p$. If $u,v \leq 0$ then 
$$\int_X \chi(u) \o_v^n \leq p2^{p}\big(E_\chi(u) + E_\chi(v)\big)$$
\end{proposition}
\begin{proof} We first show that $\chi'(2t) \leq p 2^{p-1} \chi'(t), \ t >0$. For this, after possibly adding $\delta |t|$ to $\chi(t)$, we can momentarily assume that $\chi'(t),\chi(t)>0$ for any $t > 0$. 

By convexity we have $\chi(t)/t \leq \chi'(t)$, and \eqref{eq: GrowthControl} gives $\chi(2t)/\chi(t)\leq 2^p$, hence we can write:
\begin{equation}\label{eq: chi'_ineq}
\frac{\chi'(2t)}{\chi'(t)}=\frac{2t\chi'(2t)}{\chi(2t)}\cdot\frac{\chi(2t)}{2\chi(t)}\cdot \frac{\chi(t)}{t\chi'(t)} \leq p2^{p-1}.
\end{equation}
Consequently we have the following sequence of inequalities:
$$\int_X \chi(u) \o_v^n = \int_0^\infty \chi'(t) \o_v^n\{ |u| >t \}dt \leq 2^{p}p \int_0^\infty\chi'(t)\o_v^n\{ |u| >2t \} dt$$
Noticing that $\{u < -2t \} \subset \{u < -t + v \} \cup \{ v < -t\}$ we can continue to write:
\begin{flalign*}
\int_X \chi(u) \o_v^n \leq  & 2^{p}p \Big(\int_0^\infty\chi'(t)\o_v^n\{ u < -t + v \} dt + \int_0^\infty\chi'(t)\o_v^n\{ v < -t \} dt\Big)\\
\leq & 2^{p}p \Big(\int_0^\infty\chi'(t)\o_u^n\{ u < -t + v \} dt +E_\chi(v)\Big)\\
\leq & 2^{p}p \Big(\int_0^\infty\chi'(t)\o_u^n\{ u < -t\} dt +E_\chi(v)\Big) = 2^{p}p \big(E_\chi(u) +E_\chi(v)\big),
\end{flalign*}
where in the second line we have used Proposition \ref{prop: comp_princ_E} and in the last line we have used that $\{u < -t + v \} \subset \{u < -t \}$.
\end{proof}

We note the following result about decreasing approximants and the class $\mathcal E_\chi(X,\omega)$:
 
\begin{lemma}\textup{\cite[Proposition 5.6]{gz}} \label{lem: E_semicont} Suppose $\chi\in \mathcal W^+_p$ and $\{u_k\} _{k \in \Bbb N} \subset \mathcal E_\chi(X,\o)$ is a sequence decreasing to $u \in \textup{PSH}(X,\o)$. If
$\sup_k E_\chi(u_k) < \infty$ then $u \in \mathcal E_\chi(X,\o)$ and 
$$E_\chi(u) \leq (p+1)^n\limsup_k E_\chi(u_k).$$
\end{lemma}
\begin{proof} Let $u^l_k = \max(u_k,-l)$ and $u^l=\max(u,-l)$ be the canonical cutoffs. As $-l \leq u^l \leq u^k_l$,  by Proposition \ref{prop: Energy_est} and  \cite[Theorem 4.26]{gzbook} we get that $E_\chi(u^l)=\lim_k E_\chi(u^l_k) \leq (p+1)^n \limsup_k E_\chi(u_k).$ Finally, Proposition \ref{prop: cutoff_aprox} gives that $u \in \mathcal E_\chi(X,\o)$, along with the desired estimate. 
\end{proof}

\section{Examples and singularity type of finite energy potentials}

It is clear that $\textup{PSH}(X,\o) \cap L^\infty \subset \mathcal E_p(X,\o)$ for all $p \geq 1$. In this short section we will show that this inclusion is always strict, as one can construct unbounded elements of $\mathcal E_p(X,\o)$. However the content of our first result is that the singularity type of full mass potentials is always mild, even when they are unbounded.

Given $u \in \textup{PSH}(X,\o)$ and $x_0 \in X$ it is possible to measure the local singularity of $u$ at $x_0$ using the \emph{Lelong number} $\mathcal L(u,x_0)$, whose definition we now recall:
\begin{equation}\label{eq: Lelong_def}
\mathcal L(u,x_0):= \sup \{ r \geq 0 \ | \ u(x) \leq r \log |x| + C_r \ \forall x \in U_r \ \textup{ for some } C_r>0\},
\end{equation}
where $U_r \subset \Bbb C^n$ is some coordinate neighborhood of $x_0$ that identifies $x_0$ with $0 \in \Bbb C^n$. Roughly speaking, $\mathcal L(u,x_0)$ measures the extent to which the singularity of $u$ at $x_0$ is logarithmic. For an extensive treatment of Lelong numbers we refer to \cite[Section~2.3]{gzbook}.

To begin, we observe that elements of $\mathcal E(X,\o)$ have singularity so mild that Lelong numbers can not detect them: 

\begin{proposition}\textup{\cite[Corollary 1.8]{gz}}\label{prop: Lelong_E} If $u \in \mathcal E(X,\o)$ then all Lelong numbers of $u$ are zero.
\end{proposition}
\begin{proof} Let $x_0 \in X$ and $U \subset X$ a coordinate neighborhood of $x_0$, identifying $x_0$ with $0 \in \Bbb C^n$ via a biholomorphism $\varphi : B(0,2) \to U$. Let $v \in \textup{PSH}(X,\o)$ such that $v$ is smooth on $X \setminus \{x_0 \}$ and $v \circ \varphi|_{B(0,1)} = c \log |x|$ for some $c >0$. Using a partition of unity, such $v$ can be easily constructed.

If $\mathcal L(u,x_0)>0$, then for some $\varepsilon >0$ we will have $u \leq \varepsilon v + 1/\varepsilon$. The monotonicity property (Proposition \ref{prop: monotonicity_E}) now implies that $\varepsilon v \in \mathcal E(X,\o)$. However this cannot happen, yielding a contradiction. Indeed, by the  definition of $\o_{\varepsilon v}^n$ (see \eqref{eq: CMA_general_def}) we have that 
\begin{flalign*}
\int_X \o_{\varepsilon v}^n &=   \textup{Vol}(X) - \lim_{k \to \infty}\int_{\{\varepsilon  v \leq -k\}} \o_{\max(\varepsilon v,-k)}^n\\
&\leq   \textup{Vol}(X) - \lim_{k \to \infty} \int_{B(0,1)} \big(i\ddbar \max (c\varepsilon\log|x|, -k)\big)^n  \\
&= \textup{Vol}(X) - c^n \varepsilon^n (2\pi)^n,
\end{flalign*} 
where in the last line we have used  that $ \big(i\ddbar\max(\log |z|, \log r)\big)^n =  (2\pi)^n d\sigma_{\partial B(0,r)}$, where  $d\sigma_{\partial B(0,r)}$ is the Euclidean probability surface measure of $\partial B(0,r)$ (see the exercise following \cite[Corollary 2.2.7]{bl3}).
\end{proof}

The purpose of our next proposition is to give a flexible construction for unbounded elements of $\mathcal E_p(X,\o)$:

\begin{proposition}\textup{\cite[Example 2.14]{gz},  \cite[Proposition 5]{bl2}}\label{prop: E_p_examples} Suppose that  $u \in \textup{PSH}(X,\o)$, $ u < -1$ and $\alpha \in (0,1/2)$. Then $v = -(- u)^{\alpha} \in \mathcal E_p(X,\o)$ for all $p \in [1,(1-\alpha)/\alpha)$.
\end{proposition}

\begin{proof} First assume that $u \in \mathcal H_\o$ and $u \leq -1$. For $v = -(-u)^{\alpha}$ we have
\begin{flalign*}
\o_{v}&=\alpha(1-\alpha) (-u)^{\alpha - 2} i \partial u  \wedge \dbar u + \alpha (-u)^{\alpha - 1}\o_{u} + (1 - \alpha (-u)^{\alpha - 1}) \o \nonumber\\
&\leq (-u)^{\alpha - 2} i \partial u  \wedge \dbar u  + (-u)^{\alpha -1}\o_{u} + \o. 
\end{flalign*}
Consequently $v \in \mathcal H_\o$, and for some $C:=C(\alpha) > 1$ we have
\begin{equation}\label{eq: o_v_est}
\o_v^n \leq C\Big( \sum_{k=0}^{n-1}  (-u)^{\alpha - 2 + k(\alpha - 1)} i \partial u  \wedge \dbar u \wedge \o_{u}^k \wedge \o ^{n-1 -k} + \sum_{k=0}^{n}  (-u)^{k(\alpha - 1)} \o_u^k \wedge \o^{n-k}\Big).
\end{equation}
For arbitrary $a > 0$, integration by parts gives the following estimates
\begin{flalign*}
a\int_X (-u)^{-a-1}i\partial u \wedge \dbar u \wedge \o_u^k \wedge \o^{n-k-1} &=  \int_X i \partial (-u)^{-a} \wedge \dbar u \wedge \o_u^k \wedge \o^{n-k-1}\\
&=-\int_X  (-u)^{-a} i \ddbar u \wedge \o_u^k \wedge \o^{n-k-1}\\
& \leq \int_X  (-u)^{-a} \o_u^k \wedge \o^{n-k} \leq \textup{Vol}(X),
\end{flalign*}
where in the last inequality we have used that $(-u)^{-a} \leq 1$. This estimate and \eqref{eq: o_v_est} implies  that for any $b \in (0,1-\alpha)$ we have
$$\int_X |v|^{\frac{b}{\alpha}} \o_v^n=\int_X |u|^{b} \o_v^n \leq C(\o, \alpha, b) \Big( 1 + \int_X |u| \o^n \Big).$$

Returning to the general case $ u \in \textup{PSH}(X,\o)$, let $\{ u_k\}_k\subset \mathcal H_\o$ be a sequence of smooth potentials decreasing to $u$ (Theorem \ref{thm: BK_approx}). We can assume that $u_k < -1$, hence the above inequality implies that $\int_X |v_k|^{b/\alpha} \o_{v_k}^n$ is uniformly bounded for $v_k :=-(-u_k)^{\alpha}$. Consequently, Lemma \ref{lem: E_semicont} implies that  $v \in \mathcal E_p(X,\o)$ for any $p \in (0,(1-\alpha)/\alpha)$.
\end{proof}

\section{Envelopes of finite energy classes}

Envelope constructions are ubiquitous throughout pluripotential theory. In our setting, given an usc function $f: X \to [-\infty,\infty)$, the simplest envelope one can consider is
\begin{equation}
\label{eq: P_env_def}
P(f):= \sup \{u \in \textup{PSH}(X,\o)  \textup{ s.t. }  u \leq f\}.
\end{equation}
As we know, the supremum of a family of $\o$-psh functions may not be $\o$-psh, as $P(f)$ may not be usc to begin with. However by \cite[Theorem 1.2.3]{bl3}  the usc regularization $P(f)^*$ is indeed $\o$-psh. As $f$ is usc, we obtain that $P(f)^* \leq f^* = f$, immediately giving that $P(f)^*$ is a candidate in the definition of $P(f)$, hence $P(f)=P(f)^*$, i.e., $P(f) \in \textup{PSH}(X,\o)$.

Slightly generalizing the above concept, for usc functions $\{f_1,f_2,\ldots,f_k \}$ we introduce the \emph{rooftop envelope} $$P(f_1,f_2,\ldots,f_k):=P(\min(f_1,f_2,\ldots,f_k)).$$

When $f$ is smooth (or just continuous), as $P(f)$ is usc, we obtain that the \emph{non-contact set} $\{f > P(f) \} \subset X$ is open. A classical Perron type argument (see \cite[Corollary 9.2]{BT1} or \cite[Proposition 1.4.10]{bl3}). yields that 
\begin{equation}\label{eq: coinc_vanish}
\o_{P(f)}^n(\{f > P(f) \})=0
\end{equation}  
This observation suggests that $P(f)$ can have at most bounded but not continuous second derivatives. This is mostly  confirmed by the next result, whose proof is provided in the appendix:

\begin{theorem}[Theorem \ref{thm: DR_reg}, \textup{\cite[Theorem 2.5]{DR1}}]\label{thm: reg_thm_for_P(f)} Given $f_1,...,f_k \in C^\infty(X)$, then $P(f_1,f_2,...,f_k) \in C^{1,\alpha}(X), \ \alpha \in (0,1)$. More precisely, the following estimate holds: 
$$\|P(f_1,f_2,...,f_k)\|_{C^{1,\bar 1}} \leq C(X,\o,\| f_1\|_{C^{1,\bar 1}},\| f_2\|_{C^{1,\bar 1}},\ldots,\| f_k\|_{C^{1,\bar 1}}).$$
\end{theorem}

By a  bound on the $C^{1,\bar1}$ norm of $P(f)$ we mean a uniform bound on all mixed second order derivatives $\partial^2 P(f)/\partial z_j \bar \partial z_k$. Since $P(f)$ is $\o$-psh, this is equivalent to saying that $\Delta^\o P(f)$ is bounded, and by the Calderon--Zygmund estimate \cite[Chapter 9, Lemma 9.9]{GT}, we automatically obtain that $P(f_1,f_2,\ldots,f_k) \in C^{1,\alpha}(X), \alpha  < 1$. We introduce the following subspace of $\textup{PSH}(X,\o)$:
\begin{equation}\label{eq: H_01bar1_def}
\mathcal H_\o^{1,\bar 1}:=\{u \in \textup{PSH}(X,\o) \  \textup{ s.t. } \ \|u\|_{C^{1,\bar 1}} < \infty \}. 
\end{equation}

When $k=1$, the above result was first proved in \cite{bd} using the Kiselman technique for attenuation of singularities. An independent ``PDE proof"  has been given by Berman \cite{Brm2}, and we present this in the appendix (see Theorem \ref{thm: BD_reg}). The proof of the general case $k \geq 1$ was given in the paper \cite{DR1}, that provided a detailed regularity analysis for the rooftop envelopes introduced above. By reduction to the case $k=1$,  we prove this more general result in the appendix as well (see Theorem \ref{thm: DR_reg}). Lastly, we mention that very recently it was shown by Tosatti \cite{To} and Chu-Zhou \cite{ChZh} that in fact it is possible to bound the full Hessian of $P(f_1,\ldots,f_k)$ in terms of the Hessians of $f_1,\ldots,f_k$. 

With the above introduced notation, we will use Theorem \ref{thm: reg_thm_for_P(f)} in the following form:
\begin{corollary}\label{cor: rooftop_env_reg} If $u_0,u_1,\ldots,u_k \in \mathcal H_\o^{1,\bar 1}$ then $P(u_0,u_1,\ldots,u_k) \in \mathcal H_\o^{1,\bar 1}$.
\end{corollary} 
\begin{proof} Let $f_j^i \in C^\infty(X)$ be such that $f_j^i \to u_j$ uniformly, and the mixed second derivatives of $f_j^i$ are uniformly bounded. By the previous theorem, the mixed second derivatives of $P(f^i_0,f^i_1,\ldots,f^i_k)$ are uniformly bounded as well. Since $P(f^i_0,f^i_1,\ldots,f^i_k) \to P(u_0,u_1,\ldots,u_k)$ uniformly, the result follows.
\end{proof}

\noindent As a consequence of the above corollary, we get a volume partition formula for $\o_{P(u_0,u_1)}^n$:

\begin{proposition}\textup{\cite[Proposition 2.2]{da1}} \label{prop: MA_form} For $u_0,u_1 \in \mathcal H_\o^{1,\bar 1}$, we introduce the contact sets $\Lambda_{u_0} = \{ P(u_0,u_1)=u_0\}$ and  $\Lambda_{u_1} = \{ P(u_0,u_1)=u_1\}$. Then the following partition formula holds:
\begin{equation}\label{eq: MA_partition}
\o_{P(u_0,u_1)}^n= \mathbbm{1}_{\Lambda_{u_0}}\o_{u_0}^n + \mathbbm{1}_{\Lambda_{u_1} \setminus \Lambda_{u_0}}\o_{u_1}^n.
\end{equation}
\end{proposition}
\begin{proof} As pointed out in \eqref{eq: coinc_vanish}, $\o_{P(u_0,u_1)}^n$ is concentrated on the coincidence set $\Lambda_{u_0} \cup \Lambda_{u_1}.$ Having bounded Laplacian implies that all second order partials of $P(u_0,u_1)$ are in any $L^p(X), \ p <\infty$ \cite[Chapter 9, Lemma 9.9]{GT}. It follows from \cite[Chapter 7, Lemma 7.7]{GT} that on $\Lambda_{u_0}$ all the second order partials of $P(u_0,u_1)$ and $u_0$ agree a.e., and the analogous statement holds on $\Lambda_{u_1}$. Hence, using \cite[Proposition 2.1.6]{bl3}  one can write:
\begin{flalign*}
\o_{P(u_0,u_1)}^n&=\mathbbm{1}_{\Lambda_{u_0} \cup \Lambda_{u_1}}\o_{P(u_0,u_1)}^n \nonumber =\mathbbm{1}_{\Lambda_{u_0}}\o_{u_0}^n + \mathbbm{1}_{\Lambda_{u_1} \setminus \Lambda_{u_0}}\o_{u_1}^n,
\end{flalign*}
finishing the proof.
\end{proof}
The partition formula (\ref{eq: MA_partition}) is at the core of many theorems presented later in this survey. Interestingly, it fails to hold evein in the slightly more general case of Lipschitz potentials $$u_0,u_1 \in \mathcal H^{0,1}_\o:= \textup{PSH}(X,\o) \cap C^{0,1}(X).$$ For a counterexample, suppose $\dim X = 1$ and $g_x$ is the $\o-$Green function with pole at $x \in X$. Such function is characterized by the property $\int_X g_x \o =0$ and $\o + i\partial \bar \partial g_x = \delta_x$. We choose $u_0 = \max\{g_x,0\}$ and $u_1 =0$. In this case $P(u_0,u_1)=0, \ \Lambda_{u_0} = \{ g_x \leq 0\}$ and $\Lambda_{u_1} \setminus \Lambda_{u_0}= X \setminus \Lambda_{u_0} \neq \emptyset$. As
$ \textup{Vol}(X) = \int_{\Lambda_{u_0}} \o_{u_0}^n = \int_X \o_{P(u_0,u_1)}^n,$ it is seen that the right hand side of (\ref{eq: MA_partition}) has total integral greater then the left hand side, hence they can not equal.

As $\int_X \o_{u_1}^n$ is finite, it follows that $\o_{u_1}^n(\{u_0 = u_1 +\tau\})>0$ only for a countable number of values $\tau \in \Bbb R$. Consequently, \eqref{eq: MA_partition} implies the following observation:
\begin{remark} \label{rem: MA_form_remark} Given $u_0,u_1 \in \mathcal H_\o^{1,\bar 1}$, for any $\tau \in \Bbb R$ outside a countable set we have:
$$\o_{P(u_0,u_1+\tau)}^n= \mathbbm{1}_{\Lambda_{u_0}}\o_{u_0}^n + \mathbbm{1}_{\Lambda_{u_1+\tau}}\o_{u_1}^n.$$
In particular, $\textup{Vol}(X)=\int_{\{P(u_0,u_1 + \tau)=u_0\}}\o_{u_0}^n + \int_{\{P(u_0,u_1 + \tau)=u_1 + \tau\}}\o_{u_1}^n$.
\end{remark}
 
Corollary \ref{cor: rooftop_env_reg} simply says that the operation $(u_0,u_1) \to P(u_0,u_1)$ is closed inside the class $\mathcal H_\o^{1,\bar 1}$. The next proposition tells that the same holds inside the finite energy classes as well:
 
\begin{proposition}\textup{\cite[Lemma 3.4]{da1}} \label{prop: env_exist} Suppose $\chi \in \mathcal W^+_p$ and ${u_0},{u_1} \in \mathcal E_\chi(X,\o)$. Then $P(u_0,u_1) \in \mathcal E_\chi(X,\o)$, and if $u_0,u_1 \leq 0$ then following estimate holds:
\begin{equation}\label{eq: e_est}
E_\chi(P({u_0},{u_1})) \leq (p+1)^{2n}(E_\chi({u_0}) + E_\chi({u_1})).
\end{equation}
\end{proposition}

\begin{proof} As $P(u_0-c,u_1-c)=P(u_0,u_1)-c$ for $c \in \Bbb R$, it follows that without loss of generality we can assume that $u_0,u_1 <0$.

By Theorem \ref{thm: BK_approx}, it is possible to find negative potentials $u^j_0,u^j_1 \in \mathcal H_\o$ that decrease to $u_0,u_1$. Furthermore, by Proposition \ref{prop: MA_form} we can also assume that $P(u^j_0,u^j_1) \in \mathcal H_\o^{1,\bar 1}$ satisfies
\begin{equation}\label{eq: MA_approx_part}
\o_{P(u_0^j,u_1^j)}^n \leq \mathbbm{1}_{\Lambda_{u^j_0}}\o_{u^j_0}^n + \mathbbm{1}_{\Lambda_{u_1^j}}\o_{u_1^j}^n.
\end{equation}
Using this formula we  can write:
\begin{flalign*}E_\chi(P({u^j_0},{u^j_1}))&= \int_X \chi (P({u^j_0},{u_1^j}))\o_{P({u^j_0},{u^j_1})}^n \\
&\leq \int_{\{P({u^j_0},{u^j_1}) = {u^j_0}\}}\chi({u^j_0})\o_{{u^j_0}}^n + \int_{\{P({u^j_0},{u^j_1}) = {u^j_1}\}}\chi({u_1^j})\o_{u^j_1}^n \\
&\leq \int_X\chi({u_0^j})\o_{u_0^j}^n + \int_X\chi({u_1^j})\o_{u_1^j}^n  = E_\chi({u^j_0}) + E_\chi({u^j_1})\\
&\leq (p+1)^n (E_\chi({u_0}) + E_\chi({u_1})),
\end{flalign*}
where in the last line we have used Proposition \ref{prop: Energy_est}. As $P(u^j_0,u^j_1)$ decreases to $P(u_0,u_1)$, by  Lemma \ref{lem: E_semicont} we have $P(u_0,u_1) \in \mathcal E_\chi(X,\o)$, and \eqref{eq: e_est} holds.
\end{proof}

Observe that $t u_0 + (1-t)u_1 \geq P(u_0,u_1)$ for any $t \in [0,1]$, hence as a consequence of the previous proposition and the monotonicity property of $\mathcal E_\chi(X,\o)$ (Corollary \ref{cor: monotonicity_E_chi}) we obtain that $\mathcal E_\chi(X,\o)$ is convex:
\begin{corollary}\label{cor: E_chi_convex} If $u_0,u_1 \in \mathcal E_\chi(X,\o)$ then $tu_0+(1-t)u_1 \in \mathcal E_\chi(X,\o)$ for any $t \in [0,1]$. 
\end{corollary}

Before we can establish the continuity property of the Monge--Amp\`ere operator along monotonic sequences of potentials, we need to establish the following auxilliary result, which states that if $u \in \mathcal E_\chi(X,\o)$, then it is possible to find a $\tilde \chi$ with bigger growth than $\chi$ such that $u \in \mathcal E_{\tilde \chi}(X,\o)$ still holds:
\begin{lemma}\label{lem: CompBiggerEn}
Suppose $u \in \mathcal E_\chi(X,\omega), \ \chi \in \mathcal W^+_{p}$. Then there exists $\tilde{\chi} \in \mathcal W^+_{2p+1}$ such that $ \chi(t) \leq \tilde \chi(t)$, $\chi(t)/\tilde{\chi}(t) \searrow 0$ as $t \to \infty$, and $u \in \mathcal E_{\tilde{\chi}}(X,\o)$.
\end{lemma}
\begin{proof} The weight ${\tilde{\chi}}:\Bbb R^+ \to \Bbb R^+$ will be constructed as an increasing limit of weights ${\tilde{\chi}}^j \in \mathcal W^+_{2p+1}, j \in \Bbb N$ constructed below.
We set ${\tilde{\chi}}_0(t) = \chi$. Let $t_0 \in \Bbb R^+$, to be specified later. We define ${\tilde{\chi}}_1:\Bbb R^+ \to \Bbb R^+$ by the formula
$${\tilde{\chi}}_1(t) = \begin{cases} {\tilde{\chi}}_0(t), & \mbox{if } { t \leq t_1}
\\{\tilde{\chi}}_0(t_1) + 2({\tilde{\chi}}_0(t) - {\tilde{\chi}}_0(t_1)), & \mbox{if } { t >t_1. } \end{cases}$$
Notice that ${\tilde{\chi}}_1$ satisfies $\chi \leq \tilde \chi_1$ and the following also hold
\begin{equation}\label{CompWeightEst1}
\sup_{t >0}\frac{|t{\tilde{\chi}}_1'(t)|}{|{\tilde{\chi}}_1(t)|}\leq  \sup_{t >0}\frac{2|t{\tilde{\chi}}_0'(t)|}{|{\tilde{\chi}}_0(t)|} < 2p + 1,
\end{equation}
\begin{equation}\label{CompWeightEst2}
\lim_{t \to \infty}\frac{|t{\tilde{\chi}}_1'(t)|}{|{\tilde{\chi}}_1(t)|}=p.
\end{equation}
We can choose $t_1$ to be big enough such that $E_{{\tilde{\chi}}_1}(u) < E_\chi(u) + 1$.

Now pick $t_2 > t_1$, again specified later. One defines ${\tilde{\chi}}_2:\Bbb R^+ \to \Bbb R^+$ in a similar manner:
$${\tilde{\chi}}_2(t) = \begin{cases} {\tilde{\chi}}_1(t), & \mbox{if } { t \leq t_2}
\\{\tilde{\chi}}_1(t_2) + 2({\tilde{\chi}}_1(t) - {\tilde{\chi}}_1(t_2)), & \mbox{if } { t >t_2. } \end{cases}$$
As $\eqref{CompWeightEst2}$ holds for ${\tilde{\chi}}_1$, it is possible to choose $t_2>t_1$ big enough so that the ${\tilde{\chi}}_2$-analogs of \eqref{CompWeightEst1},\eqref{CompWeightEst2} are satisfied and $E_{{\tilde{\chi}}_2}(u) < E_\chi(u) + 1$. 

We define ${\tilde{\chi}}_k, k \in \Bbb N$, following the above procedure. As $\lim_{t \to \infty} {\tilde{\chi}}_k(t)/\chi(t)=2^k$, the limit weight ${\tilde{\chi}}(t) = \lim_{k \to \infty} {\tilde{\chi}}_k(t)$ is seen to satisfy the requirements of the lemma.
\end{proof}

The following result, allowing to take weak limits of certain measures, will be used in many different contexts throughout the survey:

\begin{proposition}\label{prop: MA_cont} Assume that  $\{ \phi_k\}_{k \in \Bbb N},\{ \psi_k\}_{k \in \Bbb N},\{ v^j_k\}_{k \in \Bbb N} \subset \mathcal E_\chi(X,\o)$ decrease (increase  a.e.) to $\phi,\psi, v^j \in \mathcal E_\chi(X,\o)$ respectively, $j \in \{1,\ldots,n\}$. Suppose that \\
(i) $\psi_k \leq \phi_k$;\\ 
(ii) $h:\Bbb R \to \Bbb R$ is continuous  with $\limsup_{|l| \to \infty} \frac{|h(l)|}{\chi(l)} \leq C$ for some $C>0$; \\
Then $h(\phi_k - \psi_k)\o_{v^1_k}\wedge \ldots \wedge \o_{v^n_k} \to h(\phi-\psi)\o_{v^1}\wedge \ldots \wedge \o_{v^n}$ weakly.
\end{proposition}

We will apply this porposition mostly for $h = \chi$ and $h = 1$. In the latter case this proposition simply tells that the non-pluripolar products, as defined in \eqref{eq: CMA_general_def}, converge weakly along monotonic sequences of $\mathcal E_\chi(X,\o)$\vspace{0.1cm}. 

\noindent \textbf{Remark.} Though in this work we will only deal with $\chi \in \mathcal W^+_p$, the same argument yields the result for concave weights $\chi \in \mathcal W^-$  (introduced in \cite{gz}) as well.

\begin{proof} Let $\gamma \in C^\infty(X).$ We can suppose without loss of generality that all the potentials involved are negative. 
First we suppose that there exists $L >1$ such that $-L < \phi,\phi_k,\psi,\psi_k,v,v_k < 0$, and prove the theorem under this assumption. \cite[Theorem 4.26]{gzbook} gives this immediately, but we give a more elementary argument instead. Given $\varepsilon >0$ one can find an open $O \subset X$ such that $\textup{Cap}_X(O) < \varepsilon$ and $\phi,\phi_k,\psi,\psi_k,v^j,v^j_k$ are all continuous on $X \setminus O$ (\cite[Theorem 2.2]{bl3} or \cite[Definition 2.4, Corollary 2.8]{gz05}). We have
\begin{flalign}\label{eq: some1}
\int_X \gamma \big[h(\phi_k - \psi_k)&- h( \phi - \psi)\big]\o_{v^1_k}\wedge \ldots \wedge \o_{v^n_k}= \nonumber \\
&=\int_O + \int_{X \setminus O} \gamma \big[h(\phi_k - \psi_k) - h( \phi - \psi)\big]\o_{v^1_k}\wedge \ldots \wedge \o_{v^n_k}.
\end{flalign}
The integral on $O$ is bounded by $\varepsilon,$ where $C:=C(L,\gamma)$. The second integral tends to $0$ as on the closed set $X \setminus O$ we have $\phi_k \to \phi$ and $\psi_k \to \psi$ uniformly. We also have
\begin{flalign}\label{eq: some2}
\int_X \gamma h(\phi - \psi)\o_{v^1_k}\wedge \ldots \wedge \o_{v^n_k} - \int_X \gamma h({\phi - \psi})\o_{v^1}\wedge \ldots \wedge \o_{v^n} \to 0,
\end{flalign}
as the function $h({\phi - \psi})$ is quasi--continuous and bounded. Indeed, quasi--continuity and boundedness implies again that for all $\varepsilon >0$ one can find an open $\tilde O \subset X$ such that $\textup{Cap}_X(\tilde O) < \varepsilon$ and $h({\phi - \psi})$ is continuous on $X \setminus \tilde O$. Furthermore, by Tietze's extension theorem we can extend $h({\phi - \psi})|_{X \setminus \tilde O}$ to a continuous function $\alpha$ on $X$. As $\textup{Cap}_X(\tilde O) < \varepsilon$,  we have that $\int_X |h({\phi - \psi}) - \alpha| \o_{v^1_k}\wedge \ldots \wedge \o_{v^n_k} \leq C\varepsilon$ and $\int_X|h({\phi - \psi}) - \alpha| \o_{v^1}\wedge \ldots \wedge \o_{v^n} \leq C\varepsilon$. On the other hand $\int_X \alpha \o_{v^1_k}\wedge \ldots \wedge \o_{v^n_k} \to \int_X \alpha \o_{v^1}\wedge \ldots \wedge \o_{v^n}$ by Bedford-Taylor theory (see \cite[Theorem 2.2.5]{bl3}). Putting these facts together we get \eqref{eq: some2}. Finally, \eqref{eq: some1} and \eqref{eq: some2} together give the proposition for bounded potentials.

Now we argue that the result also holds when $\phi,\phi_k,\psi,\psi_k,v,v_k$ are unbounded. For this we only need to show that
\begin{flalign}\label{uniformlimit}
\int_X \gamma h(\phi_k - \psi_k)\o_{v^1_k}\wedge \ldots \wedge \o_{v^n_k} - \int_X \gamma h(\phi^L_k - \psi^L_k)\o_{v^{1L}_k}\wedge \ldots \wedge \o_{v^{nL}_k} \to 0
\end{flalign}
\begin{flalign}\label{uniformlimit2}
\int_X \gamma h(\phi - \psi)\o_{v^1}\wedge \ldots \wedge \o_{v^n}- \int_X\gamma  h(\phi^L - \psi^L)\o_{v^{1L}}\wedge \ldots \wedge \o_{v^{nL}} \to 0
\end{flalign}
as $L \to \infty$, uniformly with respect to $k$, where $v^{jL} = \max (v^j,-L)$ and $v^{jL}_k,\psi^L_k,\psi^L,\phi^L_k,\phi^L$  are defined similarly. 

We only argue \eqref{uniformlimit} as the proof of \eqref{uniformlimit2} is identical. By Proposition \ref{prop: env_exist} there exists $\beta \in \mathcal E_\chi(X,\omega)$ such that $\beta \leq \phi,\phi_k,\psi,\psi_k,v^j,v^j_k$ for any $k$, due to the monotonicity of our sequences. For example, in case all sequences are decreasing, we can take $\beta := P(\phi,\psi, v_1,,\ldots,v_n)$. The other cases are treated similarly.

By \eqref{eq: GZ_MP_local} the non-pluripolar product is local in the plurifine topology, hence both integrands in \eqref{uniformlimit} are the same on $\{\beta > -L\}$. As a result we can write 
\begin{flalign}\label{eq: interm_est}
&\int_X  \gamma h(\phi_k - \psi_k)\o_{v^1_k}\wedge \ldots \wedge \o_{v^n_k} - \int_X \gamma h(\phi^L_k - \psi^L_k)\o_{v^{1L}_k}\wedge \ldots \wedge \o_{v^{nL}_k} = \nonumber \\
& = \int_{\{\beta \leq -L \}} \gamma h(\phi_k - \psi_k)\o_{v^1_k}\wedge \ldots \wedge \o_{v^n_k} - \int_{\{\beta \leq -L \}} \gamma h(\phi^L_k - \psi^L_k)\o_{v^{1L}_k}\wedge \ldots \wedge \o_{v^{nL}_k}
\end{flalign}
As a result, to argue \eqref{uniformlimit} it is enough to show that both  terms in the above expression converge to zero. We now focus on  the second expression. As $0 \geq \psi^L_k - \phi^L_k \geq \beta$ we can write: 
\begin{flalign*}
\bigg|\int_{\{\beta \leq -L \}} \gamma & h(\phi^L_k - \psi^L_k)\o_{v^{1L}_k}\wedge \ldots \wedge \o_{v^{nL}_k}\bigg|  \leq C \int_{\{\beta \leq -L \}}  \chi(\beta)\o_{v^{1L}_k}\wedge \ldots \wedge \o_{v^{nL}_k}\\
& \leq C \int_{\{\beta \leq -L \}}  \chi(v^{1L}_k + \ldots + v^{nL}_k + \beta)\o^n_{\frac{v^{1L}_k + \ldots + v^{nL}_k + \beta}{n+1}}\\
& \leq C \int_{\{v^{1L}_k + \ldots + v^{nL}_k + \beta \leq -L \}}  \chi(v^{1L}_k + \ldots + v^{nL}_k + \beta)\o^n_{\frac{v^{1L}_k + \ldots + v^{nL}_k + \beta}{n+1}}
\\
& \leq C \int_{\{\frac{v^{1L}_k + \ldots + v^{nL}_k + \beta}{n+1} \leq -\frac{L}{n+1} \}}  \chi\Big(\frac{v^{1L}_k + \ldots + v^{nL}_k + \beta}{n+1}\Big)\o^n_{\frac{v^{1L}_k + \ldots + v^{nL}_k + \beta}{n+1}},
\end{flalign*}
where in the last line we have used \eqref{eq: GrowthControl}. To continue, we introduce the potentials $\eta : = ({v^{1}_k + \ldots + v^{n}_k + \beta})/({n+1})$ and $\eta^L : = ({v^{1L}_k + \ldots + v^{nL}_k + \beta})/({n+1})$. These potentials are in $\mathcal E_\chi(X,\omega)$ by Corollary \ref{cor: E_chi_convex}.  We also pick ${\tilde{\chi}} \in \mathcal W^+_{2p+1}$ as in the previous lemma, for $\eta \in \mathcal E_\chi(X,\omega)$. We can continue the above estimates in the following manner:
\begin{flalign*}
&= C \int_{\{\eta^L \leq -\frac{L}{n+1} \}}  \chi(\eta^L)\o^n_{\eta^L}  \leq C \frac{\chi(L/(n+1))}{\tilde \chi(L/(n+1))}\int_{\{\eta^L \leq -\frac{L}{n+1} \}}  \tilde \chi(\eta^L)\o^n_{\eta^L}\\
& = C \frac{\chi(L/(n+1))}{\tilde \chi(L/(n+1))} E_{\tilde \chi}(\eta^L) \leq C \frac{\chi(L/(n+1))}{\tilde \chi(L/(n+1))} E_{\tilde \chi}(\eta),
\end{flalign*}
where in the last step we used Proposition \ref{prop: Energy_est}. As $\chi(L)/\tilde \chi(L) \to 0$, this implies that the second term in \eqref{uniformlimit} converges to zero.  

Since 
\begin{flalign*}
\bigg| \int_{\{\beta \leq -L \}} \gamma h(\phi_k - \psi_k)\o_{v^1_k}\wedge \ldots \wedge \o_{v^n_k} \bigg| \leq C \int_{\{\beta \leq -L \}}  \chi(\beta)\o_{v^1_k}\wedge \ldots \wedge \o_{v^n_k} \to 0,
\end{flalign*}
the same argument gives that the first term in \eqref{eq: interm_est} converges to zero as well, finishing the proof of \eqref{uniformlimit}.
\end{proof}


Lastly, we prove the \emph{domination principle} for the class $\mathcal E_1(X,\o)$ (recall \eqref{eq: E_p_def}). We mention that these results also hold more generally for the class $\mathcal E(X,\o)$, by a theorem of S. Dinew \cite{Di,BL12}. The short proof below was pointed out to us by C.H. Lu, and it is based on the arguments of \cite{DDL17}.

\begin{proposition}
	\label{prop: domination principle}
	Let $\phi,\psi\in \mathcal E_1(X,\o)$. If $\psi\leq \phi$ almost everywhere with respect to $\o_{\phi}^n$ then $\psi\leq \phi$. 
\end{proposition}

Since $\mathcal E_p(X,\o) \subset \mathcal E_1(X,\o), \ p \geq 1$, we obtain that the domination principle trivially holds for $\phi,\psi\in \mathcal E_p(X,\o), \ p \geq 1$ as well.

\begin{proof}
We can assume without loss of generality that $\phi,\psi<0$.
As $\phi,\psi \in \textup{PSH}(X,\o)$, it suffices to prove that $\psi\leq \phi$ a.e. with respect to $\o^n$. This will then imply that $\psi \leq \phi$ globally.

Suppose that $\omega^n(\{\phi < \psi \}) > 0$. By the next lemma, this implies existence of $u \in \mathcal E_1(X,\o)$ such that $u \leq \psi$ and $\omega_u^n(\{\phi < \psi \})>0$.  

By Corollary \ref{cor: E_chi_convex} we have that $tu + (1-t)\psi \in \mathcal E_1(X,\o)$ for any $t \in [0,1]$. 
By the comparison principle (Proposition \ref{prop: comp_princ_E}) we can write:
$$t^n\int_{\{\phi< t u + (1-t) \psi \}}\o_u^n \leq \int_{\{\phi < tu + (1-t)\psi\}}\o_{tu + (1-t)\psi}^n\leq \int_{\{\phi < tu + (1-t)\psi\}}\omega_\phi^n\leq \int_{\{\phi < \psi\}}\omega_\phi^n=0. $$
We conclude that $0=\o^n_u(\{\phi<tu + (1-t)\psi\}) \nearrow \o^n_u(\{\phi<\psi\})$, as $t \searrow 0$. As a result, $\o^n_u(\{\phi<\psi\})=0$, a contradiction. 
\end{proof}

\begin{lemma} Suppose that $U \subset X$ is a Borel set with non--zero Lebesgue measure and $\psi \in \mathcal E_1(X,\o)$. Then there exists $u\in \mathcal E_1(X,\omega)$ such that $u \leq \psi$ and $\o_u^n(U)>0$.
 \end{lemma}
 \begin{proof} We can assume that $\psi <0$. We fix $C >0$ and let $\psi_C := P(\psi + C,0)$.
 Let $u_j \in \mathcal H_\o$ be a decreasing approximating sequence of $\psi$, that exists by Theorem \ref{thm: BK_approx}. We can assume that $u_j <0$. The partition formula of Proposition \ref{prop: MA_form} gives:
 \[
 \omega_{P(u_j + C,0)}^n \leq \mathbbm{1}_{\{u_j\leq -C\}} \omega_{u_j}^n + \omega^n \leq -\frac{u_j}{C}  \omega_{u_j}^n +  \omega^n.
 \]
Notice that $P(u_j + C,0) \searrow \psi_C$ as $j \to \infty$. Taking the limit, Proposition \ref{prop: MA_cont} gives us the estimate of measures $\o_{\psi_C}^n \leq -\frac{\psi}{C}\o_{\psi}^n + \omega^n$. Using this  we can write:
$$\omega_{\psi_C}^n (X\setminus U) \leq \frac{1}{C} \int_{X \setminus U} |\psi|\o_{\psi}^n + \omega^n(X\setminus U) < \frac{1}{C} \int_X |\psi|\o_{\psi}^n + \omega^n(X\setminus U).$$
As $\o_{\psi_C}^n(X)= \o^n(X)$, for $C$ big enough we get that $\omega_{\psi_C}^n (U)>0$. The choice $u:= \psi_C-C$ satisfies the requirements of the lemma.
\end{proof}


\chapter{The Finsler geometry of the space of K\"ahler potentials}

As follows from the definition \eqref{eq: H_o_def}, the space of K\"ahler potentials $\mathcal H_\o$ is a convex open subset of $C^\infty(X)$, hence one can think of it as a trivial Fr\'echet manifold. As such, one can introduce on $\mathcal H_\o$ a collection of $L^p$ type Finsler metrics whose geometry we will study in this section. If $u \in \mathcal H_\o$ and $\xi \in T_u \mathcal H_\o \simeq C^\infty(X)$, then the $L^p$-length of $\xi$ is given by the following expression:
\begin{equation}\label{eq: Lp_metric_def}
\| \xi\|_{p,u} = \bigg( \frac{1}{\textup{Vol}(X)}\int_X |\xi|^p \o_u^n\bigg)^{\frac{1}{p}},
\end{equation}
where $\textup{Vol}(X) = \int_X \o^n.$ In case $p=2$, we recover the Riemannian geometry of Mabuchi \cite{m} (independently discovered by Semmes \cite{s} and Donaldson \cite{do}). Though these metrics will be our primary object of study, unfortunately the associated weights $\chi_p(l)=l^p/p$ are not twice differentiable for $1 \leq p < 2$. This will cause problems as we shall see, and an approximation with more regular weights will be necessary to carry out even the most basic geometric arguments. For this reason, one needs to first work with even more general Finsler metrics on $\mathcal H_\o$ with smooth weights in $\mathcal W^+_p$ (see \eqref{eq: GrowthEst}) and then use approximation via Propositions \ref{prop: approx_lemma} and \ref{prop: approx_lemma2} to return to the $L^p$ metrics in the end. 

Following our discussion above, we introduce the Orlicz--Finsler length of $\xi$ for any weight $\chi \in \mathcal W^+_p$:
\begin{equation}\label{eq: FinslerDef}
\|\xi\|_{\chi,u}=\inf \Big\{ r > 0 :  \frac{1}{\textup{Vol}(X)}\int_X \chi\Big(\frac{\xi}{r}\Big) \o_u^n \leq \chi(1) \Big\}.
\end{equation}

The above expression introduces a norm on each fiber of $T\mathcal H_\o$, and the length  of a smooth curve $[0,1]\ni t \to \alpha_t \in \mathcal{H}_\o$ is computed by the usual formula:
\begin{equation}\label{eq: curve_length_def}
l_\chi(\alpha_t)=\int_0^1\|\dot \alpha_t\|_{\chi,\alpha_t}dt.
\end{equation}
To clarify, smoothness of $t \to \alpha_t$ simply means that the map $\alpha(t,x)=\alpha_t(x)$ is smooth as a map from $[0,1]\times X$ to $\Bbb R$. 
Furthermore, the distance $d_\chi({u_0},{u_1})$ between ${u_0},{u_1} \in \mathcal H_\o$ is the infimum of the $l_\chi$-length of smooth curves joining ${u_0}$ and ${u_1}$:
\begin{equation}\label{eq: d_chi_def}
d_\chi(u_0,u_1)= \inf\{l_\chi(\gamma_t): \ t \to \gamma_t \ \textup{ is smooth and } \ \gamma_0 = u_0,\gamma_1=u_1 \}.
\end{equation} 
The distance $d_\chi$ is a pseudo--metric (the triangle inequality holds), but $d_\chi(u,v)=0$ may not imply $u=v$, as our setting is infinite--dimensional. We will see in Theorem \ref{thm: XXChenThm} below  that $d_\chi$ is a bona fide metric, but this will require a careful analysis of our Finsler structures.

When dealing with the $L^p$ metric structures \eqref{eq: Lp_metric_def}, the associated curve length and path length metric will be denoted by $l_p$ and $d_p$ respectively. 

As we will see, the different weights in $\mathcal W^+_p$ induce different geometries on $\mathcal H_\o$, however the equation for the shortest length curves between points of $\mathcal H_\o$, the so called geodesics, will be essentially the same. Motivated by this, in the next section we will focus on the $L^2$ Riemannian geometry first, in which case we can get an explicit equation for these geodesics.

\section{Riemannian geometry of the space of K\"ahler potentials}

When $p=2$ the metric of \eqref{eq: Lp_metric_def} is induced by the following non-degenerate inner product:
\begin{equation}\label{eq: Riem_metric_def}
\langle \phi, \psi\rangle_u= \frac{1}{\textup{Vol}(X)}\int_X \phi \psi \o_u^n, \ \  u \in \mathcal H_\o, \ \phi,\psi \in T_u \mathcal H_\o.
\end{equation}
This Riemannian structure was first studied by Mabuchi \cite{m} and later independently by Semmes \cite{s} and Donaldson \cite{do}. For another introductory survey on the Mabuchi geometry of $\mathcal H_\o$ we refer to \cite{bl1}. 

Let us compute the Levi--Civita connection of this metric. For this we choose a smooth curve $[0,1] \ni t \to u_t \in \mathcal H_\o$ and $[0,1] \ni t \to \phi_t,\psi_t \in C^\infty(X)$, two vector fields along $t \to u_t$. 
In the future, when working with time derivatives, we will use the notation $\dot u_t = du_t/dt$, $\ddot u_t = d^2u_t/dt^2$,  etc.

We will identify the Levi--Civita connection $\nabla_{(\cdot)}(\cdot)$, using the fact that it is torsion free and satisfies the following product rule:
\begin{equation}\label{eq: Levi_Civita_product}
\frac{d}{dt}\langle \phi_t,\psi_t\rangle_{u_t}=\langle \nabla_{\dot u_t} \phi_t,\psi_t \rangle_{u_t} + \langle\phi_t,\nabla_{\dot u_t} \psi_t \rangle_{u_t}.
\end{equation}
Using the identities $\frac{d}{dt}\o_{u_t}^n=n i \ddbar \dot u_t \wedge \o_{u_t}^{n-1}=\frac{1}{2}\Delta^{\o_{u_t}} \dot u_t \o_{u_t}^n$ and $-\int_X \langle\nabla^{\o_{u_t}}f, \nabla^{\o_{u_t}}g\rangle\o_{u_t}^n=\int_X f \Delta^{\o_{u_t}}g\o_{u_t}^n$ (see the discussion following \eqref{eq: Lapl_grad_formula} in the appendix), we can start writing:
\begin{flalign*}
\frac{d}{dt}\langle \phi_t,&\psi_t\rangle_{u_t} = \frac{1}{\textup{Vol}(X)}\int_X (\dot \phi_t \psi_t +  \phi_t \dot \psi_t + \frac{1}{2}\phi_t\psi_t\Delta^{\o_{u_t}} \dot u_t) \o_{u_t}^n\\
&= \frac{1}{\textup{Vol}(X)}\int_X (\dot \phi_t \psi_t +  \phi_t \dot \psi_t - \frac{1}{2}\langle \nabla^{\o_{u_t}}(\phi_t\psi_t),\nabla^{\o_{u_t}} \dot u_t \rangle) \o_{u_t}^n\\
&= \frac{1}{\textup{Vol}(X)}\int_X (\dot \phi_t  - \frac{1}{2}\langle \nabla^{\o_{u_t}}\phi_t,\nabla^{\o_{u_t}} \dot u_t \rangle)\psi_t \o_{u_t}^n+\int_X \phi_t(\dot \psi_t  - \frac{1}{2}\langle \nabla^{\o_{u_t}}\psi_t,\nabla^{\o_{u_t}} \dot u_t \rangle) \o_{u_t}^n.
\end{flalign*} 
Comparing with \eqref{eq: Levi_Civita_product}, this line of calculation suggests the following formula for the Levi--Civita connection, and it is easy to see that the resulting connection is indeed torsion free:
\begin{equation}\label{eq: CovDerivative}
\nabla_{\dot u_t} \phi_t = \dot \phi_t - \frac{1}{2}\langle \nabla^{\o_{u_t}} \dot u_t , \nabla^{\o_{u_t}} \phi_t \rangle, \ t \in [0,1].
\end{equation}
This immidiately implies that $t\to u_t$ is a geodesic if and only if $\nabla_{\dot u_t} \dot u_t=0$, or equivalently
\begin{equation}\label{eq: geod_eq_Lev_Civ}
\ddot u_t - \frac{1}{2}\langle \nabla^{\o_{u_t}} \dot u_t , \nabla^{\o_{u_t}} \dot u_t \rangle=0, \ t \in [0,1].
\end{equation}
As discovered independently by Semmes \cite{s} and Donaldson \cite{do}, the above equation can be understood as a complex Monge--Amp\`ere equation. For this one has to introduce the trivial complexification $u \in C^\infty(S\times X)$, using the formula 
$$u(s,x) = u_{\textup{Re }s}(x),$$
where $S = \{0 < \textup{Re s} <1 \} \subset \Bbb C$ is the unit strip. We pick a coordinate patch $U \subset X$, where the metric $\omega_u$ has a potential $g_u \in C^\infty(X)$, i.e., $\o_u = i\ddbar g_u=i{g_u}_{j\bar k}dz_j \wedge d \bar z_k$. Then on $[0,1] \times U$ the geodesic equation \eqref{eq: geod_eq_Lev_Civ} is seen to be equivalent to $\ddot u - g_u^{j\bar k}\dot {u}_j\dot {u}_{\bar k}=0$, where $g_u^{j\bar k}$ is the inverse of ${g_u}_{j\bar k}$. By involving the complexified variable $s$, this identity is further seen to be equivalent to 
$u_{s\bar s} - g_u^{j\bar k} u_{j\bar s} u_{s\bar k}=0$ on $S\times U$. After multiplying with $\det({g_u}_{j\bar k})$, this last equation can be written globally on $S \times X$ as:
\begin{equation}\label{eq: geodesic_eq}
(\pi^*\omega + i \partial \overline{\partial}u)^{n+1}=0.
\end{equation}
where $\pi: S \times X \to X$ is the projection map to the second component. Consequently, the problem of joining the potentials $u_0,u_1 \in \mathcal H_\o$ with a smooth geodesic equates to finding a smooth solution $u \in \mathcal C^\infty({S} \times X)$ to the following boundary value problem:
\begin{equation}\label{eq: BVPGeod}
\begin{cases}
(\pi^* \o + i \partial \overline{\partial}u)^{n+1}=0, \\ 
\o + i \partial \overline{\partial}u\big|_{\{s \}\times X} >0, \ s \in S, \\
u(t+ir,x) =u(t,x) \ \forall x \in X, t \in (0,1), \ r \in \Bbb R.\\
\lim_{s \to 0}u(s,\cdot)=u_0 \textup{ and }\lim_{s \to 1}u(s,\cdot)=u_1.
\end{cases}
\end{equation}
To be precise, here $\lim_{s \to 0,1}u(s,\cdot)=u_{0,1}$ simply means that $u(s,\cdot)$ converges uniformly to $u_{0,1}$ as $ s \to 0,1$.
Unfortunately, as detailed below, this boundary value problem does not usually have smooth solutions, but a unique weak solution (in the sense of Bedford--Taylor) does exist. Instead, one replaces \eqref{eq: geod_eq_Lev_Civ} with the following equation for $\varepsilon$\emph{-geodesics}:
\begin{equation}\label{eq: eps_geod_eq_Lev_Civ}
(\ddot u^\varepsilon_t - \frac{1}{2}\langle \nabla^{\o_{u^\varepsilon_t}} \dot u^\varepsilon_t , \nabla^{\o_{u^\varepsilon_t}} \dot u^\varepsilon_t \rangle)\o_{u^\varepsilon_t}^n=\varepsilon \o^n, \ t \in [0,1].
\end{equation}
By an elementary calculation, similar to the one giving \eqref{eq: geodesic_eq}, the associated boundary value problem for this equation becomes:
\begin{equation}\label{eq: epsBVPGeod}
\begin{cases}
(\pi^* \o + i \partial \overline{\partial}u^\varepsilon)^{n+1}=\frac{\varepsilon}{4} (i ds \wedge d\bar{s} + \pi^* \o)^{n+1},\\
u^\varepsilon(t+ir,x) =u^\varepsilon(t,x) \ \forall x \in X, t \in (0,1), r \in \Bbb R.\\
\lim_{s \to 0}u^\varepsilon(s,\cdot)=u_0 \textup{ and }\lim_{s \to 1}u^\varepsilon(s,\cdot)=u_1.\end{cases}
\end{equation}
Since $\o_{u_0},\o_{u_1}>0$, we see that $\pi^* \o + i \partial \overline{\partial}u^\varepsilon>0$ on $S \times X$. As a result the condition $\o + i \partial \overline{\partial}u|_{\{s \}\times X} >0, \ s \in S$ is automatically satisfied.
In contrast with \eqref{eq: BVPGeod}, this Dirichlet problem is elliptic and its solutions are smooth, moreover we have the following regularity result due to X.X. Chen \cite{c1} (with complements by Blocki \cite{bl1}):
\begin{theorem}\label{thm: ueps_estimates}The boundary value problem \eqref{eq: epsBVPGeod} admits a unique smooth solution $u^\varepsilon \in C^\infty(\overline{S} \times X)$ with the following bounds that are independent of $\varepsilon>0$:
\begin{equation}\label{eq: ueps_estimates}
\| u^\varepsilon\|_{C^0(\overline{S} \times X)},\| u^\varepsilon\|_{C^{1}(\overline{S} \times X)}, \|\Delta u^\varepsilon \|_{C^0(\overline{S} \times X)} \leq C(\|u_0\|_{C^{3}(X)},\|u_1\|_{C^{3}(X)},X,\o),
\end{equation}
\begin{equation*}
\| \dot u^\varepsilon\|_{C^0(\overline{S} \times X)}, \leq C(\|u_0\|_{C^{2}(X)},\|u_1\|_{C^{2}(X)},X,\o).
\end{equation*}
\end{theorem}

Recall from our discussion preceding \eqref{eq: H_01bar1_def} that having a bound on $\Delta  u$ is equivalent to bounding all mixed second order complex derivatives of $u$ on $S\times X$.
We refer to \cite[Theorem 12]{bl1} for an elaborate treatment of Theorem \ref{thm: ueps_estimates} (see also the survey paper \cite{bo}). 

The estimate for $\dot u^\varepsilon$ is argued in the last step of the proof of \cite[Lemma 16]{bl1}, as a consequence of the comparison principle. Let us however mention that a much more precise estimate  for $\nabla u^\varepsilon$ is argued in \cite[Theorem 1]{bl5}.

In addition, recently Chu--Tosatti--Weinkove have showed that one can more generally bound the Hessian of $u^\varepsilon$ \cite{ctw} independently of $\varepsilon$. Additionally, it was shown by Berman--Demailly \cite{bd} and He \cite{he1} that one can in fact bound each $\Delta^\o u_t^\varepsilon, \ t \in [0,1],$ using  bounds on  $\Delta^\o u_0$ and $\Delta^\o u_1$.

Using the Bedford--Taylor interpretation of $(\pi^* \o + i \partial \overline{\partial}u^\varepsilon)^{n+1}$ as a Borel measure, the boundary value problems \eqref{eq: BVPGeod} and \eqref{eq: epsBVPGeod} can be stated for $u, u^\varepsilon \in \textup{PSH}(S \times X, \pi^* \o)$ that are only bounded and not necessarily smooth. Additionally, after pulling back by the $\log$ function,  we can equivalently state \eqref{eq: BVPGeod} and \eqref{eq: epsBVPGeod} as boundary value problems with circle--invariant solutions on $A \times X$, where $ A$ is the annulus $\{e^0 < |z| < e^1 \} \subset \Bbb C$. Consequently, the next result (whose proof closely follows \cite[Theorem 21]{bl1}) will assure  that uniqueness of solutions to \eqref{eq: BVPGeod} and \eqref{eq: epsBVPGeod} holds not only for smooth solutions, but also for solutions that are merely in $\textup{PSH}(S \times X, \pi^* \o) \cap L^\infty$:

\begin{theorem}\label{thm: uniqueness_BVP} Suppose $M$ is a $k$ dimensional compact complex manifold with smooth boundary and K\"ahler form $\eta$. If $u,v \in \textup{PSH}(M,\eta) \cap L^\infty$ with $\liminf_{x \to \partial M}(u-v)(x) \geq 0$ and $(\eta + i\ddbar v)^k \geq (\eta + i\ddbar u)^k $, then $u \geq v$ on $M$.
\end{theorem} 

Since $\pi^* \omega$ is only non-negative on $S \times X$, the above result is not directly applicable to our situation.  This small inconvenience can be fixed by taking $\eta := \pi^* \o + i\partial \overline{\partial} g$, where $g$ is a smooth function on $S$, such that $i\partial \overline{\partial} g >0$, $g(t + ir) =g(t)$, and $g(ir)=g(1+ir)=0$.

\begin{proof} Let $\delta >0$ and $v_\delta := \max(u,v -\delta) \in \textup{PSH}(M,\eta) \cap L^\infty$. Then $v_\delta = u$ near $\partial M$. To conclude the proof, it is enough to show that $v_\delta =u$ on $M$. 

From Bedford--Taylor theory (see \cite[Theorem 2.2.10]{bl3}) it follows that
$$\eta_{v_\delta}^k \geq \mathbbm{1}_{\{u \geq v-\delta\}}\eta_{u}^k + \mathbbm{1}_{\{u < v-\delta\}}\eta_{v}^k \geq \eta_{u}^k.$$

As $v_\delta$ and $u$ agree near the boundary and $v_\delta \geq u$, we can integrate by parts \cite[Theorem 1.3.4]{bl3} and write:
$$0 \leq \int_M (v_\delta-u) (\eta_{v_\delta}^k-\eta_u^k) = -\sum_{j=0}^{k-1} \int_M i \partial (v_\delta-u) \wedge \bar \partial (v_\delta-u) \wedge \eta_{u}^j \wedge \eta_{v_\delta}^{k-j-1}.$$
As each of the summands above is non-negative, it follows that
\begin{equation}\label{eq: somezeroid} 
\int_M i \partial (u-v_\delta) \wedge \bar \partial (u-v_\delta) \wedge \eta_{u}^j \wedge \eta_{v_\delta}^{k-j-1}=0, \ j \in \{0,1,\ldots,k-1\}.
\end{equation}
The equality $v_\delta=u$ will follow, if we can show that $\int_M i \partial (u-v_\delta)  \wedge \bar \partial (u-v_\delta) \wedge \eta^{k-1}=0$. Indeed, this would imply that all Sobolev derivatives of $u-v_\delta$ are zero. This last identity will be established as the last step in an inductive argument showing that 
\begin{equation}\label{eq: firstzeroid}
\int_M i \partial (u-v_\delta) \wedge \bar \partial (u-v_\delta) \wedge \eta_{u}^j \wedge \eta^{k-j-1}=0, \ j \in \{0,1,\ldots,k-1\}.\end{equation}
This identity holds for $j=k-1$ by the above. As all the steps are carried out similarly, we only show that 
\begin{equation} \label{eq: L^2normgradzero} \int_M i \partial (u-v_\delta) \wedge \bar \partial (u-v_\delta) \wedge \eta_{u}^{k-2} \wedge \eta=0.
\end{equation}
Denote $f = u - v_\delta$. Using \eqref{eq: firstzeroid} and integration by parts we can write:
\begin{flalign} \label{eq: random_estimate}
\int_{M} i \partial f  \wedge \dbar f \wedge \eta_{u}^{k-2} \wedge \eta &= \int_{M} i \partial f \wedge \dbar f \wedge \eta_{u}^{k-1}- \int_{M} i \partial f \wedge \dbar f \wedge i\ddbar u \wedge  \eta_{u}^{k-2} \nonumber\\
&=- \int_{M} i \partial f \wedge  i\ddbar u \wedge  \dbar f \wedge\eta_{u}^{k-2} \nonumber\\
&= -\int_{M} i \partial f \wedge \dbar u \wedge (\eta_u - \eta_{v_
\delta}) \wedge  \eta_{u}^{k-2} \nonumber\\
&= -\int_{M} i \partial f \wedge \dbar u \wedge \eta_{u}^{k-1}+\int_{M} i \partial f \wedge \dbar u \wedge \eta_{v_\delta} \wedge \eta_{u}^{k-2}.
\end{flalign}
Using the Cauchy--Schwarz inequality of Bedford--Taylor theory \cite{bl3} we can write
$$\bigg|\int_{M} i \partial f \wedge \dbar u \wedge \eta_{u}^{k-1}\bigg|^2 \leq \int_{M} i \partial u \wedge \dbar u \wedge \eta_{u}^{k-1} \cdot \int_{M} i \partial f \wedge \dbar f \wedge \eta_{u}^{k-1},$$
hence by \eqref{eq: somezeroid}, it follows that the first term in \eqref{eq: random_estimate} is zero. It can be shown similarly that the second term in \eqref{eq: random_estimate} is zero as well, implying \eqref{eq: L^2normgradzero} and finishing the proof.
\end{proof}

Now we return to \eqref{eq: BVPGeod} and \eqref{eq: epsBVPGeod}. As each $u^\varepsilon$ is invariant in the $i\Bbb R$ direction, plurisubharmonicity of $u^\varepsilon$ implies that $t \to u^\varepsilon(t,x)$ is convex for each $x \in X$. As each $u^\varepsilon$ solves the boundary value problem \eqref{eq: epsBVPGeod}, we additionally obtain that 
\begin{equation}\label{eq: ueps_upper_bound}
u^{\varepsilon}(s,x) \leq (1-\textup{Re }s) u_0(x) + \textup{Re }s u_1(x). 
\end{equation}
On top of proving uniqueness of general solutions to \eqref{eq: BVPGeod} and \eqref{eq: epsBVPGeod}, the above theorem also shows that the family $\{ u^\varepsilon\}_\varepsilon$ increases as $\varepsilon \searrow 0$. Using this, \eqref{eq: ueps_upper_bound} and the continuity of the complex Monge--Amp\`ere operator along increasing sequences \cite[Theorem 2.2.5]{bl3} it follows that solutions of \eqref{eq: epsBVPGeod} approximate solutions of \eqref{eq: BVPGeod}, i.e., 
\begin{equation}\label{eq: epsgeod_limit}
\lim_{\varepsilon \to 0} u^\varepsilon = u.
\end{equation}
From here, an application of the Arzela--Ascoli compactness theorem yields that the estimates of \eqref{eq: ueps_estimates} also hold for solutions of \eqref{eq: BVPGeod}:
\begin{theorem}\label{thm: u_estimates} For the unique solution $u \in \textup{PSH}(S \times X,\pi^* \o)$ of \eqref{eq: BVPGeod} the following estimates hold:
\begin{equation}\label{eq: u_estimates}
\| u\|_{C^0(\overline{S} \times X)},\| u\|_{C^{1}(\overline{S} \times X)}, \|\Delta u \|_{C^0(\overline{S} \times X)} \leq C(\|u_0\|_{C^{3}(X)},\|u_1\|_{C^{3}(X)},X,\o).
\end{equation}
\end{theorem}

By the examples of  \cite{lv,dl} this regularity result is essentially optimal, as no higher order estimates are possible for weak solutions.  Often we will refer to the unique solution of  \eqref{eq: BVPGeod} as the \emph{weak} $C^{1,\bar 1}$--\emph{geodesic} joining $u_0,u_1 \in \mathcal H_\o$, and it will be denoted as
\begin{equation}\label{eq: weak_geod_def}
[0,1] \ni t \to u_t \in \mathcal H_\o^{1,\bar 1}.
\end{equation}

Moving on, as follows from standard elliptic PDE theory, given a smooth family of boundary data for \eqref{eq: epsBVPGeod} one obtains a corresponding smooth family of solutions. As another simple consequence of the uniqueness theorem, we obtain a uniform bound on the rate of change of this family as the parameter changes:  

\begin{corollary}[\cite{c1}]\label{cor: duds_estimate} 
Suppose $[0,1] \ni \rho \to u_0(\cdot,\rho),u_1(\cdot,\rho) \in \mathcal H_\o$ is smooth family of boundary data for \eqref{eq: epsBVPGeod}. For the smooth family of solutions $[0,1] \ni \rho \to u^\varepsilon(\cdot,\cdot,\rho) \in C^\infty(\overline {S} \times X)$ we have the following estimate for all  $\varepsilon>0$:
$$\Big\|\frac{du^\varepsilon}{d\rho}\Big\|_{C^0( \overline{S} \times X \times [0,1])} \leq \max\Big(\Big\|\frac{du_0}{d\rho}\Big\|_{C^0(X \times [0,1])},\Big\|\frac{du_1}{d\rho}\Big\|_{C^0(X \times [0,1])}\Big).$$
\end{corollary}
\begin{proof}Let $\rho_0 \in [0,1]$ and $C > \max(\|{du_0}/{d\rho}\|_{C^0(X \times [0,1])},\|{du_1}/{d\rho}\|_{C^0(X \times [0,1])}).$ Then Theorem \ref{thm: uniqueness_BVP} implies that
$$u^\varepsilon(s,x,\rho_0) - C\delta \leq  u^\varepsilon(s,x,\rho_0+\delta) \leq u^\varepsilon(s,x,\rho_0) + C\delta, \ (s,x) \in \overline{S} \times X$$
for all $\delta \in \Bbb R$ such that $\rho_0 + \delta \in [0,1]$. This gives that ${|u^\varepsilon(s,x,\rho_0+\delta)-u^\varepsilon(s,x,\rho_0)|}/{\delta} \leq C,$
implying the desired estimate.
\end{proof}

As a last consequence of the uniqueness theorem, we note the following concrete estimate for the tangent vectors of weak $C^{1,\bar 1}$-geodesics: 
\begin{lemma}\label{lem: dot_u_est} Given $u_0,u_1 \in \mathcal H_\o$, let $[0,1]  \ni t \to u_t \in \mathcal H_\o$ be the weak $C^{1,\bar 1}$--geodesic joining $u_0,u_1$. Then the following estimate holds:
\begin{equation}\label{eq: dot_u_est}
\| \dot u_t\|_{C^0(X)} \leq \| u_0 - u_1\|_{C^0(X)}, \ \ \ t \in [0,1].
\end{equation}
\end{lemma}
\begin{proof} Let $C := \sup_X |u_0 - u_1|$. Convexity in the $t$ variable implies that $\dot u_0 \leq \dot u_t \leq \dot u_1$. Hence we only need to show that $-C \leq \dot u_0$ and $\dot u_1 \leq C$.

Examining \eqref{eq: geod_eq_Lev_Civ}, it is clear that the curve $[0,1] \in t \to v_t:= u_0 - C t \in \mathcal H_\o$ is a smooth geodesic connecting $u_0$ and $u_0 - C$. Hence its complexification $v$ is a solution to \eqref{eq: BVPGeod}.

As $u_0 - C \leq u_1$, Theorem \ref{thm: uniqueness_BVP} gives that that $v_t \leq u_t, \ t \in [0,1]$, implying that $-C \leq \dot u_0$.

The same trick applied to the ``reverse" $C^{1,\bar 1}$--geodesic $t \to u_{1-t}$ gives $-C \leq - \dot u_1$, finishing the proof.
\end{proof}

\section{The Orlicz geometry of the space of K\"ahler potentials}

As discussed in the beginning of the present chapter, in our study of $L^p$ Finsler metrics on $\mathcal H_\o$, we need to return to the full generality of Orlicz--Finsler metrics \eqref{eq: FinslerDef} on $\mathcal H_\o$ with weight in $\mathcal W^+_p$. We will do this in this section, and our main theorem connects the $d_\chi$ pseudo--distance (see \eqref{eq: d_chi_def}) with the weak $C^{1,\bar 1}$ geodesic of the $L^2$ Mabuchi geometry (see \eqref{eq: weak_geod_def}), in the process showing that $d_\chi$ is indeed a bona fide metric:

\begin{theorem}\textup{\cite[Theorem 1]{da2}}\label{thm: XXChenThm} If $\chi \in \mathcal W^+_p, \ p \geq 1$ then $(\mathcal H_\o,d_\chi)$ is a metric space and for any $u_0,u_1 \in \mathcal H_\o$ the weak $C^{1,\bar 1}$ geodesic $t \to u_t$ connecting $u_0,u_1$ satisfies:
\begin{equation}\label{eq: ChiDistGeodFormula}
d_\chi(u_0,u_1)=\|\dot u_t \|_{\chi,u_t}, \ t \in [0,1].
\end{equation}
\end{theorem}

Although the metric spaces $(\mathcal H_\o, d_\chi)$ are not quasi isometric for different $\chi$ (as we will see, they have different metric completions), it is remarkable that the \emph{same} weak $C^{1 \bar 1}$--geodesic is ``length minimzing'' for \emph{all} $d_\chi$ metric structures. Additionally, as we will see in the next sections, this same curve is an honest metric geodesic in the completion of each space $(\mathcal H_\o, d_\chi)$.

In the proof of Theorem \ref{thm: XXChenThm}, we first show the result for Finsler metrics with smooth weight $\chi$, and afterwards use approximation via Proposition \ref{prop: approx_lemma} and Proposition \ref{prop: approx_lemma2} to establish the result for all metrics with weight in $\mathcal W^+_p$, which includes as particular case the $L^p$ metrics. 

In the particular case of the $L^2$ metric, Theorem \ref{thm: XXChenThm} was obtained by X.X. Chen \cite{c1} and our proof in case of Finsler metrics with smooth weight follows his ideas and a careful differential analysis of Orlicz norms.

To ease the technical nature of future calculations, for the rest of this section we assume the normalizing condition
\begin{equation}\label{eq: vol_normalization}
V:=\textup{Vol}(X)=\int_X \o^n = 1.
\end{equation}
This can always be achieved by rescaling the K\"ahler metric $\o$.
Given a differentiable normalized Young weight $\chi$, the differentiability of the associated norm $\| \cdot\|_\chi$ (see \eqref{eq: OrliczNormDef}) is well understood (see \cite[Chapter VII]{rr}). As a first step in proving Theorem \ref{thm: XXChenThm}, we adapt \cite[Theorem VII.2.3]{rr} to our setting:

\begin{proposition}\textup{\cite[Proposition 3.1]{da2}}\label{prop: OrliczNormDiff} Suppose $\chi \in \mathcal W^+_p \cap C^\infty(\Bbb R)$. Given a smooth curve $(0,1) \ni t \to u_t \in \mathcal H_\o$, and a vector field $(0,1) \ni t \to f_t \in C^\infty(X)$ along this curve with $f_t \not \equiv 0, \ t\in (0,1)$, the following formula holds:
\begin{equation}\label{eq: OrliczNormDiffEq}
\frac{d}{dt}\|f_t\|_{\chi,u_t} = \frac{\int_X \chi' \Big(\frac{f_t}{\|f_t\|_{\chi,u_t}}\Big)\nabla_{\dot u_t} f_t \omega_{u_t}^n}{\int_X \chi'\Big(\frac{f_t}{\|f_t\|_{\chi,u_t}}\Big) \frac{f_t}{\|f_t\|_{\chi,u_t}}  \omega_{u_t}^n},
\end{equation}
where $\nabla_{(\cdot)}(\cdot)$ is the covariant derivative from \eqref{eq: CovDerivative}.
\end{proposition}
\begin{proof}
We introduce the smooth function $F: \Bbb R^+ \times (0,1) \to \Bbb R$ given by
$$F(r,t) = \int_X \chi\Big(\frac{f_t}{r}\Big)\o_{u_t}^n.$$
As $\chi \in \mathcal W^+_p$, by \eqref{eq: GrowthControl} we have $\chi'(l) >0, \ l >0$ and $\chi'(l) <0, \ l <0$. As $t \to f_t$ is non-vanishing, it follows that
$$\frac{d}{dr} F(r,t)=- \frac{1}{r^2}\int_X {f_t} \chi'\Big(\frac{f_t}{r}\Big) \omega_{u_t}^n < 0$$
for all $r > 0,\ t \in (0,1)$. Using the fact that 
$F(\|f_t\|_{\chi,u_t},t) = \chi(1)$ (see \eqref{eq: OrliczNormId}),  an application of the implicit function theorem yields that the map $t \to \|f_{t}\|_{\chi,u_t}$ is differentiable and the following formula holds:
$$\frac{d}{dt}\|f_t\|_{\chi,u_t} =
\frac{\int_X \Big[\dot f_t \chi'\Big(\frac{f_t}{\|f_t\|_{\chi,u_t}}\Big) + \frac{1}{2} \|f_t\|_{\chi,u_t} \chi\Big(\frac{f_t}{\|f_t\|_{\chi,u_t}}\Big) \Delta^{\omega_{u_t}} \dot u_t \Big] \omega_{u_t}^n}{\int_X \frac{f_t}{\|f_t\|_{\chi,u_t}} \chi'\Big(\frac{f_t}{\|f_t\|_{\chi,u_t}}\Big) \omega_{u_t}^n}.$$
Recalling the formula for the covariant derivative \eqref{eq: CovDerivative}, an integration by parts yields \eqref{eq: OrliczNormDiffEq}.
\end{proof}

The estimate for $\varepsilon$--geodesics (see \eqref{eq: eps_geod_eq_Lev_Civ}) from the following technical lemma will be of great use in our later study:

\begin{lemma} \label{lem: dotintegralest} Suppose $\chi \in \mathcal W^+_p$ and $u_0,u_1 \in \mathcal H_\o$. Then  the $\varepsilon$--geodesic $[0,1] \ni t \to u^\varepsilon_t \in \mathcal H_\o$ connecting $u_0,u_1$ satisfies the following estimate:
\begin{equation}\label{eq: dotintegralest}\int_X \chi( \dot u^\varepsilon_t ) \o_{u^\varepsilon_t}^n \geq \max\Big(\int_X \chi(\min(u_1 - u_0,0))\o_{u_0}^n,\int_X \chi(\min(u_0 - u_1,0))\o_{u_1}^n\Big) -\varepsilon R > 0,
\end{equation}
for all $t \in [0,1]$, where $R:=R(\chi, \| u_0\|_{C^2}, \| u_1\|_{C^2})$.
\end{lemma}
\begin{proof}
As $t \to u^\varepsilon_t(x)$ is convex for any $x \in X$, on the set $\{ u_0 \geq u_1 \}$ the estimate $\dot u^\varepsilon_0 \leq u_1 - u_0\leq 0$ holds, hence
$$\int_X \chi( \dot u^\varepsilon_0 ) \o_{u_0}^n \geq \int_X \chi(\min(u_1 - u_0,0))\o_{u_0}^n.$$
We can similarly deduce that
$$\int_X \chi( \dot u^\varepsilon_1 ) \o_{u_1}^n \geq \int_X \chi(\min(u_0 - u_1,0))\o_{u_1}^n.$$
Let us assume momentarily that $\chi$ is smooth. For $t \in [0,1]$, using the Riemannian connection $\nabla_{(\cdot)}(\cdot)$ (see \eqref{eq: CovDerivative}) and the fact that $t \to u^\varepsilon_t$ is an $\varepsilon$--geodesic (see \eqref{eq: eps_geod_eq_Lev_Civ}), we can write:
$$
\Big|\frac{d}{dt} \int_X \chi(\dot u^\varepsilon_t) \o_{u^\varepsilon_t}^n\Big|=\Big|\int_X \chi'(\dot u^\varepsilon_t) \nabla_{\dot u^\varepsilon_t} \dot u^\varepsilon_t\o_{u^\varepsilon_t}^n\Big|= \varepsilon \Big|\int_X \chi'(\dot u^\varepsilon_t) \o^n\Big| \leq \varepsilon R(\chi, \| u_0\|_{C^2}, \| u_1\|_{C^2}),$$
where in the last estimate we have used that $\dot u^\varepsilon_t$ is uniformly bounded in terms of  $\| u_0\|_{C^2}, \| u_1\|_{C^2}$ (Theorem \ref{thm: ueps_estimates}). For  general $\chi \in \mathcal W^+_p$, since $\chi$ is convex, an approximation argument yields the same estimate. 

After putting together the last three estimates, \eqref{eq: dotintegralest} follows. 
\end{proof}

As a consequence of the previous two results we obtain the following corollary:

\begin{corollary} \label{cor: EpsGeodDiffCor} Suppose $\chi \in \mathcal W^+_p \cap C^\infty(\Bbb R)$ and $u_0,u_1 \in \mathcal H_\o, \ u_0 \neq u_1$. Then there exists $\varepsilon_0 >0$ dependent on upper bounds for $\| u_0\|_{C^2(X)}, \| u_1\|_{C^2(X)}$ and lower bounds for $\|\chi(u_1 -u_0)\|_{L^1(\o^n)}, \o_{u_0}^n /\o^n$ and $\o_{u_1}^n /\o^n$, such that for all $\varepsilon \in (0,\varepsilon_0)$ the $\varepsilon$--geodesic $[0,1] \ni t \to u^\varepsilon_t \in \mathcal H_\o$ of \eqref{eq: epsBVPGeod}, connecting $u_0,u_1$ satisfies:
\begin{equation}\label{eq: EpsGeodDiffEq}
\frac{d}{dt}\|\dot u^\varepsilon_t\|_{\chi,u^\varepsilon_t}=\varepsilon\frac{\int_X \chi'\Big(\frac{\dot u^\varepsilon_t}{\|\dot u^\varepsilon_t\|_{\chi,u^\varepsilon_t}}\Big)\o^n} {\int_X \frac{\dot u^\varepsilon_t}{\|\dot u^\varepsilon_t\|_{\chi,u^\varepsilon_t}} \chi'\Big(\frac{\dot u^\varepsilon_t}{\|\dot u^\varepsilon_t\|_{\chi,u^\varepsilon_t}}\Big) \omega_{u^\varepsilon_t}^n}, \ t \in [0,1].
\end{equation} 
\end{corollary}
\begin{proof}This is a simple application of the formula of Proposition \ref{prop: OrliczNormDiff} for the $\varepsilon$--geodesic $t \to u^\varepsilon_t$ and the vector field $t \to f_t := \dot u_t$. Indeed, by the previous lemma, one can choose $\varepsilon_0>0$ as indicated, so that $\dot u_t$ is non-vanishing for all $t \in [0,1]$, hence the assumptions of Proposition \ref{prop: OrliczNormDiff} are satisfied.
\end{proof}

Continuing to focus on smooth weights $\chi$, we establish concrete bounds for the $\chi$--length of tangent vectors along the $\varepsilon$--geodesics and their derivatives. This is the analog of \cite[Lemma 13]{bl1} in our more general setting:

\begin{proposition} \label{prop: EpsGeodTanEst} Suppose $\chi \in \mathcal W^+_p$ and $u_0,u_1 \in \mathcal H_\o, \ u_0 \neq u_1$. Then there exists $\varepsilon_0 >0$  such that for any $\varepsilon \in(0,\varepsilon_0)$ the $\varepsilon$--geodesic $[0,1] \ni t \to u^\varepsilon_t \in \mathcal H_\o$ connecting $u_0,u_1$ satisfies:\vspace{0.1cm}\\
\noindent(i) $\| \dot u^\varepsilon_t\|_{\chi,u^\varepsilon_t} > R_0, \ t \in [0,1]$\vspace{0.1cm}.\\
\noindent(ii) $\big|\frac{d}{dt}\|\dot u^\varepsilon_t\|_{\chi,u^\varepsilon_t}\big| \leq \varepsilon R_1, \ t \in [0,1]$,\vspace{0.1cm}\\
where $\varepsilon_0>0,R_0>0,R_1>0$ depend on upper bounds for $\| u_0\|_{C^2(X)}$, $\| u_1\|_{C^2(X)}$ and lower bounds for $\|\chi(u_1 -u_0)\|_{L^1(\o^n)}, \o_{u_0}^n /\o^n$ and $\o_{u_1}^n /\o^n$.
\end{proposition}
\begin{proof}The estimate of (i) follows from \eqref{eq: dotintegralest} and Proposition \ref{prop: NormIntegralEst}. To establish (ii) we assume momentarily that $\chi \in  C^\infty(\Bbb R)$, and we shrink $\varepsilon_0$ enough to satisfy the requirements of Corollary \ref{cor: EpsGeodDiffCor}. Using the Young identity \eqref{eq: YoungIdIneq} we can write:
\begin{flalign}\label{epstangentest}
\Big|\frac{d}{dt}\|\dot u^\varepsilon_t\|_{\chi,u^\varepsilon_t}\Big|&=\varepsilon\frac{\Big|\int_X \chi'\Big(\frac{\dot u^\varepsilon_t}{\|\dot u^\varepsilon_t\|_{\chi,u^\varepsilon_t}}\Big)\o^n\Big|} {\int_X \frac{\dot u^\varepsilon_t}{\|\dot u^\varepsilon_t\|_{\chi,u^\varepsilon_t}} \chi'\Big(\frac{\dot u^\varepsilon_t}{\|\dot u^\varepsilon_t\|_{\chi,u^\varepsilon_t}}\Big) \omega_{u^\varepsilon_t}^n}=\varepsilon\frac{\Big|\int_X \chi'\Big(\frac{\dot u^\varepsilon_t}{\|\dot u^\varepsilon_t\|_{\chi,u^\varepsilon_t}}\Big)\o^n\Big|} {\chi(1)+\int_X \chi^*\Big(\chi'\Big(\frac{\dot u^\varepsilon_t}{\|\dot u^\varepsilon_t\|_{\chi,u^\varepsilon_t}}\Big)\Big) \omega_{u^\varepsilon_t}^n}\leq \nonumber \\
&\leq \frac{\varepsilon}{\chi(1)}\int_X \chi'\Big(\frac{\dot u^\varepsilon_t}{\|\dot u^\varepsilon_t\|_{\chi,u^\varepsilon_t}}\Big)\o^n.
\end{flalign}
Using (i) and the fact that $\dot u^\varepsilon_t$ is uniformly bounded in terms of  $\| u_0\|_{C^2}, \| u_1\|_{C^2}$ (Theorem \ref{thm: ueps_estimates}) the estimate of (ii) follows for smooth $\chi$. For general $\chi$ we simply use the approximation result Proposition \ref{prop: approx_lemma} to conclude that \eqref{epstangentest} still holds, where $\frac{d}{dt}$ is interpreted as a  Lipschitz derivative.
\end{proof}

We can now establish the main geometric estimate for $\varepsilon$--geodesics, which generalizes the corresponding statement for the $L^2$--metric (\cite[Theorem 14]{bl1}):

\begin{proposition}[\cite{da2}] Suppose $\chi \in \mathcal W^+_p \cap C^\infty(\Bbb R)$, $[0,1] \ni s \to \psi_s \in\mathcal H_\o$ is a smooth curve, $\phi \in \mathcal H_\o \setminus \psi([0,1])$ and  $\varepsilon >0$.  We denote by $u^\varepsilon \in C^\infty([0,1]\times [0,1] \times X)$ the smooth function for which $[0,1] \ni t \to u^\varepsilon_t(\cdot,s):= u^\varepsilon(t,\cdot,s) \in \mathcal H_\o$ is the $\varepsilon$--geodesic connecting $\phi$ and $\psi_s$. There exists $\varepsilon_0(\phi,\psi) >0$ such that for any $\varepsilon \in(0,\varepsilon_0)$ the following holds:
$$l_\chi(u^\varepsilon_t(\cdot,0)) \leq l_\chi(\psi) + l_\chi(u^\varepsilon_t(\cdot,1)) + \varepsilon R,$$
for some $R(\phi,\psi,\chi,\varepsilon_0) >0$ independent of $\varepsilon >0$.
\end{proposition}

\begin{proof}Whenever it will not cause confusion, we will drop sub/superscripts when addressing the $\varepsilon$--geodesic $t \to u^\varepsilon_t(\cdot,s)$. To avoid cumbersome notation, derivatives in the $t$--direction will be denoted by dots, derivatives in the $s$--direction will be denoted by $d/ds$, and sometimes we also omit dependence on $(t,s)$.  

Fix $s \in [0,1]$. By Proposition \ref{prop: OrliczNormDiff} and Proposition \ref{prop: EpsGeodTanEst}(i) there exists $\varepsilon_0(\phi,\psi)>0$ such that for $\varepsilon \in (0,\varepsilon_0)$ the following holds:
\begin{flalign*}
\frac{d}{ds}l_\chi(u_t(\cdot,s))&= \int_0^1 \frac{d}{ds} \|\dot u(t,\cdot, s))\|_{\chi,u(t,s)}dt = \int_0^1 \frac{\int_X \chi' \Big(\frac{\dot u}{\|\dot u\|_{\chi,u}}\Big)\nabla_{\frac{du}{ds}}\dot u \omega_{u}^n}{\int_X \chi'\Big(\frac{\dot u}{\|\dot u\|_{\chi,u}}\Big) \frac{\dot u}{\|\dot u\|_{\chi,u}}  \omega_{u}^n}dt
\end{flalign*}
Using the Young identity \eqref{eq: YoungIdIneq} and the fact that $\nabla_{(\cdot)}(\cdot)$ is a Riemannian connection, we can continue:
\begin{flalign}\label{eq: anothercalc}
&=\int_0^1 \frac{\int_X \chi' \Big(\frac{\dot u}{\|\dot u\|_{\chi,u}}\Big)\nabla_{\frac{du}{ds}}\dot u \omega_{u}^n}{\chi(1) + \int_X \chi^*\Big(\chi'\Big(\frac{\dot u}{\|\dot u\|_{\chi,u}}\Big)\Big)\omega_{u}^n}dt \nonumber\\
&=\int_0^1 \frac{\int_X \chi' \Big(\frac{\dot u}{\|\dot u\|_{\chi,u}}\Big)\nabla_{\dot u}\frac{du}{ds} \omega_{u}^n}{\chi(1) + \int_X \chi^*\Big(\chi'\Big(\frac{\dot u}{\|\dot u\|_{\chi,u}}\Big)\Big)\omega_{u}^n}dt\nonumber\\
&=\int_0^1 \frac{\frac{d}{dt}\int_X \chi' \Big(\frac{\dot u}{\|\dot u\|_{\chi,u}}\Big)\frac{du}{ds}\omega_{u}^n -\int_X \frac{du}{ds} \nabla_{\dot u}\Big(\chi' \Big(\frac{\dot u}{\|\dot u\|_{\chi,u}}\Big)\Big) \omega_{u}^n}{\chi(1) + \int_X \chi^*\Big(\chi'\Big(\frac{\dot u}{\|\dot u\|_{\chi,u}}\Big)\Big)\omega_{u}^n}dt.
\end{flalign}
We make the following side computation:
\begin{equation}\label{eq: interimnabla}
\nabla_{\dot u}\Big(\chi' \Big(\frac{\dot u}{\|\dot u\|_{\chi,u}}\Big)\Big)\o_{u}^n=\chi'' \Big(\frac{\dot u}{\|\dot u\|_{\chi,u}}\Big)\Big( \frac{\nabla_{\dot u}\dot u}{\|\dot u\|_{\chi,u}} - \frac{\dot u}{\|\dot u\|_{\chi,u}^2}\frac{d}{dt}\|\dot u\|_{\chi,u}\Big)\o_u^n.
\end{equation}
After possibly further shrinking $\varepsilon_0(\phi,\psi) >0$, from Proposition \ref{prop: EpsGeodTanEst}(i)(ii) and \eqref{eq: eps_geod_eq_Lev_Civ} it follows that $\|\dot u\|_{\chi,u}$ is uniformly bounded away from zero and both $\nabla_{\dot u} \dot u\o_u^n$ and $\frac{d}{dt} \|\dot u \|_{\chi,u}$ are of the form $\varepsilon R$, where $R$ is an uniformly bounded quantity for $\varepsilon < \varepsilon_0(\phi,\psi)$. 

Furthermore, it follows from Theorem \ref{thm: ueps_estimates}  that $\dot u$ is uniformly bounded independently of $\varepsilon$. All of this implies that the quantity of \eqref{eq: interimnabla} is also of the form $\varepsilon R$. Lastly, Corollary \ref{cor: duds_estimate} implies that $du/ds$ is uniformly bounded as well, hence putting the above together it follows that the second term in the numerator of \eqref{eq: anothercalc} is also of the form $\varepsilon R$, and we can continue to write:
\begin{flalign*}
&= \int_0^1 \frac{\frac{d}{dt}\int_X \chi' \Big(\frac{\dot u}{\|\dot u\|_{\chi,u}}\Big)\frac{du}{ds} \o_{u}^n}{\chi(1) + \int_X \chi^*\Big(\chi'\Big(\frac{\dot u}{\|\dot u\|_{\chi,u}}\Big)\Big)\o_{u}^n}dt +\varepsilon R
\end{flalign*}
As $\chi^*$ is the Legendre transform of $\chi$, it follows that ${\chi^*}' (\chi'(l))=l, \ l \in \Bbb R$. Using this, our prior observations and the chain rule, we obtain that the expression
$$\frac{d}{dt}\left( \chi(1)+\int_X \chi^*\Big(\chi'\Big(\frac{\dot u}{\|\dot u\|_{\chi,u}}\Big)\Big)\o_{u}^n\right)=\int_X \frac{\dot u}{\|\dot u\|_{\chi,u}} \chi''\Big(\frac{\dot u}{\|\dot u\|_{\chi,u}}\Big)\nabla_{\dot u}\Big(\frac{\dot u}{\|\dot u\|_{\chi,u}}\Big)\o_{u}^n$$
is again of magnitude $\varepsilon R$, hence in our sequence of calculations we can write
\begin{flalign}\label{lastestimate}
&= \int_0^1 \frac{d}{dt} \frac{\int_X \chi' \Big(\frac{\dot u}{\|\dot u\|_{\chi,u}}\Big)\frac{du}{ds} \omega_{u}^n}{\chi(1) + \int_X \chi^*\Big(\chi'\Big(\frac{\dot u}{\|\dot u\|_{\chi,u}}\Big)\Big)\omega_{u}^n}dt +\varepsilon R \nonumber\\
&=\frac{\int_X \chi' \Big(\frac{\dot u(1,s)}{\|\dot u(1,s)\|_{\chi,\psi}}\Big)\frac{d \psi(s)}{ds} \omega_{\psi}^n}{\chi(1) + \int_X \chi^*\Big(\chi'\Big(\frac{\dot u(1,s)}{\|\dot u(1,s)\|_{\chi,\psi}}\Big)\Big)\omega_{\psi}^n} + \varepsilon R \nonumber\\
&\geq -\Big\|\frac{d\psi(s)}{ds}\Big\|_{\chi, \psi} + \varepsilon R,
\end{flalign}
where in the last line we have used the Young inequality \eqref{eq: YoungIdIneq} in the following manner:
\begin{flalign*}
\frac{{\int_X \chi' \Big(\frac{\dot u(1,s)}{\|\dot u(1,s)\|_{\chi,\psi}}\Big)\frac{d \psi(s)}{ds} \omega_{\psi}^n}}{\| {d \psi}/{ds}\|_{\chi,\psi}}&\geq- \int_X \Big[ \chi \Big( \frac{{d \psi}/{ds}}{\| {d \psi}/{ds}\|_{\chi,\psi}}\Big) + \chi^*\Big(\chi' \Big(\frac{\dot u(1,s)}{\|\dot u(1,s)\|_{\chi,\psi}}\Big)\Big) \Big]\omega_{\psi}^n\\
&=-\Big[\chi(1) + \int_X \chi^*\Big(\chi' \Big(\frac{\dot u(1,s)}{\|\dot u(1,s)\|_{\chi,\psi}}\Big)\Big) \omega_{\psi}^n\Big].
\end{flalign*}
Integrating estimate \eqref{lastestimate} with respect to $s$ yields the desired inequality.
\end{proof}

With the previous result established, there is no more need to differentiate expressions involving the $l_\chi$ length of curves, hence we can return to general Finsler metrics on $\mathcal H_\o$, with possibly non--smooth weight $\chi \in \mathcal W^+_p$, and prove Theorem \ref{thm: XXChenThm} in the process:

\begin{proof}[Proof of Theorem \ref{thm: XXChenThm}] First we show the following identity for the weak $C^{1,\bar 1}$--geodesic joining $u_0,u_1$:
\begin{equation}\label{eq: distgeodformula1}
d_\chi(u_0,u_1)=l_\chi(u_t).
\end{equation}
We can assume that $u_0 \neq u_1$. By Theorem \ref{thm: ueps_estimates} and \eqref{eq: epsgeod_limit},\eqref{eq: u_estimates}, the smooth $\varepsilon$--geodesics $u^\varepsilon$ connecting $u_0,u_1$ $C^{1,\alpha}$--converge to the weak $C^{1,\bar 1}$ geodesic $u$, connecting $u_0,u_1$. As a result, $\dot u^\varepsilon_t$ converges uniformly to $\dot u_t$. Next we argue that the lengths of these tangent vectors converge as well:

\begin{claim}$\|\dot u^\varepsilon_t\|_{\chi,u^\varepsilon_t} \to \|\dot u_t\|_{\chi,u_t}$ as $\varepsilon \to 0$.
\end{claim}

From Proposition \ref{prop: EpsGeodTanEst}(i) and Theorem \ref{thm: ueps_estimates} it follows that there exists $C_2>C_1>0$ such that for small enough $\varepsilon >0$ we have
$$0 < C_1 \leq  \|\dot u^\varepsilon_t\|_{\chi,u^\varepsilon_t} \leq C_2.$$
In particular, we only have to argue that all cluster points of the set $\{ \|\dot u^\varepsilon_t\|_{\chi,u^\varepsilon_t} \}_{\varepsilon}$ are equal to $\|\dot u_t\|_{\chi,u_t}$. Let $N$ be such a cluster point, and after taking a subsequence, we can assume that $\|\dot u^\varepsilon_t\|_{\chi,u^\varepsilon_t} \to N$ as $\varepsilon \to 0$. By uniform convergence of tangent vectors, we obtain that $\dot u^\varepsilon_t/ \|\dot u^\varepsilon_t\|_{\chi,u^\varepsilon_t}$ converges to $\dot u_t/N$ uniformly as well. Since $\omega_{u^\varepsilon_t}^n \to \omega_{u_t}^n$ weakly, this allows to conclude that
$$\chi(1)=\int_X \chi\bigg(\frac{\dot u^\varepsilon_t}{\|\dot u^\varepsilon_t\|_{\chi,u^\varepsilon_t}}\bigg) \o^n_{u^\varepsilon_t} \to \int_X \chi\bigg( \frac{\dot u_t}{N}\bigg) \o^n_{u_t}, \textup{ as } \varepsilon \to 0.$$
Using \eqref{eq: OrliczNormId} we get that $N = \|\dot u_t\|_{\chi,u_t}$, finishing the proof of the claim.

Using the claim, we can apply the  dominated convergence theorem to conclude that
\begin{equation}\label{eq: length_epsgeod_lim}
\lim_{\varepsilon \to 0} l_\chi(u^{\varepsilon}_t) = l_\chi(u_t),
\end{equation}
hence $d_\chi(u_0,u_1) \leq l_\chi(u_t)$. To prove the reverse inequality, and with that establishing \eqref{eq: distgeodformula1}, we assume first that $\chi \in \mathcal W^+_p \cap C^\infty$. We have to prove that
\begin{equation}\label{eq: reverseineq}
l_\chi(\phi_t) \geq l_\chi(u_t)
\end{equation}
for all smooth curves $[0,1] \ni t \to \phi_t \in \mathcal H$ connecting $u_0,u_1$. We can assume that $u_1 \not \in \phi[0,1)$ and let $h \in [0,1)$. Letting $\varepsilon \to 0$ in the estimate of the previous result, by \eqref{eq: length_epsgeod_lim} we obtain that
$$l_\chi(u_{1-t}) \leq l_\chi(\phi_t|_{[0,h]}) + l_\chi(w^h_t),$$
where $[0,1] \ni t \to u_{1-t},w^h_t \in \mathcal H_\o^{1,\bar 1}$ are the weak $C^{1,\bar 1}$--geodesic segments joining $u_1,u_0$ and $u_1,\phi_h$ respectively. As $h \to 1$, by Lemma \ref{lem: dot_u_est} we have $l_\chi(w^h_t)\to 0$ and we obtain \eqref{eq: reverseineq} for smooth weights $\chi$.

For general $\chi \in \mathcal W^+_p$ by Proposition \ref{prop: approx_lemma2} there exists a sequence $\chi_k \in \mathcal W^+_{p_k} \cap C^\infty(\Bbb R)$ such that $\chi_k$ converges to $\chi$ uniformly on compacts. From what we just proved it follows that
$$\int_0^1 \| \dot \phi_t\|_{\chi_k,\phi_t}dt=l_{\chi_k}(\phi_t) \geq l_{\chi_k}(u_t)=\int_0^1 \| \dot u_t\|_{\chi_k,u_t}dt.$$
Using Proposition \ref{prop: approx_lemma} and the dominated convergence theorem ($\dot \phi_t, \dot u_t$ are uniformly bounded), we can take the limit in this last estimate to conclude \eqref{eq: reverseineq}, which gives \eqref{eq: distgeodformula1}.
Formula \eqref{eq: ChiDistGeodFormula} follows now from the fact that $l_\chi(u_t)=\int_0^1 \| \dot u_l\|_{\chi,u_l}dl$ and Lemma \ref{lem: chilengthgeodconst} below.

Finally, if $u_0\neq u_1$ then after taking the limit $\varepsilon \to 0$ in the estimate of Lemma \ref{lem: dotintegralest} we obtain that  $\dot u_0 \not\equiv 0$, hence $d_\chi(u_0,u_1)=\| \dot u_0\|_{\chi,u_0}>0$. This implies that $(\mathcal H, d_\chi)$ is a metric space, as claimed.
\end{proof}

According to the last lemma of this section, the $\chi$--length of tangent vectors along a $C^{1,\bar 1}$--geodesic is always constant. This parallels a similar result of Berndtsson \cite[Proposition 2.2]{br2}. 

\begin{lemma}\label{lem: chilengthgeodconst} Given $u_0,u_1 \in \mathcal H_\o$, let $[0,1] \ni t \to u_t \in \mathcal H_\o^{1,\bar 1}$ be the weak $C^{1,\bar 1}$--geodesic connecting $u_0,u_1$. Then for any $\chi \in \mathcal W^+_p$ and  $t_0,t_1 \in [0,1]$ the following hold:\\
\begin{equation}\label{eq: chilengthgeodconst}
\|\dot u_{t_0} \|_{\chi,u_{t_0}}=\|\dot u_{t_1} \|_{\chi,u_{t_1}}.
\end{equation}
\end{lemma}
\begin{proof} As we already argued in the proof of the previous result, for the $\varepsilon$--geodesics $t \to u^\varepsilon_t$ joining $u_0,u_1$ we have
$\| \dot u^\varepsilon_{t_0}\|_{\chi,u^\varepsilon_{t_0}} \to \| \dot u_{t_0}\|_{\chi,u_{t_0}}$  and $\| \dot u^\varepsilon_{t_1}\|_{\chi,u^\varepsilon_{t_1}} \to \| \dot u_{t_1}\|_{\chi,u_{t_1}}$ as $\varepsilon \to 0$. Furthermore Proposition \ref{prop: EpsGeodTanEst}(ii) implies that $|\| \dot u^\varepsilon_{t_1}\|_{\chi,u^\varepsilon_{t_1}}-\|\dot u^\varepsilon_{t_0}\|_{\chi,u^\varepsilon_{t_0}}| \leq |t_1 -t_0|\varepsilon R_1$. Putting all of this together, and letting $\varepsilon \to 0$, \eqref{eq: chilengthgeodconst} follows.
\end{proof}

\section{The weak geodesic segments of $\textup{PSH}(X,\o)$}

As noted after Theorem \ref{thm: u_estimates}, it is not possible to join smooth potentials $u_0,u_1$ with a geodesic staying inside $\mathcal H_\o$. The main point of the present section is to show that a similar phenomenon does not occur if $u_0,u_1$ is allowed to be more singular. Indeed, as we will see, if $u_0,u_1$ are bounded, or they are from a finite energy space, then it is possible to define a weak geodesic connecting them that is also bounded or stays inside the finite energy space respectively. Our study will connect properties of weak geodesics with that of envelopes, and we will make good use of the results of Section 2.4.

When considering the boundary value problem \eqref{eq: BVPGeod}, we constructed the (weak) solution $u$ as the limit of solutions to the family of elliptic problems \eqref{eq: epsBVPGeod}. Moving away from this idea, as noted by Berndtsson \cite[Section 2.1]{brn1}, it possible to describe the (weak) solution $u$  in another way, using a slight generalization of the classical Perron--Bremmerman envelope from the local theory. 

The advantage of this approach is that one can consider very general boundary data in \eqref{eq: BVPGeod}. Indeed, to begin, let  $u_0,u_1 \in \textup{PSH}(X,\o)$. In the future, we will refer to $i\Bbb R$--invariant elements of $\textup{PSH}(S\times X, \pi^*\o)$ as \emph{weak subgeodesics} (recall that $S =\{ 0<\textup{Re }z <1 \} \subset \Bbb C$). This name is justified by the following formula:
\begin{equation}\label{eq: udef1}
u = \sup_{v \in \mathcal S}v,
\end{equation}
where $\mathcal S$ is the following family of weak subgeodesics:
$$ \mathcal S = \{ (0,1) \ni t \to v_t \in \text{PSH}(X,\o) \textup{ is a subgeodesic with }\lim_{t \to 0,1}v_t \leq u_{0,1} \}.$$
As we know, the supremum of a family of $\pi^* \o$-psh functions may not be $\pi^* \o$-psh, and the first step is to show that $u$, as defined in \eqref{eq: udef1}, is $\pi^* \o$-psh nonetheless. Indeed, by convexity in the $t$ variable, each member of $w_t \in \mathcal S$ satisfies $w_t \leq (1-t)u_0 + tu_1$, hence this also holds for the supremum $u$:
\begin{equation}\label{eq: u_upperbound}
u_t \leq (1-t)u_0 + t u_1.
\end{equation}
By taking the usc regularization of the above inequality, we conclude that the same inequality holds with $u^*$ in place of $u$:
$$u^*_t \leq (1-t)u_0 + t u_1.$$
We obtain that $u^* \in \mathcal S$, hence $u^* \leq u$ by \eqref{eq: udef1}. Trivially $u \leq u^*$, and this implies that $u = u^* \in \textup{PSH}(S \times X,\pi^* \o)$.  

We will call the curve $[0,1] \ni t \to u_t \in \textup{PSH}(X,\o)$ resulting from the construction of \eqref{eq: udef1} the \emph{weak geodesic} connecting $u_0,u_1$. This terminology is justified by the following result, which says that \eqref{eq: udef1} gives the unique solution to \eqref{eq: BVPGeod} for boundary data that is merely bounded:
\begin{lemma} \label{lem: boundedBVP_solution} When $u_0,u_1 \in \textup{PSH}(X,\o) \cap L^\infty$ then the unique bounded $\pi^*\o$-psh solution of \eqref{eq: BVPGeod} is given by \eqref{eq: udef1}.
\end{lemma}

As a result of this lemma, for $u_0,u_1 \in \textup{PSH}(X,\o) \cap L^\infty$, we will call the weak geodesic $t \to u_t$ connecting $u_0,u_1$ a \emph{bounded geodesic}.

\begin{proof} Let $C := \|u_1 - u_0 \|_{L^\infty}$. It is easy to see that $t \to  (u_0 - C t)$ and $t \to (u_1 -C(1-t))$ are (sub)geodesics that are both members of $\mathcal S$, hence so is their maximum $v_t := \max(u_0 - Ct, u_1 - C(1-t)) \in \mathcal S$. 
This and \eqref{eq: u_upperbound} gives
\begin{equation}\label{eq: u_boundary_est}
\max(u_0 - Ct, u_1 - C(1-t)) \leq u_t \leq (1-t)u_0 + tu_1.
\end{equation}
Consequently, $u \in \textup{PSH}(S \times X,\pi^*\o) \cap L^\infty$ and  $\lim_{t \to 0,1}u_t = u_{0,1}$. The classical Perron--Bremmerman argument can now be adapted to this setting to give $(\pi^* \o + i\ddbar u)^{n+1}=0$. Lastly, uniqueness of $u$ is a consequence of Theorem \ref{thm: uniqueness_BVP}.
\end{proof}

Turning back to non-bounded endpoints $u_0,u_1$, it turns out that even the very general weak geodesic segment $t \to u_t$ connecting $u_0,u_1$ exhibits some structure, as we will see in the next two results:

\begin{proposition}[\cite{da1}]  \label{prop: weak_geod_approx} Suppose $u^k_0,u_0,u^k_1,u_1 \in \textup{PSH}(X,\o)$ are such that $u^k_0 \searrow u_0$ and $u^k_1 \searrow u_1$. Let $[0,1] \ni t \to u^k_t,u_t \in \textup{PSH}(X,\o)$ be the weak geodesics connecting $u^k_0,u^k_1$ and $u_0,u_1$ respectively. Then the following hold:\\
(i) $u^k_t \searrow u_t, \ t \in [0,1]$.\\
(ii) For any $t_1,t_2 \in [0,1]$ we have that $[0,1] \ni l \to u_{(1-l)t_1 + l t_2} \in \textup{PSH}(X,\o)$ is the weak geodesic joining $u_{t_1}$ and $u_{t_2}$.
\end{proposition}
\begin{proof} By the definition of $u^k\in \textup{PSH}(S \times X,\pi^*\o)$ \eqref{eq: udef1} it is clear that  $u^k $ is decreasing in $k$ and $v := \lim_k u_k \in \textup{PSH}(S \times X,\pi^*\o)$. As $u$ is a candidate in the definition of each $u_k$, it follows that $u \leq u^k$, hence also $u \leq v$. For the other direction, by \eqref{eq: u_upperbound} we have that $u^k_t \leq (1-t)u^k_0 + t u^k_1$ hence we can take the limit to obtain that
$$v \leq (1-t)u_0 + t u_1.$$
Consequently $v$ is a candidate for $u$, giving that $v \leq u$, finishing the proof of (i).

Now we turn to proving (ii). Let $u^k_0=\max(u_0,-k)$ and $u^k_1 = \max(u_1,-k)$ be the canonical cutoffs and let $t \to u^k_t$ be the bounded geodesics joining $u^k_0$ and $u^k_1$. Part (i) implies that $u^k_{t_1} \searrow u_{t_1}$ and $u^k_{t_2} \searrow u_{t_2}$. Hence, applying  (i) again, it is enough to prove that  $[0,1] \ni l \to u^k_{(1-l)t_1 + l t_2} \in \textup{PSH}(X,\o)$ is the bounded/weak  geodesic joining $u^k_{t_1},u^k_{t_2}$.

Now \eqref{eq: u_boundary_est} implies that each $t \to u^k_t$ is Lipschitz continuous in the $t$ variable, hence $u^k_{t} \to u^k_{t_{1,2}} $ uniformly as ${t \to t_{1,2}}$. By Lemma \ref{lem: boundedBVP_solution}, $[0,1] \ni l \to u^k_{(1-l)t_1 + l t_2} \in  \textup{PSH}(X,\o)$ is indeed the unique bounded/weak geodesic joining $u^k_{t_1}$ and $u^k_{t_2}$.
\end{proof}

The next result connects weak geodesics to the rooftop envelopes of Section 3.3:

\begin{lemma}[\cite{da1}] \label{lem: Leg_transf_Envelope} Suppose $u_0,u_1 \in \textup{PSH}(X,\o)$ and $t \to u_t$ is the  weak geodesic  connecting $u_0,u_1$. Then the following holds:
$$\inf_{t \in (0,1)} (u_t - t\tau) = P(u_0,u_1 -\tau), \ \tau \in \Bbb R.$$
\end{lemma}
\begin{proof}Notice that $t \to v_t:= u_t - \tau t$ is the weak geodesic connecting $u_0,u_1-\tau$, hence the proof of the general case reduces to the particlar case $\tau =0$.

By definition we have $P(u_0,u_1) \leq u_0,u_1$. As a result, for the constant (sub)geodesic $t \to h_t :=P(u_0,u_1)$ we have $h \in \mathcal S$. This trivially gives $h_t \leq u_t, \ t \in [0,1]$, hence $P(u_0,u_1) \leq \inf_{t \in (0,1)} u_t$.

For the reverse inequality, we use the Kiselman minimum principle \cite[Chapter I, Theorem 7.5]{De}, which guarantees that $w:=\inf_{t \in (0,1)}u_t \in \textup{PSH}(X,\o)$. Using this and \eqref{eq: u_upperbound} we obtain that $w \leq u_0,u_1$ hence $w$ is a candidate for $P(u_0,u_1)$, i.e., $w \leq P(u_0,u_1)$, finishing the proof.
\end{proof}

We can now relate the super--level sets of tangent vectors along weak geodesics to contact sets of rooftop envelopes:

\begin{lemma}[\cite{da1}] \label{lem: sublevel_lemma}Suppose $u_0,u_1 \in \textup{PSH}(X,\o)$. Let $t \to u_t$ be the weak geodesic joining $u_0,u_1$. Then for any $\tau \in \Bbb R$ we have
$$\{ \dot u_0 \geq \tau \} = \{ P(u_0,u_1 - \tau)=u_0\}.$$
\end{lemma}
\begin{proof} By the previous result we have $\inf_{t \in [0,1]}(u_t - t\tau)= P(u_0,u_1 - \tau).$ Given $x \in X$, it follows that $P(u_0,u_1 - \tau t)(x)=u_0(x)$ if and only if $\inf_{t \in [0,1]}(u_t(x) - \tau t)=u_0(x)$. Convexity in the $t$ variable implies that this last identity is equivalent to $\dot u_0(x) \geq \tau$.
\end{proof}

With the aid of Lemma \ref{lem: Leg_transf_Envelope}, we can show that weak geodesics with endpoints in finite energy classes stay inside finite energy classes:

\begin{proposition}[\cite{da1}] \label{prop: geod_Echi} Suppose $u_0,u_1 \in \mathcal E_\chi(X,\o), \ \chi \in \mathcal W^+_p, \ p \geq 1$. Then for the weak geodesic $t \to u_t$ connecting $u_0,u_1$ we have that $u_t \in \mathcal E_\chi(X,\o)$ for all $t \in [0,1]$. In case $u_0,u_1 \leq 0$ the following estimate holds:
$$E_\chi(u_t) \leq (p+1)^{2n}(E_\chi(u_0) + E_\chi(u_1)), \ t \in [0,1].$$
\end{proposition}

As a consequence of this proposition, for $u_0,u_1 \in \mathcal E_\chi(X,\o)$ we will call the the curve $[0,1] \ni t \to u_t \in \mathcal E_\chi(X,\o)$ the \emph{finite energy geodesic} connecting $u_0,u_1$.

\begin{proof} To start, Proposition \ref{prop: env_exist} implies that $P(u_0,u_1) \in \mathcal E_\chi(X,\o)$. By Lemma \ref{lem: Leg_transf_Envelope}, we have that $P(u_0,u_1) \leq u_t$, hence by the monotonicity property (Corollary \ref{cor: monotonicity_E_chi}) it follows that $u_t \in \mathcal E_\chi(X,\o), \ t \in [0,1]$. 

When $u_0,u_1 \leq 0$, \eqref{eq: u_upperbound} implies that $u_t\leq 0, \ t \in [0,1]$. To finish the proof, by Proposition \ref{prop: env_exist} and Proposition \ref{prop: Energy_est} we have the following estimates:
$$E_\chi(u_t) \leq (p+1)^n E_\chi(P(u_0,u_1)) \leq (p+1)^{3n}(E_\chi(u_0)+E_\chi(u_1)).$$
\end{proof}

\section{Extension of the $L^p$ metric structure to finite energy spaces}

For the rest of this chapter we will focus only various $L^p$ Finsler geometries of $\mathcal H_\o$. Most of the results we present also have analogs for the more general Orlicz Finsler structures discussed in the previous sections (see \cite{da2}). Having later applications in mind, we do not seek the greatest generality, and we leave it to the interested reader to adapt our argument to more general metrics. As done it previously, we will assume the volume normalization condition \eqref{eq: vol_normalization} throughout this section as well. 

Given $u_0,u_1 \in \mathcal E_p(X,\o)$, by Theorem \ref{thm: BK_approx} there exists decreasing sequences $u^k_0, u^k_1 \in \mathcal H_\o$ such that $u^k_0 \searrow u_0$ and $u^k_1 \searrow u_1$. We propose to define the distance $d_p(u_0,u_1)$ by the formula:
\begin{equation}\label{eq: dp_def_general}
d_p(u_0,u_1)=\lim_{k\to \infty}d_p(u^k_0,u^k_1).
\end{equation}
We will show that the above limit exists and it is also independent of the approximating sequences. Developing this further, our main result of this section is the following:
\begin{theorem}[\cite{da2}]  \label{thm: e2space}$(\mathcal E_p(X,\o), d_p)$ is a geodesic pseudo--metric space that extends $(\mathcal H_\o,d_p)$. Additionally, for $u_0,u_1 \in \mathcal E_p(X,\o)$ the finite energy geodesic $t \to u_t$ joining $u_0,u_1$ (given by Proposition \ref{prop: geod_Echi}) is a $d_p$-geodesic.
\end{theorem}

Recall that a pseudo--metric is just a metric that may not satisfy the non-degeneracy condition. Also, given a pseudo--metric space $(M,d)$, we say that a curve $[0,1] \ni t \to \gamma_t \in M$ is a $d$\emph{--geodesic} if 
\begin{equation}\label{eq: d_geod_def}
d(\gamma_{t_1},\gamma_{t_2}) = |t_1 - t_2|d(\gamma_{0},\gamma_{1}), \ \ t_1,t_2 \in [0,1].
\end{equation}
A \emph{geodesic pseudo--metric space} $(M,d)$ is pseudo--metric space in which any two points can be joining by a $d$--geodesic.

In the next section we will show that in fact $d_p$ is in fact a bona fide metric, but this will require additional machinery. Finally, as the last major theorem of this chapter, we will prove that the resulting metric space $(\mathcal E_p(X,\o), d_p)$ is the completion of $(\mathcal H_\o,d_p)$. 

The proof of Theorem \ref{thm: e2space} will be split into a sequence of lemmas and propositions. Our first one is an estimate for ``comparable" potentials:

\begin{proposition}\label{prop: Mdist_est}Suppose $u,v \in \mathcal H_\o$ with $u \leq v$. Then we have:
\begin{equation}\label{eq: Mdist_est}
\max\Big( \frac{1}{2^{n+p}}\int_X|v-u|^p \o_u^n, \int_X|v-u|^p \o_v^n \Big) \leq d_p(u,v)^p \leq \int_X|v-u|^p \o_u^n.
\end{equation}
\end{proposition}

\begin{proof} Suppose $[0,1] \ni t \to w_t \in \mathcal H_\o^{1,\bar 1}$ is the $C^{1,\bar 1}$-geodesic segment joining $w_0 = u$ and $w_1 = v$. By \eqref{eq: ChiDistGeodFormula} we have
$$d_p(u,v)^p=\int_X |\dot w_0|^p \o_u^n=\int_X |\dot w_1|^p \o_v^n.$$ 
Since $u \leq v$, we have that $u \leq w_t$, as follows from \eqref{eq: udef1} (or Theorem \ref{thm: uniqueness_BVP}). Since $(t,x) \to w_t(x)$ is convex in the $t$ variable, we get $0 \leq \dot w_0 \leq v-u \leq \dot w_1$, and together with the above identity we obtain part of \eqref{eq: Mdist_est}:
\begin{equation}\label{eq: Mdist_est_interm}
\int_X |v-u|^p \o_v^n \leq d_p(u,v)^p \leq \int_X |v-u|^p \o_u^n.
\end{equation}
Now we prove the rest of \eqref{eq: Mdist_est}. Using $ \o_{u}^n \leq 2^{n}\o_{(u+v)/2}^n$ we obtain that 
$$\frac{1}{2^{n+p}}\int_X |v - u|^p \o_{u}^n \leq \int_X \Big|u - \frac{u+v}{2}\Big|^p\o^n_{(u+v)/2}.$$
Since $u \leq (u+v)/2$, the first estimate of \eqref{eq: Mdist_est_interm} allows us to continue and write:
$$\frac{1}{2^{n+p}}\int_X |v - u|^p \o_{u}^n \leq d_p \Big(\frac{u + v}{2},u\Big)^p.$$
The lemma below implies that $d_p((u + v)/2,u) \leq d_p(v,u)$, giving the remaining estimate in \eqref{eq: Mdist_est}.
\end{proof}

\begin{lemma} \label{lem: NaiveCompare}Suppose $u,v,w \in \mathcal H_\o$ and $u \geq v \geq w$. Then $d_p(v,w) \leq d_p(u,w)$ and $d_p(u,v) \leq d_p(u,w)$.
\end{lemma}
\begin{proof} We introduce the  $C^{1,\bar 1}$ geodesics $[0,1] \ni t \to \alpha_t,\beta_t \in \mathcal H_\o^{1,\bar 1}$ connecting $\alpha_0 := w,\alpha_1 := v$ and $\beta_0 := w,\beta_1 := u$ respectively. From \eqref{eq: udef1} (or Theorem \ref{thm: uniqueness_BVP}) it follows that both of these curves are increasing in $t$. Additionally, $\alpha \leq \beta$ by  Theorem \ref{thm: uniqueness_BVP}. As $\alpha_0=\beta_0$, it  follows that $0 \leq \dot \alpha_0  \leq \dot \beta_0$. Using this and Theorem \ref{thm: XXChenThm} we obtain that $d_p(w,v) \leq d_p(w,u)$. The estimate $d_p(u,v) \leq d_p(u,w)$ is proved similarly.
\end{proof}

Next we turn our attention to smooth approximants of finite energy potentials:

\begin{lemma} Suppose $u \in \mathcal E_p(X,\o)$ and $\{ u_k\}_{k} \subset \mathcal H_\o$ is a sequence decreasing to $u$. Then $d_p(u_l,u_k) \to 0$ as $l,k \to \infty$.\label{lem: IntDistEst}
\end{lemma}

\begin{proof}
We can suppose that $l \leq k$. Then $u_k\leq u_l$, hence by Proposition \ref{prop: Mdist_est} we have:
$$d_p(u_l,u_k)^p \leq \int_X|u_k-u_l|^p\o_{ u_k}^n.$$
We clearly have $u - u_l, u_k - u_l \in \mathcal E_p(X,\o_{u_l})$ and $u - u_l\leq u_k - u_l\leq0$. Hence, applying Proposition \ref{prop: Energy_est} for the class $\mathcal E_p(X,\o_{u_l})$ we obtain that
\begin{equation}\label{eq: estimate}
d_p(u_l,u_k)^p\leq (p+1)^n\int_X|u-u_l|^p\o_u^n.
\end{equation}
As $u_l$ decreases to $u \in \mathcal E_p(X,\o)$, it follows from the dominated convergence theorem that $d_p(u_l,u_k) \to 0$ as $l,k \to \infty$.
\end{proof}

Our next lemma confirms that the way we proposed to extend the $d_p$ to $\mathcal E_p(X,\o)$ (see \eqref{eq: dp_def_general}) does not have inconsistencies:

\begin{lemma} Given $u_0,u_1 \in \mathcal E_p(X,\o)$, the limit in \eqref{eq: dp_def_general} is finite and independent of the approximating sequences $u^k_0, u^k_1 \in \mathcal H_\o$. \end{lemma}

In particular, this result implies that for $u_0,u_1 \in \mathcal H_\o$ the distance $d_p(u_0,u_1)$ will be the same according to both \eqref{eq: dp_def_general} and our original definition in \eqref{eq: d_chi_def}. Lastly, the triangle inequality will also hold, hence $d_p$ is a pseudo--metric on $\mathcal E_p(X,\o)$, as claimed in Theorem \ref{thm: e2space}.

\begin{proof} By the triangle inequality and Lemma \ref{lem: IntDistEst} we can write:
$$|d_p(u^l_0,u^l_1)-d_p(u^k_0,u^k_1)| \leq d_p(u^l_0,u^k_0) +d_p(u^l_1,u^k_1) \to 0, \ l,k \to \infty, $$
proving that $d_p(u^k_0,u^k_1)$ is indeed convergent.

Now we prove that the limit in \eqref{eq: dp_def_general} is independent of the choice of approximating sequences. Let $v^l_0, v^l_1 \in \mathcal H_\o$ be different approximating sequences. By adding small constants if necessary, we can arrange that the sequences $u^l_0, u^l_1$, respectively $v^l_0, v^l_1$, are strictly decreasing to $u_0,u_1$.

Fixing $k$ for the moment, the sequence $\{\max\{ u^{k+1}_0,v^j_0\}\}_{j \in \Bbb N}$ decreases pointwise to $u^{k+1}_0$. By Dini's lemma the convergence is uniform, hence there exists $j_k\in \Bbb N$ such that for any $j \geq j_k$ we have $v^j_0 < u^k_0$. By repeating the same argument we can also assume that $v^j_1 < u^k_1$ for any $j \geq j_k$. By the triangle inequality again
$$|d_p(u^k_0,u^k_1)-d_p(v^j_0,v^j_1)| \leq d_p(u^k_0,v^j_0) +d_p(u^k_1,v^j_1), \ j \geq j_k. $$
From \eqref{eq: estimate} it follows that for $k$ big enough the quantities $d(u^j_0,v^k_0)$, $d(u^j_1,v^k_1), \ j \geq j_k$ are arbitrarily small, hence $d_p(u_0,u_1)$ is independent of the choice of approximating sequences. 

When $u_0,u_1 \in \mathcal H_
\o$, then it is possible to approximate with a constant sequence, hence our argument implies that the restriction of $d_p$ (as extended in \eqref{eq: dp_def_general}) to $\mathcal H_\o$ coincides with the original definition in \eqref{eq: d_chi_def}. 
\end{proof}

By the previous result, the triangle inequality is inherited by the extension of $d_p$ to $\mathcal E_p(X,\o)$, making $(\mathcal E_p(X,\o),d_p)$ a pseudo--metric space. In the last part of  this section we turn our attention to showing that finite energy geodesic segments of Proposition \ref{prop: geod_Echi} are in fact $d_p$--geodesics:

\begin{lemma} Suppose $u_0,u_1 \in \mathcal E_p(X,\o)$ and $[0,1]\ni t \to u_t \in \mathcal E_p(X,\o)$ is the finite energy geodesic segment connecting $u_0,u_1$. Then $t\to u_t$ is a $d_p$-geodesic.
\end{lemma}
\begin{proof} First we prove that
\begin{flalign}\label{eq: geod_half}
 d_p(u_0,u_l)=ld_p(u_0,u_1), \ \ l \in [0,1].
\end{flalign}
By Theorem \ref{thm: BK_approx}, suppose $u^k_0,u^k_1 \in \mathcal H_\o$ are strictly decreasing approximating sequences of $u_0,u_1$ and let $[0,1]\ni t \to u^k_t \in \mathcal H_\o^{1,\bar 1}$ be the decreasing sequence of $C^{1,\bar 1}$ geodesics connecting $u_0^k,u_1^k$. By the definition of \eqref{eq: dp_def_general} and Theorem \ref{thm: XXChenThm} we can write: 
$$d_p(u_0,u_1)^p=\lim_{k\to \infty}d_p(u^k_0,u^k_1)^p=\lim_{k\to \infty}\int_X |\dot u_0^k|^p \o_{u^k_0}^n.$$

By Proposition \ref{prop: weak_geod_approx}(i), the geodesic segments $[0,1]\ni t \to u^k_t \in \mathcal H_\o^{1,\bar 1}$ are decreasing pointwise to $[0,1]\ni t \to u_t\in \mathcal E_p(X,\o)$. In particular, this implies that $u^k_l \searrow u_l$.
We want to find a decreasing sequence $\{w^k_l\}_k \subset \mathcal H_\o$ such that $u^k_l \leq w^k_l$, $w^k_l \searrow u_l$ and
\begin{equation}\label{eq: approx_lim}
l^pd_p(u_0^k,u^k_1)^p - d_p(u^k_0, w^k_l)^p=l^p\int_X |\dot u^k_0|^p \o_{u^k_0}^n - \int_X |\dot w^k_0|^p \o_{u^k_0}^n \to 0 \textup{ as } k \to \infty,
\end{equation}
where $t \to w^k_t$ is the $C^{1,\bar 1}$--geodesic segment connecting $u^k_0$ and $w^k_l$. By the definition of $d_p$, letting $k \to \infty$ in \eqref{eq: approx_lim} would give us \eqref{eq: geod_half}.

Finding such sequence $w_l^k$ is always possible by an (iterated) application of Lemma \ref{lem: geod_tangent_limit} applied to $v_0 := u^k_0$ and $v_1 := u^k_l$, as we observe that the bounded geodesic segment connecting $u^k_0$ and $u^k_l$ is exactly $t \to u_{lt}^k$ (see Proposition \ref{prop: weak_geod_approx}(ii)).

Finally, to finish the proof, we argue that for $t_1,t_2 \in [0,1], \ t_1 \leq t_2$ we have
\begin{equation}\label{eq: geod_def1}
d_p(u_{t_1},u_{t_2})=(t_2 -t_1)d_p(u_0,u_1).
\end{equation}
Let $h_0 = u_{t_2}$ and $h_1 = u_0$. From Proposition \ref{prop: weak_geod_approx}(ii) it follows that $[0,1] \ni t \to h_t := u_{t_2(1-t)} \in \mathcal E_p(X,\o)$ is the finite energy geodesic connecting $h_0,h_1$. Applying \ref{eq: geod_half} to $t \to h_t$ and $l = 1 - t_1/t_2$ we obtain
$$(1 - t_1/t_2)d_p(u_{t_2},u_0)=d_p(u_{t_2},u_{t_1}).$$ Now applying \eqref{eq: geod_half} for $t \to u_t$ and $l = t_2$ we have
$$d_p(u_0, u_{t_2})=t_2 d_p(u_0,u_1).$$ Putting these last two formulas together we obtain \eqref{eq: geod_def1}, finishing the proof.
\end{proof}

\begin{lemma} \label{lem: geod_tangent_limit} Suppose $v_0,v_1 \in \textup{PSH}(X,\o) \cap L^\infty$ and $\{v^j_1 \}_{j \in \Bbb N}\subset \textup{PSH}(X,\o) \cap L^\infty$ is sequence decreasing to $v_1$. By $[0,1] \ni t \to v_t,v_t^j \in \textup{PSH}(X,\o) \cap L^\infty$ we denote the bounded geodesics connecting $v_0,v_1$ and $v_0, v^j_1$ respectively. By convexity in the $t$ variable, we can introduce $\dot v_0 = \lim_{t \to 0}(v_t - v_0)/t$ and $\dot v^j_0 = \lim_{t \to 0}(v^j_t - v_0)/t$, and the following holds:
$$\lim_{j \to \infty}\int_X |\dot v^j_0|^p \o_{v_0}^n = \int_X |\dot {v_0}|^p \o_{v_0}^n.$$
\end{lemma}
\begin{proof} By \eqref{eq: u_boundary_est} there exists $C >0$ such that $\|\dot v_0\|_{L^\infty(X)}, \|\dot v^j_0\|_{L^\infty(X)} \leq C$. We also have $v \leq v^j, \ j \in \Bbb N$ by \eqref{eq: udef1} (or Theorem \ref{thm: uniqueness_BVP}). As we have convexity in the $t$ variable and all our bounded geodesics share the same starting point, it also follows that $\dot v^j_0 \searrow \dot v_0$ pointwise. Consequently, the lemma follows now from Lebesgue's dominated convergence theorem.
\end{proof}

\section{The Pythagorean formula and applications}

In this section we explore the geometry of the operator $(u,v) \to P(u,v)$ restricted to the spaces $\mathcal E_p(X,\o)$. The main focus will be on the following result, establishing a metric relationship between the vertices of the ``triangle" $(u,v,P(u,v))$. This will help in proving that $d_p$ is non--degenerate, that $v \to P(u,v)$ is a $d_p$--contraction, and a number of other properties. Again, throughout this section we  will assume the volume normalization condition \eqref{eq: vol_normalization} holds. The main result of this section is the Pythagorean formula of \cite{da2}: 

\begin{theorem}
 [Pythagorean formula, \textup{\cite[Corollary 4.14]{da2}}] \label{thm: pythagorean} Given $u_0,u_1 \in \mathcal E_p(X,\o)$, we have $P(u_0,u_1)\in \mathcal E_p(X,\o)$ and
$$d_p(u_0,u_1)^p = d_p(u_0,P(u_0,u_1))^p +  d_p(P(u_0,u_1),u_1)^p.$$
\end{theorem}

The name of the above formula comes from the particular case $p=2$, in which case it suggests that $u_0,u_1$ and $P(u_0,u_1)$ form a right triangle with hypotenuse $t \to u_t$. 

Before we give the proof of this result, we  argue how it implies the non--degeneracy of $d_p$, giving that $(\mathcal E_p(X,\o),d_p)$ is a metric space:

\begin{proposition}\label{prop: dp_nondegeneracy} Given $u_0,u_1 \in \mathcal E_p(X,\o)$ if $d_p(u_0,u_1)=0$ then $u_0=u_1$.
\end{proposition}

\begin{proof} By Theorem \ref{thm: pythagorean} it follows that $d_p(u_0,P(u_0,u_1))=0$ and $d_p(u_1,P(u_0,u_1))=0$. By the first estimate of the next lemma, which generalizes Proposition \ref{prop: Mdist_est}, it follows that $u_0=P(u_0,u_1)$ a.e. with respect to $\o_{P(u_0,u_1)}^n$, and similarly, $u_1= P(u_0,u_1)$ a.e. with respect to $\o_{P(u_0,u_1)}^n$. By the domination principle (Proposition \ref{prop: domination principle}) it follows that $u_0 \leq P(u_0,u_1)$ and $u_1 \leq P(u_0,u_1)$. As the reverse inequalities are trivial, we obtain that $u_0 = P(u_0,u_1)=u_1$. 
\end{proof}

\begin{lemma}\label{lem: Mdist_est_gen} Suppose $u,v \in \mathcal E_p(X,\o)$ with $u \leq v$. Then we have:
\begin{equation}\label{eq: Mdist_est_gen}
\max\Big( \frac{1}{2^{n+p}}\int_X|v-u|^p \o_u^n, \int_X|v-u|^p \o_v^n \Big) \leq d_p(u,v)^p \leq \int_X|v-u|^p \o_u^n.
\end{equation}
\end{lemma}

\begin{proof} By Proposition \ref{prop: mixed_finite_prop: Energy_est}, all the integrals in \eqref{eq: Mdist_est_gen} are finite. By Theorem \ref{thm: BK_approx}, we can choose $u_k,v_k \in \mathcal H_\o$ such that $v_k \searrow v$ and $u_k \searrow u$ and $u_k \leq v_k$. By Proposition \ref{prop: Mdist_est}, estimate \eqref{eq: Mdist_est_gen} holds for $u_k,v_k$. Using Proposition \ref{prop: MA_cont} and definition \eqref{eq: dp_def_general} we can take the limit $k \to \infty$ in the estimates for $u_k,v_k$, and obtain \eqref{eq: Mdist_est_gen}.
\end{proof}

By the next corollary we will only need to prove Theorem \ref{thm: pythagorean} for smooth potentials $u_0,u_1$: 

\begin{corollary} \label{cor: d_p_monotone_limit}
If $\{w_k\}_{k \in \Bbb N} \subset \mathcal E_p(X,\o)$ decreases (increases a.e.) to $w \in \mathcal E_p(X,\o)$ then $d_p(w_k,w)\to 0$.
\end{corollary}

\begin{proof} By the previous lemma, we have $d_p(w,w_k)^p \leq \int_X |w -w_k|^p (\o_{w_k}^n + \o_{w}^n)$. We can use Proposition \ref{prop: MA_cont} again to conclude that $d_p(w,w_k) \to 0$.
\end{proof}

As $P(u_0,u_1) \in \mathcal H_\o^{1,\bar 1}$ for $u_0,u_1 \in \mathcal H_\o$ (Corollary \ref{cor: rooftop_env_reg}), we need to generalize Theorem \ref{thm: XXChenThm} for endpoints in $\mathcal H_\o^{1,\bar 1}$, before we can prove Theorem \ref{thm: pythagorean}. The first step is the next lemma:

\begin{lemma} Suppose $u_0,u_1 \in \mathcal H_\o^{1,\bar 1}$ and $t \to u_t$ is the bounded geodesic connecting them. Then the following holds:
\begin{equation}\label{eq: lengthequal}
\int_X |\dot u_0|^p \o_{u_0}^n=\int_X |\dot u_1|^p \o_{u_1}^n
\end{equation}
\end{lemma}

Recall that it is not yet known if $t \to u_t$ is $C^1$ in the $t$-direction when $u_0,u_1 \in \mathcal H_\o^{1,\bar 1}$. However, since $t \to u_t$ is $t$-convex, it makes sense to define $\dot u_0$ as the right derivative at $t =0$ and $\dot u_1$ as the left derivative at $t =0$.

\begin{proof} To obtain \eqref{eq: lengthequal} we prove:
\begin{equation}\label{eq: lengthequal1}\int_{\{\dot u_0 >0\}} |\dot u_0|^p\o_{u_0}^n=\int_{\{\dot u_1 > 0\}} |\dot u_1|^p \o_{u_1}^n,
\end{equation}
\begin{equation}\label{eq: lengthequal2} \int_{\{\dot u_0 <0\}} |\dot u_0|^p\o_{u_0}^n=\int_{\{\dot u_1 <0\}} |\dot u_1|^p\o_{u_1}^n.
\end{equation}
Using Remark \ref{rem: MA_form_remark} and Lemma \ref{lem: sublevel_lemma} multiple times we can write:
\begin{flalign*} \int_{\{\dot u_0 > 0\}} |\dot u_0|^p \o_{u_0}^n&= p\int_0^{\infty} \tau^{p-1}\o_{u_0}^n(\{{\dot u_0 \geq \tau}\})d\tau\\
&= p\int_0^{\infty} \tau^{p-1}\o_{u_0}^n(\{ P(u_0,u_1 - \tau)=u_0\})d\tau\\
&= p\int_0^{\infty} \tau^{p-1}(\textup{Vol}(X)-\o_{ u_1}^n(\{ P(u_0,u_1 - \tau)=u_1-\tau\}))d\tau\\
&= p\int_0^{\infty} \tau^{p-1} \o_{u_1}^n(\{ P(u_0,u_1 - \tau)<u_1-\tau\})d\tau\\
&= p\int_0^{\infty} \tau^{p-1}\o_{u_1}^n(\{ P(u_0+\tau,u_1)<u_1\})d\tau\\
&= p\int_0^{\infty} \tau^{p-1}\o_{u_1}^n(\{ \dot u_1 > \tau\})d\tau\\
&=\int_{\{\dot u_1 > 0\}} |\dot u_1|^p\o_{ u_1}^n,
\end{flalign*}
where in the second line we used Lemma \ref{lem: sublevel_lemma}, in the third line we used Remark \ref{rem: MA_form_remark} and in the sixth line we used Lemma \ref{lem: sublevel_lemma} again, but this time for the ``reversed" geodesic $t \to u_{1-t}$. Formula \eqref{eq: lengthequal2} follows if we apply \eqref{eq: lengthequal1} to the bounded geodesic $t \to u_{1-t}$.
\end{proof}

\begin{proposition} \label{prop: distgeod_general} Suppose $u_0,u_1 \in \mathcal H_\o^{1,\bar 1}$ and $t \to u_t$ is the bounded geodesic connecting them. Then we have:
\begin{equation}\label{eq: distgeod_formula}
d_p(u_0,u_1)^p = \int_X |\dot u_0|^p \o_{u_0}^n=\int_X |\dot u_1|^p\o_{u_1}^n.
\end{equation}
\end{proposition}
\begin{proof} As usual, let $u^k_0,u^k_1 \in \mathcal H_\o$ be a sequence of potentials decreasing to $u_0,u_1$. Let $[0,1] \ni t \to u^{kl}_t \in \mathcal H_\o^{1,\bar 1}$ be the $C^{1,\bar 1}$--geodesic joining $u_0^k,u^l_1$. By Theorem \ref{thm: XXChenThm} we have
$$d_p(u_0^k,u^l_1)^p = \int_X |\dot u^{kl}_0|^p\o_{u_0^k}^n.$$
If we let $l \to \infty$, by Lemma \ref{lem: geod_tangent_limit} and \eqref{eq: dp_def_general} we obtain that that
$$d_p(u_0^k,u_1)^p = \int_X |\dot u^k_0|^p\o_{u_0^k}^n,$$
where $t \to u^k_t$ is bounded geodesic connecting $u^k_0$ with $u_1$. Using the previous lemma we can write:
$$d_p(u_0^k,u_1)^p = \int_X |\dot u^k_1|^p\o_{u_1}^n.$$
Letting $k \to \infty$, another application of Lemma \ref{lem: geod_tangent_limit} yields \eqref{eq: distgeod_formula} for $t=1$, and the case $t=0$ follows by symmetry.
\end{proof}

\begin{proof}[Proof of Theorem \ref{thm: pythagorean}] By Corllary  \ref{cor: d_p_monotone_limit}, it is enough to prove the Pythgaorean formula for $u_0,u_1 \in \mathcal H_\o$. According to Corollary \ref{cor: rooftop_env_reg} we have $P(u_0,u_1) \in \mathcal H_\o^{1,\bar 1}$. Suppose $[0,1] \ni t\to u_t \in \mathcal H_\o^{1,\bar 1}$ is the $C^{1,\bar 1}$--geodesic connecting $u_0,u_1$. By Theorem \ref{thm: XXChenThm} we have:
$$d_p(u_0,u_1)^p = \int_X |\dot u_0|^p \o_{u_0}^n.$$
To complete the argument we will prove the following:
\begin{equation}\label{eq: u1dist}
d_p(u_1, P(u_0,u_1))^p = \int_{\{\dot u_0 >0\}} |\dot u_0|^p \o_{u_0}^n,
\end{equation}
\begin{equation}\label{eq: u0dist}
d_p(u_0, P(u_0,u_1))^p = \int_{\{\dot u_0 < 0\}} |\dot u_0|^p \o_{u_0}^n.
\end{equation}
We prove now \eqref{eq: u1dist}. Using Lemma \ref{lem: sublevel_lemma} we can write:
\begin{flalign*} \int_{\{\dot u_0 > 0\}} |\dot u_0|^p \o_{u_0}^n&= p\int_0^{\infty} \tau^{p-1}\o_{u_0}^n(\{{\dot u_0 \geq \tau}\})d\tau\\
&= p\int_0^{\infty} \tau^{p-1}\o_{u_0}^n(\{ P(u_0,u_1 - \tau)=u_0\})d\tau.
\end{flalign*}
Suppose $t\to \tilde u_t $ is the bounded geodesic connecting $P(u_0,u_1),u_1$. Since $P(u_0,u_1) \leq \tilde u_t, \ t \in [0,1]$, the correspondence $(t,x) \to \tilde u_t(x)$ is increasing in the $t$ variable, hence $\dot{\tilde {u}}_0 \geq 0$. By  \eqref{eq: distgeod_formula}, Lemma \ref{lem: sublevel_lemma} and Proposition  \ref{prop: MA_form} we can write:
\begin{flalign*} d_p (P(u_0,u_1),u_1)^p &=\int_X |\dot{\tilde {u}}_0|^p \o_{P(u_0,u_1)}^n=\int_{\{\dot{\tilde {u}}_0 > 0\}} |\dot{\tilde {u}}_0|^p \o_{P(u_0,u_1)}^n\\
&= p\int_0^{\infty} \tau^{p-1} \o_{P(u_0,u_1)}^n(\{ \dot{\tilde {u}}_0 \geq \tau\})d\tau\\
&= p\int_0^{\infty} \tau^{p-1} \o_{ P(u_0,u_1)}^n(\{ P(P(u_0,u_1),u_1 - \tau)=P(u_0,u_1)\})d\tau\\
&= p\int_0^{\infty} \tau^{p-1} \o_{ P(u_0,u_1)}^n(\{ P(u_0,u_1,u_1 - \tau)=P(u_0,u_1)\})d\tau\\
&= p\int_0^{\infty} \tau^{p-1}\o_{ P(u_0,u_1)}^n(\{ P(u_0,u_1-\tau)=P(u_0,u_1)\})d\tau\\
&= p\int_0^{\infty} \tau^{p-1}\o_{u_0}^n(\{ P(u_0,u_1-\tau)=P(u_0,u_1)=u_0\})d\tau\\
&= p\int_0^{\infty} \tau^{p-1}\o_{u_0}^n(\{ P(u_0,u_1-\tau)=u_0\})d\tau,
\end{flalign*}
where in the third line we have used Lemma \ref{lem: sublevel_lemma}, in the sixth line we have used Proposition  \ref{prop: MA_form} together with the fact that $\{ P(u_0,u_1)=u_1\}\cap\{ P(u_0,u_1 -\tau)=P(u_0,u_1)\}$ is empty for $\tau>0$.

Comparing our above calculations \eqref{eq: u1dist} follows. One can conclude \eqref{eq: u0dist} from \eqref{eq: u1dist} after reversing the roles of $u_0,u_1$ and then using \eqref{eq: lengthequal1}.
\end{proof}

As another application of Theorem \ref{thm: pythagorean} we will show that the $d_p$ metric is comparable to a concrete analytic expression:

\begin{theorem}[\cite{da2}] \label{thm: Energy_Metric_Eqv} For any $u_0,u_1 \in \mathcal E_p(X,\o)$ we have
\begin{equation}\label{eq: Energy_Metric_Eqv}
\frac{1}{2^{p-1}} d_p(u_0,u_1)^p \leq \int_X |u_0 - u_1|^p \o_{u_0}^n + \int_X |u_0 - u_1|^p \o_{u_1}^n\leq {2^{2n + 3p + 3}} d_p(u_0,u_1)^p.
\end{equation}
\end{theorem}
\begin{proof} To obtain the first estimate we use the triangle inequality and Lemma \ref{lem: Mdist_est_gen}:
\begin{flalign*}
d_p(u_0,u_1)^p &\leq
(d_p(u_0,\max(u_0,u_1)) + d_p(\max(u_0,u_1),u_1))^p\\
& \leq 2^{p-1}(d_p(u_0,\max(u_0,u_1))^p + d_p(\max(u_0,u_1),u_1)^p)\\
&\leq 2^{p-1}\Big(\int_X |u_0-\max(u_0,u_1)|^p \o_{u_0}^n + \int_X |\max(u_0,u_1)-u_1|^p \o_{u_1}^n\Big)\\
&= 2^{p-1}\Big(\int_{\{u_1 > u_0\}} |u_0-u_1|^p \o_{u_0}^n + \int_{\{u_0>u_1\}} |u_0-u_1|^p \o_{u_1}^n\Big)\\
&\leq 2^{p-1}\Big(\int_X |u_0-u_1|^p \o_{u_0}^n + \int_X |u_0-u_1|^p \o_{u_1}^n\Big).
\end{flalign*}
Now we deal with the second estimate in \eqref{eq: Energy_Metric_Eqv}. By the next result result and Theorem \ref{thm: pythagorean} we can write
\begin{flalign*}
2^{n+p+1}d_p(u_0,u_1)^p &\geq d_p\Big(u_0,\frac{u_0 + u_1}{2}\Big)^p \geq d_p\Big(u_0,P\Big(u_0,\frac{u_0 + u_1}{2}\Big)\Big)^p\\
&\geq \int_X \Big|u_0 - P\Big(u_0,\frac{u_0 + u_1}{2}\Big)\Big|^p \o_{u_0}^n.
\end{flalign*}
By a similar reasoning as above, and the fact that $2^n \o^n_{(u_0 + u_1)/2} \geq \o^n_{u_0}$ we can write:
\begin{flalign*}
2^{n+p+1}d_p(u_0,u_1)^p &\geq d_p\Big(u_0,\frac{u_0 + u_1}{2}\Big)^p \geq d_p\Big(\frac{u_0+u_1}{2},P\Big(u_0,\frac{u_0 + u_1}{2}\Big)\Big)^p\\
&\geq \int_X \Big|\frac{u_0+u_1}{2} - P\Big(u_0,\frac{u_0 + u_1}{2}\Big)\Big|^p \o_{(u_0 + u_1)/2}^n\\
&\geq \frac{1}{2^n} \int_X \Big|\frac{u_0+u_1}{2} - P\Big(u_0,\frac{u_0 + u_1}{2}\Big)\Big|^p \o^n_{u_0}.
\end{flalign*}
Adding the last two estimates, and using the convexity of $t \to t^p$ we obtain:
\begin{flalign*}
2^{2n+p+2} d_p(u_0,u_1)^p&\geq  \int_X \Big(\Big|u_0 - P\Big(u_0,\frac{u_0 + u_1}{2}\Big)\Big|^p+\Big|\frac{u_0+u_1}{2} - P\Big(u_0,\frac{u_0 + u_1}{2}\Big)\Big|^p \Big) \o^n_{u_0}\\
&\geq \frac{1}{2^{2p}}\int_X |u_0 - u_1|^p  \o^n_{u_0}.
\end{flalign*}
By symmetry we also have $2^{2n +3p + 2}d_p(u_0,u_1)^p \geq \int _X |u_0 - u_1|^p\o_{u_1}^n$, and adding these last two estimates together the second inequality in \eqref{eq: Energy_Metric_Eqv}  follows.
\end{proof}

\begin{lemma} \label{lem: halwayest} Suppose $u_0,u_1 \in \mathcal E_p(X,\o)$. Then we have
$$d_p\Big(u_0,\frac{u_0+u_1}{2}\Big)^p \leq 2^{n+p+1} d_p(u_0,u_1)^p.$$
\end{lemma}

\begin{proof} Using Theorem \ref{thm: pythagorean} and Lemma \ref{lem: NaiveCompare} we can start writing:
\begin{flalign*}
d_p\Big(u_0,\frac{u_0 + u_1}{2}\Big)^p &= d_p\Big(u_0, P\Big(u_0,\frac{u_0 + u_1}{2}\Big)\Big)^p + d_p\Big(\frac{u_0 + u_1}{2},P\Big(u_0,\frac{u_0 + u_1}{2}\Big)\Big)^p\\
&\leq d_p(u_0, P(u_0,u_1))^p + d_p\Big(\frac{u_0 + u_1}{2},P(u_0,u_1)\Big)^p\\
&\leq  \int_X |u_0 - P(u_0,u_1)|^p \o^n_{P(u_0,u_1)} + \int_X\Big|\frac{u_0+u_1}{2} - P(u_0,u_1)\Big|^p\o^n_{P(u_0,u_1)}\\
&\leq \frac{3}{2}\int_X |u_0 - P(u_0,u_1)|^p \o^n_{P(u_0,u_1)} + \frac{1}{2}\int_X|u_1 - P(u_0,u_1)|^p\o^n_{P(u_0,u_1)}\\
& \leq 3 \cdot 2^{n+p-1} d_p(u_0, P(u_0,u_1))^p + 2^{n+p-1}d_p(u_1,P(u_0,u_1))^p\\
&\leq 2^{n+p+1}d_p(u_0,u_1)^p,
\end{flalign*}
where in the second line we have used Lemma \ref{lem: NaiveCompare} and the fact that $P(u_0,u_1) \leq P(u_0,(u_0 + u_1)/2)$, in the third and fifth line Lemma \ref{lem: Mdist_est_gen}, and in the sixth line we have used Theorem \ref{thm: pythagorean} again.
\end{proof}

\section{The complete metric spaces $(\mathcal E_p(X,\o),d_p)$}

To show completeness of $(\mathcal E_p(X,\o),d_p)$ we have to argue that limits of Cauchy sequences land in $\mathcal E_p(X,\o)$. According to our first result, which extends Corollary \ref{cor: d_p_monotone_limit}, monotone $d_p$--bounded sequences are Cauchy, and their limit is in $\mathcal E_p(X,\o)$:  

\begin{lemma}\label{lem: lemma mononton_seq} Suppose $\{u_k\}_{k \in \Bbb N} \subset \mathcal E_p(X,\o)$ is a decreasing/increasing  $d_p$--bounded sequence. Then $u = \lim_{k \to \infty} u_k \in \mathcal E_p(X,\o)$ and additionally $d_p(u,u_k) \to 0$.
\end{lemma}
\begin{proof} First we assume that $u_k$ is decreasing.  Without loss of generality, we can also suppose that $u_k < 0$. By Lemma \ref{lem: Mdist_est_gen} it follows that
\begin{equation}\label{eq: some_int_est}
\max(\int_X |u_j|^p \o_{u_j}^n,\int_X |u_j|^p \o^n) \leq 2^{n+p}d_p(u_j,0)^p \leq D.
\end{equation}
As $\int_X |u_j|^p \o^n$ is uniformly bounded, it follows that the limit function $u= \lim_k u_k$ exists and $u \in \textup{PSH}(X,\o)$. Since $E_p(u_j)=\int_X |u_j|^p \o_{u_j}^n$ is uniformly bounded we can use Lemma \ref{lem: E_semicont}, to obtain that $u \in \mathcal E_p(X,\o)$ and Corollary \ref{cor: d_p_monotone_limit} gives that $d_p(u_k,u) \to 0$.

Now we turn to the case when  $u_k$ is increasing. We know that $d_p(u_k,0)$ is uniformly bounded, hence using Theorem \ref{thm: Energy_Metric_Eqv} we obtain that $\int_X |u_k|^p \o^n$ is uniformly bounded as well. By the $L^1$--compactness of subharmonic functions  \cite[Chapter I, Proposition 4.21]{De} we obtain that $u_k \nearrow u \in \textup{PSH}(X,\o)$ a.e. with respect to $\o^n$. By the monotonicity property (Corollary \ref{cor: monotonicity_E_chi}) we obtain that $u \in \mathcal E_p(X,\o)$ and another application of Corollary \ref{cor: d_p_monotone_limit}  gives that $d_p(u_k,u) \to 0$.
\end{proof}

As another consequence of the Pythagorean identity, we note the contractivity of the operator $v \to P(u,v)$, when restricted to finite energy spaces. This result will be essential in proving that $(\mathcal E_p(X,\o),d_p)$ is a complete metric space.

\begin{proposition} \label{prop: contractivity} Given $u,v,w \in \mathcal E_p(X,\o)$ we have
$$d_p (P(u,v),P(u,w)) \leq d_p(v,w).$$
\end{proposition}
\begin{proof}By Corollary \ref{cor: d_p_monotone_limit} we can assume that $u,v,w \in \mathcal H_\o^{1,\bar 1}$ and $P(u,v),P(u,w) \in \mathcal H_\o^{1,\bar 1}$ (see Corollary \ref{cor: rooftop_env_reg}).

First we assume that $v \leq w$.  Let $t \to \phi_t$ be the $C^{1,\bar 1}$ geodesic connecting $v,w$, and $t \to \psi_t$ be the bounded geodesic connecting $P(u,v),P(u,w)$. By Proposition \ref{prop: distgeod_general} we have to argue that
\begin{equation}\label{eq: dist_est}
\int_X |\dot{\psi_0}|^p\o_{P(u,v)}^n\leq\int_X |\dot{\phi_0}|^p \o_{v}^n.
\end{equation}
Proposition \ref{prop: MA_form} implies that
$$\int_X |\dot{\psi_0}|^p\o_{P(u,v)}^n\leq \int_{\{P(u,v)=u\}} |\dot{\psi_0}|^p\o_{u}^n + \int_{\{P(u,v)=v\}} |\dot{\psi_0}|^p \o_{v}^n.$$
We argue that the first term in this sum is zero. As $P(u,v) \leq P(u,w) \leq u$, by \eqref{eq: udef1} (or Theorem \ref{thm: uniqueness_BVP}), it is clear that $P(u,v) \leq \psi_t \leq u, \ t \in [0,1]$. Hence, if $x \in \{ P(u,v)=u\}$ then $\psi_t(x)=u(x), \ t \in [0,1]$, implying $\dot \psi_0 \big|_{\{ P(u,v)=u\}}\equiv 0$.

At the same time, using Theorem \ref{thm: uniqueness_BVP}  again, it follows that $\psi_t \leq \phi_t, \ t \in [0,1]$. This implies that $0 \leq \dot \psi_0 \big|_{\{ P(u,v)=v\}}\leq\dot \phi_0 \big|_{\{ P(u,v)=v\}},$ which in turn implies \eqref{eq: dist_est}.

The general case follows now from an application of the Pythagorean formula (Theorem \ref{thm: pythagorean}) and what we just proved above:
\begin{flalign*}
d_p(P(u,v),P(u,w))^p&= d_p(P(u,v),P(u,v,w))^p + d_p(P(u,w),P(u,v,w))^p\\
&= d_p(P(u,v),P(u,P(v,w)))^p + d_p(P(u,w),P(u,P(v,w)))^p\\
&\leq d_p(v,P(v,w))^p + d_p(w,P(v,w))^p\\
&= d_p(v,w)^p.
\end{flalign*}
\end{proof}

We arrive to the main result of this section. Using the previous proposition we will show that any Cauchy sequence of $(\mathcal E_p,d_p)$ is equivalent to a monotonic Cauchy sequence. By the Lemma \ref{lem: lemma mononton_seq}, such sequences have limit in $\mathcal E_p(X,\o)$, showing that this space is complete with respect to $d_p$:

\begin{theorem}[\cite{da2}]  \label{thm: EpComplete}  $(\mathcal E_p (X, \o),d_p)$ is a geodesic metric space, which is the metric completion of $(\mathcal H_\o,d_p)$. Additionally, the finite energy geodesic $t \to u_t$ connecting $u_0,u_1 \in \mathcal E_{p}(X, \o)$ is a $d_p$-geodesic.
\end{theorem}

\begin{proof} By Theorem \ref{thm: BK_approx} and Corollary \ref{cor: d_p_monotone_limit} $\mathcal H_\o$ is a $d_p$--dense subset of $\mathcal E_p(X,\o)$. The statement about geodesics was addressed in Theorem \ref{thm: e2space}, hence we only need to argue completeness. 

Suppose $\{u_k\}_{k \in \Bbb N} \subset \mathcal E_p(X,\o)$ is a $d_p$--Cauchy sequence. We will prove that there exists $v \in \mathcal E_p(X,\o)$ such that $d_p (u_k,v) \to 0.$ After passing to a subsequence we can assume that
$$d_p(u_l,u_{l+1}) \leq 1/2^l, \ l \in \Bbb N.$$
By Proposition \ref{prop: env_exist} we can introduce $v^k_l = P(u_k,u_{k+1},\ldots,u_{k+l}) \in \mathcal E_p(X,\o), \ l,k \in \Bbb N$. We argue first that each decreasing sequence $\{ v^k_l\}_{l \in \Bbb N}$ is $d_p$--Cauchy. 
We observe that $v^k_{l+1}=P(v^k_l,u_{k+ l+1})$ and $v^k_l=P(v^k_l,u_{k+l})$. Using this and Proposition \ref{prop: contractivity} we can write:
$$d_p(v^k_{l+1},v^k_l) = d_p(P(v^k_l,u_{k+l+1}),P(v^k_l,u_{k+l})) \leq d_p(u_{k+l+1}, u_{k+l})\leq \frac{1}{2^{k+l}}.$$

By Lemma \ref{lem: lemma mononton_seq}, it follows now that each sequence $\{ v^k_l\}_{l \in \Bbb N}$  is $d_p$--convergening to some $v^k \in \mathcal E_p(X,\o)$. Using the same trick as above, we can write:
\begin{flalign*}
d_p(v^k,v^{k+1}) &=\lim_{l \to \infty}d_p(v^k_{l+1},v^{k+1}_l)= \lim_{l \to \infty}d_p(P(u_k,v^{k+1}_{l}),P(u_{k+1},v^{k+1}_l))\\
&\leq d_p (u_k,u_{k+1}) \leq \frac{1}{2^k},
\end{flalign*}\vspace{-0.7cm}
\begin{flalign*}
d_p(v^k,u_k) &=\lim_{l \to \infty}d_p(v^k_l,u_k)=\lim_{l \to \infty}d_p(P(u_k,v^{k+1}_{l-1}),P(u_k,u_k))\\
&\leq\lim_{l \to \infty}d_p(v^{k+1}_{l-1},u_k)=\lim_{l \to \infty}d_p(P(u_{k+1},v^{k+2}_{l-2}),u_k)\\
&\leq \lim_{l \to \infty}d_p(P(u_{k+1},v^{k+2}_{l-2}),u_{k+1}) + d_p(u_{k+1},u_k)\\
&\leq \lim_{l \to \infty} \sum_{j=k}^{l+k} d_p (u_j,u_{j+1}) \leq \frac{1}{2^{k-1}}.
\end{flalign*}
Consequently, $\{v^k\}_{k \in \Bbb N}$ is an increasing $d_p$--bounded $d_p$--Cauchy sequence that is equivalent to $\{u_k\}_{k \in \Bbb N}$. By Lemma \eqref{lem: lemma mononton_seq} there exists $v \in \mathcal E_p(X,\o)$ such that $d_p(v_k,v) \to 0$, which in turn implies that $d_p(u_k,v)\to 0$, finishing the proof.
\end{proof}

\section{Special features of the $L^1$ Finsler geometry}

When it comes to applications of measure theory, the most important $L^p$ spaces are the $L^1$ space, its dual $L^\infty$, and the Hilbert space $L^2$. A similar pattern can be observed with the $L^p$ Finsler geometries on $\mathcal H_\o$. One can show that the $d_\infty$ metric is a multiple of the usual $L^\infty$ metric. As follows from the work of Calabi--Chen \cite{cc}, the completion of $(\mathcal H_\o,d_2)$ is non--postively curved and as such it provides fertile ground to the study of geometric gradient flows explored in \cite{st1,st2,bdl1}.

In this section we will focus exclusively on the $L^1$ Finsler geometry of $\mathcal H_\o$, whose path length metric structure has a number of interesting properties making it suitable in our study of canonical K\"ahler metrics, detailed in the next chapter.

The starting point is the Monge--Amp\`ere energy $I: \textup{PSH}(X,\o)\cap L^\infty \to \Bbb R$ (sometimes called Aubin--Yau, or Aubin--Mabuchi energy):
\begin{equation}\label{eq: I_energy_def}
I(u):= \frac{1}{(n+1)V}\sum_{j=0}^{n}\int_X u \o_{u}^j \wedge \o^{n-j}.
\end{equation}
The following lemma explains the choice of name for the $I$ energy, as it turns out that the first order variation of this functional is exactly the complex Monge--Amp\`ere operator:
\begin{lemma} \label{lem: I_differential} Suppose $[0,1]\ni t \to v_t \in \mathcal H_\o$ is a smooth curve. Then $t \to I(v_t)$ is differentiable and
$$\frac{d}{dt}I(v_t)= \frac{1}{V}\int_X \dot v_t \o_{v_t}^n.$$
\end{lemma}
\begin{proof}Using the definition of $I$ we can calculate the differential of $t \to I(v_t)$ directly:
\begin{flalign*}
\frac{d}{dt}I(v_t) &= \frac{1}{(n+1)V} \bigg ( \sum_{j=0}^{n}\int_X \dot v_t \o_{v_t}^j \wedge \o^{n-j} + \sum_{j=0}^{n} j \int_X v_t i\ddbar \dot v_t \wedge \o_{v_t}^{j-1} \wedge \o^{n-j}\bigg)\\
&= \frac{1}{(n+1)V} \bigg ( \sum_{j=0}^{n}\int_X \dot v_t \o_{v_t}^j \wedge \o^{n-j} + \sum_{j=0}^{n} j \int_X \dot v_t i\ddbar v_t \wedge \o_{v_t}^{j-1} \wedge \o^{n-j}\bigg)\\
&= \frac{1}{(n+1)V} \bigg ( \sum_{j=0}^{n}\int_X \dot v_t \o_{v_t}^j \wedge \o^{n-j} + \sum_{j=0}^{n} j \int_X \dot v_t ( \o_{v_t}  - \o) \wedge \o_{v_t}^{j-1} \wedge \o^{n-j}\bigg)\\
&= \frac{1}{V} \int_X \dot v_t \o_{v_t}^n.
\end{flalign*}
\end{proof}
Based on the last lemma, we deduce a number of properties of the Monge--Amp\`ere energy:

\begin{proposition}\label{prop: I_energy_prop}
Given $u,v \in \textup{PSH}(X,\omega) \cap L^\infty(X)$, the following hold:\\
\begin{equation}\label{eq: I_energy_diff}
I(u)-I(v) = \frac{1}{(n+1)V}\sum_{j=0}^{n}\int_X (u-v) \o_{u}^j \wedge \o_v^{n-j},
\end{equation}
\begin{equation}\label{eq: I_energy_diff_est}
 \frac{1}{V}\int_X (u-v)\o_u^{n} \leq I(u)-I(v) \leq \frac{1}{V}\int_X (u-v)\o_v^{n}.
\end{equation}
\end{proposition}
\begin{proof}First we show \eqref{eq: I_energy_diff} for $u,v \in \mathcal H_\o$. Let $[0,1] \ni t \to h_t := (1-t)v + t u \in \mathcal H_\o$ be the smooth affine curve joining $v$ and $u$. By the previous lemma and the binomial theorem we can write
\begin{flalign*}
I(u)-I(v)&= \int_0^1 \frac{d}{dt} I(h_t)dt = \frac{1}{V}\int_X (u-v) \int_0^1 ((1-t)\o_v + t \o_u)^n dt\\
&= \frac{1}{V} \sum_{j=0}^n {n \choose j} \cdot \int_0^1 t^j (1-t)^{n-j} dt  \cdot  \int_X (u-v) \o_u^j \wedge \o_v^{n-j} dt
\end{flalign*}
After an elementary calculation involving integration by parts we obtain that 
$$\int_0^1 t^j (1-t)^{n-j}dt = \frac{j! (n-j)!}{(n+1)!},$$
giving \eqref{eq: I_energy_diff} for $u,v \in \mathcal H_\o$. For the general case $u,v \in \mathcal E^1$, by Theorem \ref{thm: BK_approx} we can find $u_k,v_k \in \mathcal H_\o$ decreasing to $u$ and $v$ respectively. We can use Proposition \ref{prop: MA_cont} to conclude that $I(u_k) \to I(u)$, $I(v_k) \to I(v)$ and $\int_X (u_k - v_k)\o_{u_k}^j \wedge \o_{v_k}^{n-j} \to \int_X (u - v)\o_{u}^j \wedge \o_{v}^{n-j}$, giving \eqref{eq: I_energy_def} for elements of $\textup{PSH}(X,\omega) \cap L^\infty(X)$.

Now we turn to the estimates in \eqref{eq: I_energy_diff_est}. First we prove the following estimate:
\begin{flalign}\label{eq: I_energy_int_parts_ineq}
\int_X (u-v) \o_u^j &\wedge \o_v^{n-j} = \int_X (u-v) \o_u^{j-1} \wedge \o_v^{n-j + 1} + \int_X (u-v) i\ddbar (u-v) \wedge \o_u^{j-1} \wedge \o_v^{n-j} \nonumber\\
& = \int_X (u-v) \o_u^{j-1} \wedge \o_v^{n-j + 1} - \int_X i\partial (u-v) \wedge i\dbar (u-v) \wedge \o_u^{j-1} \wedge \o_v^{n-j}\nonumber \\
& \leq \int_X (u-v) \o_u^{j-1} \wedge \o_v^{n-j + 1}.
\end{flalign}
Using this estimate inductively for the terms of \eqref{eq: I_energy_diff} yields \eqref{eq: I_energy_diff_est}.
\end{proof}
As a consequence of \eqref{eq: I_energy_diff} we note the monotonicity property of $I$:
\begin{corollary} Suppose $u,v \in \textup{PSH}(X,\o) \cap L^\infty$ such that $u \geq v$. Then $I(u) \geq I(v)$.
\end{corollary}
This results allows to extend the definition of $I$ to $\textup{PSH}(X,\o)$. Indeed, if $u \in \textup{PSH}(X,\o)$ then using the canonical cutoffs $u_k = \max(u,-k)$ we can write:
\begin{equation}\label{eq: I_def_general}
I(u) := \lim_{k \to \infty}I(u_k).
\end{equation}
Though the limit in the above definition is well defined, it may equal $-\infty$ for certain potentials $u$. By our next result $I$ is finite exactly on $\mathcal E_1(X,\o)$. What is more, $I$ is $d_1$--continuous on $\mathcal E_1(X,\o)$, and Proposition \ref{prop: I_energy_prop} also extends to this bigger space:
\begin{proposition}\label{prop: I_finiteness_and_cont}
Let $u \in \textup{PSH}(X,\o)$. Then $I(u) > -\infty$ if and only if $u \in \mathcal E_1(X,\o)$. 
In addition, the following hold for $u_0,u_1 \in \mathcal E_1(X,\omega)$:

\begin{equation}\label{eq: I_Lipschitz_cont}
|I(u_0)-I(u_1)| \leq d_1(u_0,u_1),
\end{equation}
\begin{equation}\label{eq: I_energy_diff_E1}
I(u_0)-I(u_1) = \frac{1}{(n+1)V}\sum_{j=0}^{n}\int_X (u_0-u_1) \o_{u_0}^j \wedge \o_{u_1}^{n-j},
\end{equation}
\begin{equation}\label{eq: I_energy_diff_est_E1}
 \frac{1}{V}\int_X (u_0-u_1)\o_{u_0}^{n} \leq I(u_0)-I(u_1) \leq \frac{1}{V}\int_X (u_0-u_1)\o_{u_1}^{n}.
\end{equation}
\end{proposition}
\begin{proof}

By definition, $I(u+c)=I(u) +c$ for all $c \in \Bbb R$, hence we can assume that $u \leq 0$. Consequently, by \eqref{eq: I_energy_diff_est} for the canonical cutoffs $u_k=\max(u,-k)$ we have:
$$ -\frac{E_1(u_k)}{V} = \frac{1}{V} \int_X u_k \o_{u_k}^n \leq I(u_k) \leq \frac{1}{(n+1)V} \int_X u_k \o_{u_k}^n=-\frac{E_1(u_k)}{(n+1)V}.$$
If $u \in \mathcal E_1(X,\o)$ then by the fundamental estimate (Proposition \ref{prop: Energy_est}) we have that $E_1(u_k)$ is uniformly bounded, hence $I(u)=\lim_k I(u_k)$ is finite. On the other hand, if $I(u)$ is finite then $I(u_k)$ is uniformly bounded, hence so is $E_1(u_k)$. By Lemma \ref{lem: E_semicont} we obtain that $u \in \mathcal E_1(X,\o)$.

Now we turn to \eqref{eq: I_Lipschitz_cont}. By Corollary \ref{cor: d_p_monotone_limit} and \eqref{eq: I_def_general} it follows that it is enough to prove \eqref{eq: I_Lipschitz_cont} for bounded potentials. Furthermore, using Theorem \ref{thm: BK_approx} it is actually enough to prove \eqref{eq: I_Lipschitz_cont} for potentials in $\mathcal H_\o$. For $u_0,u_1 \in \mathcal H_\o$ let $t \to u_t$ be the $C^{1,\bar 1}$ geodesic connecting $u_0,u_1$. 
By \eqref{eq: dIdt_C11geod}, proved in the next lemma, we can finish our argument:
$$|I(u_1)-I(u_0)| = \Big| \int_0^1 \frac{d}{dt} I(u_t)dt\Big| \leq \frac{1}{V} \int_0^1 \int_X |\dot u_t|\o_{u_t}^n dt = d_1(u_0,u_1),$$
where in the last equality we have used Theorem \ref{thm: XXChenThm}.

Finally,  \eqref{eq: I_energy_diff_E1} and \eqref{eq: I_energy_diff_est_E1} follow from Proposition \ref{prop: I_energy_prop} and Proposition \ref{prop: MA_cont} 
\end{proof}
\begin{lemma}Suppose $t \to u_t$ is the $C^{1,\bar 1}$ geodesic connecting $u_0,u_1 \in \mathcal H_\o$. Then $t \to I(u_t)$ is differentiable, moreover
\begin{equation}\label{eq: dIdt_C11geod} 
\frac{d}{dt} I(u_t) = \frac{1}{V}\int_X \dot u_t \o_{u_t}^n.
\end{equation}
\end{lemma}
\begin{proof} As $t \to u_t$ is $C^{1,\bar 1}$ it follows that $(u_{t +\delta}-u_t)/\delta \to \dot u_t$ uniformly as $\delta \to 0$. Using this and \eqref{eq: I_energy_diff} we can write:
\begin{flalign}
\frac{d}{dt}I(u_t)&=\lim_{\delta \to 0} \frac{I(u_{t + \delta}) - I(u_t)}{\delta} = \frac{1}{(n+1)V} \sum_{j=0}^n\lim_{\delta \to 0} \int_X \frac{u_{t + \delta} - u_t}{\delta}\o_{u_{t + \delta}}^j \wedge \o_{u_{t}}^{n-j} \nonumber\\
&= \frac{1}{V}\int_X \dot u_t \o_{u_t}^n,
\end{flalign}
where we have used that $\o_{u_{t+\delta}}^j\wedge\o_{u_t}^{n-j}\to \o_{u_t}^n$ weakly.
\end{proof}
Further linking the Monge--Amp\`ere energy to the $L^1$ Finsler geometry is the fact that $I$ is affine along finite energy geodesics:
\begin{proposition}\label{prop: I_linrear} Suppose $u_0,u_1 \in \mathcal E_1(X,\o)$ and $t \to u_t$ is the finite energy geodesic connecting $u_0,u_1$. Then $t \to I(u_t)$ is affine: $I(u_t)=(1-t)I(u_0) + t I(u_1), \ t \in [0,1].$
\end{proposition}
\begin{proof} First we show that $t \to I(u_t)$ is affine for the $C^{1,\bar 1}$ geodesic connecting $u_0,u_1 \in \mathcal H_\o$. Let $[0,1] \ni t \to u^{\varepsilon}_t \in \mathcal H_\o$ be the smooth $\varepsilon$--geodesic connecting $u_0,u_1$ (see \eqref{eq: epsBVPGeod}). By Lemma \ref{lem: I_differential} we have that 
$$\frac{d}{dt}I(u_t^\varepsilon)= \frac{1}{V} \int_X \dot u_t^\varepsilon \o_{u_t^\varepsilon}^n  = \langle 1,\dot u_t^\varepsilon \rangle_{u_t^\varepsilon}.$$
As we take one more derivative of the above formula, we can use the $L^2$ Mabuchi Levi--Civita connection (see \eqref{eq: CovDerivative}) and \eqref{eq: eps_geod_eq_Lev_Civ} to obtain:
\begin{equation}\label{eq: I_second_deriv_epsgeod}
\frac{d^2}{dt^2}I(u_t^\varepsilon)= \langle 1,\nabla_{\dot u_t^\varepsilon}\dot u_t^\varepsilon \rangle_{u_t^\varepsilon} = \frac{1}{V}\int_X \nabla_{\dot u_t^\varepsilon}\dot u_t^\varepsilon \o_{u_t^\varepsilon}^n= \frac{1}{V}\int_X \varepsilon \o^n= \varepsilon.
\end{equation}
For fixed $t \in [0,1]$ we have that $u^\varepsilon_t \nearrow u_t$ uniformly, hence \eqref{eq: I_energy_diff} gives that $I(u^\varepsilon_t) \to I(u_t)$. By an elementary argument, \eqref{eq: I_second_deriv_epsgeod} implies that $t \to I(u_t) = \lim_{\varepsilon \to 0} I(u_t^\varepsilon)$ is affine.

Next we show the result for $u_0,u_1 \in \mathcal E_1(X,\o)$. In this case let $u^k_0,u^k_1 \in \mathcal H_\o$ be  decreasing approximating sequences that exist by Theorem \ref{thm: BK_approx}. Let $ t \to u^{k}_t$ be the $C^{1,\bar 1}$ geodesics joining $u^k_0,u^k_1$, and $t \to u_t$ be the finite energy geodesic joining $u_0,u_1$. By Proposition \ref{prop: weak_geod_approx} it follows that for fixed $t$ we have $u^k_t \searrow u_t$. Corollary \ref{cor: d_p_monotone_limit} now gives that $d_1(u^k_t,u_t) \to 0$. We can now use Proposition \ref{prop: I_finiteness_and_cont} to conclude that $I(u^k_t) \to I(u_t)$. Since  $t \to I(u^k_t)$ is affine, so is $t \to I(u_t)$.
\end{proof}

By revisiting the Pythagorean formula of the $L^1$ geometry (Theorem \ref{thm: pythagorean}) we obtain the following concrete formula for the $d_1$ metric in terms of the Monge--Amp\`ere energy:
\begin{proposition} \label{prop: d_1_I_formula} If $u,v \in \mathcal E_1(X,\o)$ then 
\begin{equation}\label{eq: d_1_I_formula}
d_1(u,v)=I(u) + I(v) -2 I(P(u,v)).
\end{equation}
In particular, if $u \geq v$ then $d_1(u,v)=I(u) - I(v)$.
\end{proposition}
\begin{proof} We show that $d_1(u,v)=I(u) - I(v)$ when $u \geq v$. Then the Pythagorean formula (Theorem \ref{thm: pythagorean}) will imply \eqref{eq: d_1_I_formula}. By Theorem \ref{thm: BK_approx} , Corollary \ref{cor: d_p_monotone_limit} and Proposition \ref{prop: I_finiteness_and_cont} it is enough to show $d_1(u,v)=I(u) - I(v)$ for $u,v \in \mathcal H_\o$. Let $[0,1] \ni t \to u_t \in \mathcal H_\o^{1,\bar 1}$ be the $C^{1,\bar 1}$ geodesic joining $u_0:=v$ and $u_1:=u$. As $u_1 \geq u_0$, by Theorem \ref{thm: uniqueness_BVP} we obtain that $u_0 \leq u_t$. Using convexity in the $t$ variable, we obtain that $0 \leq \dot u_0 \leq \dot u_t.$
Since $\dot u_t \geq 0$, by Theorem \ref{thm: XXChenThm} we can conclude that
$$d_1(u,v)= \frac{1}{V}\int_0^1 \int_X \dot u_t \o_{u_t}^n dt = \int_0^1 \frac{d}{dt} I(u_t) dt=I(u_1)-I(u_0)=I(u)-I(v),$$
where in the second equality we have used \eqref{eq: dIdt_C11geod}.   
\end{proof}

Recall that the correspondence $u \to \o_u$ is one-to-one up to a constant. To  make the correspondence between metrics and potentials one-to-one, we need to restrict the map $u \to \o_u$ to a hyper--surface of $\mathcal H_\o$. There are many normalizations that come to mind, but the most convenient choice is to use the following space and its completion:
\begin{equation}\label{eq: H_0_def}
\mathcal H_\o \cap I^{-1}(0)= \{u \in \mathcal H_\o \textup{ s.t. } I(u)=0 \} \ \  \textup{ and }  \ \ \mathcal E_1(X,\o) \cap I^{-1}(0).
\end{equation}
As $I:\mathcal E_1(X,\o) \to \Bbb R$ is $d_1$-continuous, we see that $\mathcal E_1(X,\o) \cap I^{-1}(0)$ is indeed the $d_1$--completion of  $\mathcal H_\o \cap I^{-1}(0)$. We focus on the preimage of $I$ due to the conclusion of Proposition \ref{prop: I_linrear}. Indeed, according to this result, if $u_0,u_1 \in \mathcal E_1(X,\o) \cap I^{-1}(0)$ then $t \to u_t$, the finite energy geodesic connecting $u_0,u_1$, satisfies $u_t \in \mathcal E_1(X,\o) \cap I^{-1}(0)$, hence the ``hypersurface" $\mathcal E_1(X,\o) \cap I^{-1}(0)$ is totally geodesic.

The $J$ energy is closely related to the Monge--Amp\`ere energy and is given by the following formula:
\begin{equation}\label{eq: J_def}
J(u) = \frac{1}{V}\int_X u \o^n - I(u).
\end{equation}
By \eqref{eq: I_energy_diff_est_E1} it follows that $J(u) \geq 0$ for all $u \in \mathcal E_1(X,\o)$. 

For $u \in \mathcal E_1(X,\o) \cap I^{-1}(0)$, the $J$ energy is given by the especially simple formula $J(u)=\frac{1}{V}\int_X u \o^n$. Additionally, on $\mathcal E_1(X,\o) \cap I^{-1}(0)$ the growth of the $d_1$ metric and the $J$ energy is closely related:
\begin{proposition}[\cite{da2,dr2}] \label{prop: d_1_growth_J}
There exists $C=C(X,\o)> 1$ such that
\begin{equation}\label{eq: d_1_growth_J}
\frac{1}{C}J(u) -C \leq d_1(0,u) \leq C J(u) + C, \ \ u \in \mathcal E_1(X,\o) \cap I^{-1}(0).
\end{equation}
\end{proposition}
\begin{proof} Let  $u \in \mathcal E_1(X,\o) \cap I^{-1}(0)$. By Theorem \ref{thm: Energy_Metric_Eqv} we have 
$$J(u)=\frac{1}{V} \int_X u \o^n \leq \frac{1}{V} \int_X |u| \o^n \leq Cd_1(0,u),$$
implying the first estimate in \eqref{eq: d_1_growth_J}. For the second estimate, since $I(u)=0$, we have that $\sup_X u \geq 0$ and 
\begin{equation}\label{eq: somethingd_1}
d_1(0,u)=-2 I(P(0,u)).
\end{equation}
Clearly, $u - \sup_X u \leq \min(0,u)$, so 
$u - \sup_X u \leq P(0,u)$. Thus, $-\sup_X u = I(u - \sup_X u) \leq I(P(0,u)).$ 
Combined with \eqref{eq: somethingd_1}, we obtain that $
d_1(0,u)=-2I(P(0,u))\leq 2 \sup_X u.$
Finally, according to the next basic lemma, $\sup_X u \leq C' \int_X u \o^n + C'$ for some $C'(X,\o) > 1$, finishing the proof.
\end{proof}

\begin{lemma}\label{lem: sup_int_psh_eqv}
There exists $C := C(X,\o) > 1$ such that for any $u \in \textup{PSH}(X,\o)$ we have
$$\frac{1}{V} \int_X u \o^n \leq \sup_X u \leq  \frac{1}{V} \int_X u \o^n + C.$$
\end{lemma}

This is a well known result in pluripotential theory. A proof using compactness can be found in \cite[Proposition 1.7]{gz05}. Below we give a constructive argument using only the sub-mean value property of psh functions.

\begin{proof}
The first estimate is trivial. To argue the second estimate, it is enough to show that $\int_X u \omega^n$ is uniformly bounded from below, for all $u \in \textup{PSH}(X,\omega)$ with $\sup_X u =0$. We fix such potential $u$ for the rest of the proof.

We pick nested coordinate charts $U_k \subset W_k \subset X$, such that $\{U_k\}_{1 \leq k \leq N}$ covers $X$, and there exist holomorphic diffeomorphisms $\varphi_k: B(0,4) \to W_k $ such that  $\varphi_k(B(0,1))=U_k$ and there exists $\psi_k \in C^\infty(W_k)$ such that $\o|_{W_k} = i\ddbar \psi_k$. Notice that it is enough to show the existence of $C:=C(X,\omega)<0$ such that
\begin{equation}\label{eq: patch_estimate}
\int_{B(0,1)} u \circ \varphi_j d\mu \geq C, \ \ \ j \in \{1,\ldots,N\},
\end{equation}
where $d\mu$ is the Euclidean measure on $\Bbb C^n$.

Using our setup, we obtain that $\psi_k + u \in \textup{PSH}(W_k)$ for all $k \in \{1,\ldots,N\}$.
As $u$ is usc, its supremum is realized at some $x_1 \in X$, i.e. $u \leq u(x_1)=0$. As $\{U_k\}_k$ covers $X$, $x_1 \in U_l$ for some $l \in \{1,\ldots,N \}$. For simplicity, we can assume that $l=1$.

Let $z_1:=\varphi^{-1}_1(x_1) \in B(0,1)$. As $B(z_1,2)\subset B(0,4)$, by the  sub-meanvalue property of psh functions we can write:
\begin{equation}\label{eq: sup_int_est_interm}
\psi_1(x_1)=u\circ \varphi_1(z_1) + \psi_1 \circ \varphi_1(z_1) \leq \frac{1}{\mu(B(z_1,2))}\int_{B(z_1,2)} (u\circ \varphi_1 +\psi_1 \circ \varphi_1)d\mu.
\end{equation}
Since $u \leq 0$ and $B(0,1) \subset B(z_1,2)$, there exists $C_1 < 0$, independent of $u$, such that
\begin{equation}\label{eq: patch_estimate_1}
\int_{B(0,1)} u \circ \varphi_1 d\mu \geq C_1.
\end{equation}
As $\{U_k \}_{k}$ is a covering of $X$, $U_1$ needs to intersect at least another member of the covering. We can assume that $U_1 \cap U_2$ is non-empty.

Since $u \leq 0$ and \eqref{eq: patch_estimate_1} holds, there exists $x_2 \in U_2\cap U_1,r_2 \in (0,1)$ and $\tilde C_1<0$, independent of $u$, such that for $z_2 = \varphi_2^{-1}(x_2)$ we have $\varphi_2(B(z_2,r_2)) \subset U_1 \cap U_2$ and 
$$ \frac{1}{\mu(B(z_2,r_2))}\int_{B(z_2,r_2)}( u \circ \varphi_2 + \psi_2 \circ \varphi_2)d\mu \geq \tilde C_1.$$
Since $ u \circ \varphi_2 + \psi_2 \circ \varphi_2$ is psh in $B(0,4)$, we can further write that:
$$ \frac{1}{\mu(B(z_2,2))}\int_{B(z_2,2)} ( u \circ \varphi_2 + \psi_2 \circ \varphi_2) d\mu \geq \frac{1}{\mu(B(z_2,r_2))}\int_{B(z_2,r_2)}  ( u \circ \varphi_2 + \psi_2 \circ \varphi_2) d\mu \geq \tilde C_1.$$
Since $B(0,1) \subset B(z_2,2)$ and $u \leq 0$, from here we conclude the existence of $C_2 <0$, independent of $u$, such that 
\begin{equation}\label{eq: patch_estimate_2}
\int_{B(0,1)} u \circ \varphi_2 d\mu \geq C_2.
\end{equation}
Continuing the above process, after $N-2$ more steps we eventually arrive at \eqref{eq: patch_estimate}.
\end{proof}

\section{Relation to classical notions of convergence}

Theorem \ref{thm: Energy_Metric_Eqv} gave a  characterization of $d_p$--convergence using concrete analytic expressions. When dealing with $d_1$, it turns out that an even more convenient characterization can be given with the help of the Monge--Amp\`ere energy. This is the content of the next theorem, which also shows that $d_1$--convergence implies classical $L^1$ convergence of potentials and also the weak convergence of the associated complex Monge--Amp\`ere measures:

\begin{theorem}[\cite{da2}] \label{thm: d_1-convergence} Suppose $u_k,u \in \mathcal E_1(X,\o)$. Then the following hold:\\
\noindent (i) $d_1(u_k,u) \to 0$ if and only if $\int_X |u_k -u|\o^n \to 0$ and $I(u_k) \to I(u)$.\\
\noindent (ii) If $d_1(u_k,u) \to 0$ then $\o_{u_k}^n \to \o_u^n$ weakly, and $\int_X |u_k - u|\o_v^n \to 0$ for any $v \in \mathcal E_1(X,\o)$.
\end{theorem}

The proof of this result is given in a number of propositions and lemmas below. The new analytic input will be given by the bi--functional $\mathcal I(\cdot,\cdot):\mathcal E_1(X,\o)\times\mathcal E_1(X,\o) \to \Bbb R$ and its properties:
\begin{equation}\label{eq: I_script_def}
\mathcal I(u_0,u_1)= \int_X (u_0 - u_1)(\o_{u_1}^n - \o_{u_0}^n).
\end{equation}

Observe that by Lemma \ref{prop: mixed_finite_prop: Energy_est} the above expression is indeed finite. Moreover, if $d_1(u_j,u) \to 0$ then by Theorem \ref{thm: Energy_Metric_Eqv} we get $\mathcal I(u_j,u) \to 0$. As it turns out, the relationship between $d_1$ and $\mathcal I$ is much deeper then this simple observation might suggest.

When $u_0,u_1$ are bounded, then Bedford--Taylor theory allows to integrate by parts and obtain:
$$\mathcal I(u_0,u_1) =\sum_{j=0}^{n-1} \int_X i \partial (u_0 - u_1) \wedge \dbar (u_0 - u_1) \wedge \o_{u_0}^{j} \wedge \o_{u_1}^{n-j-1}.$$
In particular, $\mathcal I(u_0,u_1) \geq 0$ for bounded potentials. Applying the next lemma to canonical cutoffs, we deduce that this also holds in general:
\begin{lemma}\label{lem: I_script_Pyt} Suppose $u_0,u_0^j, u_1, u_1^j \in \mathcal E_1(X,\o)$ and $u^j_0 \searrow u_0$, $u^j_1 \searrow u_1$. Then the following hold:\\
\noindent (i) $\mathcal I(u_0,u_1)=\mathcal I(u_0,\max(u_0,u_1)) + \mathcal I(\max(u_0,u_1),u_1).$\\
\noindent (ii) $\lim_{j \to \infty}\mathcal I(u^j_0,u^j_1)=\mathcal I(u_0,u_1).$
\end{lemma}
\begin{proof} (i) follows from the locality of the complex Monge--Amp\`ere measure on plurifine open sets (see \eqref{eq: GZ_MP_local}). Indeed, $\mathbbm{1}_{\{u_1 > u_0\}}(\o_{\max(u_0,u_1)}^n - \o_{u_0}^n) = \mathbbm{1}_{\{u_1 > u_0\}}(\o_{u_1}^n - \o_{u_0}^n)$, consequently
$$\mathcal I(u_0,\max(u_0,u_1))=\int_{\{u_1>u_0 \}} (u_0 - u_1)(\o_{u_1}^n-\o_{u_0}^n),$$
and for similar reasons, $\mathcal I(\max(u_0,u_1),u_1)=\int_{\{u_0>u_1 \}} (u_0 - u_1)(\o_{u_1}^n-\o_{u_0}^n)$. Adding these last two identities yields (i).

To address (ii), we start with $\mathcal I(u^j_0,u^j_1) = \mathcal I(u^j_0,\max(u^j_0,u^j_1)) + \mathcal I(\max(u^j_0,u^j_1),u^j_1)$. Trivially, $u^j_0,u^j_1 \leq\max(u^j_0,u^j_1)$, hence we can apply Proposition \ref{prop: MA_cont} to each term on the right hand side of \eqref{eq: I_script_def} to conclude that $\mathcal I(u^j_0,\max(u^j_0,u^j_1)) \to \mathcal I(u_0,\max(u_0,u_1))$ and $\mathcal I(\max(u^j_0,u^j_1),u^j_1) \to \mathcal I(\max(u_0,u_1),u_1)$. Another application of (i) yields (ii).
\end{proof}

The next proposition contains the main analytic properties of the $\mathcal I$ functional:

\begin{proposition}\textup{\cite[Lemma 3.13, Lemma 5.8]{BBGZ}}\label{prop: I_script_est} Suppose $C>0$ and $\phi,\psi,u,v \in \mathcal E_1(X,\o)$ satisfies
\begin{equation}\label{eq: C_estimates} 
-C \leq I(\phi),I(\psi),I(u),I(v),\sup_X \phi, \sup_X \psi,\sup_X u,\sup_X v \leq C. 
\end{equation}
Then there exists $f_C: \Bbb R^+ \to \Bbb R^+ $ (only dependent on $C$) continuous with $f_C(0)=0$ such that
\begin{equation}\label{eq: I_script_est1}\Big|\int_X \phi (\o_u^n - \o_v^n)\Big| \leq f_C(\mathcal I(u,v)),
\end{equation}
\begin{equation}\label{eq: I_script_est2}\Big|\int_X (u-v) (\o_\phi^n-\o_\psi^n)\Big| \leq f_C(\mathcal I(u,v)).
\end{equation}
\end{proposition}

Before we get into the argument, we note that by the lemma following this proposition, the condition \eqref{eq: C_estimates} is seen to be equivalent with
\begin{equation}\label{eq: C_estimates_d1} 
d_1(0,\phi),d_1(0,\psi),d_1(0,u),d_1(0,v) \leq C'.
\end{equation} 

\begin{proof}By repeated application of 
the dominated convergence theorem, \eqref{eq: CMA_general_def} and Lemma \ref{lem: I_script_Pyt}(ii), one can see that it is enough to show \eqref{eq: I_script_est1} and \eqref{eq: I_script_est2} for bounded potentials. 

We focus now on \eqref{eq: I_script_est1} and introduce the quantities $a_k: = \int_X \phi \o_{u}^k \wedge \o_v^{n-k}$. We note that \eqref{eq: I_script_est1} follows if we are able to prove  
\begin{equation}\label{eq: ak_ak1}
|a_{k+1} - a_k| \leq f_C(\mathcal I(u,v)), \ \ k \in \{0,\ldots,n-1\}.
\end{equation}
As our potentials are bounded, we can use integration by parts and start writing:
\begin{flalign*}
a_{k+1}-a_k & = \int_X \phi i\ddbar (u - v) \wedge \o_u^{k} \wedge \o_v^{n-k-1}\\
&=\int_X i \partial \phi \wedge \dbar (v - u) \wedge \o_u^{k} \wedge \o_v^{n-k-1},  
\end{flalign*}
Consequently, by the Cauchy--Schwarz inequality we can write:
\begin{flalign*}
|a_{k+1}-a_k|^2 & \leq \int_X i \partial \phi \wedge \dbar \phi \wedge \o_u^{k} \wedge \o_v^{n-k-1} \cdot \int_X i \partial (u-v) \wedge \dbar (u-v) \wedge \o_u^{k} \wedge \o_v^{n-k-1}\\
& \leq \int_X i \partial \phi \wedge \dbar \phi \wedge \o_u^{k} \wedge \o_v^{n-k-1} \cdot \mathcal I(u,v).
\end{flalign*}
By the above, \eqref{eq: ak_ak1} would follow if we can show that $\int_X i \partial \phi \wedge \dbar \phi \wedge \o_u^{k} \wedge \o_v^{n-k-1}$ is uniformly bounded, and this exactly what we argue:
\begin{flalign*}
\int_X i \partial \phi \wedge \dbar \phi \wedge \o_u^{k} \wedge \o_v^{n-k-1}&=\int_X \phi \o \wedge \o_u^{k} \wedge \o_v^{n-k-1} - \int_X \phi \o_\phi \wedge \o_u^{k} \wedge \o_v^{n-k-1}\\
& \leq \int_X |\phi| \o \wedge \o_u^{k} \wedge \o_v^{n-k-1} + \int_X |\phi| \o_\phi \wedge \o_u^{k} \wedge \o_v^{n-k-1} \\
& \leq D \int_X |\phi| \o^n_{ \phi/4 + u/8 + v/8}\\
& \leq D'd_1(u/8 + v/8,\phi/4 + u/8 + v/8).
\end{flalign*}
where in the last estimate we have used Theorem \ref{thm: Energy_Metric_Eqv}. To finish, by the triangle inequality we have to argue that $d_1(0,u/8+v/8)$ and $d_1(0,\phi /4+u/8+v/8)$ are bounded. This can be deduced from \eqref{eq: C_estimates_d1} by repeated application of Lemma \ref{lem: halwayest} and the triangle inequality for $d_1$.

Now we turn to \eqref{eq: I_script_est2}. The proof has similar philosophy, but it is slightly more intricate. We introduce $\alpha := u-v$, and  the quantities $b_k := \int_X \alpha  \o_{u}^k \wedge \o_\phi^{n-k}$. We will show that
\begin{equation}\label{eq: b_k_main_est}
|b_{k+1}-b_k| \leq f_C(\mathcal I(u,v)), \ \ k \in \{0,\ldots,k-1\}.
\end{equation}
This will imply that $\Big|\int_X (u -v)(\o_u^n - \o_{\phi}^n) \Big| \leq f_C(\mathcal I(u,v))$, and using the symmetry in $u,v$, and basic properties of the absolute value, we obtain \eqref{eq: I_script_est2}.
Integration by parts yields the following:
\begin{flalign*}
b_{k+1}-b_k &=\int_X \alpha i\ddbar (u-\phi) \wedge \o_u^{k} \wedge  \o_\phi^{n-k-1}= -\int_X i\partial \alpha \wedge \dbar (u-\phi) \wedge \o_u^{k} \wedge  \o_\phi^{n-k-1}
\end{flalign*}
Introducing $\beta := (u + v)/2$ and $c_k := \int_X i\partial \alpha \wedge \dbar \alpha \wedge \o_\beta^{k} \wedge  \o_\phi^{n-k-1} $, by the Cauchy--Schwarz inequality we deduce that
\begin{flalign}\label{eq: c_k_est}
|b_{k+1}-b_k|^2 &\leq \mathcal I(u,\phi) \int_X i\partial \alpha \wedge \dbar \alpha \wedge \o_u^{k} \wedge  \o_\phi^{n-k-1}  \leq 2^k \mathcal I(u,\phi) c_k \leq D c_k,
\end{flalign}
where in the last inequality we have used that $\mathcal I(u,\phi)$ is bounded. Indeed, this follows from \eqref{eq: C_estimates_d1}, the triangle inequality for $d_1$, and Theorem \ref{thm: Energy_Metric_Eqv}. Consequently, to prove \eqref{eq: b_k_main_est} it suffices to show that
\begin{equation}\label{eq: c_k_main_est}
c_k \leq f_C(\mathcal I(u,v)).
\end{equation}  
To obtain this, we integrate by parts again:
\begin{flalign}\label{eq: someiniedaf}
c&_{k+1}-c_k = \int_X i \partial \alpha \wedge \dbar \alpha \wedge i\ddbar (\beta-\phi) \wedge \o_{\beta}^{k} \wedge \o_{\phi}^{n-k-2} \nonumber \\
&= \int_X i \partial \alpha \wedge \dbar (\beta-\phi) \wedge i\ddbar \alpha \wedge \o_{\beta}^{k} \wedge \o_{\phi}^{n-k-2}\\
&= \int_X i \partial \alpha \wedge \dbar (\beta-\phi) \wedge \o_u \wedge \o_{\beta}^{k} \wedge \o_{\phi}^{n-k-2}-\int_X i \partial \alpha \wedge \dbar (\beta-\phi) \wedge \o_v \wedge \o_{\beta}^{k} \wedge \o_{\phi}^{n-k-2}. \nonumber
\end{flalign}
Using that $\o_u \leq 2 \o_\beta$ and the Cauchy--Schwarz inequality we can estimate the first term from the right hand side in the following manner:
\begin{flalign*}
\Big|\int_X i \partial \alpha \wedge \dbar (\beta-\phi) \wedge \o_u \wedge \o_{\beta}^{k-1} \wedge \o_{\phi}^{n-k-1}\Big|^2 &\leq D' \mathcal I(\beta,\phi) \int_X i \partial \alpha \wedge \dbar \alpha  \wedge \o_{\beta}^{k+1} \wedge \o_{\phi}^{n-k-2}\\
&= D' \mathcal I\Big(\frac{u+v}{2},\phi\Big)c_{k+1} \leq D'' c_{k+1},
\end{flalign*}
where in the last inequality we used that $\mathcal I({(u+v)}/{2},\phi)$ is bounded. This follows from \eqref{eq: C_estimates_d1}, as repeated application of Lemma \ref{lem: halwayest} and the triangle inequality for $d_1$ yields that $d_1(\phi,(u+v)/2)$ is bounded. 

We can similarly estimate the other term on the right hand side of \eqref{eq: someiniedaf} and putting everything together we obtain that  $c_k \leq c_{k+1} + 2D'' c_{k+1}^{1/2}$. After a successive application of this inequality, we find that $c_{k} \leq f_C(c_{n-1})$, which is equivalent to \eqref{eq: c_k_main_est}.
\end{proof}
\begin{lemma}\label{lem: d1bounded_char} Suppose $u \in \mathcal E_1(X,\o)$ and $-C \leq I(u) \leq \sup_X u \leq C$ for some $C>0$. Then $d_1(0,u) \leq 3C$.
\end{lemma}
\begin{proof}By the triangle inequality and Proposition \ref{prop: d_1_I_formula} we can write:
\begin{flalign*}
d_1(0,u) &\leq d_1(0,\sup_X u) + d_1(\sup_X u,u)=d_1(0,\sup_X u) + d_1(0,u-\sup_X u)\\
&=  |\sup_X u| + I(0) - I(u - \sup_X u)\\
&= |\sup_X u| +\sup_X u  - I(u) \\
& \leq 3C.
\end{flalign*}
\end{proof}

We note the following important corollary of Proposition \ref{prop: I_script_est}:

\begin{corollary} \label{cor: int_d1_est} For any $C>0$ there exists $\tilde f_C:\Bbb R^+ \to \Bbb R^+$ continuous with $\tilde f_C(0)=0$ such that
\begin{equation}\label{eq: int_d_1_est} 
\int_X |u-v|\o_{\psi}^n \leq \tilde f_C(d_1(u,v)),
\end{equation}
for $u,v,\psi \in \mathcal E_1(X,\o)$  satisfying $d_1(0,u),d_1(0,v),d_1(0,\psi) \leq C$. 
\end{corollary}

\begin{proof}By Lemma \ref{prop: I_finiteness_and_cont} and Theorem \ref{thm: Energy_Metric_Eqv} we note that $d_1(0,u),d_1(0,v),d_1(0,\psi) \leq C$ implies that  $I(u),I(v),I(\psi)$ and $\sup_X u,\sup_X v, \sup_X \psi$ are uniformly bounded. Consequently \eqref{eq: I_script_est2} gives that
$$\Big|\int_X (u-v)\o_{\psi}^n\Big| \leq f_C(\mathcal I(u,v)) + \int_X |u-v|\o_u^n \leq \tilde f_C(d_1(u,v)).$$

Next, by Theorem \ref{thm: Energy_Metric_Eqv}, $d_1(u,\max(u,v))$ is uniformly bounded. By the triangle inequality, so is $d_1(0,\max(u,v))$. Consequently, since $\mathcal I(\max(u,v),v) \leq \mathcal I(u,v)$ (see Lemma \eqref{lem: I_script_Pyt}) we also have:
$$\Big|\int_X (\max(u,v)-v)\o_{\psi}^n\Big| \leq \tilde f_C(d_1(u,v)).$$
Using $|u - v|= 2 (\max(u,v) - v) - (u-v)$, we can add these last two inequalities to obtain \eqref{eq: int_d_1_est}.\end{proof}

As another corollary of Proposition \ref{prop: I_script_est} we obtain the second part of Theorem \ref{thm: d_1-convergence}:

\begin{corollary}Suppose $u_k,u,v \in \mathcal E_1(X,\o)$ with $d_1(u_k,u) \to 0$. The following hold:\\
\noindent (i) $\o_{u_k}^n \to \o_u^n$ weakly.\\
\noindent (ii) $\int_X |u_k - u|\o_v^n \to 0.$ 
\end{corollary}
\begin{proof} Given $\phi \in C^\infty(X)$, after possibly multiplying with a small positive constant, we can assume that $\phi \in \mathcal H_\o$. As $d_1(u_k,u) \to 0$ implies that $\mathcal I(u_k,u) \to 0$ (Theorem \ref{thm: Energy_Metric_Eqv}), we see that \eqref{eq: I_script_est1} gives (i). The convergence statement of (ii) follows immediately from the previous corollary.
\end{proof}

Now we prove the rest of Theorem \ref{thm: d_1-convergence}:

\begin{proposition}Suppose $u_k,u \in \mathcal E_1(X,\o)$. Then $d_1(u,u_k)\to 0$ if and only if $\int_X |u - u_k|\o^n \to 0$ and $I(u_k) \to I(u)$.
\end{proposition}

\begin{proof} The fact that $d_1(u_k,u) \to 0$ implies $\int_X |u - u_k|\o^n \to 0$, follows from the previous corollary. $I(u_k) \to I(u)$ follows from the $d_1$--continuity of $I$.

We focus on the reverse direction. The first step is to show that 
\begin{equation}\label{eq: v_conv}
\int_X u_k \o_u^n \to \int_X u\o_u^n.
\end{equation}
By Theorem \ref{thm: BK_approx}, pick $v_k\in \mathcal H_\o$ decreasing to $u$. Consequently $\sup_X v_k$ and $I(v_k)$ is uniformly bounded, hence by \eqref{eq: I_script_est1} we can write:
$$
\Big | \int_X u_k (\o_u^n - \o_{v_j}^n)\Big| \leq f_C(\mathcal I(v_j,u)), \ \ j,k \in \Bbb N.
$$
As $v_j$ is smooth and $\int_X |u_k - u|\o^n \to 0$, we can write
$$\Big|\limsup_k \int_X u_k \o_u^n - \int_X u \o_{v_j}^n\Big| \leq f_C(\mathcal I(v_j,u)), \ \ j \in \Bbb N.$$
As $v_j \searrow u$, we can use Proposition \ref{prop: MA_cont} to get $\limsup_k \int_X u_k \o_u^n =  \int_X u \o_u^n$. The analogous statement also holds for $\liminf$, and we obtain \eqref{eq: v_conv}.

Next we use the following estimate, which is a consequence of \eqref{eq: I_energy_diff_E1} and \eqref{eq: I_energy_int_parts_ineq}:
$$\frac{\mathcal I(u,u_k)}{(n+1)V} \leq I(u_k) - I(u) - \frac{1}{V}\int_X (u - u_k) \o_u^n.$$
By \eqref{eq: v_conv} we have $\mathcal I(u,u_k) \to 0$. Using \eqref{eq: v_conv} again and \eqref{eq: I_script_est2} we can write
\begin{equation}\label{eq: v_k_conv}
\int_X (u_k-u) \o_{u_k}^n \to 0.
\end{equation}
Since, $\mathcal I(u,u_k) \to 0$, Lemma \ref{lem: I_script_Pyt}(i) gives that  $
\mathcal I(\max(u_k,u),u) \to 0$.  Also, by our assumptions, $\int_X (\max(u_k,u)-u) \o^n \to 0$, hence \eqref{eq: I_script_est2} implies that 
$$\int_X (\max(u_k,u)-u) \o_{u}^n \to 0 \ \ \textup{ and } \ \ \int_X (\max(u_k,u)-u) \o_{u_k}^n \to 0,$$ 
where in the last limit we also used the locality of the complex Monge--Amp\`ere operator with respect to the plurifine topology (see \eqref{eq: GZ_MP_local}). Using the fact that $|u_k - u|= 2 (\max(u_k,u) - u) - (u_k - u)$, together with \eqref{eq: v_conv} and \eqref{eq: v_k_conv} we get that
$$\int_X |u-u_k|\o_u^n + \int_X |u-u_k|\o_{u_k}^n \to 0.$$
Theorem \ref{thm: Energy_Metric_Eqv} now gives $d_1(u_k,u) \to 0$.
\end{proof}

For the last result of this section we return to the general $d_p$ metric topologies. First observe that for any $u \in \mathcal H_\o$ and $\xi \in T_u \mathcal H_\o$ the H\"older inequality gives $\| \xi\|_{u,p'} \leq \| \xi\|_{u,p}$ for any $p' \leq p$. This in turn implies that the $L^{p'}$ length of smooth curves in $\mathcal H_\o$ is shorter then their $L^p$ length, ultimately giving that the $d_p$ metric dominates $d_p'$. Consequently, all $d_p$ metrics dominate the $d_1$ metric, hence Theorem \eqref{thm: d_1-convergence}(ii) holds for $d_p$--convergence as well. We record this (in a slightly stronger form) in our last result:

\begin{proposition}\label{prop: dp_mixed_convergence} Suppose $v,u,u_k \in \mathcal E_p(X,\o)$ and $d_p(u_k,u) \to 0$. Then $\o_{u_k}^n \to \o_u^n$ weakly and $
\int_X |u -u_k|^p \o_v^n \to 0.$
\end{proposition}
\begin{proof} By the previous result, we only need to argue that $\int_X |u - u_k|^p \o_v^n \to 0.$ In fact, it is enough to prove this convergence for a subsequence of $\{u_k\}_k$.

Given an arbitrary subsequence of $u_k$, there exists a sub-subsequence, again denoted by $u_k$, satisfying the sparsity condition:
\begin{equation}\label{sparsity}
d_p(u_k,u_{k+1}) \leq \frac{1}{2^k}, \ \ k \in \Bbb N.
\end{equation}
Using this sparsity condition and Proposition \ref{prop: contractivity} we can write
\begin{flalign*}
d_p(P(u,u_0,\ldots,u_k),P(u,&u_0,\ldots,u_{k+1}))= \\ &=d_p(P(P(u,u_0,\ldots,u_k),u_k),P(P(u,u_0,\ldots,u_k),u_{k+1}))\\
&\leq d_p(u_k,u_{k+1}) \leq \frac{1}{2^k}.
\end{flalign*}
Hence, the decreasing sequence $h_k = P(u,u_0,u_1, \ldots,u_k), \ k \geq 1$ is $d_p$-bounded and Lemma \ref{lem: lemma mononton_seq} (or completeness) implies that the limit satisfies $h:=\lim_k h_k \in \mathcal E_p(X,\o).$
As $d_p(u_k,0),d_p(u,0)$ are bounded, by Theorem \ref{thm: Energy_Metric_Eqv} and Lemma \ref{lem: sup_int_psh_eqv} there exists $M>0$ such that 
$h \leq u_k,u \leq M$. Putting everything together we get
\begin{equation}\label{eq: u_uk_domination}
h - M \leq u_k - u \leq M - h.
\end{equation}   
By Theorem \ref{thm: d_1-convergence}(ii) we have $\int_X |u - u_k| \o_v^n \to 0$. Hence, after passing to a further subsequence, $u_k \to u$ a.e. with respect to $\o_v^n$. Finally, using \eqref{eq: u_uk_domination}, the dominated convergence theorem gives $\int_X |u_k - u|^p \o_v^n \to 0.$
\end{proof}

\paragraph{Brief historical remarks.} In this chapter we put extensive focus on the particular case of the $L^1$ geometry. Historically however, the study of the  $L^2$ structure was the one developed first. In particular, Calabi-Chen showed that $C^{1,1}$ geodesics of $\mathcal H_\o$ satisfy the CAT(0) inequality \cite{cc}. As shown in \cite{da1}, this implies that the completion $(\mathcal E_2(X,\o),d_2)$ is a CAT(0) geodesic metric space. The possibility of identifying $\mathcal E_2(X,\o)$ with the metric completion of $\mathcal H_\o$ was conjectured by Guedj \cite{g}, who worked out the case of toric K\"ahler manifolds. As detailed in this chapter, this conjecture was confirmed in \cite{da1}, and later generalized to the case of $L^p$ Finsler structures \cite{da2}.

As noticed by J. Streets \cite{st1,st2}, the CAT(0) property allows to study of the Calabi flow \cite{ch} in the context of the metric completion, leading to precise convergence results for this flow in \cite{bdl1}.

Endowing the space of K\"ahler metrics with natural geometries goes back to the work of Calabi in the 50's \cite{clb}. Calabi's Riemannian metric is defined in terms of the Laplacian of the potentials, and the resulting geometry differs from that of Mabuchi. The study of this structure was taken up by Calamai \cite{clm} and Clarke-Rubinstein \cite{cr}. In the latter work the completion of the Calabi path lengh metric was identified, and was compared to the Mabuchi geometry in \cite{da5}. 

It is also possible to introduce a Dirichlet type Riemannian metric on $\mathcal H_\o$ in terms of the gradient of the potentials \cite{clm,czh}. Not much is known about the metric theory of this structure. However properties of this space seem to be  closely immeresed with the study of the K\"ahler-Ricci flow \cite{bl}.

\chapter{Applications to K\"ahler--Einstein metrics}


Given a K\"ahler manifold $(X,\o)$, we will be interested 
in picking a metric with special curvature properties from $\mathcal H$, the space of all K\"ahler metrics $\o'$ whose de Rham cohomology class equals that of $\o$:
\begin{equation}\label{eq: H_def}
\mathcal H = \{\o' \textup{ is a K\"ahler metric on }X \textup{ such that } [\o']=[\o] \in H^2(X,\Bbb R)\}. 
\end{equation}
By the $\ddbar$--lemma of Hodge theory, there exists $u \in \mathcal H_\o$ \eqref{eq: H_o_def}, unique up to a constant, s.t. $\o' = \o_u$ (see \cite[Theorem 3]{bl1}). As follows from Stokes' theorem (see more generally Lemma \ref{lem: const_volume_lem}), the total volume of each metric in $\mathcal H$ is the same, and we introduce the constant
$$V := \int_X \o_u^n=\int_X \o^n, \ \ u \in \mathcal H_\o.$$
As a result of the above observations, we will focus on the space of potentials $\mathcal H_\o$ instead of $\mathcal H$, and our goal will be to find $u \in \mathcal H_\o$ whose Ricci curvature is a multiple of the metric $\o_u$:
\begin{equation}\label{eq: KE_eq}
\textup{Ric } \o_u = \lambda \o_u.
\end{equation}
Such metrics are called \emph{K\"ahler--Einstein} (KE) metrics. Recall from Appendix 6.1 that $\textup{Ric } \o_u$ is always closed. Moreover, by the well known formula \eqref{prelricdifference} for the change of Ricci curvature
\begin{equation}\label{eq: Ricci_diff_eq}
\textup{Ric } \o_u - \textup{Ric } \o_v = i \ddbar \log \Big(\frac{\o_v^n}{\o_u^n}\Big), \ \ u,v \in \mathcal H_\o
\end{equation}
we deduce that the de Rham class $[\textup{Ric }\o_u]$ does not depend on the choice of $u \in \mathcal H_\o$. What is more, $[\textup{Ric }\o_u]$ agrees with $c_1(X)$, the first Chern class of $X$ (see \cite[Section III.3]{we} for more details). Consequently, if \eqref{eq: KE_eq} holds then $X$ needs to have special cohomological properties depending on the sign of $\lambda \in \Bbb R$: \vspace{0.2cm}\\
\noindent (i) if $\lambda =0$ then $c_1(X)=0$, i.e., $X$ is \emph{Calabi--Yau}. \\
\noindent (ii) if $\lambda <0$ then $K_X$ is an ample line bundle, i.e., $X$ is of \emph{general type}.\\
\noindent (iii) if $\lambda >0$ then $-K_X$ is an ample line bundle, i.e., $X$ is \emph{Fano}. \vspace{0.2cm}

After close inspection it turns out that \eqref{eq: KE_eq} is a fourth order PDE in terms of the derivatives of $u$. 
As we show now, with the help of \eqref{eq: Ricci_diff_eq} one can write down a scalar equation equivalent to \eqref{eq: KE_eq}, that is merely a second order PDE, greatly simplifying our subsequent treatment. Indeed, from $\lambda [\o]=c_1(X)$ it follows that for each $u \in \mathcal H_\o$ there exists a unique $f_u \in C^\infty(X)$ such that $\int_X e^{f_u}\o_u^n = V$ and
\begin{equation}\label{eq: Ricci_pot_def}
\textup{Ric }\o_u = \lambda \o + i \ddbar f_u.
\end{equation}
Fittingly, the potential $f_u$ is called the \emph{Ricci potential} of $\o_u$, and $\o_u$ is KE if and only if $f_u=0$. Consequently, by \eqref{eq: Ricci_diff_eq}, \eqref{eq: KE_eq} is equivalent to 
$$
i \ddbar \log \Big(\frac{\o_u^n}{\o^n}\Big) = \textup{Ric } \o - \textup{Ric } \o_u = \lambda \o + i \ddbar f_0 - \lambda \o_u = i\ddbar (f_0 - \lambda u).
$$
And now the (magical!) drop of order takes place. As there are only constants in the kernel of $i\ddbar$, there exists $c \in \Bbb R$ such that
$$\log \Big(\frac{\o_u^n}{\o^n}\Big)= f_0 - \lambda u + c.$$
When $\lambda \neq 0$, the constant $c$ can be ``contracted" into $u$. When $\lambda =0$, then the normalization condition on $f_0$ implies that $c=0$. Summarizing, we arrive at the scalar KE equation:
\begin{equation}\label{eq: KE_scalar_eq}
\o_u^n = e^{- \lambda u + f_0} \o^n, \ \ u \in \mathcal H_\o.
\end{equation}
When $\lambda <0$, the existence of unique solutions was proved by Aubin and Yau \cite{A,Y}. In the case $\lambda =0$, existence and uniqueness was obtained by Yau, as a particular case of the solution of the Calabi conjecture \cite{Y}. 

When $X$ is Fano ($\lambda>0$) the situation is more involved. Uniqueness up to holomorphic automorpshims was shown by Bando--Mabuchi (\cite{bm}, see also Section 4.5 below). As it turns out, on a  general Fano manifold $(X,\o)$, there are numerous obstructions to existence of KE metrics (see \cite{mat,fu}). Recently the algebro--geometric notion of K--stability has been found to be equivalent with existence of KE metrics (\cite{cds,t2}), with this verifying an important particular case of the Yau--Tian--Donaldson conjecture. Our goal in this chapter is to give equivalent characterizations, more in line with the variational study of partial differential equations, eventually verifying related conjectures of Tian. 

As mentioned above, in the rest of this chapter we will assume that $(X,\o)$ is Fano. Also, after possibly rescaling $\o$, we can also assume that $\lambda =1$, i.e., $[\o]=c_1(X)$. Next we introduce Ding's $\mathcal F$  functional \cite{ding}:
\begin{equation}\label{eq: F_def}
\mathcal F(u) = -I(u) - \log \frac{1}{V}\int_X e^{-u+f_0}\o^n, \ \ u \in \mathcal H_\o,
\end{equation}
where $I$ is the Monge--Amp\`ere energy (see \eqref{eq: I_energy_def}). By definition, $\mathcal F$ is invariant under adding constants, i.e., $\mathcal F(u+c)=\mathcal F(u)$, and as a consequence of Lemma \ref{lem: I_differential} it follows that the critical points of $\mathcal F$ are exactly the KE metrics:
\begin{lemma}\label{lem: F_func_differential}Suppose $[0,1] \ni t \to v_t \in \mathcal H_\o$ is a smooth curve. Then
$$V \frac{d}{dt}\mathcal F(v_t) = \int_X \dot v_t \bigg ( - \o_{v_t}^n + \frac{V}{\int_X e^{-v_t + f_0} \o^n}e^{-v_t + f_0} \o^n \bigg).$$
\end{lemma}
 
Comparing with \eqref{eq: KE_scalar_eq}, we deduce that existence of KE potentials in $\mathcal H_\o$ is equivalent with existence of critical points of $\mathcal F$. As we will see, the natural domain of definiton of $\mathcal F$ is not $\mathcal H_\o$, but rather its $d_1$ metric completion $\mathcal E_1(X,\o)$. In addition to this, $\mathcal F$ will be shown to be $d_1$--continuous and it will be convex along the finite energy geodesics of $\mathcal E_1(X,\o)$. Using convexity we will deduce that critical points of $\mathcal F$ are exactly the minimizers of $\mathcal F$, and we can use properness properties of $\mathcal F$ to characterize existence of these minimizers. All this will be done in the forthcoming sections.

\section{The action of the automorphism group}

We continue to assume that $(X,\o)$ is Fano with $[\o]=c_1(X)$. Let $\textup{Aut}_0(X,J)$ denote the connected component of the complex Lie group of biholomorphisms of $(X,J)$, and denote by  $\textup{aut}(M,J)$ its complex Lie algebra of holomorphic vector fields. $\textup{Aut}_0(X,J)$  acts on $\mathcal H$ by pullback of metrics. Indeed, 
$f^\star\eta \in \mathcal H$ for any $f\in\textup{Aut}_0(X,J)$ and $\eta\in\mathcal H.$

Given the one-to-one correspondence between $\mathcal H$ and $\mathcal H_\o \cap I^{-1}(0)$ (recall \eqref{eq: H_0_def}), the group $\textup{Aut}_0(X,J)$  also acts on $\mathcal H_\o \cap I^{-1}(0)$ and we describe this action more precisely in the next lemma:
\begin{lemma}\textup{\cite[Lemma 5.8]{dr2}}
For $\varphi \in \mathcal H_\o \cap I^{-1}(0)$ and $f \in \textup{Aut}_0(X,J)$ let $f.\varphi\in\mathcal H_\o \cap I^{-1}(0)$ be the unique element such that $f^*\o_\varphi=\o_{f.\varphi}$.
Then,
\begin{equation}
\label{eq: factionIEq}
f.\varphi=f.0+\varphi\circ f.
\end{equation}
\end{lemma}

\begin{proof}
First we note that the right hand side of \eqref{eq: factionIEq} is a K\"ahler  potential for $f^*\o_\varphi$. Indeed, since  $f\in\textup{Aut}_0(X,J)$ we have $f^\star i\ddbar\varphi=i\ddbar \varphi\circ f$.
The identity $I(f.0+\varphi\circ f)=0$ follows from \eqref{eq: I_energy_diff} as we have:
\begin{flalign*}
I(f.0+\varphi\circ f)&=I(f.0+\varphi\circ f) - I(f.0)=\frac{1}{(n+1)V}\int_X \varphi \circ f \sum_{j=0}^n f^\star \o^{n-j}\wedge f^\star \o_\varphi^j \\
&=\frac{1}{(n+1)V}\int_X \varphi  \sum_{j=0}^n  \o^{n-j}\wedge \o_\varphi^j = I(\varphi)=0.
\end{flalign*}
\end{proof}
With the formula of the above lemma, we show that $\textup{Aut}_0(X,J)$ acts on $\mathcal H_\o \cap I^{-1}(0)$ by $d_p$--isometries:
\begin{lemma}
\label{lem: dpIsomLemma}
The action of $\textup{Aut}_0(X,J)$ on $\mathcal H_\o \cap I^{-1}(0)$ is by $d_p$--isometries, for any  $p \geq 1$.
\end{lemma}

\begin{proof}
From \eqref{eq: factionIEq} it follows that $
{d}(f.\varphi_t)/dt=\dot\varphi_t\circ f,$
for any smooth curve $[0,1] \ni t \to \varphi_t \in \mathcal H_\o \cap I^{-1}(0)$. Thus, the $d_p$--length of $t \to f.\varphi_t$ satisfies:
$$
l_p(f.\varphi_t)=\int_0^1 \bigg(\frac{1}{V}\int_X|\dot\varphi_t\circ f|^p f^\star \o_{\varphi_t}^n\bigg)^{\frac{1}{p}} dt
=
\int_0^1 \bigg(\frac{1}{V}\int_X|\dot\varphi_t|^p \o_{\varphi_t}^n \bigg)^{\frac{1}{p}}dt.
$$
Since this last quantity is exactly the the $d_p$-length of $t \to \varphi_t$, it follows that $\textup{Aut}_0(X,J)$ acts by $d_p$--isometries.
\end{proof}

As a consequence of this last result (see also Lemma \ref{lem: LipschitzExt} below), the action of $\textup{Aut}_0(X,J)$ on $\mathcal H_\o \cap I^{-1}(0)$ has a unique $d_p$--isometric extension to the $d_p$--metric completion  $\mathcal E_p(X,\o) \cap I^{-1}(0)$. In what follows we will focus on the case $p=1$.

Recall the definition of the $J$ functional \eqref{eq: J_def} and let $H\leq \textup{Aut}_0(X,J)$ be a subgroup. The ``$H$--dampened" functional $J_H:\big(\mathcal E_1 \cap I^{-1}(0)\big)/H \to \Bbb R$ is introduced by the formula
\begin{equation}
\label{eq: JGEq}
J_H(Hu):= \inf_{f \in H}  {J}(f.u).
\end{equation}
As a direct consequence of Proposition \ref{prop: d_1_growth_J} we have the following estimates for this functional:
\begin{lemma} There exists $C:=C(X,\o)>1$ such that for any $u \in \mathcal E_1 \cap I^{-1}(0)$ we have 
\label{lem: JGPropernessLemma}
\begin{equation}
\label{lem: JGdGEqv}
 \frac{1}{C} J_H(Hu) -C \leq d_{1,H}(H0,Hu) 
\leq  C J_H(Hu) + C,
\end{equation}
where $d_{1,H}$ is the pseudo--metric of the quotient $\big(\mathcal E_1 \cap I^{-1}(0)\big)/H$ given by the formula $d_{1,H}(Hu,Hv):=\inf_{f\in H} d_1(u,f.v)$.
\end{lemma}

Lastly, we note that $\mathcal F$ is affine along one parameter subgroups of $\textup{Aut}_0(X,J)$, and that the map $(u,v) \to \mathcal F(u)-\mathcal F(v)$ is $\textup{Aut}_0(X,J)$--invariant:
\begin{lemma} \label{lem: F_one_par_linear} Suppose $\Bbb R \ni t \to \rho_t \in \textup{Aut}_0(X,J)$ is a one parameter subgroup. Then:\\
(i) for any $u \in \mathcal H_\o \cap I^{-1}(0)$ the map $t \to \mathcal F(\rho_t. u)$ is affine.\\
\noindent (ii) if $u,v \in \mathcal H_\o \cap I^{-1}(0)$ and $g \in \textup{Aut}_0(X,J)$ then $\mathcal F(u)-\mathcal F(v)=\mathcal F(g.u)-\mathcal F(g.v)$.
\end{lemma} 

\begin{proof} First we show that for any $u \in \mathcal H_\o$ we have that
\begin{equation}\label{eq: Ricci_pot_eq}
\frac{e^{-u + f_0}}{\int_X e^{-u + f_0}\o^n}{\o^n} = \frac{e^{f_u}}{V}\o_u^n,
\end{equation}
where $f_u$ is the Ricci potential of $u$ (see \eqref{eq: Ricci_pot_def}). By definition, we have $\textup {Ric } \o_u = \o_u + i\ddbar f_u$ and $\textup {Ric } \o = \o + i\ddbar f_0$. Substituting this into \eqref{eq: Ricci_diff_eq} we see that \eqref{eq: Ricci_pot_eq} holds.

Now we argue (i). Let $h \in C^\infty(X)$ be the unique  function such that $\int_X h \o_u^n = 0$ and $\frac{d}{dt}\big|_{t=0} \rho_t^* \o_u = i\ddbar h$. We claim that
\begin{equation}\label{eq: action_diff}
\frac{d}{dt} \rho_t.u = h \circ \rho_t, \ \ t \in \Bbb R.
\end{equation}
Indeed, this follows from a simple calculation: $i\ddbar \frac{d}{dt}\rho_t.u = \frac{d}{dt} \o_{\rho_t.u} = \frac{d}{dt} \rho_t^* \o_{u} = \rho^*_t i\ddbar h = i\ddbar h \circ \rho_t$. 

Finally, using Lemma \ref{lem: F_func_differential}, \eqref{eq: Ricci_pot_eq}, \eqref{eq: action_diff}, and the fact that $f_u \circ \rho_t = f_{\rho_t .u}$ we can conclude that $t \to \mathcal F(\rho_t.u)$ is indeed affine:
\begin{flalign}\label{eq: F_deriv_calc}
V \frac{d}{dt}\mathcal F(\rho_t.u) &= \int_X \frac{d}{dt} \rho_t.u \bigg ( -\o_{\rho_t.u}^n + \frac{V}{\int_X e^{-\rho_t.u + f_0} \o^n}e^{-\rho_t.u + f_0} \o^n \bigg) \nonumber\\
&=\int_X \frac{d}{dt}\rho_t.u \Big( -\o_{\rho_t.u}^n +e^{f_{\rho_t.u}} \o^n_{\rho_t. u}\Big) \nonumber\\
&=\int_X h \circ \rho_t \Big( -\rho_t^*\o_{u}^n + e^{f_{u}\circ \rho_t}\rho_t^* \o^n_{u}\Big)=\int_X h \big(-1 +   e^{f_{u}}\big) \o_{u}^n.\nonumber
\end{flalign}
Now we focus on (ii). Let $[0,1] \ni t \to \gamma_t\in \mathcal H_\o \cap I^{-1}(0)$ be any smooth segment connecting $u,v$. For example, one can take $ \gamma_t:= (1-t)u + tv - I((1-t)u + tv)$. Consequently, $t \to g.\gamma_t$ connects $g.u,g.v$. Using Lemma \ref{lem: F_func_differential}, \eqref{eq: factionIEq},
 and \eqref{eq: Ricci_pot_eq} we can finish the proof:
\begin{flalign*}
V(\mathcal F(g.v) - \mathcal F(g.u)) &= V\int_0^1 \frac{d}{dt}\mathcal F(\gamma_t) dt = \int_0^1 \int_X \dot \gamma_t \circ g (-1 + e^{f_{g.\gamma_t}})\o_{g.\gamma_t}^n dt.\\
&= \int_0^1 \int_X \dot \gamma_t \circ g (-1 + e^{f_{\gamma_t} \circ g})g^*\o_{\gamma_t}^n dt\\
&= \int_0^1 \int_X \dot \gamma_t  (-1 + e^{f_{\gamma_t}})\o_{\gamma_t}^n dt=V(\mathcal F(v) - \mathcal F(u)).
\end{flalign*}
\end{proof}

\section{The existence/properness principle and relation to Tian's conjectures}

Staying with a Fano manifold $(X,\o)$, our main goal is to show that existence of KE metrics in $\mathcal H$ is equivalent with properness of the Ding energy (and later the K-energy). As it turns out, our proof will rest on a very general existence/properness principle for abstract  metric spaces, and we describe this now as it provides the skeleton for our later arguments concerning K\"ahler geometry. 

To begin, we introduce $(\mathcal R,d,F,G)$, a metric space structure with additional data satisfying the following \emph{axioms}:
\begin{enumerate}[label = (A\arabic*)]
  \item\label{a1} $(\mathcal R,d)$ is a metric space with a 
distinguished 
element $0\in\mathcal R$,
        whose metric completion is denoted by $(\overline{\mathcal R},d)$.
  \item\label{a2} $F : \mathcal R \to \Bbb R$ 
is $d$--lsc. Let  
        $F: \overline{\mathcal R} \to \Bbb R \cup \{ +\infty\}$ 
        be the largest $d$--lsc extension
        of $F: \mathcal R \to \Bbb R$:
$$F(u) = \sup_{\varepsilon > 0 } \bigg(\inf_{\substack{v \in \mathcal R\\ d(u,v) \leq \varepsilon}} F(v) \bigg), \ \ u \in \overline{ \mathcal R}.$$
  \item\label{a3} 
By $\mathcal M$ we denote the set of minimizers of $F$ on $\overline{\mathcal R}$:
$$
\mathcal M:= 
\Big\{ u \in \overline{\mathcal R} \ : \ F (u)= 
\inf_{v \in \overline{\mathcal R}} F(v)
\Big\}.
$$
  \item\label{a4} $G$ is a group  
  acting on ${\mathcal R}$ by 
  $G\times{\mathcal R}\ni(g,u) \to g.u\in {\mathcal R}$. Denote by
${\mathcal R} /G$ the orbit space, by $Gu\in{\mathcal R} /G$ the 
  orbit of $u\in{\mathcal R}$, 
  and~define~$d_G:{\mathcal R}/G\times {\mathcal R}/G\to\Bbb R_+$ by
$$
d_G(Gu,Gv):=\inf_{f,g\in G}d(f.u,g.v).
$$
\end{enumerate}

In addition to the above, our data $(\mathcal R,d,F,G)$ also enjoys the following \emph{properties}:
\begin{enumerate}[label = (P\arabic*)]
  \item\label{p1}
 For any $u_0,u_1 \in \mathcal R$ there exists a $d$--geodesic segment $[0,1] \ni t \mapsto u_t \in \overline{\mathcal R}$ connecting $u_0,u_1$ (see \eqref{eq: d_geod_def}) for which
 $t\mapsto F(u_t) \textup{ is continuous and convex on }[0,1].$
  \item \label{p2} 
If  $\{u_j\}_j\subset \overline{\mathcal R}$  satisfies 
$\lim_{j\to\infty}F(u_j)= \inf_{\overline{\mathcal R}} F$, and
for some $C>0$, $d(0,u_j) \leq C$ for all $j$, then there exists $u\in \mathcal M$ and a subsequence $\{u_{j_k}\}_k$ s.t.   $d(u_{j_k},u) \to 0$.
 \item[(P2)] \hspace{-0.2cm}$^*$ \label{p2*} If  $\{u_j\}_j\subset \overline{\mathcal R}$  satisfies 
$F(u_j) \leq C$, and
 $d(0,u_j) \leq C, \ j \geq 0$ for some $C>0$, then there exists $u\in \mathcal R$ and a subsequence $\{u_{j_k}\}_k$ s.t.   $d(u_{j_k},u) \to 0$.
  \item\label{p3} 
  $\mathcal M \subset \mathcal R.$
  \item\label{p4} $G$ acts on ${\mathcal R}$ by $d$-isometries.
  \item\label{p5} $G$ acts on $\mathcal M$ transitively.
\item \label{p7} For all $u,v \in \mathcal R$ and $g \in G$,
$F(u) - F(v)=F(g.u) - F(g.v)$.

\end{enumerate}
We make three remarks.  First, by \ref{a2}, 
\begin{equation}
\label{eq: InfRealization}
\inf_{v \in \overline{\mathcal R}} F(v)
=
\inf_{v \in \mathcal R} F(v).
\end{equation}
Second, condition $\textup{\ref{p2*}}^*$ is stronger than \ref{p2} and we will require that only one of these conditions holds.

Third, thanks to \ref{p4} and the next lemma,
the action of $G$, originally defined on $\mathcal R$ \ref{a4}, 
extends to an action by $d$--isometries on the completion $\overline{\mathcal R}$.

\begin{lemma}
\label{lem: LipschitzExt} 
Let $(X,\rho)$ and $(Y,\delta)$ be two complete metric spaces, $W$ a dense subset of $X$ and $f:W \to Y$ a $C$-Lipschitz function, i.e.,
\begin{equation}\label{eq: lipineq}
\delta(f(a),f(b)) \leq C \rho(a,b), \ \ \forall\, a,b \in W.
\end{equation}
Then $f$ has a unique $C$--Lipschitz continuous extension 
to a map $\bar f:X\to Y$.
\end{lemma}
\begin{proof}Let $w_k \in W$ be a sequence converging to some $w \in X$. Lipschitz continuity gives
$$\delta(f(w_k),f(w_l)) \leq C \rho(w_k,w_l),$$
hence $\bar f(w) := \lim_k f(w_k) \in Y$ exists and independent of the choice of approximating sequence $w_k$. Choose now another sequence $z_k \in W$ with limit $z \in X$, plugging in $w_k,z_k$ in \eqref{eq: lipineq} and taking the limit gives that $\bar f: X \to Y$ is $C$-Lipschitz continuous.  
\end{proof}

A  $d$-geodesic ray $[0,\infty) \ni t \to u_t \in \overline{\mathcal R}$ is  $G$-\emph{calibrated}  if the curve $t \to Gu_t$  is a $d_G$-geodesic with the same speed as $t \to u_t$, i.e.,
$$d_G(Gu_0,Gu_t)=d(u_0,u_t), \ \ \  t \geq 0.$$
The next result will provide the framework that relates existence of canonical K\"ahler metrics
to energy properness and uniform geodesic stability.
\begin{theorem}\textup{\cite[Theorem 3.4]{dr2}}
\label{thm: ExistencePrinc}
Suppose $(\mathcal R,d,F,G)$ satisfies \ref{a1}--\ref{a4} and \ref{p1}--\ref{p7}. 
The following are equivalent:\vspace{0.1cm}\\
\noindent
(i)(existence of minimizers) $\mathcal M$ is nonempty\vspace{0.1cm}.

\noindent
(ii)(energy properness) $F:{\mathcal R}\to \Bbb R$ 
is $G$-invariant, and
for some $C,D>0$,
\begin{equation}\label{eq: Dproperness}
    F(u) \geq Cd_G(G0,Gu)-D,\ \  \textup{for all\ } u \in \mathcal R.
\end{equation}
\noindent
If $\textup{\ref{p2*}}^*$ holds instead of \ref{p2}, then the above are additionally equivalent to:\vspace{0.1cm}\\
(iii)(uniform geodesic stability) Fix $u_0 \in \mathcal R$ with $F(u_0) <+\infty$. Then $F:{\mathcal R}\to \Bbb R$ 
is $G$-invariant, and
there exists $C>0$ such that for all geodesic rays $[0,\infty) \ni t \to u_t \in \overline{\mathcal R}$ we have that
\begin{equation}\label{eq: Dgeodstab1}
    \limsup_{t \to \infty} \frac{F(u_t)-F(u_0)}{t}\geq C \limsup_{t \to \infty}{\frac{d_G(Gu_0,Gu_t)}{t}}.
\end{equation}
(iv)(uniform geodesic stability) Fix $u_0 \in \mathcal R$ with $F(u_0) <+\infty$. Then $F:{\mathcal R}\to \Bbb R$ 
is $G$-invariant, and
there exists $C>0$ such that for all $G$-calibrated geodesic rays $[0,\infty) \ni t \to u_t \in \overline{\mathcal R}$ we have that
\begin{equation}\label{eq: Dgeodstab}
    \limsup_{t \to \infty} \frac{F(u_t)-F(u_0)}{t}\geq C d(u_0,u_1).
\end{equation}
\end{theorem}

Before arguing the above theorem we recall standard facts from metric geometry in the form of the following two lemmas:
\begin{lemma}\label{lem: pseudo_metr} If  \ref{p4} holds, 
then $(\overline{\mathcal R}/G,d_G)$ and $(\mathcal R/G,d_G)$ are pseudo--metric spaces.\end{lemma}

\begin{proof}It is enough to show that $(\overline{\mathcal R}/G,d_G)$ is a pseudo--metric space. Since $d$ is symmetric,
$$
d_G(Gu,Gv):=\inf_{f,g\in G}d(f.u,g.v)=\inf_{f,g\in G}d(g.v,f.u)=d_G(Gv,Gu).
$$
Hence $d_G$ is also symmetric. Given $u,v,w \in \overline{\mathcal R}$ and $\varepsilon >0$, there exist $f,g \in G$ such that 
$d_G(Gu,Gw) > d(f.u,w)-\varepsilon$ 
and $d_G(Gv,Gw)> d(g.v,w)-\varepsilon$. 
The triangle inequality for $d$ and \ref{p4} give
\begin{flalign*}
d_{G}(Gu,Gv)&\leq 
d(u,f^{-1}g.v)=d(f.u,g.v) \cr
&\leq d(f.u,w)+d(g.v,w) < 2\varepsilon+d_{G}(Gu,Gv) + d_{G}(Gv,Gw).
\end{flalign*}
Letting $\varepsilon \to 0$ shows $d_G$ satisfies the triangle inequality. Thus $d_G$ is a pseudo--metric.
\end{proof}

\begin{lemma}\label{lem: GeodDescentLemma} Suppose \ref{p4}  holds, $u_0,u_1 \in \mathcal R$ and $[0,1] \ni t \to u_t \in \overline{\mathcal R}$ is a $d$--geodesic connecting $u_0,u_1$. If $d(u_0,u_1) -\varepsilon \leq d_G(Gu_0,Gu_1)\leq d(u_0,u_1)$ for some $\varepsilon >0$ then
$$
d(u_a,u_b) -\varepsilon \leq d_G(Gu_a,Gu_b) \leq d(u_a,u_b), \ \  \forall\, a,b \in [0,1].
$$
\end{lemma}

\begin{proof} The proof is by contradiction. Suppose $d_G(Gu_a,Gu_b) < d(u_a,u_b) - \varepsilon$.
Since $d(u_0,u_a)+ d(u_a,u_b)+d(u_b,u_1)=d(u_0,u_1)$ we can write
\begin{flalign}
d_{G}(Gu_0,Gu_1) 
&\leq d_{G}(Gu_0,Gu_a)+ d_G(Gu_a,Gu_b)+d_{G}(Gu_b,Gu_1) 
\cr
&<  d(u_0,u_a)+ d(u_a,u_b)+d(u_b,u_1)-\varepsilon
\cr
&=d(u_0,u_1) - \varepsilon \leq d_G(Gu_0,Gu_1).
\end{flalign}
This is a contradiction, finishing the proof.
\end{proof}

\begin{proof}[Proof of Theorem \ref{thm: ExistencePrinc}]

First we show (ii)$\Rightarrow$(i). If condition (ii) holds, then  $F$ is bounded from below. 
By \eqref{eq: InfRealization}, \eqref{eq: Dproperness}, the $G$--invariance of $F$, and the definition of $d_G$ there 
exists $u_j \in \mathcal R$ such that 
$\lim_{j}F(u_j) = \inf_{\overline{\mathcal R}} F$ 
and $d(0,u_j) \leq d_G(G0,Gu_j) + 1<C$ for $C$ independent of $j$. 
By \ref{p2}, $\mathcal M$ is non-empty.

We now show that (i)$\Rightarrow$(ii). First we argue that $F:\mathcal R \to \Bbb R$ is $G$--invariant. Let $v\in\mathcal M$.
By \ref{p3}, $v\in\mathcal R$.
By \ref{p4}, 
$f.v \in \mathcal M$ for any $f \in G$. 
Thus, $F(v)=F(f.v)$. Consequently,
$F(u) - F(v)=F(u) - F(f.v)$. By \ref{p7}, we get 
$F(u) - F(v)=F(f^{-1}.u) - F(v)$, so 
$F(u)=F(f^{-1}.u)$ for every $f\in G$, i.e., $F$ is $G$--invariant.

For $v \in \mathcal M \subset \mathcal R$ we define 
$$
C:= \inf \bigg\{
\frac{F(u)-F(v)}{d_G(Gv,Gu)}
\,:\, u \in \mathcal R,\; d_G(Gv,Gu) \geq 2\bigg\} \geq 0.
$$
If $C>0$, then we are done. Suppose  $C=0$. Then 
there exists $\{u^k\}_k \subset \mathcal R$ such that 
$$
(F(u^k)-F(v))/{d_G(Gv,Gu^k)} \to 0
$$ 
and $d_G(Gv,Gu^k) \geq 2$. By $G$--invariance of $F$ we can also assume that $d(v,u^k) -1/k \leq d_G(Gv,Gu^k) \leq d(v,u^k)$. Thus, 
$$
\lim_{k\to\infty}
\frac{F(u^k) - F(v)}{d(v,u^k)} = 0.
$$
Using \ref{p1}, let $[0,d(v,u^k)] \ni t \mapsto u^k_t \in \overline{\mathcal R}$ be a unit 
speed $d$--geodesic connecting $u^k_0:=v$ and $u^k_{d(v,u^k)}:=u^k$ such that $t \mapsto F(u^k_t)$ is convex. 
As $v$ is a minimizer of $F$, by convexity we obtain 
\begin{equation}
\label{eq: ConvConseqEq}
0 \leq 
F(u^k_1) -F(v)\leq  
\frac{F(u^k)-F(v)}{d(v,u^k)} \to 0.
\end{equation}
Trivially, $d(v,u^k_1)=1$, hence \ref{p2} and \eqref{eq: ConvConseqEq} imply that $d(u^k_1,\tilde v) \to 0$ for some $\tilde v \in \mathcal M$ (after perhaps passing to a subsequence of $u^k_1$). By 
\ref{p5}, $\tilde v=f.v$ for some $f \in G$, hence $d_G(Gv,G\tilde v)=0$.

From Lemma \ref{lem: GeodDescentLemma} we obtain $1 - 1/k \leq d_G(Gv,Gu^k_1) \leq 1$. Since $d(u^k_1,\tilde v) \to 0$, we also have $d_G(Gu^k_1,G\tilde v) \to 0$, which gives $d_G(Gv,G\tilde v)=1$, a contradiction with $d_G(Gv,G\tilde v)=0$. This implies that $C>0$, finishing the proof of the implication (i)$\Rightarrow$(ii).\vspace{0.1cm}

The implications (ii)$\Rightarrow$(iii)$\Rightarrow$(iv) is trivial and we finish the proof by showing that (iv)$\Rightarrow$(i). Suppose (i) does not hold but (iv) does. We will derive a contradiction. Since $F$ is $G$-invariant there exists $\{ u^k\} \subset \mathcal R$ such that $d(u_0,u^k) - 1/k \leq d_G(Gu_0,Gu^k)$ and $F(u^k)$ decreases to $\inf_{u \in \mathcal R} F(u)$. By $\textup{\ref{p2}}^*$ we must have $d(u_0,u^k) \to \infty$, otherwise there would exists $u \in \overline{\mathcal R}$ such that $F(u)=-\infty$, a contradiction.

Let $[0,d(u_0,u^k)] \ni t \to u^k_t \in \overline{\mathcal R}$ be the $d$-geodesic joining $u_0,u^k$ from \ref{p1}. We note that by Lemma \ref{lem: GeodDescentLemma} it follows that 
\begin{equation}\label{1/kEst}
d(u_0,u^k_t) - 1/k \leq d_G(Gu_0,Gu^k_t), \ \ t \in [0,d(u_0,u^k)].
\end{equation}
Fix $l \in \Bbb Q_+$. For big enough $k$ using convexity of $F$ we can write:
\begin{equation}\label{eq: u_k_t_F_est}
\frac{F(u^k_l)-F(u_0)}{l} \leq \frac{F(u^k)-F(u_0)}{d(u^k,u_0)}\leq 0.
\end{equation}
As $d(u^k_l,u_0)=l$, we can use $\textup{\ref{p2}}^*$ and a Cantor diagonal process, to conclude the existence of a sequence $k_j \to \infty$ and $u_l \in \overline{\mathcal R}$ for all $l \in \Bbb Q_+$ such that $d(u^{k_j}_{l},u_l) \to 0$. As each curve $t \to u^k_t$ is $d$-Lipschitz, it follows that in fact we can extend the curve $\Bbb Q_+ \ni l \to u_l \in \overline{\mathcal R}$ to a curve  $[0,\infty) \ni t \to u_t \in \mathcal \overline{\mathcal R}$ such that $d(u^{k_t}_{l},u_t) \to 0$.

For elementary reasons $t \to u_t$ is a $d$-geodesic. By \ref{a2} and \eqref{eq: u_k_t_F_est} we get that  
\begin{equation}
\frac{F(u_l)-F(u_0)}{l} \leq {d(u_l,u_0)}\leq 0, \ \ l \geq 0.
\end{equation}
Finally, we argue that $t \to u_t$ is a $G$-calibrated geodesic ray, yielding a contradiction with \eqref{eq: Dgeodstab}. Let $g \in G$ be arbitrary, from \eqref{1/kEst} it follows that $d(u_0,u^{k_j}_t) - 1/{k_j} \leq d(g.u_0,u^{k_j}_t)$. Letting $k_j \to \infty$, we obtain $d(u_0,u_t)\leq d(g.u_0,u_t)$ for $t \in \Bbb Q_+$ and by density for all $t \geq 0$. Consequently, $t \to u_t$ is $G$-calibrated.
\end{proof}

Later, when dealing with the K-energy functional, we will make use of the following observation:

\begin{remark} 
\label{rem: existence_princ}
The direction (ii)$\Rightarrow$(i) in the above argument only uses the compactness condition~\ref{p2}.
\end{remark}

The next result, together with Theorem \ref{thm: K-energy_properness}, represents the main application of  Theorem \ref{thm: ExistencePrinc}:

\begin{theorem} \textup{\cite[Theorem 7.1]{dr2}}
\label{thm: F_func_properness}
Suppose $(X,\o)$ is Fano and set $G:=\textup{Aut}_0(X,J)$.
The following are equivalent: \vspace {0.1cm} \\
\noindent (i) There exists a KE metric in $\mathcal H$.

\noindent (ii) 
For some $C,D >0$ the following holds:
\begin{equation}\label{eq: F_d_1_properness}
\mathcal F(u) \geq C d_{1,G}(G0,Gu) - D, \ \  u\in\mathcal H_\omega \cap I^{-1}(0).
\end{equation}

\noindent (iii) 
For some $C,D>0$ the following holds:
\begin{equation}\label{eq: F_J_properness}
\mathcal F(u) \geq C J_G(Gu) - D, \ \ u\in\mathcal H_\o \cap I^{-1}(0).
\end{equation}
\end{theorem}
\begin{proof}
The equivalence between (ii) and (iii) is the content of Lemma \ref{lem: JGPropernessLemma}.

For the equivalence between (i) and (ii) 
we wish to apply Theorem \ref{thm: ExistencePrinc} to the data
\begin{equation}\label{eq: F_choice_of_data}
\mathcal R=\mathcal H_\o \cap I^{-1}(0), \ \ \ d=d_1, \ \ \ F = \mathcal F, \ \ \ G:=\textup{Aut}_0(X,J).
\end{equation}
We first have to show that \eqref{eq: F_d_1_properness} implies a bit of extra information: that $\mathcal F$ is invariant under the action of $\textup{Aut}_0(X,J)$. Indeed, \eqref{eq: F_d_1_properness} implies that $\mathcal F$ is bounded from below. This and Lemma \ref{lem: F_one_par_linear}(i) implies that $\Bbb R \ni  t \to \mathcal F(\rho_t.u) \in \Bbb R$ is affine and bounded for all one parameter subgroups $\Bbb R \ni t \to \rho_t \in \textup{Aut}_0(X,J)$. Consequently $t \to \mathcal F(\rho_t.u)$ has to be constant equal to $\mathcal F(u)$. As $\textup{Aut}_0(X,J)$ is connected, we obtain that $\mathcal F$ is invariant under the action of this group, as claimed.

To finish the proof, we go over the axioms and properties of $(\mathcal R,d,F,G)$, as imposed in the statement of Theorem \ref{thm: ExistencePrinc}:
\begin{enumerate}
  \item[\ref{a1}] By Theorem \ref{thm: EpComplete} and the fact that $I$ is $d_1$--continuous (Proposition \ref{prop: I_finiteness_and_cont}) we obtain that $\overline{(\mathcal H_\o \cap I^{-1}(0),d_1)}=(\mathcal E_1(X,\o) \cap I^{-1}(0),d_1).$
  \item[\ref{a2}] That $\mathcal F$ is $d_1$--continuous on $\mathcal H_\o$ will be proved in Theorem \ref{thm: F_d_1_cont}. As we will see, the $d_1$--continuous extension $\mathcal F: \mathcal E_1(X,\o) \to \Bbb R$ is given by the original formula for smooth potentials (see \eqref{eq: F_def}).
  \item[\ref{a3}] We choose $\mathcal M$ as the minimizer set of the extended functional $\mathcal F: \mathcal E_1(X,\o) \cap I^{-1}(0) \to \Bbb R$.
\end{enumerate}
\begin{enumerate}
  \item[\ref{p1}] This fact is due to
Berndtsson \cite[Theorem 1.1]{brn1}, and we present this result in Theorem \ref{thm: F_convex} below.
  \item[\ref{p2}] This property will be verfied in Theorem \ref{thm: E_1_F_min_compactness}.
  \item[\ref{p3}]  That elements of $\mathcal M$ are in fact smooth KE potentials follows after combination of results due to Berman, Tosatti--Sz\'ekelyhidi and Berman--Boucksom--Guedj--Zeriahi. We present this in Theorem \ref{thm: F_min_regularity} below.  
  \item[\ref{p4}] This is Lemma \ref{lem: dpIsomLemma} for $p=1$.
  \item[\ref{p5}] This follows from \ref{p3} and the Bando--Mabuchi uniqueness theorem  that we will prove  in Theorem \ref{thm: BMuniqueness} below. 
  \item[\ref{p7}] This is exactly the content of  Lemma \ref{lem: F_one_par_linear}(ii) above.\vspace{-.8cm}
\end{enumerate}
\end{proof}

\begin{remark} The equivalence between (i) and (iii) in the above theorem verifies the analogue of a  conjecture of Tian for the Ding functional $\mathcal F$ (see \cite[p. 127]{t3},\cite[Conjecture 7.12]{t1}). 

In Theorem \ref{thm: K-energy_properness} below we will verify Tian's original conjecture for the K--energy as well, giving another characterization for existence of KE metrics on Fano manifolds. For other results of the same spirit, as well as relation to the literature, we refer to \cite{dr2,bdl2,cc2,cc3,dl18}.
\end{remark}

\section{Continuity and compactness properties of the Ding functional}

Given a Fano manifold $(X,\o)$, in this section we will show that the $\mathcal F$ functional is $d_1$--continuous on $\mathcal H_\o$, hence naturally extends to the $d_1$--completion $\mathcal E_1(X,\o)$ (Theorem \ref{thm: F_d_1_cont}). In addition to this, we will show that $d_1$--bounded  sequences that are $\mathcal F$ minimizing are $d_1$--subconvergent, with this establishing an important compactness property of $\mathcal F$ (Theorem \ref{thm: E_1_F_min_compactness}). 

To start, we need to prove the following preliminary compactness lemma:

\begin{lemma}\label{lem: L_1_closedness} For $B,D \in \Bbb R$ we consider the following subset of $\mathcal E_1(X,\o)$:
$$\mathcal C =\{u \in \mathcal E_1(X,\o) \ : \ B \leq I(u) \leq \sup_X u \leq D \}.$$
Then $\mathcal C$ is compact with respect to the weak  $L^1(\o^n)$ topology of $\textup{PSH}(X,\o)$.
\end{lemma}
\begin{proof}
Since $\sup_X u$ is bounded for any $u \in \mathcal C$, by Lemma \ref{lem: sup_int_psh_eqv} it follows that $\int_X |u|\o^n$ is bounded as well. By the Montel property of psh functions (\cite[Proposition I.4.21]{De}) it follows that $\mathcal C$ is precompact with respect to the  weak $L^1(\o^n)$ topology.

Now we argue that $\mathcal C$ is $L^1$--closed. Let $u_k \in \mathcal C$ and $u \in \textup{PSH}(X,\o)$ such that $\int_X |u -u_k|\o^n \to 0$. After possibly taking a subsequence, we can suppose that $u_k \to u$ a.e.. As $u_k \leq D$ and $u_k \to u$ a.e., it follows that $\sup_X u \leq D$.

Since $u_k \to u$ a.e., it follows that $v_k \searrow u$, where $v_k := (\sup_{j \geq k}u_j)^* \in \textup{PSH}(X,\o)$. By the monotonicity property of $I$ we have that $B \leq I(u_k) \leq I(v_k) \leq \sup_X v_k \leq D, \ k \geq 1$. By Lemma \ref{lem: d1bounded_char} and Lemma \ref{lem: lemma mononton_seq} it follows that  $d_1(v_k,u) \to 0$. Consequently $I(v_k) \to I(u)$ (Proposition \ref{prop: I_finiteness_and_cont}), hence $u \in \mathcal C$.
\end{proof}

In contrast with Proposition \ref{prop: I_finiteness_and_cont}, 
as a consequence of the above argument,  we also obtain that $I$ is $L^1(\o^n)$--usc:
\begin{corollary}\label{cor: I_L_1_semicont} Suppose that $u_k,u \in \mathcal E_1(X,\o)$ satisfies $u_k \to_{L^1} u$. Then  $$\limsup_{k \to \infty}I(u_k) \leq I(u).$$
\end{corollary}

Before we proceed, we recall Zeriahi's uniform version of the famous Skoda integrability theorem \cite{ze}:

\begin{theorem}\label{thm: ZerSkoda} Suppose $\mathcal S \subset \textup{PSH}(X,\o)$ is an $L^1(\o^n)$--compact family whose elements have zero Lelong numbers. Then for any $p \geq 1$ there exists $C:=C(p,\mathcal S,\o) >1$ such that
$$\int_X e^{-pu} \o^n \leq C, \ \ u \in \mathcal S.$$
\end{theorem}

For a full account of this result we refer to \cite[Theorem 2.50]{gzbook}. Since full mass potentials have zero Lelong numbers (Propostion \ref{prop: Lelong_E}), we obtain the following corollary:

\begin{corollary}\label{cor: Skodacor} For $D,p \geq 1$ there exists $C: = C(p,D,\o)>0$ such that for any $u \in \mathcal E_1(X,\o)$ with $d_1(0,u) \leq D$ we have
$$\int_X e^{-pu}\o^n \leq C.$$
\end{corollary}

\begin{proof}From Theorem \ref{thm: Energy_Metric_Eqv} it follows that  $\int_X u \o^n$ is uniformly bounded. By Lemma \ref{lem: sup_int_psh_eqv} so is $\sup_X u$. By Proposition \ref{prop: I_finiteness_and_cont} it follows that $I(u)$ is bounded as well. This allows to apply Lemma \ref{lem: L_1_closedness}, and together with Proposition \ref{prop: Lelong_E} the conditions of of Theorem \ref{thm: ZerSkoda} are satisfied to conclude the result.
\end{proof}

In particular, this last corollary implies that the original definition of the $\mathcal F$ functional for smooth potentials (see \eqref{eq: F_def}) extends to $\mathcal E_1(X,\o)$ as well. Additionally, our next theorem shows that this extension is in fact $d_1$--continuous:

\begin{theorem}\label{thm: F_d_1_cont} The map 
$$\mathcal E_1(X,\o) \ni u \to \mathcal F(u):= -I(u) - \log \frac{1}{V}\int_X e^{-u+f_0}\o^n \in \Bbb R$$ 
is $d_1$--continuous.
\end{theorem}

\begin{proof}We know that $u \to I(u)$ is $d_1$--continuous and finite on $\mathcal E_1(X,\o)$, hence we only have to argue that so is $u \to \int_X e^{-u + f_0} \o^n$. For $u,v \in \mathcal E_1(X,\o)$, using the inequality $|e^x - e^y| \leq |x-y|(e^{x} + e^y), \ x,y \in \Bbb R$, we have the following estimates:
$$\Big|\int_X e^{-u + f_0 }\o^n-\int_X e^{-v + f_0 }\o^n\Big| \leq \int_X e^{f_0 }\Big|e^{-u} -e^{-v}\Big|\o^n \leq \int_X e^{f_0 }|u-v|(e^{-u} + e^{-v})\o^n.$$
Using the fact that $f_0$ is bounded, the H\"older inequality gives 
\begin{equation}\label{eq: exp_d_1_cont}
\Big|\int_X e^{-u + f_0 }\o^n-\int_X e^{-v + f_0 }\o^n\Big|^2 \leq C \Big(\int_X (e^{-2u} + e^{-2v})\o^n \Big)\cdot \Big(\int_X |u-v|^2 \o^n\Big).
\end{equation}
Suppose $u_k \in \mathcal E_1(X,\o)$ such that $d_1(u_k,u) \to 0$. Then Proposition \ref{prop: dp_mixed_convergence} implies that $\int_X |u_k - u|\o^n \to 0$. As all $L^p$ topologies on $\textup{PSH}(X,\o)$ are equivalent (see \cite[Theorem 4.1.8]{ho}) we obtain that $\int_X |u_k - u|^2\o^n \to 0$. Finally,  \eqref{eq: exp_d_1_cont} and Corollary \ref{cor: Skodacor}  implies that $\int_X e^{-u_k + f_0 }\o^n \to \int_X e^{-u + f_0 }\o^n$.
\end{proof}

Lastly, we argue the compactness property of $\mathcal F$ that is a vital ingredient in the existence/properness principle of the previous section:

\begin{theorem}\label{thm: E_1_F_min_compactness} Suppose $u_k \in \mathcal E_1(X,\o)$ such that $\mathcal F(u_k) \to \inf_{\mathcal E_1(X,\o)}\mathcal F$ and $d_1(u_k,0) \leq C$ for some $C>0$. After possibly taking a subsequence, there exists $u \in \mathcal E_1(X,\o)$ such that $d_1(u_k,u) \to 0$.  In particular, $\mathcal F(u)=\inf_{\mathcal E_1(X,\o)}\mathcal F$.
\end{theorem}
\begin{proof} 
First we construct a candidate for the minimizer $u\in \mathcal E_1(X,\o)$. From Theorem \ref{thm: Energy_Metric_Eqv} it follows that  $\int_X u_k \o^n$ is uniformly bounded. By Lemma \ref{lem: sup_int_psh_eqv} so is $\sup_X u_k$. By Proposition \ref{prop: I_finiteness_and_cont} it follows that $I(u_k)$ is bounded as well. Now Lemma \ref{lem: L_1_closedness} implies that after possibly taking a subsequence we  can find $u \in \mathcal E_1(X,\o)$ such that $\int_X |u_k - u| \o^n \to 0$.

We now show that $u$ is actually a minimizer of $\mathcal F$. 
Using \eqref{eq: exp_d_1_cont} in the same way as in the proof of the previous result, we obtain that $\int_X e^{-u_k + f_0}\o^n \to \int_X e^{-u + f_0}\o^n$. By Corollary \ref{cor: I_L_1_semicont} $I$ is 
usc with respect to the weak $L^1(\o^n)$ topology, hence we can write
\begin{equation}
\label{liminfAM}
\lim_{k}\mathcal F(u_{k}) 
\geq 
- \limsup_{k}I(u_{k}) 
-\log\frac{1}{V}\int_M e^{-u+f_0}\o^n \geq \mathcal F(u).
\end{equation}
As $\{u_{k}\}_k$ minimizes $\mathcal F$, 
it follows that all the inequalities above are equalities.
Thus,  $u$ minimizes $\mathcal F$.

Lastly, we show that there is a subsequence of $u_k$ that $d_1$-converges to $u$. As $\limsup_{k}I(u_{k}) = I(u)$, after possibly passing to a subsequence, 
$\lim_k I(u_{k})=I(u)$. 
This together with $|u_{k}-u|_{L^1(\o^n)} \to 0$ 
and Theorem \ref{thm: d_1-convergence} gives that $d_1(u_{k},u)\to 0$.
\end{proof}

\section{Convexity of the Ding functional}

We stay with a Fano manifold $(X,\o)$ for this section as well. Previously we extended Ding's functional to $\mathcal E_1(X,\o)$. In this short section we show that this extension is convex along the finite energy geodesics of $\mathcal E_1(X,\o)$ (Theorem \ref{thm: F_convex}). We start by computing the Hessian of the Ding functional with respect to the $L^2$ Mabuchi metric of $\mathcal H_\o$:

\begin{proposition}\label{prop: F_Hess_prop} Suppose $u \in \mathcal H_\o$ and $\phi,\psi \in C^\infty(X) \simeq T_u \mathcal H_\o$. We have the following formula for the Hessian of $\mathcal F$ with respect to the $L^2$ Mabuchi metric:
\begin{flalign}\label{eq: F_Hess_formula}
\nabla^2& \mathcal F(u)(\phi,\psi)=\\
&=\frac{1}{V}\int_X\bigg[ \frac{1}{2}\langle \nabla^{\o_u} \phi, \nabla^{\o_u} \psi \rangle_{\o_u} - \Big( \phi -  \frac{1}{V}\int_X \phi e^{f_u}\o_u^n\Big) \cdot \Big( \psi -  \frac{1}{V}\int_X \psi e^{f_u}\o_u^n\Big)\bigg]  e^{f_u} \o_u^n. \nonumber
\end{flalign}
\end{proposition}
\begin{proof}

First we compute the Hessian of the Aubin--Yau energy $I$. From Lemma \ref{lem: I_differential} it follows that for small $t >0$ we have:
$$dI(u + t \psi)(\phi)= \frac{1}{V}\int_X \phi \o_{u + t \psi}^n.$$
Using the formulas $\nabla^2 I(u + t \psi)(\phi,\psi) = \frac{d}{dt}d I(u + t \psi)(\phi) - dI(u+t\psi)(\nabla_{\frac{d}{dt}}\phi)$ and $\nabla_{\frac{d}{dt}} \phi = -\frac{1}{2}\langle \nabla^{\o_{u}} \phi,\nabla^{\o_{u}} \psi \rangle_{\o_u}$, we can write:
\begin{flalign}\label{eq: I_Hess_eq}
\nabla^2 I(u)(\phi,\psi)&=\frac{n}{V}\int_X \phi i \ddbar \psi \wedge \o_u^{n-1} + \frac{1}{2V}\int_X \langle \nabla^{\o_{u}} \phi,\nabla^{\o_{u}} \psi \rangle_{\o_u} \o_u^n  \nonumber\\
&=\frac{1}{2V}\int_X \phi  (\Delta^{\o_{u}} \psi) \o_u^{n} + \frac{1}{2V}\int_X \langle \nabla^{\o_{u}} \phi,\nabla^{\o_{u}} \psi \rangle_{\o_u}\o_u^n = 0,
\end{flalign}
where in the last line we have used  \eqref{eq: Lapl_grad_formula} in the appendix and the formula below it. Next we introduce $\mathcal B: \mathcal H_{\o} \to \Bbb R$ by the formula
\begin{equation}\label{eq: B_oper_def}
\mathcal B(u)= - \log \frac{1}{V}\int_X e^{-u + f_0}\o^n, \ u \in \mathcal H_\o. 
\end{equation}
Using \eqref{eq: Ricci_pot_eq} and $d\mathcal B(u+t\psi)(\phi) = {\int_X \phi e^{-u - t\psi + f_0}\o^n}/{\int_X e^{-u - t\psi + f_0}\o^n},$ we obtain that $d\mathcal B(u+t\psi)(\phi) =\frac{1}{V} \int_X \phi e^{f_u}\o_u^n$. Another  differentiation gives
\begin{flalign*}
\frac{d}{dt}\bigg|_{t=0}d \mathcal B(u + t\psi)(\phi) &= -\frac{\int_X \phi\psi e^{-u + f_0}\o^n}{\int_X e^{-u + f_0}\o^n} + \frac{\int_X \phi e^{-u + f_0}\o^n}{\int_X e^{-u + f_0}\o^n} \cdot \frac{\int_X \psi e^{-u + f_0}\o^n}{\int_X e^{-u + f_0}\o^n} \\ 
&= -\frac{1}{V} {\int_X \phi\psi e^{f_u}\o^n_u} + \frac{1}{V^2}{\int_X \phi e^{f_u}\o_u^n} \cdot {\int_X \psi e^{f_u}\o_u^n}.
\end{flalign*}
where in the last line we have used \eqref{eq: Ricci_pot_eq} again. Reorganizing terms in this identity and using $\nabla^2 \mathcal B(u + t \psi)(\phi,\psi) = \frac{d}{dt}d \mathcal B(u + t \psi)(\phi) - d\mathcal B(u + t\psi)(\nabla_{\frac{d}{dt}}\phi)$ we conclude that  
\begin{flalign*}
\nabla^2 & \mathcal B(u)(\phi,\psi) =\\
&=\frac{1}{V}\int_X\bigg[ \frac{1}{2}\langle \nabla^{\o_u} \phi, \nabla^{\o_u} \psi \rangle_{\o_u} - \Big( \phi -  \frac{1}{V}\int_X \phi e^{f_u}\o_u^n\Big) \cdot \Big( \psi -  \frac{1}{V} \int_X \psi e^{f_u}\o_u^n\Big)\bigg]  e^{f_u} \o_u^n.
\end{flalign*}
The proof is finished after we subtract \eqref{eq: I_Hess_eq} from this last formula.
\end{proof}

Next we will show that $\nabla^2 \mathcal F(u)(\cdot,\cdot)$ is positive semi--definite for all $u \in \mathcal H_\o$. Before this we introduce and study the following complex valued weighted complex Laplacian 
\begin{equation}\label{eq: L_f_def}
 L^{f_u} h= \partial ^* \partial h - \langle\partial h, \partial f_u \rangle_{\o_u}, \ \ \ \ h \in C^\infty(X,\Bbb C).
\end{equation}
To clarify, by $\langle\partial h, \partial f_u \rangle_{\o_u}$ we mean the quantity $g_u^{j\bar k} h_{j} {f_u}_{\bar k}$ (expressed in local coordinates). Also $\partial ^*$ is the Hermitian $L^2$ adjoint of $\partial$ with respect to $\o_u$.
For $g,h \in C^\infty(X,\Bbb C)$ integration by parts gives the following:
\begin{equation}\label{eq: self_adj}
\int_X  ( L^{f_u} g ) \bar h e^{f_u}\o_u^n =\int_X \langle \partial  g, \partial h \rangle_{\o_u} e^{f_u}\o_u^n=\int_X g \overline{( L^{f_u} h )} e^{f_u}\o_u^n.
\end{equation}
Consequently $L^{f_u}$ is a self--adjoint elliptic operator with respect to the Hermitian inner product
$$\langle\alpha,\beta \rangle = \int_X \alpha \bar \beta e^{f_u} \o_u, \ \ \ \alpha,\beta \in C^\infty(X,\Bbb C).$$
We conclude that $L^2(e^{f_u}\o_u^n)$ has an orthonormal base composed of eigenfunctions corresponding to the  eigenvalues $\lambda_0 < \lambda_1 < \ldots$ of $L^{f_u}$. As another application of   \eqref{eq: self_adj} we see that $\lambda_0 = 0$ and the eigenspace of this eigenvalue is  composed by the constant functions. Moreover, we have the following general result about the eigenfunctions of $L^{f_u}$ due to Futaki:
\begin{proposition}[\cite{fu}] \label{prop: Fut_ineq}Suppose $v \in C^\infty(X,\Bbb R)$ and $u \in \mathcal H_\o$ such that $\int_X v e^{f_u} \o_u^n =0$. In addition, let $h$ be an eigenfunction $L^{f_u}$, i.e., $L^{f_u} h = \lambda h$.
Then the following hold:
\begin{equation}\label{eq: Fut_id}
\lambda \int_X \langle \partial h, \partial h \rangle_{\o_u} e^{f_u}\o_u^n  = \int_X \langle \partial h, \partial h \rangle_{\o_u} e^{f_u}\o_u^n + \int_X \langle \mathcal L h, \mathcal L h \rangle_{\o_u} e^{f_u}\o_u^n, 
\end{equation} 
\begin{equation}\label{eq: Fut_ineq}
\int_X |v|^2 e^{f_{u}}\o^n_{u} \leq \frac{1}{2}\int_X \langle \nabla^{\o_u} v, \nabla^{\o_u} v \rangle_{\o_u} e^{f_{u}}\o_{u}^n=
\int_X \langle \partial  v, \partial v \rangle_{\o_u} e^{f_u}\o_u^n=\int_X  ( L^{f_u} v ) \bar v e^{f_u}\o_u^n,
\end{equation} 
where $\mathcal L$ is the Lichnerowitz operator \textup{(}see \eqref{eq: Lich_def} in the appendix\textup{)}. Also, one has equality in \eqref{eq: Fut_ineq} if and only if $\mathcal L v=0$, or equivalently, $\nabla^{\o_u}_{0,1} v \in T^{1,0}_{\Bbb C}X$ is a holomorphic vector field.
\end{proposition}

Recall from the discussion following \eqref{eq: Lich_def} in the appendix that the condition $\mathcal L v=0$ is indeed equivalent with $\nabla^{\o_u}_{0,1} v \in T^{1,0}_{\Bbb C}X$ being holomorphic.

\begin{proof} By the the discussion preceding the proposition, the condition $\int_X v e^{f_u} \o_u^n =0$ simply means that $v$ is orthogonal to the eigenspace of $\lambda_0$, hence to prove \eqref{eq: Fut_ineq} we only need to argue \eqref{eq: Fut_id}. Indeed, the rightmost term in \eqref{eq: Fut_id} is nonnegative hence either $\lambda = \lambda_0 =0$ (in which case $h$ is a constant) or $\lambda \geq \lambda_1 \geq 1$. Consequently the inequality between the first and last term of \eqref{eq: Fut_ineq} follows after expressing $v$ using the orthonormal base of $L^2(e^{f_u}\o_u^n)$ composed of eigenfunctions of $L^{f_u}$. 

Lastly, we note that the middle identities of \eqref{eq: Fut_ineq} are simply a consequence of \eqref{eq: self_adj} and the formula following \eqref{eq: Lapl_grad_formula}, since $v$ is real valued.

We now argue \eqref{eq: Fut_id}. To ease notation, we will drop the subscript of $f_u$ and $\o_u$ in the rest of the proof. Also, recall that by a choice of normal coordinates identifying a neighborhood of $x \in X$ with that of $0 \in \Bbb C^n$, we can assure that locally $\o = i\ddbar g$, with $g_{j\bar k}(0)=\delta_{jk}$ and $g_{j\bar k l}(0)=g_{j\bar k \bar l}(0)=0$ (see Proposition \ref{prop: prelnormcoord} below). With such a choice of coordinates we also have $\textup{Ric}\ \o _{j\bar k}|_x = - g_{j \bar k a \bar a}(0)$ (see \eqref{eq: prelcurvnormalcoord}). Making use of this and  integrating by parts multiple times we get
\begin{flalign*}\lambda \int_X \langle \partial h,\partial h \rangle e^{f}\o^n & = \int_X \langle \partial (L^{f} h), \partial h \rangle e^{f}\o^n= \int_X g^{j\bar k}(L^{f}h)_j \bar h_{\bar k} e^{f}\o^n\\
&= \int_X g^{j\bar k}(-g^{a\bar b} h_{a\bar b} - g^{a\bar b} h_{a} {f}_{\bar b})_j \bar h_{\bar k} e^{f}\o^n\\
&= \int_X (g_{\bar a jb} h_{a\bar b} \bar h_{\bar j} - h_{a\bar a j} \bar h_{\bar j} -h_{aj}f_{\bar a}\bar h_{\bar j} - h_a f_{\bar a j} \bar h_{\bar j})e^{f}\o^n\\
&= \int_X (g_{\bar a jb} h_{a\bar b} \bar h_{\bar j} + h_{aj} \bar h_{\bar a \bar j} + h_{aj}f_{\bar a}\bar h_{\bar j}-h_{aj}f_{\bar a}\bar h_{\bar j} - h_a f_{\bar a j} \bar h_{\bar j})e^{f}\o^n\\
&= \int_X (g_{\bar a jb} h_{a\bar b} \bar h_{\bar j} + h_{aj} \bar h_{\bar a \bar j} - h_a f_{\bar a j} \bar h_{\bar j})e^{f}\o^n\\
&= \int_X (-g_{\bar a jb \bar b} h_{a} \bar h_{\bar j} + h_{aj} \bar h_{\bar a \bar j} - h_a f_{\bar a j} \bar h_{\bar j})e^{f}\o^n\\
&= \int_X (\textup{Ric} \ \o_{u}(\partial h,\dbar \bar h) + h_{aj} \bar h_{\bar a \bar j} - i\ddbar f_u(\partial h,\dbar h))e^{f}\o^n\\
&= \int_X (\langle \partial h,\partial h \rangle + h_{aj} \bar h_{\bar a \bar j})e^{f}\o^n = \int_X \langle \partial h,\partial h \rangle e^{f}\o^n + \int _X \langle \mathcal L h, \mathcal L h  \rangle e^{f}\o^n,
\end{flalign*}
where in the last line we have used the identity $\textup{Ric }\o_u - \o_u = i\ddbar f_u$, and the expression of $\mathcal L$ in normal coordinates (see \eqref{eq: Lich_def}). 
\end{proof}

From \eqref{eq: Fut_ineq} and Proposition \ref{prop: F_Hess_prop} it follows that $\nabla^2 \mathcal F(u)(\cdot,\cdot)$ is indeed positive semi--definite. As an additional consequence we obtain that $\mathcal B$ is convex along $\varepsilon$--geodesics:

\begin{lemma}\label{lem: F_eps_convex}Suppose $u_0,u_1 \in \mathcal H_\o$ and $\varepsilon > 0$. Let $[0,1] \ni t \to u^\varepsilon_t \in \mathcal H_\o$ be the smooth $\varepsilon$--geodesic joining $u_0,u_1$ \textup{(}see \eqref{eq: eps_geod_eq_Lev_Civ}\textup{)}. Then $t \to \mathcal B(u^\varepsilon_t):=-\log\big(\int_X e^{-u_t^{\varepsilon} + f_0}\o^n\big)$ is convex.
\end{lemma}
\begin{proof} As it will not cause confusion, we will drop the reference to $\varepsilon$ in our argument. 

Using the $L^2$ Mabuchi structure of $\mathcal H_\o$ we have the following formula:
\begin{flalign*}
\frac{d^2}{dt^2}\mathcal B(u_t) = \nabla^2 \mathcal B(u_t)(\dot u_t,\dot u_t) + d \mathcal B(u_t)(\nabla_{\dot u_t} \dot u_t).
\end{flalign*}
From Proposition \ref{prop: Fut_ineq} it follows that $\nabla^2 \mathcal B(u_t)(\dot u_t,\dot u_t) \geq 0$. In addition to this, the equation of $\varepsilon$--geodesics (see \eqref{eq: eps_geod_eq_Lev_Civ}) gives $\nabla_{\dot u_t} \dot u_t > 0$, hence we obtain $d \mathcal B(u_t)(\nabla_{\dot u_t} \dot u_t)  \geq 0$. Putting the last two facts together we get that $\frac{d^2}{dt^2}\mathcal B(u_t) \geq 0$.
\end{proof}

Finally, we argue that $\mathcal F$ is convex along the finite energy geodesics of $\mathcal E_1(X,\o)$. 

\begin{theorem}\label{thm: F_convex} Suppose $u_0,u_1 \in \mathcal E_1(X,\o)$ and let $t \to u_t$ be the finite energy geodesic connecting $u_0,u_1$. Then $t \to \mathcal F(u_t)$ is convex and continuous on $[0,1]$.
\end{theorem}
\begin{proof} As $t \to I(u_t)$ is known to be affine (Proposition \ref{prop: I_linrear}), we only need to argue that $t \to \mathcal B(u_t):=-\log \big( \int_X e^{-u_t + f_0}\o^n\big)$ is also affine.
As $u \to \mathcal F(u)$ and $u \to I(u)$ is $d_1$-continuous (see Proposition \ref{prop: I_finiteness_and_cont} and Theorem \ref{thm: F_d_1_cont}), we obtain that so is $u \to \mathcal B(u)$. 

Let us assume first that $u_0,u_1 \in \mathcal H_\o$. In this case  $u^\varepsilon_t \nearrow u_t$ as $\varepsilon \to 0$, where $t \to u^\varepsilon_t$ is the $\varepsilon$--geodesic joining $u_0,u_1$ (see \eqref{eq: epsgeod_limit}). Consequently $d_1(u^\varepsilon_t,u_t) \to 0, \ t \in [0,1]$. Hence $\mathcal B(u^\varepsilon_t) \to \mathcal B(u_t), \ t \in [0,1]$. By Lemma \ref{lem: F_eps_convex} it follows that $t \to \mathcal B(u_t)$ is convex.

In the general case $u_0,u_1 \in \mathcal E_1(X,\o)$, let $u_0^j,u_1^j \in \mathcal H_\o$ be smooth decreasing approximants that exist by Theorem \ref{thm: BK_approx}. Let $t \to u^j_t$ be the $C^{1,\bar 1}$--geodesics joining $u^j_0,u^j_1$. By Proposition \ref{prop: weak_geod_approx}(i) it follows that $d_1(u^j_t,u_t) \to 0, \ t \in [0,1]$ and consequently $\mathcal B(u^j_t) \to \mathcal B(u_t), \ t \in [0,1]$, implying that $t \to \mathcal B(u_t)$ is convex, finishing the proof.  
\end{proof}

For weak geodesics joining bounded potentials the above theorem was proved by Berndtsson in much more general context \cite[Theorem 1.2]{brn0}. The argument that we presented in this section follows more closely the simplified treatment in \cite{he2}.

\section{Uniqueness of KE metrics and reductivity of the automorphism group}

In this section we will give the proof of an important theorem of Bando--Mabuchi according to which on a Fano manifold $(X,\o)$ K\"ahler--Einstein metrics are unique up to pullback by an automorphism:

\begin{theorem}[\cite{bm}] \label{thm: BMuniqueness} Suppose $u,v \in \mathcal H_\o$ both solve \eqref{eq: KE_scalar_eq}, i.e., they are both KE potentials. Then there exists $g \in \textup{Aut}_0(X,J)$ such that $g^* \o_u = \o_v$.
\end{theorem}

As we plan to use some of the machinery that we developed in previous parts, we will not follow the original proof of Bando--Mabuchi. Instead our proof will be a combination of the arguments of Berndtsson \cite{brn1} and Berman--Berndtsson \cite[Section 4]{bb}, and bears similarities with the treatment in \cite{Li}.

The proof will need a sequence of preliminary results about the automorphism group of K\"ahler--Einstein manifolds and will also use the classical theory of self adjoint elliptic differential operators on compact manifolds (see \cite[Chapter 4]{we}).

\begin{lemma}\label{lem: Fano_coh_vanishing} Suppose $X$ is a Fano manifold. Then $H^{0,q}(X,\Bbb C)$ is trivial for $q \in \{1,\ldots,n\}$.
\end{lemma}

\begin{proof}We start with the observation $H^{0,q}(X,\Bbb C) \simeq H^{n,q}(X,-K_X)$. As $-K_X >0$ by the Fano condition, the triviality of $H^{n,q}(X,-K_X)$ follows from the Kodaira vanishing theorem (\cite[Theorem VII.3.3]{De}).
\end{proof}

Note that by the Hodge decomposition we also have $H^1(X,\Bbb C)\simeq H^{1,0}(X,\Bbb C) \oplus H^{0,1}(X,\Bbb C)$, hence this group is trivial as well. Though we will not use this, we mention that by an argument involving the Bonnet--Myers theorem and the Euler characteristic, we can further deduce that $X$ is in fact simply connected.

Now we focus on the Lie algebra $\mathfrak g$ of $G:=\textup{Aut}_0(X,J)$. Pick $U \in \mathfrak g$.
As $U= U^{1,0} + \overline{U^{1,0}}$ is real holomorphic it follows that $\dbar (U^{1,0} \lrcorner \o_u) =0$ for all $u \in \mathcal H_\o$. Indeed, this is immediate after one computes $\partial_{\bar l}(U^{1,0}_j {g^u}_{j\bar k})=0, \ l \in \{1,\ldots,n\}$ in normal coordinates (here and below $g^u$ is a local potential of $\o_u$). 

Using the previous lemma it follows that there exists  a unique $v_{\o_u}^U \in C^\infty(X,\Bbb C)$ with $\int_X v_{\o_u}^U\o_u^n =0$ such that
$$U^{1,0} \lrcorner \o_u = \dbar v_{\o_u}^U.$$ 
Equivalently, using Hamiltonian formalism this can be written as $X^{\o_u}_{1,0}({ v_{\o_u}^U})=U^{1,0},$ and we have the following identification for $\mathfrak g$ using $\o_u$:
\begin{equation}\label{eq: mathfrak_g_def}
\mathfrak g \simeq \mathfrak g_{\o_u}:=\{v \in C^\infty(X,\Bbb C)\ : \ X^{\o_u}_{1,0} v \in T^{1,0}_{\Bbb C} X  \textup{ is holomorphic and } \int_X v \o_u^n = 0\}.
\end{equation}
Recall that the ``complex" gradient $\nabla^{\o_u}_{1,0} v = JX^{\o_u}_{1,0} v$ (see \eqref{eq: grad_Ham_relation}) is holomorphic precisely when $\mathcal L v=0$, where $\mathcal L$ is the Lichnerowitz operator of the metric $\o_u$ (see \eqref{eq: Lich_def}). 

When $\mathcal H_\o$ contains a KE potential $u$ then $S_{\o_u}$ is trivially  constant, hence by Proposition \ref{prop: prelLichform} we have that 
$$\mathcal L^* \mathcal L(f)= \frac{1}{4}\Delta^{\o_u}(\Delta^{\o_u} f) + \langle Ric_{\o_u}, i\partial \bar \partial f\rangle_{\o_u}=\frac{1}{4}\Delta^{\o_u}(\Delta^{\o_u} f) + \frac{1}{2}\Delta^{\o_u} f$$
is a real differential operator. As a result, $v \in \ker \mathcal L^* \mathcal L = \ker \mathcal L$ if and only if $\textup{Re }v, \textup{Im }v \in \ker \mathcal L$. Consequently, for KE Fano manifolds the above description of  $\mathfrak g$ can be sharpened, to imply that $\textup{Aut}_0(X,J)$ is reductive, which was one of the first known obstructions to existence of KE metrics:

\begin{proposition}[\cite{mat}]\label{prop: Matsushima_thm} Suppose $(X,\o_u)$ is a Fano KE manifold. Introducing $\mathfrak k_{\o_u} = \mathfrak g_{\o_u} \cap C^\infty(X,\Bbb R)$ we can write 
$$\mathfrak g_{\o_u} = \mathfrak k_{\o_u} \oplus i \mathfrak k_{\o_u}.$$
In particular, $\textup{Aut}_0(X,J)$ is the complexification of the compact connected Lie group $\textup{Isom}_0(X,\o_u,J)$, with Lie algebra $\mathfrak k_{\o_u}$.
\end{proposition}
Here $\textup{Isom}_0(X,\o_u,J)$ is the identity component of the group of holomorphic isometries of the KE metric $\o_u$. As $X$ is compact, the group of smooth isometries of $(X,\o_u)$ is compact as well, hence so is its subgroup $\textup{Isom}_0(X,\o_u,J)$. 
\begin{proof} The decomposition $\mathfrak g_{\o_u} = \mathfrak k_{\o_u} \oplus i \mathfrak k_{\o_u}$ follows from the discussion preceding the proposition. We have to argue that the Lie algebra of $\textup{Isom}_0(X,\o_u,J)$ is exactly $\mathfrak k_{\o_u}$.

Suppose $U \in \frak k_{\o_u}$. Trivially $v := v^U_{\o_u} \in C^\infty(X,\Bbb R)$, and since $U^{1,0}=X^{\o_u}_{1,0} v$, by conjugation we obtain that in fact $U = X^{\o_u}_v$. Consequently,  
$$d (U \lrcorner \o_u) =  d (U^{1,0} \lrcorner \o_u +  U^{0,1} \lrcorner \o_u)= d(\dbar h + \partial h)=d d h=0,$$ hence $U$ represents an infinitesimal symplectomorphism. Since $U$ is holomorphic, by \eqref{eq: form_Riem_rel} $U$ represents an infinitesimal isometry as well, i.e., $U \in \textup{Lie}(\textup{Isom}_0(X,\o_u,J))$ as claimed.

Conversely, if $U \in \textup{Lie}(\textup{Isom}_0(X,\o_u,J))$ then $0=d (U \lrcorner \o_u) = d (U^{1,0}+ U^{0,1}) \lrcorner \o_u = d(\dbar v^U_{\o_u} + \partial \overline{v^U_{\o_u}})= 2i \ddbar \textup{Im }v^U_{\o_u}$. Consequently $v^U_{\o_u} \in \mathfrak g_{\o_u} \cap C^\infty(X,\Bbb R)=\mathfrak k_{\o_u}$.
\end{proof}

According to the next result the action of a one parameter subgroup of automorphisms in the direction of $i\mathfrak k$ gives $d_p$--geodesic rays inside $\mathcal H_\o$:

\begin{lemma}\label{lem: aut_geod_dp} Suppose $u \in \mathcal H_\o \cap I^{-1}(0)$ is a KE potential. Let  $U \in i\mathfrak k_{\o_u}$ and $\Bbb R \ni t \to \rho_t \in \textup{Aut}_0(X,J)$ be the associated one parameter subgroup. Then $t \to u_t :=\rho_t.u$ is a smooth $d_p$--geodesic ray for any $p \geq 1$.
\end{lemma}
\begin{proof} As $U \in i \mathfrak k_{\o_u}$, it follows that $v^U_{\o_u}=ih$ for some $h \in C^\infty(X,\Bbb R)$ with $\int_X h \o_u^n =0$. Differentiating $\rho_t^* \o_u = \o_{u_t}$ we find that $i\ddbar \dot u_t = \rho_t^*d (U \lrcorner \o_{u})=\rho_t^*d (\dbar v^U_{\o_u} + \partial \overline{v^U_{\o_u}}) = 2i \rho_t^*\ddbar (\textup{Im }v^U_{\o_u}) = 2i\ddbar h \circ \rho_t$. 

Since $I(u_t)=0, \ t  \geq 0$, Lemma \ref{lem: I_differential} gives that $\int_X \dot u_t \o_{u_t}^n = 0$. Also, $\int_X h \circ \rho_t \o_{u_t}^n=\int_X h \circ \rho_t \rho_t^*\o_{u}^n=\int_X h \o_{u}^n=0$, so we conclude that 
\begin{equation}\label{eq: tang_vect_id}
\dot u_t = 2h \circ \rho _t.
\end{equation}
Differentiating this identity we obtain
\begin{flalign*}
\ddot u _t &= 2\rho^*_t (U \lrcorner d h) = 2\rho^*_t \langle \nabla^{\o_u} h , U \rangle_{\o_u} = 2\rho^*_t \langle \nabla^{\o_u} h , J X^{\o_u} h \rangle_{\o_u}\\
&=2\langle \nabla^{\o_{u_t}} h\circ \rho_t , \nabla^{\o_{u_t}} h\circ \rho_t \rangle_{\o_{u_t}}=\frac{1}{2} \langle \nabla^{\o_{u_t}} \dot u_t, \nabla^{\o_{u_t}} \dot u_t \rangle_{\o_{u_t}},
\end{flalign*}
where we used \eqref{eq: grad_Ham_relation}   in the second to last equality, and \eqref{eq: tang_vect_id} again in the last equality. The above arguments show that $t \to u_t$ satisfies \eqref{eq: geod_eq_Lev_Civ}, hence by Theorem \ref{thm: EpComplete} we obtain that $t \to u_t$ is a geodesic ray for any $p \geq 1$. 
\end{proof}

The following result will play an important role in the proof of Theorem \ref{thm: BMuniqueness}:

\begin{proposition}\label{prop: J_group_minimizer} Let $(X,\o)$ be Fano. Suppose $u \in \mathcal H_\o \cap I^{-1}(0)$ is a KE potential. Then the map $\textup{Aut}_0(X,J) \ni h \to J_\o(h.u) \in \Bbb R$ admits a minimizer $g \in \textup{Aut}_0(X,J)$ that satisfies 
\begin{equation}\label{eq: v_orthog}
\int_X v \o^n =0 \ \ \textup{ for all } \ \ v \in \mathfrak g_{\o_{g.u}}.
\end{equation}
\end{proposition}

Recall the definition of the $J$ functional from \eqref{eq: J_def}. To avoid the possibility of confusion with the complex structure, we denoted this functional with $J_\o$ in the above proposition, and will continue to do so in the rest of this section.

\begin{proof} By Proposition \ref{prop: d_1_growth_J} the functional $J_\o$  has the same growth as the metric  $d_1$. We turn to the group $\textup{Aut}_0(X,J)$ which is reductive by Proposition \ref{prop: Matsushima_thm}, hence we can apply Proposition \ref{prop: PartialCartanProp} in the appendix to deduce that the map $C:  \textup{Isom}_0(X,\o_u,J)\oplus \mathfrak k_{\o_u} \to \textup{Aut}_0(X,J)$ given by $C(k,U)=k \textup{exp}_{I}(JU)$ is surjective. 

For any $k \in \textup{Isom}_0(X,\o_u,J)$  and any $U \in \mathfrak k_{\o_u}$ the previous proposition gives that $t \to k\textup{exp}_I(tJU).u =\textup{exp}_I(tJU).u=: u_t \in \mathcal H_\o \cap I^{-1}(0)$ is a smooth $d_1$--geodesic. As the growth of $J_\o$ is equivalent with the growth of the $d_1$ metric, it follows that the map
$$(K,\mathfrak k_{\o_u}) \ni (k,U) \to \Theta(k,U):=J_\o(C(k,U).u)=J_\o(\textup{exp}_I(JU).u) =\Theta(I,U)\in \Bbb R$$
is proper (meaning that $\Theta(k_j,U_j)=\Theta(I,U_j) \to \infty$ if $|U_j | \to \infty$), hence it admits a minimizer. As $C$ is surjective, it follows that $\textup{Aut}_0(X,J) \ni h \to J_\o(h.u) \in \Bbb R$ admits a  minimizer $g \in \textup{Aut}_0(X,J)$ as well. 

Fix $v \in \mathfrak k_{\o_{g.u}}$. We introduce the vector field $W = X^{\o_{g.u}}v$, and by the previous proposition $[0,\infty) \ni t \to h_t := \textup{exp}(tJW).(g.u) \in \mathcal H_0 \cap I^{-1}(0)$ is a geodesic ray.

As the identity element minimizes $\textup{Aut}_0(X,J) \ni h \to J_\o(h.(g.u)) \in \Bbb R$, we can write
$$0=\frac{d}{dt}\Big|_{t=0}J_\o(\textup{exp}(tJW).(g.u))=\frac{1}{V}\frac{d}{dt}\Big|_{t=0} \int_X h_t \o^n = \frac{1}{V}\int_X \dot h_0 \o^n.$$
Using \eqref{eq: tang_vect_id} we conclude that $\dot h_0 = 2v$, hence $\int_X v \o^n =0$. Finally, the decomposition formula of Proposition \ref{prop: Matsushima_thm} implies that \eqref{eq: v_orthog} in fact holds for any $v \in \mathfrak g_{\o_{g.u}}$.
\end{proof}

For $u \in \mathcal H_\o$, recall the self adjoint differential operator $L^{f_u}$ from \eqref{eq: L_f_def}. Motivated by the explicit formula for the Hessian of $\mathcal F$ (see \ref{eq: F_Hess_formula}) we introduce the differential operator $\mathcal D^u:C^\infty(X,\Bbb C) \to \Bbb R$:
$$\mathcal D^u (h) = L^{f_u}(h) - h + \frac{1}{V}\int_X h e^{f_u}\o_u^n.$$
Integrating by parts in \eqref{eq: F_Hess_formula} we get the following formula, relating $\nabla^2 \mathcal F(u)$ and $\mathcal D^u$:
\begin{flalign}\label{eq: F_Hess_D_rel}
\nabla^2 \mathcal F(u)(\phi,\psi)&= \frac{1}{V}\int_X \phi\Big( \overline{L^{f_u}\psi -\psi + \frac{1}{V}\int_X \psi e^{f_u}\o_u^n}\Big) e^{f_u} \o_u^n = \frac{1}{V}\int_X \phi\overline{\mathcal D^{u}(\psi)} e^{f_u} \o_u^n. \nonumber\\
&= \frac{1}{V}\int_X \Big({L^{f_u}\phi -\phi + \frac{1}{V}\int_X \phi e^{f_u}\o_u^n}\Big)\psi e^{f_u} \o_u^n = \frac{1}{V}\int_X {\mathcal D^{u}(\phi)}\overline{\psi} e^{f_u} \o_u^n.
\end{flalign}
This implies that $\mathcal D^u$ is a self--adjoint differential operator as well. Additionally, the kernel of $\mathcal D^u$ is exactly equal to the eigenspace of $\mathcal L^{f_u}$ corresponding to the eigenvalue $\lambda = 1$. In Proposition \ref{prop: Fut_ineq} (see especially the identity \eqref{eq: Fut_id}) we gave an exact description of this space that we now recall:
$$\textup{Ker } \mathcal D^u = \{v \in C^\infty(X,\Bbb C) \textup{ s.t.} \int_X v e^{f_u}\o^n_u =0 \textup{ and }\nabla^{\o_u}_{1,0}v \in T^{1,0}_{\Bbb C} X\textup{ is holomorphic} \}.$$
In case $u \in \mathcal H_\o \cap I^{-1}(0)$ is a KE potential we trivially have $f_u =0$, and comparing with \eqref{eq: mathfrak_g_def} we get the following identification:
\begin{equation}\label{eq: Ker_D_u_mathfrak_g}
\textup{Ker } \mathcal D^u = \mathfrak g_{\o_u}.
\end{equation}
As $\mathcal D^u$ is self--adjoint and elliptic, for any $h \in C^\infty(X,\Bbb C)$ such that $ h \perp \textup{Ker } \mathcal D^u$ there exists $v \in C^\infty(X,\Bbb C)$ such that $\mathcal D^u(v)=h$ (see \cite[Theorem IV.4.11]{we}). We note this fact in slightly more precise form in the following lemma:

\begin{lemma}\label{lem: F_Hess_range} Let $(X,\o_u)$ is a KE manifold. Suppose that $\int_X v \o^n =0$ for all $v \in \mathfrak g_{\o_u}$. Then there exists $g \in C^\infty(X,\Bbb R)$ such that
$$\nabla^2 \mathcal F(u)(f,g)= \frac{1}{V} \int_X f \overline{\mathcal D^u(g)} \o_u^n = \frac{1}{V}\int_X f\o^n, \ \ \forall \ f \in C^\infty(X,\Bbb C).$$ 
\end{lemma}
\begin{proof} Denote $h := \o^n/\o_u^n \in C^\infty(X,\Bbb R)$. By our assumption we have that $h \perp \mathfrak g_{\o_u}$ with respect to the Hermitian product $\langle \alpha, \beta \rangle= \frac{1}{V} \int_X  \alpha \overline{\beta} \o_u^n$, hence by our above remarks there exists $g \in C^\infty(X,\Bbb C)$ such that $\mathcal D^u(g)=h$. As $u$ is a KE potential, we have $f_u =0$, and as a result $\mathcal D^u$ is a real differential operator. Consequently, we can make sure that $g \in C^\infty(X,\Bbb R)$, finishing the argument.
\end{proof}

\begin{proof}[Proof of Theorem \ref{thm: BMuniqueness}] Without loss of generality we can assume that our KE potentials $u,v$ satisfy $u,v \in \mathcal H_\o \cap I^{-1}(0)$. Also, by Proposition \ref{prop: J_group_minimizer}, after possibly pulling back $u$ and $v$ by an element of $\textup{Aut}_0(X,J)$ we can assume that  
$$\int_X h \o^n=0, \ \ \forall \in h \in \mathfrak g_{\o_u} \cup \mathfrak g_{\o_v}.$$
Using this, by the previous lemma we can find $g_u,g_v \in C^\infty(X,\Bbb R)$ such that 
\begin{equation}\label{eq: Hessian_F_u}
\nabla^2 \mathcal F(u)(f,g_u) =  \nabla^2 \mathcal F(v)(f,g_v) = -\frac{1}{V}\int_X f \o^n, \ \ \forall \ f \in C^\infty(X,\Bbb R).
\end{equation}
For the rest of the proof we will be working with the twisted Ding functional $\mathcal F_s: \mathcal H_\o \to \Bbb R, \ s \geq 0$ given by the formula:
$$\mathcal F_s(h)= \mathcal F(h) + s J_\o(h).$$

For small enough $s>0$  we can suppose that the potentials $u_0^s := u + s g_u$ and  $u_1^s := v + s g_u$ satisfy $u^s_0,u^s_1 \in \mathcal H_\o$.

The differential of $\mathcal F_s$ is equal to $d \mathcal F + s d J$. Choosing  $w \in C^\infty(X,\Bbb R)$ with $\int_X w \o_{u}^n=0$ a simple differentiation gives 
$$\frac{d}{ds}\Big|_{s=0} d\mathcal F_s(u^s_0)(w)= \nabla^2 \mathcal F(u)(w,g_u) + d \mathcal F(u)(\nabla_{\frac{d}{ds}} w) + dJ_\o(u)(w).$$
Since $u$ minimizes $\mathcal F$ we have that $d \mathcal F(u)(\nabla_{\frac{d}{ds}} w)=0$. Since $\int_X w \o_u^n=0$ we also have $d J_\o(u)(w)=\frac{1}{V}\int_X w\o^n$, so we can continue the above identity and write:
$$\frac{d}{ds}\Big|_{s=0} d\mathcal F_s(u^s_0)(w)= \nabla^2 \mathcal F(u)(w,g_u) + \frac{1}{V}\int_X w\o^n=0,$$
where in the last identity we have used \eqref{eq: Hessian_F_u}. Consequently $s \to d \mathcal F_s(u^s_0)(w)=O(s^2)$. Since $d \mathcal F_s(u^s_0)(w)= \frac{1}{V}\int_X w f_s \o^n$ for some smooth curve $s \to f_s$, we can conclude that $f_s = O(s^2)$ and we have 
\begin{equation}\label{eq: dF_s_est}
\big| d\mathcal F_s(u^s_0)(w) \big|\leq C s^2 \sup_X |w| \textup{ for all }w \in C(X,\Bbb R).
\end{equation}
A similar estimate holds for $\big| d\mathcal F_s(u^s_1)(w) \big|$ as well.

Let $[0,1] \ni t \to u^s_t \in \mathcal H_\o^{1,\bar 1}$ be the $C^{1,\bar 1}$--geodesic connecting $u_0^s$ and $u_1^s$. By convexity of $t \to \mathcal F(u^s_t)$ and $t \to J_\o(u^s_t)$ it follows that 
\begin{flalign*}
0 \leq s \Big( \frac{d}{dt}\Big|_{t = 1} J_\o(u^s_t)- \frac{d}{dt}\Big|_{t = 0} J_\o(u^s_t)\Big)
&\leq  \frac{d}{dt}\Big|_{t = 1} \mathcal F_s(u^s_t)- \frac{d}{dt}\Big|_{t = 0} \mathcal F_s(u^s_t) \\
&=d \mathcal F_s(u^s_1)(\dot u^s_1) -  d \mathcal F_s(u^s_0)(\dot u^s_0) \leq C s^2,
\end{flalign*}
where the last inequality is a consequence of \eqref{eq: dF_s_est}. Taking the limit $s \searrow 0$ in the above estimate, (by convexity) we obtain that $t \to J_\o(u^0_t)$ is affine. By the lemma below, this implies that $u = u^0_0=u^0_1 =v$, finishing the proof. 
\end{proof}

\begin{lemma}
Suppose $u_0,u_1 \in \mathcal H_\o \cap I^{-1}(0)$ and $t \to u_t$ is the $C^{1,\bar 1}$--geodesic connecting $u_0,u_1$. If $t \to J_\o(u_t)$ is affine then $u_0 = u_1$.
\end{lemma}

\begin{proof} We know that $t \to I(u_t)$ is affine (Proposition \ref{prop: I_linrear}). Using this, our assumption implies that $t \to \int_X u_t \o^n$ is linear as well, hence $\frac{d}{d_t}\big|_{t =0} \int_X u_t \o^n=\frac{d}{d_t}\big|_{t =1} \int_X u_t \o^n$. This implies that $\int_X (\dot u_1 - \dot u_0)\o^n =0$. By convexity in the  $t$ variable we have $\dot u_1 \geq \dot u_0$, so we conclude that $\dot u_0 = \dot u_1$, hence $t \to u_t$ is affine, i.e., 
\begin{equation}\label{eq: dot_u_t_special}
\dot u_0 = \dot u_1 = u_1 - u_0.
\end{equation}
Since $t \to I(u_t)$ is affine too, we have $\frac{d}{dt}\big|_{t =0} I(u_t)=\frac{d}{dt}\big|_{t =1} I(u_t)$, and by   \eqref{eq: dIdt_C11geod}  and \eqref{eq: dot_u_t_special} we get that $\int_X (u_1 - u_0)\o_{u_0}^n = \int_X (u_1 - u_0)\o_{u_0}^n$. Subtracting the right hand side from the left and integrating by parts we obtain
$$\sum_{j=0}^{n-1}\int_X i \partial (u_1 - u_0) \wedge \dbar (u_1 - u_0) \wedge \o_{u_0}^{j} \wedge \o_{u_1}^{n-1-j}=0.$$ 
As all terms in the above sum are nonnegative, we get that $\int_X i \partial (u_1 - u_0) \wedge \dbar (u_1 - u_0) \wedge \o_{u_0}^{n-1}=0$. Consequently, $\langle \nabla^{\o_{u_0}} (u_1 - u_0),\nabla^{\o_{u_0}} (u_1 - u_0) \rangle_{\o_{u_0}}=0$, hence $u_1= u_0 + c$. As $I(u_0)=I(u_1)=0$, we have in fact $u_0 = u_1$.
\end{proof}

\section{Regularity of weak minimizers of the Ding functional}

In this short section we will show that minimizers of the extended $\mathcal F$ functional are actually smooth, with this verifying another important condition in the existence/properness principle described earlier (see Theorem \ref{thm: ExistencePrinc}):

\begin{theorem}If $u \in \mathcal E_1(X,\o)$ minimizes $\mathcal F: \mathcal E_1(X,\o) \to \Bbb R$ then $u$ is a smooth K\"ahler--Einstein potential. \label{thm: F_min_regularity} 
\end{theorem}

In case $u \in \mathcal H_\o$ minimizes $\mathcal F$, then  after computing the first order variation of $t \to \mathcal F(u + tv)$ for all $v \in C^\infty(X)$, Lemma \ref{lem: F_func_differential} allows to conclude that:\begin{equation}\label{eq: KE_weak_eq}
\o_u^n= \frac{V}{\int_X e^{- u + f_0}\o^n} e^{-u + f_0}\o^n,
\end{equation}
hence $u$ is indeed a KE potential. In case $u \in \mathcal E_1(X,\o)$ we can't even guarantee that $t \to u +tv \in \mathcal E_1(X,\o)$ for small $t$. Getting around this obstacle will represent one of main technical ingredients in the proof of Theorem \ref{thm: F_min_regularity}. 

Eventually we will be able to show that \eqref{eq: KE_weak_eq} holds for minimizers from $\mathcal E_1(X,\o)$ as well. 
Notice that by Zeriahi's version of Skoda's theorem (Corollary \ref{cor: Skodacor}) the right hand side of this equation does indeed makes sense for potentials of $\mathcal E_1(X,\o)$. The proof is completed by appealing to work of Kolodziej and Tosatti--Sz\'ekelyhidi on the apriori regularity theory of such equations:

\begin{theorem}\label{thm: SzToregularity} If $u \in \mathcal E_1(X,\o)$ solves \eqref{eq: KE_weak_eq} then $u \in \mathcal H_\o$, i.e., $u$ is a smooth KE potential.
\end{theorem}

\begin{proof}[Sketch of proof.] As $u \in \mathcal E_1(X,\o)$ Corollary \ref{cor: Skodacor} implies that $e^{-u + f_0} \in L^p(\o^n)$ for all $p > 1$. By Kolodziej's theorem \cite{k0}, since $u$ solves  \eqref{eq: KE_weak_eq}, we obtain that $u$ is bounded (for a full proof of this fact see \cite[Theorem 14.1]{gzbook}). Using a result of Tosatti--Sz\'ekelyhidi \cite[Theorem 1.1]{szto} we obtain that $u$ is actually a smooth KE potential (see also \cite[Theorem 14.1]{gzbook}).
\end{proof}

\begin{proof}[Proof of Theorem \ref{thm: F_min_regularity}] Let $v \in C^\infty(X)$. Recall from \eqref{eq: P_env_def} that $P(u+tv) \in \textup{PSH}(X,\o),$ $t \in \Bbb R$ is defined as follows:
$$P(u+tv)= \sup\{v \in \textup{PSH}(X,\o) \textup{ s.t. } v \leq u + tv \}.$$
Since $u - t \sup_X |v| \leq P(u + tv)$, Corollary \ref{cor: monotonicity_E_chi} implies that $P(u+tv) \in \mathcal E_1(X,\o)$, and this allows to introduce the function
$$g(t)= - I(P(u + tv)) - \log \int_X e^{-u -tv+ f_0}\o^n + \log (V).$$
We claim that $g(t)$ is differentiable at $t=0$ and the following formula holds:
\begin{equation}\label{eq: Dg(t)}
\frac{d}{dt}\Big|_{t=0} g(t) = -\frac{1}{V}\int_X v\bigg(\o_u^n - \frac{V}{\int_X e^{-u + f_0}\o^n}e^{-u + f_0}\o^n\bigg).
\end{equation}
From the proposition below it follows that $\frac{d}{dt}\big|_{t=0} I(P(u+tv))= \frac{1}{V}\int_X v \o_u^n.$ Consequently, to prove \eqref{eq: Dg(t)} it suffices to show that
\begin{equation}\label{eq: interm_dexp(t)}
\frac{d}{dt}\Big|_{t=0} \log \int_X e^{-u -tv + f_0}\o^n=\frac{1}{\int_X e^{-u + f_0}\o^n}e^{-u + f_0}\o^n.
\end{equation}
Using the elementary inequality $|e^{x}-e^y| \leq |x-y|(e^x + e^y)$ we can write that 
$$\Big|\frac{e^{-u -tv + f_0} - e^{-u -lv + f_0}}{t - l}\Big| \leq |v|( e^{-u - tv + f_0} + e^{-u -lv + f_0}) \leq C e^{-u}, \ \ \ l,t \in (-1,1).$$
Corollary \ref{cor: Skodacor} implies that the right most quantity in this inequality is integrable, hence we can conclude \eqref{eq: interm_dexp(t)} using the dominated convergence theorem.

Since $P(u+tv) \leq u +tv$, we notice that $g(0)= \mathcal F(u) \leq \mathcal F(P(u + tv)) \leq g(t), \  t \in \Bbb R.$ This implies that $\frac{d}{dt}\big|_{t=0} g(t)=0$, and by 
\eqref{eq: Dg(t)} we can conclude that \eqref{eq: KE_weak_eq} holds for $u \in \mathcal E_1(X,\o)$. Now using Theorem \ref{thm: SzToregularity} we conclude that $u$ is smooth, finishing the argument.
\end{proof}

\begin{proposition}\textup{\cite[Lemma 3.10]{bebo}} Suppose $u \in \mathcal E_1(X,\o)$ and $v \in C^\infty(X)$. Then $P(u + tv) \in \mathcal E_1(X,\o)$ for all $t \in \Bbb R$, and $t \to I(P(u+tv))$ is differentiable at $t=0$. More precisely,
$$\frac{d}{dt}\Big|_{t=0} I(P(u+tv))= \frac{1}{V}\int_X v \o_u^n.$$ 
\end{proposition}
In our approach we will follow closely the simplified argument proposed by Lu and Nguyen \cite{ln}. 
\begin{proof} We want to show that
$$\frac{I(P(u + tv)) - I(u)}{t} \to \frac{1}{V} \int_X v \o_u^n.$$
After changing $v$ to $-v$, it suffices to consider $t > 0$ in the above limit.
Using \eqref{eq: I_energy_diff} we can write 
$$\frac{I(P(u + tv)) - I(u)}{t}\leq  \frac{1}{V}\int_X \frac{P(u + tv)-u}{t} \o_u^n \leq \frac{1}{V}\int_X v \o_u^n, \ t >0.$$
By the same inequality we also have that 
$$\frac{1}{V}\int_X \frac{P(u + tv)-u}{t} \o_{P(u + tv)}^n \leq \frac{I(P(u + tv)) - I(u)}{t}.$$
By the lemma below $\o_{P(u + tv)}^n$ is concentrated on the coincidence set $\{P(u + tv)=u+tv \}$,  thus we have 
$$\frac{1}{V}\int_X v \o_{P(u + tv)}^n \leq \frac{I(P(u + tv)) - I(u)}{t}.$$
We also have  $|P(u + tv) - u| \leq t \sup_X |v|$, and 
since the complex Monge--Amp\`ere operator is continuous under uniform convergence we conclude that
$$\frac{1}{V}\int_X v \o_{u}^n \leq \liminf_{t \to 0}\frac{I(P(u + tv)) - I(u)}{t},$$
finishing the argument.
\end{proof}

Finally, we provide the lemma promised in the proof of the above proposition:
\begin{lemma} Suppose $u \in \mathcal E_1(X,\o)$ and $v \in C^\infty(X)$. Then $\int_{\{u + tv > P(u + tv)\}} \o_{P(u+tv)}^n=0.$
\end{lemma}

\begin{proof} Using Theorem \ref{thm: BK_approx} we choose $u_k  \in \mathcal H_\o$ such that $u_k \searrow u$. By a  classical Perron type argument it follows that $\int_{\{u_k + tv > P(u_k + tv)\}} \o_{P(u_k+tv)}^n=0$ (see \cite[Corollary 9.2]{BT1}). This is equivalent to $\int_X (u_k + tv - P(u_k + t v))\o_{P(u_k + tv)}^n =0.$ 
As $u_k + t \sup_X |v| \geq P(u_k +tv)$,  Proposition \ref{prop: MA_cont} allows to take the limit $k \to \infty$ and conclude that 
$$\int_X (u + tv - P(u + t v))\o_{P(u + tv)}^n =0,$$
which is equivalent to $\int_{\{u + tv > P(u + tv)\}} \o_{P(u+tv)}^n=0.$ 
\end{proof}

\section{Properness of the K--energy and existence of KE metrics}

Given $u \in \mathcal H_\o$, the average of the scalar curvature of the metric $\o_u$ is independent of $u$, as by \eqref{prelricdifference} and integration by parts yields
\begin{flalign}\label{eq: av_Scal_curv}
\bar S = \frac{1}{V}\int_X S_{\o_u} \o_u^n=\frac{n}{V} \int_X \textup{Ric } \o_{u} \wedge \o_u^{n-1}= \frac{n}{V} \int_X \textup{Ric } \o \wedge \o^{n-1}.
\end{flalign}
Next we introduce Mabuchi's K--energy functional $\mathcal K: \mathcal H_\o \to \Bbb R$ \cite{m}:
\begin{equation}\label{eq: Ken_def}
\mathcal K(u)=\frac{1}{V} \int_X [\log\Big(\frac{\o_u^n}{\o^n}\Big)\o_u^n - {u}\sum_{j=0}^{n-1}\text{Ric }\o\wedge\o_u^j \wedge \o^{n-j-1}] + \bar S I(u).
\end{equation}
The reason behind this specific definition is the following variational  formula, which shows that critical points of the K--energy are exactly the \emph{constant scalar curvature K\"ahler} (csck) metrics:
\begin{proposition} \label{prop: prelKenergyvariation} For a smooth curve $(0,1) \ni t \to u_t \in \mathcal H_\o$ we have
$$\frac{d}{dt}\mathcal K(u_t)= \frac{1}{V} \int_X \dot u_t (\bar S - S_{\o_{u_t}})\o_{u_t}^n.$$
\end{proposition}

\begin{proof} By straightforward calculations we arrive at the identities:
$$\frac{d}{dt}\Big[\log\Big(\frac{\o_{u_t}^n}{\o^n}\Big)\o_{u_t}^n \Big] = \frac{1}{2}\Delta^{\o_{u_t}}\dot u_t\o_{u_t}^n + n\log\Big(\frac{\o_{u_t}^n}{\o^n}\Big) i\partial \bar \partial \dot u_t \wedge \o_{u_t}^{n-1},$$
$$\frac{d}{dt}\Big[{u_t}\sum_{j=0}^{n-1}\textup{Ric } \o\wedge\o_{u_t}^j \wedge \o^{n-j-1}\Big]= \sum_{j=0}^{n-1}\big({\dot u_t} \o_{u_t}^j +  j u_t \cdot i \ddbar \dot u_t \wedge \o_{u_t}^{j-1})\wedge \textup{Ric } \o \wedge \o^{n-j-1}.$$
Consequently, integration by parts and \eqref{prelricdifference} gives:
$$\frac{d}{dt}\int_X \log\Big(\frac{\o_{u_t}^n}{\o^n}\Big)\o_{u_t}^n = n \int_X \dot u_t \Big(i\partial\bar \partial \log\frac{\o_{u_t}^n}{\o^n} \Big)  \wedge \o_{u_t}^{n-1}=n \int_X \dot u_t \Big(\textup{Ric }\o -  \textup{Ric }\o_u\Big)  \wedge \o_{u_t}^{n-1},$$
$$\frac{d}{dt}\int_X{u_t}\sum_{j=0}^{n-1}\textup{Ric }\o\wedge\o_{u_t}^j \wedge \o^{n-j-1}= n\int_X\dot u_t \textup{Ric }\o \wedge\o_{u_t}^{n-1}.$$
The desired formula now follows after differentiating $t \to \mathcal K(u_t)$, and using the last two identities together with  Lemma \ref{lem: I_differential}.
\end{proof}

Trivially, KE metrics are csck. By the next result, in case $c_1(X)$ is a multiple of $[\o]$, the reverse is also true:

\begin{lemma} Suppose $c_1(X)=\lambda [\o]$. Then $u \in \mathcal H_\o$ is a csck potential if and only if it is a KE potential.
\end{lemma}

\begin{proof} Suppose $u$ is a csck potential, i.e., $S_{\o_u}=\bar S$. Since $c_1(X)=\lambda [\o]$, by definition of the Ricci potential (see \eqref{eq: Ricci_pot_def}) and invariance of $\bar S$ (see \eqref{eq: av_Scal_curv}) we have 
$$\bar{S }=S_{\o_u} = n \lambda  + \frac{1}{2}\Delta^{\o_u} f_u=\overline{S }+ \frac{1}{2}\Delta^{\o_u} f_u,$$
hence  $\Delta^{\o_u} f_u =0$. The condition $\int_X e^{f_u}\o^n=V$ gives $f_u=0$, i.e., $\o_u$ is a KE metric.
\end{proof}

By this last lemma, in case $(X,\o)$ is Fano, the critical points of $\mathcal K$ are exactly the KE potentials, hence the K--energy plays a role similar to the $\mathcal F$ functional. Developing this analogy further, our main result in this section parallels Theorem \ref{thm: F_func_properness}, giving another characterization of existence of KE metrics, confirming a related conjecture of Tian (see \cite[Conjecture 7.12]{t1},\cite[p. 127]{t3}):

\begin{theorem}\textup{({\cite[Theorem 2.4]{dr2}})} 
\label{thm: K-energy_properness}
Suppose $(X,\o)$ is Fano with $c_1(X)=[\o]$, and set $G:=\textup{Aut}(X,J)_0$. The following are equivalent:\vspace{0.1cm}\\
\noindent (i) there exists a KE metric in $\mathcal H$.

\noindent (ii) 
For some $C,D >0$ the following holds:
\begin{equation}\label{eq: K_d_1_properness}
\mathcal K(u) \geq C d_{1,G}(G0,Gu) - D, \ \  u\in\mathcal H_\omega \cap I^{-1}(0).
\end{equation}

\noindent (iii) 
For some $C,D>0$ the following holds:
\begin{equation}\label{eq: K_J_properness}
\mathcal K(u) \geq C J_G(Gu) - D, \ \ u\in\mathcal H_\o \cap I^{-1}(0).
\end{equation}
\end{theorem}

 For the resolution of other closely related conjectures we refer to \cite{dr2}. It is possible to give a proof for this theorem by verifying the conditions of the existence/properness principle (Theorem \ref{thm: ExistencePrinc}) directly, the same way as we did with Theorem \ref{thm: F_func_properness}. Instead of doing this, we will rely on the special relationship between the K--energy and the $\mathcal F$ functional (see Proposition \ref{prop: K_F_relation} below). 

First we have to show that $\mathcal K$ is $d_1$--lsc, and \eqref{eq: Ken_def} gives the $d_1$--lsc extension of this functional to $\mathcal E_1(X,\o)$. This will be done in a series of lemmas and propositions:

\begin{lemma}\textup{\cite[Lemma 5.23]{dr2}}
\label{lem: I_twist_ExtLemma}
Suppose $\alpha$ is a smooth closed $(1,1)$-form on $X$. The functional $I_\alpha: \mathcal H_\o \to \Bbb R$ given by
$$ I_\alpha(u)=  \sum_{j=0}^{n-1}\int_X u \alpha \wedge \o_u \wedge \o^{n-1-j}
$$
is $d_1$--continuous, and extends uniquely to a $d_1$--continuous functional $I_\alpha:\mathcal E_1(X,\o)\to \Bbb R$. Additionally, $I_\alpha$ is bounded on $d_1$--bounded subsets of $\mathcal E_1(X,\o)$.
\end{lemma}

\begin{proof}Let $u_k \in \mathcal H_\o$ and $u \in \mathcal E_1(X,\o)$ 
be such that $d_1(u_k,u) \to 0$.  An argument similar to that yielding \eqref{eq: I_energy_diff} shows that
$$
I_\alpha(u_l) -I_\alpha(u_k)= \sum_{j=0}^{n-1} \int_X (u_l-u_k) \alpha \wedge \o_{u_l}^{j} \wedge \o_{u_k}^{n-1-j}.
$$
For some $D>0$ we have $-D\o \leq \alpha \leq D\o$. Thus,
$\o_{(u_l+u_k)/{4}}=\o/2+\o_{u_l}/4+\o_{u_k}/4$ and for some $C>0$ we can write 
\begin{equation}\label{eq: eqintegralest}
\big|I_\alpha(u_l) - I_\alpha(u_k)\big| \leq C \int_X |u_l-u_k| 
\o_{(u_l+u_k)/{4}}^n.
\end{equation}
Using Lemma \ref{lem: halwayest} and the triangle inequality, $d_1(0,(u_l + u_k)/4)$ is bounded. By Corollary \ref{cor: int_d1_est}, we obtain that $\{I_\alpha(u_k)\}_k$ is a Cauchy sequence, showing that $I_\alpha$ is $d_1$--continuous and it that extends $d_1$--continuously to $\mathcal E_1(X,\o)$.

To argue $d_1$--boundedness of $I_\alpha$, we turn again to \eqref{eq: eqintegralest} (with $u_k=0$). By this estimate and Corollary \ref{cor: int_d1_est}, it is enough to show that if $d_1(0,u)$ is bounded then so is $d_1(0,u/4)$. This is a consequence of Lemma \ref{lem: halwayest}.
\end{proof}

As we will see shortly, the entropy functional $u \to \int_X \log\big(\frac{\o_u^n}{\o^n}\big)\o_u^n$ is only $d_1$--lsc. In fact, this functional is already lsc with respect to weak convergence of measures, as it follows from our discussion below and the next proposition.

Suppose $\nu,\mu$ are Borel probability measures on $X$. If $\nu$ is absolutely continuous with respect to $\mu$ then the entropy of $\nu$ with respect to $\mu$ is $\textup{Ent}(\mu,\nu)= \int_X \log \big(\frac{\nu}{\mu}\big)\frac{\nu}{\mu} \mu$, otherwise $\textup{Ent}(\mu,\nu)= \infty.$

The next well known result follows from the classical Jensen  inequality:
\begin{lemma} \label{lem: Ent_ineq_class} Suppose $\nu,\mu$ are Borel probability measures on $X$. Then $\textup{Ent}(\mu,\nu) \geq 0$ and equality holds if and only if $\nu = \mu$.
\end{lemma}
\begin{proof} $\textup{Ent}(\mu,\nu) \geq 0$ follows from an application of Jensen's inequality to the convex weight $\rho(x):=x \log x$. As $\rho$ is strictly convex on $\Bbb R^+$, the proof of Jensen's inequality implies that $\textup{Ent}(\mu,\nu) = 0$ if and only if $\frac{\mu}{\nu}=1$  (see \cite[Chapter 3, Theorem 3.3]{rud}).  
\end{proof}

Let us recall the following classical formula for the entropy of two measures:

\begin{proposition}\label{prop: ent_dual_form} Suppose $\mu,\nu$ are probability Borel measures on $X$. Then the following holds: 
\begin{equation}\label{eq: Ent_formula}
\textup{Ent}(\mu,\nu) = \sup_{f \in B(X)}\Big(\int_X f \nu - \log \int_X e^{f}\mu \Big),
\end{equation}
where $B(X)$ is the set of bounded Borel measurable functions on $X$.
\end{proposition}
\begin{proof} In case $\nu$ is not absolutely continuous with respect to $\mu$ then there exists a Borel set $M \subset X$ with $\mu(M)=0$ but $\nu(M) >0$. Then we trivially have that $\int_X c\mathbbm{1}_M \nu - \log \int_X e^{c\mathbbm{1}_M}\mu =\int_X c\mathbbm{1}_M \nu - \log \int_X e^0 \mu= c \nu(M)$ for any $c>0$ and consequently,
$$\sup_{f \in B(X)}\Big(\int_X f \nu - \log \int_X e^{f}\mu \Big) = \infty,$$ which is equal to $\textup{Ent}(\mu,\nu)$ by definition.

We assume now that $\nu$ is absolutely continuous with respect to $\mu$, i.e. $\nu = g \mu$ for some non--negative Borel measurable function $g$. To conclude \eqref{eq: Ent_formula} we need to show that
\begin{equation}\label{eq: Ent_formula_alternative}
\int_X g \log g \mu = \sup_{f \in B(X)}\Big(\int_X f g \mu - \log \int_X e^{f}\mu \Big).
\end{equation}
By choosing $f_k = \log g_k:=\log (\min(\max(g,1/k),k)), \ k \in \Bbb N$. We get that the right hand side of \eqref{eq: Ent_formula_alternative} is greater then $\int_X g \log g_k \mu - \log \int_X g_k \mu$. Letting $k \to \infty$, we conclude that the right hand side of \eqref{eq: Ent_formula_alternative} is greater then the left hand side.

For the other direction, we need to argue that  $\int_X g \log g \mu  \geq \int_X f g \mu - \log \int_X e^{f}\mu$ for any $f \in B(X)$. For this it is enough to invoke Jensen's inequality:
$$\log \int_X e^f \mu \geq \log \int_{\{g >0\}} \frac{e^f}{g}g \mu \geq \int_{\{g >0\}} ({f} - \log g)g \mu=\int_{X} ({f} - \log g)g \mu.$$
\end{proof}
As the supremum of continuous functionals is lsc, it follows that $\nu \to \textup{Ent}(\mu,\nu)$ is lsc with respect to weak convergence of Borel measures.

Theorem \ref{thm: d_1-convergence}(ii) implies that for any $u_k,u \in \mathcal E_1(X,\o)$ we have that $d_1( u_k,u) \to 0$ implies $\o_{u_k}^n \to \o_u^n$ weakly. 
We arrive at the following important corollary:

\begin{corollary} \label{cor: Ent_d1_lsc} The functional $\mathcal E_1(X,\o) \ni u \to \textup{Ent}\big(\frac{1}{V}\o^n,\frac{1}{V}\o_u^n\big) \in \Bbb R$ is $d_1$--lsc.
\end{corollary}
Comparing with \eqref{eq: Ken_def}, we observe that for $u \in \mathcal H_\o$ we actually have 
\begin{equation}\label{eq: K-en_alt_def}
\mathcal K(u) = \textup{Ent}\Big(\frac{1}{V}\o^n,\frac{1}{V}\o_u^n\Big) - \frac{1}{V} I_{\textup{Ric }\o}(u) + \bar S I(u).
\end{equation} 
This observation together with Proposition \ref{prop: I_finiteness_and_cont}, Lemma \ref{lem: I_twist_ExtLemma} and Corollary \ref{cor: Ent_d1_lsc} allows to conclude that $\mathcal K$ is $d_1$--lsc on $\mathcal H_\o$ and it extends to $\mathcal E_1(X,\o)$ in a $d_1$--lsc manner, using the formula of \eqref{eq: K-en_alt_def}.

Lastly, before we prove Theorem \ref{thm: K-energy_properness}, we provide a precise inequality between the K--energy and $\mathcal F$ functional, and we point out the relationship between the minimizers of these functionals:

\begin{proposition}[\cite{brm0}] \label{prop: K_F_relation}Suppose $(X,\o)$ is a Fano manifold with $c_1(X)=[\o]$. For any $u \in \mathcal E_1(X,\o)$ we have
\begin{equation}\label{eq: K_F_ineq}
\mathcal F(u) \leq \mathcal K(u) - \frac{1}{V} \int_X f_0 \o^n.
\end{equation}
Moreover, for $u \in \mathcal E_1(X,\o)$ the following are equivalent:\\
(i) $\mathcal F(u) = \mathcal K(u)- \frac{1}{V} \int_X f_0 \o^n$.\\
(ii) $u$ minimizes $\mathcal F$.\\
(iii) $u$ minimizes $\mathcal K$.\\
(iv) $u$ is  a smooth KE potential.
\end{proposition}

Consequently, the minimizers of $\mathcal K$ and $\mathcal F$ on $\mathcal E_1(X,\o)$ are the same and coincide with the set of smooth KE potentials. 

\begin{proof} Let $u \in \mathcal H_\o$. As both $\mathcal K$ and $\mathcal F$ are constant invariant, we can assume that $\int_X e^{-u + f_0}\o^n = V$ and note the following identity:
$$\textup{Ent}\Big(\frac{1}{V}e^{-u + f_0}\o^n, \frac{1}{V} \o_u\Big) = \textup{Ent}\Big(\frac{1}{V}e^{f_0}\o^n, \frac{1}{V} \o_u\Big) + \frac{1}{V}\int_X u \o^n_u.$$
By the formula of the next lemma, we can write that
$$\mathcal K(u) - \frac{1}{V} \int_X f_0 \o^n= \textup{Ent}\Big(\frac{1}{V}e^{-u + f_0}\o^n, \frac{1}{V} \o_u\Big) - I(u)= \textup{Ent}\Big(\frac{1}{V}e^{-u + f_0}\o^n, \frac{1}{V} \o_u\Big) + \mathcal F(u).$$
By Lemma \ref{lem: Ent_ineq_class} we have $\textup{Ent}\big(\frac{1}{V}e^{-u + f_0}\o^n, \frac{1}{V} \o_u\big) \geq 0$, hence \eqref{eq: K_F_ineq} follows. 

Moreover, by this same lemma, equality holds in \eqref{eq: K_F_ineq} if and only if $\o_u = e^{-u + f_0}\o^n$, which is equivalent with $u$ being a smooth KE potential (see Theorem \ref{thm: SzToregularity}). By Theorem \ref{thm: F_min_regularity} it immediately follows that (i), (ii) and (iv) are equivalent.

If (iv) holds, then $u$ minimizes $\mathcal F$,  and $\mathcal F(u)=\mathcal K(u) - \frac{1}{V}\int_X f_0 \o^n$, by (i). Consequently $u$ minimizes $\mathcal K$ as well, hence (iii) holds.

Suppose (iii) holds. By Lemma \ref{lem: Ricci_it}(ii) proved below, there exists $v \in \textup{PSH}(X,\omega) \cap L^\infty$ such that
$$\mathcal K(v) - \frac{1}{V}\int_X f_0 \o^n \leq \mathcal F(u) \leq \mathcal K(u) - \frac{1}{V}\int_X f_0 \o^n.$$
As $u$ minimizes $\mathcal K$, it follows that the inequalities above are actually equalities, hence (i) holds. 
\end{proof}

As pointed out in the above argument, we need a special expression for the K-energy:

\begin{lemma}\label{lem: K_en_alt_formula} Suppose $(X,\o)$ is a Fano manifold with $c_1(X)=[\o]$. For $u \in \mathcal H_\o$ we have
$$\mathcal K(u) = \textup{Ent}\Big(\frac{1}{V} e^{f_0} \o^n, \frac{1}{V}\o_u^n \Big) - I(u)  + \frac{1}{V} \int_X u \o_u^n + \frac{1}{V} \int_X f_0 \o^n.$$
\end{lemma}
\begin{proof} We start by deriving an alternative formula for $I_{\Ric \o}$:
\begin{flalign*}
I_{\Ric \o}(u)& =  \sum_{j=0}^{n-1}\int_X u \Ric \o \wedge \o_u^j \wedge \o^{n-1-j}=\sum_{j=0}^{n-1}\int_X u (\o + i \ddbar f_0) \wedge \o_u \wedge \o^{n-1-j}\\
& =  \sum_{j=0}^{n-1}\int_X u \o_u^j \wedge \o^{n-j} +\sum_{j=0}^{n-1}\int_X f_0 i \ddbar u \wedge \o_u \wedge \o^{n-1-j}\\
& =  \sum_{j=0}^{n-1}\int_X u \o_u \wedge \o^{n-j} +\sum_{j=0}^{n-1}\int_X f_0 \o_u^{j+1} \wedge \o^{n-1-j} - \sum_{j=0}^{n-1}\int_X f_0 \o_u^{j} \wedge \o^{n-j}\\
& =  \sum_{j=0}^{n-1}\int_X u \o_u \wedge \o^{n-j} + \int_X f_0 \o_u^{n} -\int_X f_0 \o^{n}.
\end{flalign*}
Since $c_1(X)=[\o]$, we have $\bar S = n$ (see \eqref{eq: av_Scal_curv}), so by the above we conclude that 
\begin{flalign*}
-\frac{1}{V}I_{\Ric \o}(u) + \bar S I(u) &=-\frac{1}{V}I_{\Ric \o}(u) + n I(u)\\
&=-I(u) + \frac{1}{V}\Big(\int_X u \o_u^n - \int_X f_0 \o_u^{n} +\int_X f_0 \o^{n}\Big).
\end{flalign*}
We note that $\textup{Ent}\big(\frac{1}{V} \o^n, \frac{1}{V}\o_u^n \big)=\textup{Ent}\big(\frac{1}{V} e^{f_0} \o^n, \frac{1}{V} \o_u^n \big) + \frac{1}{V}\int_X f_0\o_u^n$. Adding this identity to the above, and comparing with \eqref{eq: K-en_alt_def} finishes the proof.
\end{proof}

In the next lemma, we will make us of the inverse Ricci operator, introduced in \cite{Ru08} in connection with the so-called Ricci iteration.
Given a potential $u \in \mathcal E_1(X,\o)$, Corollary \ref{cor: Skodacor} and Kolodziej's estimate \cite{k0,K1} give another potential $\Ric^{-1}(u) \in \textup{PSH}(X,\o) \cap L^\infty$, unique up to a constant, such that
$$\o_{\Ric^{-1}(u)}^n = \frac{V}{\int_X e^{-u + f_0}\o^n}e^{-u + f_0}\o^n.$$
In case $u \in \mathcal H_\o$, we notice that $\Ric \o_{\Ric^{-1}u} = \o_u,$ motivating the terminology.

By the next lemma, the inverse Ricci operator decreases the $\mathcal F$ functional and sheds further light on the intimate relationship between $\mathcal K$ and $\mathcal F$: 

\begin{lemma}[\cite{Ru08}] \label{lem: Ricci_it} Suppose $u \in \mathcal E_1(X,\o)$. Then the following hold:\\
(i) $\mathcal F(\Ric ^{-1}(u)) \leq \mathcal F(u).$  \\
(ii)$\ \mathcal K(\Ric ^{-1}(u)) - \frac{1}{V}\int_X f_0 \o^n \leq \mathcal F(u).$\\
(iii)$\ \mathcal K(\Ric ^{-1}(u))\leq \mathcal K(u).$
\end{lemma}

\begin{proof} First we argue (ii). We introduce $v:=\textup{Ric}^{-1}(u)$. As both $\mathcal F$ and $\mathcal K$ are constant invariant, we can assume that $\int_X e^{-u + f_0}\o^n = V$. By the previous lemma, notice that we have
\begin{flalign*}
\mathcal K(v)- \frac{1}{V}\int_X {f_0}\o^n &=\textup{Ent}\Big(\frac{1}{V}e^{f_0}\o^n, \frac{1}{V}\o_v^n\Big) -I(v) + \frac{1}{V}\int_X v\o_v^n \\
&=\textup{Ent}\Big(\frac{1}{V}e^{f_0}\o^n, \frac{1}{V}e^{-u + f_0}\o^n\Big) -I(v) + \frac{1}{V}\int_X v\o_v^n\\
&= -I(v) + \frac{1}{V}\int_X (v-u)\o_v^n \leq -I(u)= \mathcal F(u),
\end{flalign*}
where in the penultimate estimate we have used \eqref{eq: I_energy_diff_est_E1}. 

Lastly, (i) and (iii) follow from (ii) and \eqref{eq: K_F_ineq}.
\end{proof}

We now prove the main compactness theorem of the space $(\mathcal E_1(X,\o),d_1)$:

\begin{theorem}\label{thm: E_1_compactness} Let $u_j\in \mathcal E_1(X,\o)$ be a $d_1$--bounded sequence for which $\textup{Ent}\big(\frac{1}{V}\o^n,\frac{1}{V}\o_{u_j}^n\big)$ is bounded. Then $\{u_j \}_j$ contains a $d_1$--convergent subsequence.
\end{theorem}

\begin{proof} As a consequence of  Corollary \ref{cor: Skodacor}, for any $p>0$ there exists $C(p)>0$ such that that  $\int_X e^{-pu_j}\o^n \leq C$. Since $|\sup_X u_j|$ is bounded, we get that 
\begin{equation}\label{eq: exp_int_est}\int_X e^{|pu_j|}\o^n \leq C.
\end{equation}
Consider $\phi,\psi: \Bbb R \to \Bbb R^+$ given by 
$$
\phi(|t|) = \frac{(|t|+1)\log(|t|+1)-|t|}{\log 2} \ \textup{ and } \ \psi(t)=\frac{e^{|t|} - |t| -1}{e-1}.
$$
An elementary calculation verifies that both of these functions are normalized Young weights and $\phi^* = \psi$ (in the sense of \eqref{eq: Legendre_trans_def}).

Since $\textup{Ent}\big(\frac{1}{V}\o^n,\frac{1}{V}\o_{u_j}^n\big)$ is bounded, so is $\int_X \phi\big(\omega_{u_j}^n/\omega^n \big) \omega^n$. As $\phi$ is convex and $\phi(0)=0$, for some $D \in (0,1)$ we get that 
\begin{equation}\label{eq: Ent_less_one}
\int_X \phi\bigg(D \frac{\omega_{u_j}^n}{\omega^n} \bigg) \frac{1}{V}\omega^n \leq 1.
\end{equation}

From $d_1$-boundedness we have that $|\sup_X u_j|$  and $I(u_j)$ are bounded. By Lemma \ref{lem: L_1_closedness}, after possibly taking a subsequence, we can find $u \in \mathcal E_1(X,\o)$ such that $\int_X |u_j-u| \o^n  \to 0$.

Recalling the definition of Orlicz norms from \eqref{eq: OrliczNormDef}, we can use the H\"older inequality \eqref{eq: HolderIneq} to deduce that
\begin{equation}\label{eq: intintintest}
\frac{1}{V}\int_X |u_j - u|\o_{u_j}^n =\frac{1}{V}\int_X |u_j - u|\frac{\o_{u_j}^n}{\o^n}\o^n \leq \Big\|\frac{1}{D}(u_j - u)\Big\|_{\psi, \frac{1}{V}\o^n} \Big \| D  \frac{\o_{u_j}^n}{\o^n} \Big \|_{\phi, \frac{1}{V}\o^n}.
\end{equation}

From \eqref{eq: Ent_less_one} it follows that $ \| D {\o_{u_j}^n}/{\o^n} \|_{\phi, \frac{1}{V}\o^n} \leq 1$. Since $\psi(t) \leq t^2 e^{|t|}$, for any $r>0$ we can write
\begin{flalign*}
\int_X \psi\Big(\frac{r}{D}(u_j - u)\Big) \o^n &\leq \int_X \frac{r^2}{D^2}|u_j - u|^2e^{\frac{r}{D}|u_j - u|}\o^n\\
&\leq \frac{r^2}{D^2} \| (u_j - u)^2 \|_{L^3(\o^n)} \big\|e^{\frac{r}{D}|u|}\big\|_{L^3(\o^n)}\big\|e^{\frac{r}{D}|u_j|}\big\|_{L^3(\o^n)}.
\end{flalign*}
The last two terms on right hand side are bounded by  \eqref{eq: exp_int_est}. As $\int_X |u_j - u|\o^n \to 0$ it follows that $\int_X |u_j - u|^6\o^n \to 0$, hence $\lim_j \|D^{-1}(u_j - u)\|_{\psi, \frac{1}{V}\o^n} \leq \varepsilon$ for any $\varepsilon >0$.  

Using \eqref{eq: intintintest}, we conclude that $\int_X |u_j-u|\o_{u_j}^n \to 0$. As a result, \eqref{eq: I_energy_diff_est_E1}  gives that $\liminf_j I(u_j) \geq I(u)$. Together with Corollary \ref{cor: I_L_1_semicont} we obtain that $\lim_j I(u_j) = I(u)$. Since additionally $\int_X |u_j - u|\o^n \to 0$, Theorem \ref{thm: d_1-convergence}(i) implies that $d_1(u_j,u) \to 0$, finishing the proof.
\end{proof}

Lastly, we prove our main theorem:
\begin{proof}[Proof of Theorem \ref{thm: K-energy_properness}] The equivalence between (ii) and (iii) follows from Lemma \ref{lem: JGPropernessLemma}.

Suppose (i) holds. Then Proposition \ref{prop: K_F_relation} and Theorem \ref{thm: F_func_properness} implies  properness of $\mathcal K$, giving (ii). 

Now assume that (ii) holds. Then $\mathcal K$ is bounded from below and we can find $u_j \in \mathcal E_1(X,\o)$ that is $d_1$--bounded and $\lim_{j} \mathcal K(u_j) = \inf_{v \in \mathcal E_1(X,\o)} \mathcal K(v)$. 

In particular, $\{\mathcal K(u_j)\}_j$ is bounded. Recalling \eqref{eq: K-en_alt_def}, and the fact that both $I$ and $I_{\Ric \o}$ are bounded on $d_1$--bounded sets (see Proposition \ref{prop: I_finiteness_and_cont} and Lemma \ref{lem: I_twist_ExtLemma}), it follows that $\textup{Ent}\big(\frac{1}{V}\o^n,\frac{1}{V}\o_{u_j}^n\big)$ is bounded as well. 

By the previous compactness theorem, after possibly passing to a subsequence, we have $d_1(u_j,u) \to 0$ for some $u \in \mathcal E_1(X,\o)$. Since $\mathcal K$ is $d_1$--lsc, we immediately obtain that $\mathcal K(u)= \inf_{v \in \mathcal E_1(X,\o)} \mathcal K(v)$, i.e., $u$ minimizes $\mathcal K$.

By the equivalence between (iii) and (iv) in Theorem \ref{prop: K_F_relation}, $u$ is in fact a smooth KE potential.
\end{proof}

\paragraph{Brief historical remarks.}  Motivated by results in conformal geometry, the relationship between energy properness and existence of canonical metrics in K\"ahler geometry goes back to the work of Tian and collaborators in the nineties \cite{t0,t1}. Numerous conjectures were proposed during this time, a number of which where adressed in the case of Fano manifolds without vectorfields \cite{dt,t,tz}. For general Fano manifolds, all the remaining conjectures where addressed in \cite{dr2}, and we refer to this work for more details. 

The sharpest form of the energy properness condition was identified in \cite{pssw} and was later adopted in the literature, including in the present work. 

Regarding general K\"ahler manifolds, in \cite{dr2} the equivalence between energy properness and existence of consant scalar curvature (csck) metrics is linked to a regularity problem for fourth order PDE's. In case a csck metric exists, this regularity conjecture was confirmed in \cite{bdl2}, showing that on such manifolds the K-energy is indeed proper, partially generalizing Theorem \ref{thm: K-energy_properness}.

In this chapter we tackled problems related to energy properness directly via the existence/properness principle of \cite{dr2} (Theorem \ref{thm: ExistencePrinc}). The use of geodesic convexity in this context was initially proposed by X.X.  Chen, however he advocated for the use of the $L^2$ Mabuchi geometry instead \cite{che2}.

One of the advantages of our method (that uses pluripotential theory predominantly) over previous approaches in the literature is its adaptability to K\"ahler structures with mild singularities \cite{da4,DiG}. For generalizations in other directions, as well as a more thourough overview of the vast related literature we refer to \cite{dr2,r2}.

\appendix 
\chapter[Appendix]{}
\section{Basic formulas of K\"ahler geometry}

In this section we recall the most important facts about K\"ahler manifolds. Our minimalist approach follows \cite{bl1} closely, and we refer to \cite{De},\cite{we} and \cite{sze} for more exhaustive treatments.

Suppose $(X,J)$ is compact connected complex manifold with holomorphic structure $J$. A Hermitian structure $(X,h)$ is the complex analog of a Riemannian structure, i.e., a smooth choice of $J-$compatible Hermitian metrics on each fiber of $TX.$ In local complex coordinates $h$ can be expressed as
$$h = g_{j\bar k}d z_j \otimes d\bar z_k,$$
where we have used the Einstein summation. The real part of $h$ induces a Riemannian structure: 
$$\langle \cdot, \cdot \rangle = 2 \textup{Re }  h=g_{j\bar k}d z_j \otimes d\bar z_k + g_{k\bar j}d \bar z_j \otimes d z_k.$$ 
This metric is compatible with $J$ in the sense that $\langle\cdot,\cdot\rangle= \langle J(\cdot),J(\cdot) \rangle$. 

By $T_{\Bbb C}X = TX \otimes \Bbb C$ we denote the complexification of $TX$ and extend
$J$ and $\langle \cdot,\cdot \rangle$ to $T_{\Bbb C}X$ in a $\Bbb C$--linear way. In local coordinates $z_j = x_j + iy_j$ the vector fields $\partial/\partial x_j$ , $\partial /\partial y_j$ span $TX$ over $\Bbb R$. We also have $J \partial/\partial x_j=\partial/\partial y_j, \ J \partial/ \partial y_j=-\partial /\partial x_j$ and the eigenbase of $J$ composed of  the vector fields
$$\Big\{\partial_j = \frac{\partial}{\partial z_j} =\frac{1}{2}\Big( \frac{\partial}{\partial x_j}-i\frac{\partial}{\partial y_j}\Big)\Big\}_{\{j=1,\ldots,n\}}, \ \  \Big\{\partial_{\bar j} = \frac{\partial}{\partial \bar z_j} =\frac{1}{2}\Big( \frac{\partial}{\partial x_j}+i\frac{\partial}{\partial y_j}\Big)\Big\}_{\{j=1,\ldots,n\}}$$
that span $T^{(1,0)}_\Bbb C X$ and $T^{(0,1)}_\Bbb C X$ respectively over $\Bbb C$. We also have the identities $$J \partial_j = i\partial_j \ \textup{ and } \ J \partial_{\bar j} = -i\partial_{\bar j}.$$ 

Compared to other works, we chose to multiply by a factor of two in the definition of $\langle \cdot,\cdot \rangle$ so that the following formula holds:
$$h(\partial_j,\partial_{\bar k})=\langle \partial_j,\partial_{\bar k} \rangle=g_{j\bar k}.$$
Consequently, for any real vector field $Y \in C^\infty(X,T X)$ we have that $Y = Y^j \partial_j + \overline{Y^ j} \partial_{\bar j}$ and $\| Y\|^2=\langle Y,Y\rangle= 2 Y^j\overline{Y^{k}} g_{j\bar k}$.

We will be also interested in the imaginary part of $h$:
\begin{equation}\label{eq: prelKahlerlocal}
\o = -2\text{Im }h = ig_{j\bar k}d z_j \wedge d\bar z_k.
\end{equation}
It is straightforward to see that 
\begin{equation}\label{eq: form_Riem_rel}
\o(\cdot,\cdot)= \langle J(\cdot),\cdot\rangle.
\end{equation}

We say that $\o$ is a \emph{K\"ahler form} if $d \o =0$. In this case $(X,\o)$ is called a \emph{K\"ahler manifold} and we fix such a manifold for the remainder of this section.

By the Poincar\'e lemma \cite[Lemma II.2.15]{we}, the K\"ahler condition implies that for any $x \in X$ there exists an open neighborhood $U \ni x$ and a \emph{local K\"ahler potential} $g \in C^\infty(\overline{U})$ satisfying
$$\o|_U = i\partial \bar \partial g = i \frac{\partial^2 g}{\partial z_j \partial \bar z_k} dz_j \wedge d \bar z_k.$$
From this and \eqref{eq: prelKahlerlocal} it follows that $g_{j\bar k}=\partial ^2 g/\partial z_j \partial \bar z_k$, hence in K\"ahler geometry, when doing calculations in local coordinates, one can interchange indices of the metric with its partial derivatives. We will do this quite frequently in our study.

\paragraph{The Christoffel symbols.} Given a hermitian structure $(X,J,h)$, it is possible to derive the following identity relating $\o$, the Riemannian metric $\langle\cdot,\cdot \rangle$ and its Levi--Civita connection:
$$3 d \o(W,Y,Z)- 3d\o(W,JY,JZ)= 2\langle (\nabla_W J) Y, Z\rangle,$$
where $W,Y,Z$ are smooth vector fields on $X$. Using the K\"ahler condition, this identity gives that $$\nabla_{(\cdot)} J =0.$$
As a consequence of this we obtain the following formulas for Christoffel symbols of $\nabla_{(\cdot)}(\cdot)$ in holomorphic local coordinates: 
\begin{flalign}\label{prelCristoffelform}
\nabla_{\partial_j}\partial_{\bar k}&= \Gamma_{j\bar k}^l \partial_l+ \Gamma_{j\bar k}^{\bar l} \partial_{\bar l}=0, \ \nabla_{\partial_{\bar j}}\partial_{ k}= \Gamma_{\bar j  k}^l \partial_l+\Gamma_{ \bar j  k}^{\bar l} \partial_{\bar l}=0, \nonumber\\
\nabla_{\partial_{ j}}\partial_{k}&= \Gamma_{ j k}^{ l} \partial_{ l}= g^{ l\bar h}\partial_{j} g_{ k\bar h} \partial_{ l}=g^{ l\bar h} g_{ j  k\bar h} \partial_{ l},\\
\nabla_{\partial_{\bar j}}\partial_{\bar k}&= \Gamma_{\bar j\bar k}^{\bar l} \partial_{\bar l}= g^{h \bar l}\partial_{\bar j} g_{h \bar k} \partial_{\bar l}=g^{ h\bar l} g_{h \bar j \bar k} \partial_{\bar l}. \nonumber
\end{flalign}
More concretely, to derive the above formulas we used that :\\
(i) the Levi--Civita connection is torsion free, i.e. $\Gamma^{j}_{l \bar k}=\Gamma^{j}_{\bar k l}, \ \Gamma^{j}_{l  k}=\Gamma^{j}_{k l}$, etc.\\
(ii) the Levi--Civita connection is real, i.e., $\overline{\Gamma^{j}_{l \bar k}}=\Gamma^{\bar j}_{\bar lk}, \ \overline{\Gamma^{j}_{l  k}}=\Gamma^{\bar j}_{\bar l \bar k},$ etc.\\
(iii) $\nabla_{(\cdot)} J =0$, i.e., $i\nabla_{\partial_j} \partial_k =\nabla_{\partial_j}J \partial_k=J \nabla_{\partial_j} \partial_k, \ -i\nabla_{\partial_j} \partial_{\bar k} =\nabla_{\partial_j}J \partial_{\bar k}=J \nabla_{\partial_j} \partial_{\bar k}$. \\
(iv) the product rule, $g_{lj\bar k}=\partial_l\langle \partial_j,\partial_{\bar k}\rangle = \langle \nabla_{\partial_l}\partial_j,\partial_{\bar k}\rangle + \langle \partial_j,\nabla_{\partial_l} \partial_{\bar k}\rangle=\langle \nabla_{\partial_l}\partial_j,\partial_{\bar k}\rangle=\Gamma_{lj}^h g_{h \bar k}.$

\paragraph{The volume form, gradient, Laplacian and Hamiltonian.} For K\"ahler geometers the volume form is equal to $\o^n$. This is a non--degenerate top form that is a scalar multiple of the usual volume form of Riemannian geometry.

For a smooth function $v \in C^\infty(X)$ the K\"ahler gradient and Laplacian are given as follows:
\begin{equation}\label{eq: Lapl_grad_formula}
\nabla^\o v = \nabla^\o_{1,0} v + \nabla^\o_{1,0} v = g^{j\bar k}v_{\bar k}\partial_j + g^{k\bar j}v_{ k}\partial_{\bar j}, \ \ \ \Delta^\o v =2g^{j\bar k}v_{j\bar k}= 2n i \partial \bar \partial v \wedge \o^{n-1}/\o^n.
\end{equation}
The length squared of the gradient is $\langle \nabla^\o v,\nabla^\o v \rangle = 2n i \partial v \wedge \bar \partial v \wedge \o^{n-1}/\o^n= 2g^{j\bar k}v_jv_{\bar k}.$  Also, given $u,v \in C^\infty(X)$, integration by parts gives the well known identity
$$\int_X u (\Delta^\o v) \o^n = -\int_X \langle \nabla^\o u,  \nabla^\o v\rangle\o^n.$$
Lastly, from \eqref{eq: form_Riem_rel} it follows that $d v = \langle \nabla^\o v, \cdot\rangle = - \o(J \nabla^\o v, \cdot)$, hence we have the following formula between the gradient and Hamiltonian of $v$:
\begin{equation}\label{eq: grad_Ham_relation}
\nabla^\o v = J X^\o v.
\end{equation}

\paragraph{The Ricci and scalar curvatures.} We recall the definition of the curvature tensor. Suppose $T,Y,Z,W$ are smooth sections of the complexified tangent bundle $TX^{\Bbb C}$. The Riemannian curvature tensor is introduced by the formulas:
$$R(T,Y)Z = \nabla_T \nabla_Y Z - \nabla_Y \nabla_T Z - \nabla_{[T,Y]}Z,$$
$$R(T,Y,Z,W)=\langle R(T,Y)Z,W \rangle.$$
The classical curvature identities are:
\begin{flalign}\label{prelcurvidentites}
R&(T,Y,Z,W)=-R(Y,T,Z,W),\nonumber\\
R&(T,Y,Z,W)=R(Z,W,T,Y),\\
R&(T,Y,Z,W)+R(Z,T,Y,W)+R(Y,Z,T,W)=0. \nonumber
\end{flalign}
From $\nabla_{(\cdot )} J =0$ it follows that $R(T,Y)J = J R(T,Y)$, hence also
$$R(T,Y,Z,W)=R(T,Y,JZ,JW).$$
From the curvature identity $R(T,Y,Z,W)=R(Z,W,T,Y)$ we also obtain
$$R(JT,JY,Z,W)=R(T,Y,Z,W),$$
hence $R(JT,JY)Z=R(T,Y)Z.$ We use this to compute the curvature in local coordinates:
\begin{flalign}\label{eq: prelcurvaturelocal}
R&(\partial_{\bar j},\partial_{\bar k})\partial_l=R(\partial_{\bar j},\partial_{\bar k})\partial_{\bar l}=R(\partial_j,\partial_k)\partial_l=R(\partial_j,\partial_k)\partial_{\bar l}=0, \nonumber\\
R&(\partial_{ j},\partial_{\bar k})\partial_l=-\nabla_{\partial_{\bar k}}\nabla_{\partial j} \partial_l=-(\partial_{\bar k} \Gamma_{jl}^h )\partial_h,\\
R&(\partial_{ j},\partial_{\bar k})\partial_{\bar l}=\nabla_{\partial_{j}}\nabla_{\partial_{\bar k}} \partial_{\bar l}=(\partial_{j} \Gamma_{\bar k\bar l}^{\bar h}) \partial_{\bar h}. \nonumber
\end{flalign}
The Ricci curvature is the trace of the curvature tensor:
$$\textup{Ric}(T,Y)= \textup{Tr}\{Z \to R(Z,T)Y\}.$$
From the curvature identities \eqref{prelcurvidentites} it follows that $\textup{Ric}$ is symmetric. Moreover, from  \eqref{eq: prelcurvaturelocal} it also follows that $\textup{Ric}(JT,JY)=\textup{Ric}(T,Y)$ and in local coordinates we have:
\begin{flalign*}
\textup{Ric}&_{jk}=0, \ \textup{Ric}_{\bar j \bar k}=0, \nonumber\\
\textup{Ric}&_{j\bar k} = R(\partial_{\bar l},\partial_j,\partial_{\bar k},\partial_{ m})g^{m \bar l }=-\partial_{ j} \Gamma_{\bar l \bar k}^{\bar l}= -\partial_{j} (g^{h\bar l}g_{\bar k h\bar l})= -\partial_{ j}\partial_{\bar k}\log(\det(g_{p\bar q})).
\end{flalign*}
Similar to the contruction of the K\"ahler form \eqref{eq: prelKahlerlocal} we consider the Ricci form:
\begin{equation}\label{prelriclocal}
\textup{Ric } \omega=-2\textup{Im Ric}=i\textup{Ric}_{j\bar k}dz_j\wedge d \bar z_k= -i\partial \bar \partial \log(\det(g_{p\bar q})).
\end{equation}
This last formula tells us that the Ricci form is closed, and given two K\"ahler metrics $\o,\o'$ on $X$ we have
\begin{equation}\label{prelricdifference}
\textup{Ric }\o - \textup{Ric }{\o'} = i\partial \bar \partial \log\Big(\frac{\o'^n}{\o^n}\Big).
\end{equation}
The scalar curvature $S_\o$ is the trace of the Ricci curvature. In local coordinates it can be expressed as
\begin{equation}\label{prelscallocal}
S_\o = g^{j\bar k}\textup{Ric}_{j\bar k}.
\end{equation}
We compute the variation of the Ricci and scalar curvatures in the next proposition:
\begin{proposition} \label{prop: PrelRicScalVariation}Suppose $[0,1] \ni t \to \o_t$ is a smooth curve of K\"ahler metrics on $X$. Then the following holds for the variation of the Ricci and scalar curvatures:
$$\frac{d}{dt}\textup{Ric }{\o_t}=-i\partial\bar\partial \textup{Tr}_{\o_t}\Big(\frac{d}{dt}\o_t\Big),$$
$$\frac{d}{dt}S_{\o_t}=-\frac{1}{2}\Delta_{\o_t}\textup{Tr}_{\o_t}\Big(\frac{d}{dt}\o_t\Big) - \big \langle \textup{Ric }{\o_t}, \frac{d}{dt}\o_t\big \rangle_{\o_t}.$$
\end{proposition}
\begin{proof} The proof of the formulas follow after we differentiate \eqref{prelriclocal} and \eqref{prelscallocal}.
\end{proof}

\paragraph{Normal coordinates.} One of the advantages of dealing with K\"ahler metrics is the fact that for local calculations one can choose coordinates in a convenient way:
\begin{proposition}\label{prop: prelnormcoord}Suppose $(X,\o)$ is a K\"ahler manifold. Given $x \in X$ there exist holomorphic local coordinates $z = (z_1, \ldots, z_n)$ mapping a neighborhood of $x$ to a neighborhood $U$ of $0 \in \Bbb C^n$ such that in these coordinates $\o$ can be expressed as:
$$\o|_U =i  dz_j \wedge d \bar z_j +  i g_{j\bar k p \bar q}z_p \bar z_q dz_j \wedge d \bar z_k +O(|z^3|)$$
\end{proposition}
\begin{proof} First choose an arbitrary coordinate patch that maps $x$ and a neighborhood of $x$ to $0 \subset \Bbb C^n$ and $V \subset \Bbb C^n$ respectively. It is possible to apply a linear change of coordinates in $\Bbb C^n$ so that $\o|_{x} = i\ddbar g(0) =i d z_j \wedge d \bar z_j$. Now we deal with the first order coefficients by making the following local change of coordinates:
$$F_m(z)= z_m + c_{mjk} z_j z_k, \ m \in \{1,\ldots,n\}.$$
An elementary calculation for $F^* \o = i \ddbar g \circ F = : i \ddbar \tilde g$ yields
$$\tilde g_{j\bar k l}(0)= g_{j \bar k l}(0) + 2c_{k jl}.$$
Choosing $c_{kjl}=- g_{j \bar k l}(0)/2$, we obtain that after composing our local coordinate map with $F$ the first order coefficients of $\o$ vanish. Now we deal with the second order coefficients.
For this we introduce a local diffeomorphism of the form  
$$H_m(z)= z_m + b_{mjkl} z_j z_k z_l, \ m \in \{1,\ldots,n\}.$$
Pulling back again by  $H$, for $\tilde g = g \circ H$ we can write
$$\tilde g_{j \bar k a b}(0)=g_{j \bar k a b}(0) + 6 b_{kjab}.$$
Choosing $b_{kjab} = -g_{j \bar k a b}(0)/6$ we get $\tilde g_{j \bar k a b}(0)=0$. By taking conjugates and applying the Leibniz rule for derivatives we also get $\tilde g_{j \bar k \bar a \bar b}(0)=0$, finishing the proof.
\end{proof}
It immediately follows that all Christoffel symbols of $\o$ vanish  at $x$ after a choice of \emph{normal coordinates} around this point, as described in the above proposition. Also, from \eqref{eq: prelcurvaturelocal},\eqref{prelriclocal} and \eqref{prelscallocal} we obtain convenient formulas for geometric quantities related to curvature:
\begin{flalign}
R(\partial_j,\partial_{\bar k},\partial_p,\partial_{\bar q})(x)= - g_{j \bar k p \bar q}, \ \ \
\textup{Ric}_{j\bar k}(x)=  - g_{j \bar k p \bar p},\label{eq: prelcurvnormalcoord} \ \ \ 
S_\o(x)=  - g_{j \bar j p \bar p}.
\end{flalign}

\paragraph{The Lichnerowicz operator.} Given a complex valued function $u \in C^\infty(X,\Bbb C)$, the Riemannian Hessian of $u$ is computed by the well known formula $\nabla ^2 u(X,Y) = X du(Y) - du(\nabla_X Y)$. We denote by $\mathcal L u$ the $(0,2)$ part of $\nabla ^2 u$. In local coordinates this can be expressed as
\begin{equation}\label{eq: Lich_def}
\mathcal L u = (\mathcal L u)_{\bar j\bar k}d\bar z_j  \otimes d\bar z_k=(u_{\bar j\bar k} - \Gamma_{\bar j\bar k}^{\bar l}u_{\bar l})d\bar z_j  \otimes d\bar z_k=(u_{\bar j\bar k} - g^{h\bar l}g_{h\bar j\bar k}u_{\bar l})d\bar z_j  \otimes d\bar z_k.
\end{equation}
After switching to normal coordinates, one can easily see that $\mathcal L u=0$ if and only if $(g^{j\bar k} u_{\bar k})_{\bar l} = 0$ for all $\bar l \in \{1,\ldots,n\}$ which is equivalent to $\nabla^\o_{1,0} u$ being a  holomorphic vector field. $\mathcal L$ is the \emph{Lichnerowicz operator} and the following formula for the self--adjoint complex operator $\mathcal L^* \mathcal L$ will be very useful for us:
\begin{proposition} \label{prop: prelLichform} ${\mathcal L}^* \mathcal L u = \frac{1}{4}\Delta^\o(\Delta^\o u) + \langle \textup{Ric}_\o, i\partial \bar \partial u\rangle_{\o_u} +\langle \partial S_\o,\partial u\rangle, \ u \in C^\infty(X).$
\end{proposition}
To clarify, in the above formula $\mathcal L^*$ is the Hermitian dual of $\mathcal L$, as will be detailed in the proof below. Also, $\langle \partial S_\o,\partial u\rangle$ denotes the Hermitian inner product of $\partial S_\o$ and $\partial u$ (equaling $g^{j\bar k} {S_\o}_j u_{\bar k}$ in local coordinates).

\begin{proof} Suppose $x \in X$. One has the following local expression for $\langle\mathcal L u,\mathcal L v \rangle$:
$$\langle\mathcal L u,\mathcal L v \rangle=(\mathcal L u)_{\bar j\bar k}g^{p\bar j}g^{q\bar k}\overline{(\mathcal L v)_{\bar p\bar q}}=(u_{\bar j\bar k} - g^{h\bar l}g_{h\bar j\bar k}u_{\bar l})g^{p\bar j}g^{q\bar k}(v_{ pq} - g^{ a\bar b}g_{ p q \bar b}v_{a})$$
We choose normal coordinates around $x$ as in Proposition \ref{prop: prelnormcoord}. Using integration by parts and \eqref{eq: Lich_def} one can write
$$\mathcal L^* \mathcal L u = ((u_{\bar j\bar k} - g^{h\bar l}g_{h\bar j\bar k}u_{\bar l})g^{p\bar j}g^{q\bar k})_{pq} - (((u_{\bar j\bar k} - g^{h\bar l}g_{h\bar j\bar k}u_{\bar l})g^{p\bar j}g^{q\bar k})g^{ a \bar b}g_{ p q \bar b})_a.$$

As we are working in normal coordinates, all terms of type $g_{abc\bar d},g_{ab\bar c}$ vanish, hence the second term in the above sum is 0. By expanding the first term we arrive at:
$$\mathcal L^* \mathcal L u=u_{jk\bar j\bar k} - g_{l j k \bar j\bar k}u_{\bar l}-g_{lj\bar j \bar k}u_{k\bar l} -g_{lk\bar j \bar k}u_{j\bar l}.$$
By \eqref{eq: Lapl_grad_formula} and \eqref{eq: prelcurvnormalcoord}, in normal coordinates we have $4\Delta^\o(\Delta^\o u) = g^{a \bar b} (g^{c\bar d}u_{c\bar d})_{a\bar b}=u_{a\bar a c \bar c} - u_{c\bar d}g_{a\bar a c \bar d}$, $\langle \partial S_\o, \partial u \rangle=g_{l j k \bar j\bar k}u_{\bar l}$ and $\langle \textup{Ric}_\o,i\ddbar u \rangle=-g_{lj\bar j \bar k}u_{k\bar l}$  hence the desired identity follows.
\end{proof}

\section{Approximation of $\o$--psh functions on K\"ahler manifolds}

Given a compact K\"ahler manifold, we show that any $\o$--psh function can be approximated by a decreasing sequence of smooth K\"ahler potentials. Much stronger results have been derived by Demailly \cite{de1,de2}, using  sophisticated techniques. Here we will follow closely the arguments of  Blocki--Kolodziej \cite{bk}. Our main result is the following:

\begin{theorem}\label{thm: BKapprox_app}  Let $(X,\o)$ be a compact K\"ahler manifold. Given $u \in \textup{PSH}(X,\o)$, there exists a decreasing sequence $\{u_k\}_k \subset \mathcal H_\o$ such that $u_k \searrow u$. 
\end{theorem}

The main technical ingredient will be the following proposition, which provides an intermediate result:

\begin{proposition}\label{prop: BKapprox_app} Let $(X,\o)$ be a K\"ahler manifold (not necessarily compact). Let $X' \subset X$ be a relatively compact open set. If the Lelong numbers of $\varphi$ are zero at any $x\in X$, then there exists $\varphi_k \in \textup{PSH}(X',(1 + \frac{1}{k})\o) \cap C^\infty(X')$ such that $\varphi_k \searrow \varphi$.  
\end{proposition}

Based on this proposition we quickly give the proof of Theorem \ref{thm: BKapprox_app}:

\begin{proof}[Proof of Theorem \ref{thm: BKapprox_app}] We can assume that $u \leq -1$, and consider the cutoffs $v_k := \max(u,-k) \in \textup{PSH}(X,\o) \cap L^\infty$. An application of Proposition \ref{prop: BKapprox_app} to $\{v_k\}_k$ gives  a sequence $w_k \in \mathcal H_{(1+ \frac{1}{k})\o}$ such that $w_k \searrow u$ and $v_k \leq w_k \leq 0$. Introducing $u_k:= \frac{1}{k} + (1-\frac{1}{k})w_k \in \mathcal H_\o$,  we see that $\{u_k\}_k$ is decreasing and $u_k \searrow u$.
\end{proof} 

Returning to the local situation of a moment, let $v \in \textup{PSH}(U)$ for some open set $U \subset \Bbb C^n$. It is well known that the correspondence $r \to f(r)=\sup_{{B(z_0,e^r)}} v=\max_{\partial{B(z_0,e^r)}} v$ is  convex for any $z_0 \in U$ (this follows from the comparison principle for the complex Monge--Amp\`ere operator and the fact that $i\ddbar (\log |z -z_0|)^n =0$ on $\Bbb C^n \setminus \{z_0 \}$). Convexity of $r \to f(r)$ gives in particular that this map is continuous, hence so are the maps $v_\delta \in \textup{PSH}(U_\delta)$, defined by the formula
\begin{equation}\label{eq: def_sup_conv}
v_\delta(x) := \sup_{y \in B(x,\delta)} v(y),
\end{equation}
where $U_\delta= \{x \in U, \ B(x,\delta) \subset U \}.$ Lastly, as $v$ is usc, we also obtain that $v_\delta \searrow v$. 

Given $u \in \textup{PSH}(X,\o)$ and $x_0 \in X$, recall the definition of the Lelong number $\mathcal L(u,x_0)$ from \eqref{eq: Lelong_def}:
$$\mathcal L(u,x_0):= \sup \{ r \geq 0 \ | \ u(x) \leq r \log |x| + C_r, \ \forall \ x \in U_r, \textup{ for some } \ C_r>0\},$$
where $U_r \subset \Bbb C^n$ is some coordinate neighborhood of $x_0$ (dependent on $r$) that identifies $x_0$ with $0 \in \Bbb C^n$. We now give an alternative description of $\mathcal L(u,x_0)$ that will be of great use in the proof of Proposition \ref{prop: BKapprox_app}. Choose a coordinate neighborhood $U \subset \Bbb C^n$ of $x_0$, that identifies $x_0$ with $0 \in \Bbb C^n$, and a potential $f \in C^\infty(X)$ such that $i\ddbar f = \o$ on $U$. Clearly $u + f \in \textup{PSH}(U)$, and the following is well known:
\begin{equation}\label{eq: Lelong_id}
\mathcal L(u,x_0) = \lim_{r \to 0} \frac{\sup_{B(x_0,r)} (u + f)}{\log r}.
\end{equation}
The limit on the right hand side is well defined as $\log r \to \sup_{B(x_0,r)} (u + f)$ is convex, hence the quotients involved are  decreasing as $r \to 0$. Also, this limit is clearly independent of the choice of potential $f$. For an extensive treatment of Lelong numbers we refer to \cite[Section 2.3]{gzbook}.

Finally, we arrive at the following result which will allow to compare approximation via \eqref{eq: def_sup_conv} using different coordinate charts: 
\begin{lemma}\label{lem: change_var_convolution} Let $U, V\subset \Bbb C^n$ and $F: U \to V$ a biholomorphic
map. Suppose $u \in \textup{PSH}(U)$ has zero Lelong numbers. For $u^F_\delta :=(u \circ F^{-1})_\delta \circ F$, the difference $u_\delta - u^F_\delta$ converges to zero uniformly on compact sets as $\delta \to 0$.
\end{lemma}

\begin{proof} Fix $a > 1, \ r >0$ and $z \in U$. As $\log \delta \to u_\delta(z)$ is convex, for small enough $\delta>0$ we can write
$$0 \leq u_{a\delta}(z) - u_\delta(z) \leq  \frac{\log a}{\log r - \log \delta} (u_r(z) - u_\delta(z)).$$
Consequently, as $\mathcal L(u,z)=0$, \eqref{eq: Lelong_id} implies that $u_\delta - u_{a \delta} \to 0$ uniformly on compact sets as $\delta \to 0$. 

Next we notice that
$$u^F_\delta(z) = \max_{
F^{-1}(\overline{B(F(z),\delta)})}u.$$
As $F$ is a biholomorphism, for a fixed compact set $K \subset U$ it is possible to find $a>1$ and $\delta_0 >0$ such that for $\delta  \in (0,\delta_0)$ and $z \in K$ we have
$$B(F(z), \delta) \subset F(B(z, a\delta)), \ \ F(B(z, \delta)) \subset B(F(z), a\delta).$$
As a result of these containments, we get that $u^F_\delta \leq u_{a\delta}$ and $u_\delta \leq u^F_{a\delta}$ on $K$. Since $(u_\delta - u_{a \delta}) \to 0$ uniformly on compacts, the statement of the lemma follows.
\end{proof}

\begin{proof}[Proof of Proposition \ref{prop: BKapprox_app}] First we find $\tilde \varphi_j \in \textup{PSH}(X',(1 + \frac{1}{j}) \o) \cap C(X')$ satisfying the requirements of the proposition.

Fix $\varepsilon >0$. We can find a finite number of nested charts $V_\alpha \subset U_\alpha$
such that $\{V_\alpha\}_\alpha$ covers $X'$, $\overline{V_\alpha} \subset U_\alpha$, and $i\ddbar f_\alpha =\o$ on $U_\alpha$,
for some potentials $f_\alpha \in C^\infty(U_\alpha)$. 

Then $\varphi_\alpha := \varphi + f_\alpha \in \textup{PSH}(U_\alpha)$ and by the previous lemma we have
\begin{equation}\label{eq: overlap_conv}
\varphi_{\alpha,\delta} - \varphi_{\beta,\delta} = \varphi_{\alpha,\delta} - \varphi^F_{\alpha,\delta} + (\varphi_\alpha - \varphi_\beta)^F_\delta \to f_\alpha - f_\beta
\end{equation}
locally uniformly on $U_\alpha \cap U_\beta$ as $\delta \to 0$, where $F$ is the change of coordinates on the overlap $U_\alpha \cap U_\beta$. Let $\eta_\alpha$ be smooth on $U_\alpha$ such that $\eta_\alpha=0$ on $V_\alpha$ and $\eta_\alpha = -1$ in a neighborhood of $\partial U_\alpha$. We have $i\ddbar \eta_\alpha > -C \o$
some $C >0$. 

We set
$$\tilde \varphi_\delta:= \max_\alpha \Big(\varphi_{\alpha,\delta} - f_\alpha + \frac{\varepsilon \eta_\alpha}{C}\Big).$$
By \eqref{eq: overlap_conv}, for $\delta>0$  sufficiently small, the values of $\varphi_{\alpha,\delta} - f_\alpha + {\varepsilon \eta_\alpha}/{C}$ do not contribute to the maximum on the set  $\{\eta_\alpha = -1\}$, thus $\tilde \varphi_\delta$ is continuous on $X'$ and also $\tilde \varphi_\delta \in \textup{PSH}(X',(1 + \varepsilon) \o)$. It is also clear that $\tilde \varphi_\delta$ decreases to $\varphi$ as $\delta \searrow 0$, allowing to construct $\tilde \varphi_j \in \textup{PSH}(X',(1+ \frac{1}{j}) \o) \cap C(X')$ such that $\tilde \varphi_j \searrow \varphi$.

Next we argue that is also possible to approximate with smooth potentials. For this we only need to show that any $\psi$ that is continuous and $(1 + \varepsilon) \o$--psh on an open neighborhood $Y$ of $\overline{X'}$ can be approximated uniformly on $X'$ by smooth $(1 + 2\varepsilon \o)$--psh potentials.

This can be done by using classical Richberg approximation. Indeed, let $V_\alpha \subset U_\alpha$, $f_\alpha$ and $\eta_\alpha$ be as in the first part of the proof. Additionally, let $\rho \in C^\infty_c(B(0,1))$ be a smooth, non--negative, compactly supported and spherically invariant bump function with $\int_{ \Bbb C^n} \rho = 1$. 

We introduce the smooth mollifications $\psi_{\alpha,\delta} := (\psi + (1 + \varepsilon)f_\alpha) * \rho_\delta \in \textup{PSH}(U_{\alpha,\delta})$ that converge uniformly to $\psi + (1 + \varepsilon)f_\alpha$ on $V_\alpha$, since $\psi$ is continuous. As a result, we can apply the Richberg regularized maximum (\cite[Corollary I.5.19 and Theorem II.5.21]{De}) to get that
$$\psi_\delta := M_\alpha\Big(\psi_{\alpha,\delta} - (1 + \varepsilon)f_\alpha + \frac{\varepsilon \eta_\alpha}{C}\Big)$$
is smooth on $\cup_\alpha V_\alpha$ (for small enough $\delta$) and also $\psi_\delta \in \textup{PSH}(\cup_\alpha V_\alpha,(1+2\varepsilon)\o)$. Consequently, $\psi_\delta \to \psi$ uniformly on $X' \subset \cup_\alpha V_\alpha$, finishing the proof.
\end{proof}

\section{Regularity of envelopes of $\o$--psh functions}

Suppose $(X,\o)$ is a K\"ahler manifold and $f$ is an usc function on $X$. Recall the definition of the envelope $P(f)$ from \eqref{eq: P_env_def}:
\begin{equation}
P(f)= \sup \{u \in \textup{PSH}(X,\o)  \textup{ s.t. }  u \leq f\}.
\end{equation}
As it was argued in Section 3.4, $P(f)$ is $\o$--psh and the purpose of this short section is to show the following regularity result:
\begin{theorem}\label{thm: BD_reg} If $f \in C^\infty(X)$ then $\|P(f)\|_{C^{1,\bar 1}} \leq C(X,\o,\| f\|_{C^{1,\bar 1}}).$
\end{theorem}

A bound on the $C^{1,\bar1}$ norm of $P(f)$ simply means a uniform bound on all mixed second order derivatives $\partial^2 P(f)/\partial z_j \bar \partial z_k$. Since $P(f)$ is $\o$-psh, this is equivalent to saying that $\Delta^\o P(f)$ is bounded, and by the Calderon--Zygmund estimate \cite[Chapter 9, Lemma 9.9]{GT}, we automatically obtain that $P(f_1,f_2,\ldots,f_k) \in C^{1,\alpha}(X)$.

Theorem \ref{thm: BD_reg} was first proved by Berman--Demailly using methods from pluripotential theory\cite{bd}. Here we will follow an alternative path proposed by Berman, that uses more classical PDE techniques \cite{Brm2}. The point is to consider the following complex Monge--Amp\`ere equation:
\begin{equation}\label{eq: MAeq_beta}
\o_{u_\beta}^n = e^{\beta(u_\beta - f)}\o^n.
\end{equation}  
As $f \in C^\infty(X)$, it follows from work of Aubin and Yau that this equation always has a unique smooth solution $u_\beta \in \mathcal H_\o$  for any $\beta >0$ \cite{A,Y} (for a survey see \cite[Theorem 14.1]{gzbook}). Theorem \ref{thm: BD_reg} will follow from the following regularity result:
\begin{proposition}[\cite{Brm2}] \label{prop: zero_temp_est} The unique solutions $\{ u_\beta\}_{\beta >0}$ of \eqref{eq: MAeq_beta} satisfy the following:\\
\noindent (i) $\|u_\beta - P(f)\|_{C^0} \to 0$ as $\beta \to \infty$.\\
\noindent (ii) there exists $\beta_0 >0$ and $C>0$ such that $\|\Delta^{\o} u_\beta\|_{C^0} < C$ for all $\beta \geq \beta_0$.
\end{proposition}

To argue Proposition \ref{prop: zero_temp_est}(i), first we prove the following comparison principle :

\begin{lemma}\label{lem: comp_princ_MA_beta} Assume that $u,v \in \textup{PSH}(X,\o) \cap L^\infty$ such that
$$\o_v^n \geq e^{\beta(v-f)}\o^n \ \textup{ and } \ \o_u^n \leq e^{\beta(u-f)}\o^n.$$ 
Then $v \leq u$.
\end{lemma}
\begin{proof}According to the comparison
principle  $\int_{\{u < v\}} \o_v^n \leq \int_{\{u < v\}} \o_u^n$ (Proposition \ref{prop: comp_princ_E}). Consequently, we can write that $\int_{\{u < v\}} e^{\beta(v -f)}\o^n \leq \int_{\{u < v\}} e^{\beta(u-f)}\o^n$. As a result, $v \leq u$ a.e. with respect to the Lebesque measure. Since $u$ and $v$ are $\o$--psh, we obtain that in fact $v \leq u$ globally on $X$.   
\end{proof}

Now we argue Proposition \ref{prop: zero_temp_est}(i):
\begin{lemma}  There exists $C:=C(f)>0$ such that the unique solutions $\{ u_\beta\}_{\beta >0}$ of \eqref{eq: MAeq_beta} satisfy the following estimate:
$$\sup_X |u_\beta - P(f)| \leq C \frac{\log \beta}{\beta}, \ \beta >2.$$
\end{lemma}
\begin{proof} Since $u_\beta$ is smooth, at the point $x_0 \in X$ where the maximum of $u_\beta-f$ is attained we have $i\ddbar (u_\beta-f) \leq  0$. Consequently,  \eqref{eq: MAeq_beta} implies that at $x_0$ we have 
$$u_\beta -f \leq \frac{1}{\beta} \log \Big(\frac{\o_f^n}{\o^n}\Big).$$
Denoting $C := \sup_X \log \big(\frac{\o_f^n}{\o^n}\big)$, we obtain that $u_\beta - \frac{C}{\beta} \leq f$, hence
\begin{equation}\label{eq: C_0_est_half}
u_\beta - P(f) \leq \frac{C}{\beta}.
\end{equation}

Conversely, fix $v \in \mathcal H_\o$ and $D>0$ such that $\o_v \geq D \o$, and $v \leq f$. Also, choose $\delta,\varepsilon \in (0,1)$ such that 
\begin{equation}\label{eq: delta_eps_cond}
\delta^n D^n \geq e^{-\beta \varepsilon}. 
\end{equation}
As a result, $u_{\delta,\varepsilon} := (1-\delta)P(f) + \delta v - \varepsilon \leq f -\varepsilon$ and
$$\o_{u_{\delta,\varepsilon}}^n \geq \delta^n D^n\o^n \geq e^{-\beta\varepsilon} \o^n \geq  e^{\beta(u_{\delta,\varepsilon} - f)}\o^n.$$
By Lemma \ref{lem: comp_princ_MA_beta} we obtain that $u_{\delta,\varepsilon} \leq u_\beta$. Also, for $\beta:=\beta(D)>2$ big enough, the choice  $\delta := 1/\beta$ and $\varepsilon := {2n \log \beta}/{\beta}$  satisfies \eqref{eq: delta_eps_cond}, and we obtain that $(1- \frac{1}{\beta}) P(f) + \frac{1}{\beta} v - \frac{2n \log \beta}{\beta} \leq u_\beta$. Putting this together with \eqref{eq: C_0_est_half}, we arive at
$$u_\beta - \frac{C}{\beta} \leq P(f) \leq \frac{\beta}{\beta - 1} u_\beta + \frac{2n \log \beta}{\beta - 1} - \frac{1}{\beta - 1}\inf_X v.$$
In particular, since $P(f)$ is bounded, this implies that $u_\beta$ is uniformly bounded, and the estimate of the lemma follows.
\end{proof}

Lastly, we argue Proposition \ref{prop: zero_temp_est}(ii):

\begin{lemma}  There exists $C>0$ and $\beta_0 >0$ such that the unique solutions $\{ u_\beta\}_{\beta}$ of \eqref{eq: MAeq_beta} satisfy the following estimate:
$$-C \leq  \Delta^\o u_\beta \leq C, \ \ \beta \geq \beta_0,$$
where $C$ only depends on an upper bound for $\Delta^\omega f$.
\end{lemma}
\begin{proof} The lower bound on $\Delta^\o u_\beta$ follows immediately from $\o + i\ddbar u_\beta \geq 0$. To prove the upper bound, we recall a variant of a Laplacian estimate due to Aubin--Yau (provided by Siu in this context \cite[page 99]{siu}, for a survey we refer to \cite[Proposition 4.1.2]{beg}): if $u \in \mathcal H_\o$ satisfies $\o_u^n = e^g \o^n$, then
$$\Delta^{\o_u} \log \textup{Tr}_{\o} \o_u \geq \frac{\Delta^\o g}{\textup{Tr}_{\o} \o_u} - 2B \textup{Tr}_{\o_u}\o,$$
where $B>0$ depends only on the magnitude of the holomorphic bisectional curvature of $\o$. Applying this estimate to \eqref{eq: MAeq_beta} we obtain
$$2B \textup{Tr}_{\o_{u_\beta}}\o + \Delta^{\o_{u_\beta}} \log \textup{Tr}_{\o} \o_{u_\beta} \geq \beta \frac{\Delta^\o (u_\beta -f)}{\textup{Tr}_{\o} \o_{u_\beta}}.$$
Rearranging terms we conclude:
$$2nB + \Delta^{\o_{u_\beta}}( \log \textup{Tr}_{\o} \o_{u_\beta}  - B u_\beta)\geq \beta \frac{ \textup{Tr}_{\o} \o_{u_\beta} - 2n - \Delta^\o f}{\textup{Tr}_{\o} \o_{u_\beta}}.$$
Denoting $C := \sup_X (2n + \Delta^\o f) >0$, and multiplying this inequality with $\textup{Tr}_{\o} \o_{u_\beta} e^{- Bu_\beta}$, we arrive at:
$$(C\beta + 2nB \textup{Tr}_{\o} \o_{u_\beta}) e^{- Bu_\beta} + \Delta^{\o_{u_\beta}}( \log \textup{Tr}_{\o} \o_{u_\beta}  - B u_\beta)\textup{Tr}_{\o} \o_{u_\beta} e^{- Bu_\beta}\geq \beta \textup{Tr}_{\o} \o_{u_\beta} e^{- Bu_\beta}.$$
Let $s:= \sup_X \textup{Tr}_{\o} \o_{u_\beta} e^{- Bu_\beta}$ and suppose that this last supremum is realized at $x_0 \in X$. By the previous estimate we can write:
$$\beta s \leq 2nB s + C\beta e^{-B u_\beta(x_0)}.$$
By the previous lemma, $u_\beta$ is uniformly bounded, hence for $\beta \geq \beta_0 := 3nB$ we obtain an upper bound for $s$. Using again that $u_\beta$ is uniformly bounded, we conclude that $\Delta^\o u_\beta \leq C$ for any $\beta \geq \beta_0$. 
\end{proof}

Notice that $g := \min(f_1,f_2,\ldots,f_k) = - \max(-f_1,-f_2,\ldots,-f_k)$. From this it follows that $\Delta^\omega g$ is uniformly bounded from above by a common upper bound for $\Delta^\omega f_j, \ j \in \{1,\ldots,k \}$. Using the fact that the constant $C$ in the previous lemma only depends on an upper bound for the Laplacian of $f$, we can conclude the following more general result, proved in \cite{DR1} using different methods:

\begin{theorem} \label{thm: DR_reg} Given $f_1,...,f_k \in C^\infty(X)$, then $P(f_1,f_2,...,f_k) \in C^{1,\alpha}(X), \ \alpha \in (0,1)$. More precisely, the following estimate holds: 
$$\|P(f_1,f_2,...,f_k)\|_{C^{1,\bar 1}} \leq C(X,\o,\| f_1\|_{C^{1,\bar 1}},\| f_2\|_{C^{1,\bar 1}},\ldots,\| f_k\|_{C^{1,\bar 1}}).$$
\end{theorem}

\section{Cartan type decompositions of Lie groups} 

In the proof of the Bando--Mabuchi uniqueness theorem, some basic facts about Lie groups are needed. This section is based on \cite[Section 6.1]{dr2}. For an extensive study of such groups we refer to \cite{Bump, Helg}.

Recall the following form of the classical Cartan decomposition for complexifications of compact semisimple Lie groups (see \cite[Proposition 32.1, Remark 31.1]{Bump}):

\begin{theorem}
\label{thm: CartanThm}
Let $K$ be a compact connected semisimple Lie group. Denote
by $(K^\Bbb C,J)$ the complexification of $K$, namely the unique connected
complex Lie group whose Lie algebra is the complexification 
of $\mathfrak k$, the Lie algebra of $K$. Then
the map $C: K\times\mathfrak k \to K^\Bbb C$ given by $C(k,X):=k\exp_IJ X$ is a diffeomorphism.
\end{theorem}

The following result is a partial extension of the above classical theorem  to  compact but not necessary semisimple Lie groups.
We state the result in a form that will be most useful for our applications in K\"ahler geometry, albeit it is likely not optimal.

\begin{proposition}
\label{prop: PartialCartanProp}
Let $K$ be a compact connected subgroup of a connected complex Lie group $(G,J)$ and denote by $\mathfrak k$ and $\mathfrak g$ their Lie algebras. If $\mathfrak g ={\mathfrak k}\oplus J{\mathfrak k}$ then the map 
$C:K \times {\mathfrak k} \to G$ given by 
$C(k,X)=k\exp_{I}J X$ is surjective.
\end{proposition}

\begin{proof}
First we note that
\begin{equation}\label{eq: tfkLemma}
{\mathfrak k}=\mathfrak z({\mathfrak k})\otimes[{\mathfrak k},{\mathfrak k}],
\end{equation}
where $\mathfrak z(\mathfrak k)$ is the Lie algebra of $Z(K)$. This follows from \cite[Proposition 6.6 (ii), p. 132]{Helg}, as $K$ is compact.
Next we note the following identity for the Lie algebra of the center $Z(G)$:
\begin{equation}
\label{eq: centerLemma}
\mathfrak z(\mathfrak g)=\mathfrak z({\mathfrak k})\oplus J\mathfrak z({\mathfrak k}).
\end{equation}
Since $\mathfrak z(\mathfrak g)$ is complex, we immediately obtain that
$\mathfrak z(\mathfrak g)\supset \mathfrak z({\mathfrak k})\oplus J\mathfrak z({\mathfrak k})$.
For the reverse inclusion, we use that $\mathfrak g = \mathfrak k \oplus J \mathfrak k$. Consequently, for any $X\in\mathfrak z(\mathfrak g)$ we have  $X=X_1+X_2$, 
with $X_2\in {\mathfrak k} \cap \mathfrak z(\mathfrak g)=\mathfrak z({\mathfrak k})$, 
and $X_2\in J{\mathfrak k}\cap\mathfrak z(\mathfrak g)=J({\mathfrak k}\cap\mathfrak z(\mathfrak g))=J\mathfrak z({\mathfrak k})$,
since $\mathfrak z(\mathfrak g)$ is complex. This finishes the proof of \eqref{eq: centerLemma}.

Next we claim that the map 
$$\Theta_1:Z(K)\times\mathfrak z({\mathfrak k})\to Z(G)$$ given by $(z,X)\mapsto z\exp_IJ X$ is surjective. 
Indeed, \eqref{eq: centerLemma} implies that $\dim Z(K)+\dim {\mathfrak z(\mathfrak k)}=\dim Z(G)$ and the differential of
$\Theta_1$ at $(I,0)$ is invertible (see \cite[Proposition 1.6, p. 104]{Helg}). As we are dealing with abelian groups it follows that $\Theta_1$ is a Lie group homomorphism,
thus it must be surjective as its image is a connected subgroup of the
same dimension as that of $Z(G)$.

Let $L$ denote the connected compact Lie subgroup of $K$ whose Lie algebra is $[{\mathfrak k},{\mathfrak k}]$ (since the Killing form is negative definite on $[{\mathfrak k},{\mathfrak k}]$, $L$ is indeed compact). 
By \cite[Proposition 6.6 (i), p. 132]{Helg}, $L$ is semisimple.
By Theorem \ref{thm: CartanThm}, the map 
$$\Theta_2:L\times[{\mathfrak k},{\mathfrak k}]\to L^\Bbb C$$ 
given by $\Theta_2(l,X)=l\exp_IJ X, $
is a diffeomorphism, where $L^\Bbb C$ is the complexification of $L$ inside $G$. 

Next we note that the multiplication maps $Z(K)\times L\to K$, $Z(G)\times L^\Bbb C\to G$ are surjective.
By \eqref{eq: tfkLemma} the multiplication map $Z(K)\times L\to K$ is a local isomorphism near $(I,I)$ by dimension count. The map is also a group homomorphism since elements of $Z(K)$ commute with elements of $L$. Thus, it is surjective. The same argument works for the multiplication map $Z(G)\times L^\Bbb C \to G$, with dimension count provided by \eqref{eq: centerLemma}.

Now we put all the above ingredients together. Given $k \in K$ and $X \in \mathfrak k$, observe that
$$
C(k,X)=zl\exp_{I}J X,
$$
for some $z\in Z(K)$ and $l\in L$ such that  $k=zl$
(these exist by the surjectivity of the multiplication map $Z(K) \times L \to K$). Now let $X_1$ and $X_2$  
be the unique elements such that
$X_1\in \mathfrak z({\mathfrak k})$, $X_2\in [{\mathfrak k},{\mathfrak k}]$, and $X=X_1+X_2$,
given by \ref{eq: tfkLemma}. Since $\exp_IJ X_1\in Z(G)$ we can write 
$$
C(k,X)=z\exp_{I}J X_1 l\exp_IJ X_2
=\Theta_1(z,X_1)\Theta_2(l,X_2).
$$
Since both $\Theta_2$ and $\Theta_1$ are surjective, as well as the multiplication map $Z(G)\times L^\Bbb C\to G$, it follows that $C$ is surjective, concluding the proof.
\end{proof}

\let\OLDthebibliography\thebibliography 
\renewcommand\thebibliography[1]{
  \OLDthebibliography{#1}
  \setlength{\parskip}{1pt}
  \setlength{\itemsep}{1pt plus 0.3ex}
}

\end{document}